\documentclass[11pt]{article}
\usepackage{amsmath, amssymb, amscd, epsfig, amsthm}
\usepackage{graphicx}
\usepackage{caption, subcaption}
\usepackage{leftidx}
\usepackage{tikz-cd}

\newtheorem{thm}{Theorem}
\newtheorem{prop}[thm]{Proposition}
\newtheorem{lem}[thm]{Lemma}

\newtheorem{cor}[thm]{Corollary}
\newtheorem{rem}[thm]{Remark}
\newtheorem{df}[thm]{Definition}

\renewcommand{\epsilon}{\varepsilon}
\renewcommand{\phi}{\varphi}
\renewcommand{\deg}{\operatorname{deg}}
\renewcommand{\P}{\operatorname{P}}
\newcommand{\Iso}{\operatorname{Iso}}

\newcommand{\BB}{\mathbb}
\newcommand{\g}{\mathfrak}
\newcommand{\separate}{\vskip5pt}

\newcommand{\tr}{\operatorname{Tr}}

\newcommand{\HC}{\BB H_{\BB C}}
\newcommand{\HR}{\BB H_{\BB R}}

\newcommand{\Span}{\text{-}\operatorname{Span}}
\newcommand{\M}{\operatorname{Mx}}

\newcommand{\Res}{\operatorname{Res}}

\usepackage[OT2,T1]{fontenc}
\newcommand\textcyr[1]{{\fontencoding{OT2}\fontfamily{wncyr}\selectfont #1}}
\newcommand{\Zh}{\textit{\textcyr{Zh}}}
\newcommand{\Sh}{\textit{\textcyr{Sh}}}

\begin{document}
\title{\bf Reduction of Symmetry in Quaternionic Analysis and
  Invariant Trilinear Forms}
\author{Igor Frenkel and Matvei Libine}
\maketitle

\begin{abstract}
In our previous papers we repeatedly emphasized the special role in
Quaternionic Analysis of the conformal group $SU(2,2)$ and other real forms
of its complexification $SL(4,\BB C)$.
In particular, the natural product map of the left and right regular functions
into a larger representation that contains the doubly regular functions as
a subquotient is an intertwining operator. In this paper we show, however,
that the spaces of regular and doubly regular functions do not ``interact''
-- there is no invariant trilinear form on the tensor product of these
representations. 

To construct a natural invariant trilinear form, we reduce the conformal
group symmetry to the symplectic subgroup $Sp(4,\BB R)$.
This suggests a new approach to the Quaternionic Analysis in general,
and we make the first steps in this paper.
It turns out that the spaces of regular and doubly regular functions are
still irreducible after the restriction to the symplectic subgroup and have a
composed structure arising from the metaplectic representation of the
double cover of $Sp(4,\BB R)$ -- the metaplectic group.
This also leads us to consider the double covers of the quaternionic spaces
and non-trivial pairings between them. Our study of Quaternionic Analysis
based on the symplectic symmetry group culminates in the construction of
the invariant trilinear forms on the products of spaces of doubly regular
functions and certain counterparts of regular and quasi regular functions.

Additional motivation for constructing invariant trilinear forms comes from
their application to spinor representations of certain quaternionic algebras
based on the doubly regular functions.
The latter can be viewed as the space of solutions of the Maxwell equation,
and their spinor representations are of a great importance to quantum field
theory.
These spinor representations will be the subject of a forthcoming paper.
\end{abstract}

\section{Introduction}

The notions of left and right regular functions form the foundation of
quaternionic analysis; these functions can be viewed as analogues of
holomorphic functions of complex analysis.
We denote by ${\cal V}$ and ${\cal V}'$ the spaces of left and right
regular functions on $\BB H^{\times} = \BB H \setminus \{0\}$
(non-zero quaternions).
These spaces provide unitary representations of the conformal group $SU(2,2)$,
and they belong to the most degenerate series of representations of $SU(2,2)$.
The product of left and right regular regular functions forms a space of
quaternionic valued functions on $\BB H^{\times}$
\begin{equation}  \label{VV'->W}
  {\cal V} \otimes {\cal V}' \to {\cal W}.
\end{equation}
As a representation of the conformal group, ${\cal W}$ contains a subquotient
consisting of the so-called (left and right) doubly regular functions;
these also belong to the most degenerate discrete series.

The multiplication of regular functions \eqref{VV'->W} yields an
$SU(2,2)$-invariant  trilinear form
\begin{equation}  \label{3form-obvious-intro}
  {\cal V} \otimes {\cal W}' \times {\cal V}' \to \BB C,
\end{equation}
where ${\cal W}'$ is the dual of ${\cal W}$, which is another space of
quaternionic valued functions on $\BB H^{\times}$.
As a dual space, ${\cal W}'$ also contains the spaces  of (left and right)
doubly regular functions. A detailed analysis of the structure of ${\cal W}$
and ${\cal W}'$ was performed in \cite{ATMP}.
Unfortunately, the restriction of the trilinear form \eqref{3form-obvious-intro}
to the minimal invariant subspace of ${\cal W}'$ containing the doubly regular
functions yields the zero map. In other words, the spaces of regular and doubly
regular functions do not ``interact''.

In \cite{qreg} we also studied the spaces of left and right quasi (anti) regular
functions ${\cal U}$ and ${\cal U}'$ whose product yields precisely the dual
space to ${\cal W}$:
\begin{equation}  \label{UU'->W'}
  {\cal U} \otimes {\cal U}' \to {\cal W}'.
\end{equation}
Similarly, the multiplication of quasi regular functions \eqref{VV'->W} yields
another $SU(2,2)$-invariant  trilinear form
\begin{equation}  \label{3form-easy-intro}
  {\cal U} \otimes {\cal W} \times {\cal U}' \to \BB C.
\end{equation}
While the quasi regular and doubly regular functions do interact,
any submodule of ${\cal W}$ containing the doubly regular functions
is ``much larger'' than the doubly regular functions themselves
(has higher Gelfand-Kirillov dimension). For this reason, we are not
interested in the trilinear form \eqref{3form-easy-intro} either.

On the other hand, the complex analogues of the above quaternionic spaces
belong to the discrete series of representations of $SU(1,1)$, and there exist
multiplications and trilinear forms which -- unlike the quaternionic case --
provide the ``interactions'' between the representations.
One important application of the complex counterparts of the invariant maps
\eqref{VV'->W}-\eqref{3form-easy-intro} is a construction of the
spinor representations of infinite-dimensional algebras based on the discrete
series of $SU(1,1)$. This leads to the so-called boson-fermion correspondence
in the two-dimensional conformal field theory as a corollary of an isomorphism
of two representations \cite{Fr}.

The most natural way to restore the interaction of the degenerate series
representations of $SU(2,2)$ is to reduce the conformal symmetry group.
In \cite{split} we discussed the relation between different real forms of
quaternions -- the classical quaternions $\BB H$, the split quaternions $\HR$
and the Minkowski space $\BB M$ as well as the hierarchy of the symmetry
groups of these three spaces, which was presented as follows:
\[
\begin{array}{ccccccccc}
SO(5,1) & \: & \: & \: & SO(4,2) & \: & \: & \: & SO(3,3) \\
\: & \diagdown & \: & \diagup & \: & \diagdown & \: & \diagup & \: \\
\: & \: & SO(4,1) & \: & \: & \: & SO(3,2) & \: & \:  \\
\: & \diagup & \: & \diagdown & \: & \diagup & \: & \diagdown & \: \\
SO(4) & \: & \: & \: & SO(3,1) & \: & \: & \: & SO(2,2) \\
\text{\Large$\wr$} & \: & \: & \: & \text{\Large$\wr$} & \: & \: & \: &
\text{\Large$\wr$} \\
\BB H \simeq \BB R^4 & \: & \: & \: & \BB M \simeq \BB R^{3,1} & \: & \: & \: &
\HR \simeq \BB R^{2,2}
\end{array}
\]
(In this diagram some groups are replaced with locally isomorphic ones.)
The groups at the bottom of the diagram are the metric-preserving linear
transformations of the three real forms of the complexified quaternions $\HC$,
while the groups at the top row are the corresponding conformal groups.
The groups in the middle row appear as subgroups of the conformal groups
preserving certain quadratic surfaces.
One of our motivations for this project is a study of the middle row groups
in the quaternionic setting and related representation theory.
In particular, we demonstrate that restricting the conformal group
$SU(2,2)$ -- which is locally isomorphic to $SO(4,2)$ --
to a middle row subgroup
\[
Sp(4,\BB R) \approx SO(3,2) \qquad \text{(local isomorphism)}
\]
allows us to restore the missing interactions between the spaces that include
the doubly regular functions.

Many facts about the restriction of the unitary representations of $SU(2,2)$
to its subgroup $Sp(4,\BB R)$ are well-known, and we want to restate some
of them in the setting of quaternionic analysis.
We also add some new facts that will be used in application to a
representation theoretic approach to the 4-dimensional quantum field theory
in our forthcoming paper \cite{future}.

The first fact that we recall is that the restriction of the degenerate series
of representations of $SU(2,2)$ to $Sp(4,\BB R)$ eliminates the distinction
between the left and right $n$-regular representations. Namely, that there is
an isomorphism of representations of $Sp(4,\BB R)$
\[
{\cal V}^{\pm}_n \simeq {\cal V}'^{\pm}_n, \qquad n=1,2,3,\dots.
\]
Additionally, restricting to $Sp(4,\BB R)$ preserves the irreducibility.
However, these irreducible representations of $Sp(4,\BB R)$ are no longer
the smallest unitary representations of this group (and its cover).
There is a well-known metaplectic representation of the double cover of
$Sp(4, \BB R)$, denoted ${\cal MP}^{\pm}$, where $\pm$ stands for the highest
and lowest weight realization, such that
\[
{\cal MP}^+ \otimes {\cal MP}^+ = \bigoplus_{n \in \BB Z} {\cal V}^+_n, \qquad
{\cal MP}^- \otimes {\cal MP}^- = \bigoplus_{n \in \BB Z} {\cal V}^-_n, \qquad
{\cal V}^{\pm}_n = {\cal V}^{\pm}_{-n},
\]
as representations of $Sp(4,\BB R)$
(see, for example, \cite{FF} and references therein).
This decomposition of the tensor products suggests a composite structure of
the spaces of $n$-regular functions, and, consequently, interactions of
these spaces viewed as representations of $Sp(4,\BB R)$.
In this paper we study the counterparts of the multiplication maps and
trilinear forms \eqref{VV'->W}-\eqref{3form-easy-intro} for the group
$Sp(4,\BB R)$. It is crucial that they all involve the spaces of
doubly regular functions for $SU(2,2)$ restricted to $Sp(4,\BB R)$.

We start with the counterpart of the multiplication \eqref{UU'->W'}, namely
\begin{equation}  \label{UU'->W'-res}
  {\cal U}_{res} \otimes {\cal U}'_{res} \to {\cal W}'_{res},
\end{equation}
where ${\cal U}_{res}$ are the left quasi (anti) regular functions
restricted to the symmetric quaternions
\[
\BB H^{\times}\text{-sym} = \bigl\{ \bigl(\begin{smallmatrix}
  z & it \\ it & \bar z \end{smallmatrix}\bigr);\:
z \in \BB C,\: t \in \BB R \bigr\}
\]
and regarded as representations of $Sp(4,\BB R)$.
Spaces ${\cal U}'_{res}$ and ${\cal W}'_{res}$ are defined similarly.
We show that ${\cal W}'_{res}$ contains the space of doubly regular
functions restricted to $Sp(4,\BB R)$.
In order to derive from \eqref{UU'->W'-res} the counterpart of the
trilinear form \eqref{3form-easy-intro}, we need an explicit construction
of the dual space to ${\cal W}'_{res}$.
This can be done in several steps: First, we construct a new version of
the space of regular functions for the group $Sp(4,\BB R)$, namely
regular functions on $\BB H^{\times}\text{-sym}$
\begin{align*}
{\cal V}_{new} &= \{\text{solutions of $\partial^+f=0$ on
  $\BB H^{\times}\text{-sym}$}\},  \\
{\cal V}'_{new} &=
\{\text{solutions of $g \overleftarrow{\partial^+}=0$ on
  $\BB H^{\times} \text{-sym}$}\},
\end{align*}
where $\partial^+$ is a Dirac-type operator
\[
\partial^+ = \begin{pmatrix}
  \frac{\partial}{\partial \bar z} & \frac{i}2 \frac{\partial}{\partial t} \\
  \frac{i}2 \frac{\partial}{\partial t} & \frac{\partial}{\partial z}
\end{pmatrix},
\]
with new actions of $Sp(4,\BB R)$:
\begin{align*}
\pi^{new}_l(h): \: f(Z) \: &\mapsto \: \bigl( \pi^{new}_l(h)f \bigr)(Z) =
\frac {(cZ+d)^{-1}}{\sqrt{\det(cZ+d)}} \cdot f \bigl( (aZ+b)(cZ+d)^{-1} \bigr),
\\
\pi^{new}_r(h): \: g(Z) \: &\mapsto \: \bigl( \pi^{new}_r(h)g \bigr)(Z) =
g \bigl( (a'-Zc')^{-1}(-b'+Zd') \bigr)
\cdot \frac {(a'-Zc')^{-1}}{\sqrt{\det(a'-Zc')}},
\end{align*}
where $f \in {\cal V}_{new}$, $g \in {\cal V}'_{new}$,
$h = \bigl(\begin{smallmatrix} a' & b' \\ c' & d' \end{smallmatrix}\bigr)
\in Sp(4,\BB R)$ and 
$h^{-1} = \bigl(\begin{smallmatrix} a & b \\ c & d \end{smallmatrix}\bigr)$.
Taking the product of the new left and right regular functions on
$\BB H^{\times}\text{-sym}$, we can construct a counterpart of
\eqref{VV'->W}:
\begin{equation}  \label{VV'->W-new}
  {\cal V}_{new} \otimes {\cal V}'_{new} \to {\cal W}_{new}.
\end{equation}
We show that ${\cal W}_{new}$ still contains the space of doubly regular
functions restricted to $Sp(4,\BB R)$.
Finally, ${\cal W}'_{res}$ and ${\cal W}_{new}$ are almost dual to each other.
To make them precisely dual, we need a double cover of the space of functions on
$\BB H^{\times}\text{-sym}$, which is formed by multiplication by
$\det(Z)^{\frac12}$. As a result, we obtain two pairs of dual spaces:
\begin{align*}
{\cal W}'_{res} \otimes \widetilde{\cal W}_{new} &\to \BB C,  \\
{\cal W}_{new} \otimes \widetilde{\cal W}'_{res} &\to \BB C.
\end{align*}
These dualities combined with the multiplication maps \eqref{UU'->W'-res}
and \eqref{VV'->W-new} yield the counterparts of the invariant trilinear
forms \eqref{3form-obvious-intro} and \eqref{3form-easy-intro}:
\begin{align*}
{\cal V}_{new} \otimes \widetilde{\cal W}'_{res} \otimes {\cal V}'_{new}
&\to \BB C,  \\
{\cal U}_{res} \otimes \widetilde{\cal W}_{new} \otimes {\cal U}'_{res}
&\to \BB C.
\end{align*}

We would like to make a few remarks about the relation of this paper to
the quantum field theory. Although our motivations come from
the development of quaternionic analysis and representation theory, some
relations to the constructions from quantum field theory are impossible to
ignore, and these relations can be beneficial for both mathematicians and
physicists. The restriction of the conformal group of the Minkowski space to
the subgroup $Sp(4,\BB R)$ in the context of quantum electrodynamics
was studied in the '80s -- see \cite{FF} and references therein.
In the past few years there were related studies of the anti-deSitterian
deformations of the Minkowskian representation of the Poincar\'e group.
In particular, it was shown that the unitary representations of $Sp(4,\BB R)$
play an important role in the relation between the two groups,
\cite{EGdOP, GdOP}. In our work we also use extensively non-unitary representations
that might appear in physics in the future. It is interesting to note that the
space of doubly regular functions is precisely the space of solutions of the
Maxwell equations in the vacuum (photons). Thus, the spinor representations
of the algebras containing the doubly regular functions can be regarded as
a photon quantization.
In our forthcoming paper \cite{future} we study the structure of these
algebras, which in turn leads to a generalization of the classical
boson-fermion correspondence to the higher dimensional setting.

Another motivation for studying invariant trilinear forms comes from the
relation with the 4-dimensional quantum field theory, particularly with the
construction of the spinor representation associated to the spaces of the
(left and right) regular functions on $\BB H^{\times}$.
Here is a brief preview of our forthcoming paper \cite{future}.
Let $\{ f_i(Z) \}_{i \in I^{\pm}}$ be a basis of ${\cal V}^{\pm}$, where
$I^{\pm}$ are suitable indexing sets.
And let $\{ g_i(Z) \}_{i \in I^{\mp}}$ be the dual basis of ${\cal V}'^{\mp}$.
Form a vector space ${\cal V} \oplus {\cal V}'$ with basis
$\beta_i \leftrightarrow f_i(Z)$ and $\gamma_i \leftrightarrow g_i(Z)$,
$i \in I = I^- \cup I^+$. Define a symmetric bilinear form on
${\cal V} \oplus {\cal V}'$ by
\[
\langle\beta_i,\gamma_j\rangle = \langle\gamma_j,\beta_i\rangle
= -\tfrac12 \delta_{ij},
\qquad \langle\beta_i,\beta_j\rangle = \langle\gamma_i,\gamma_j\rangle = 0.
\]
Associated to this bilinear form is the (universal) Clifford algebra
$Cl({\cal V} \oplus {\cal V}')$.
It is generated by $1$, $\{\beta_i\}_{i \in I}$ and
$\{\gamma_i\}_{i \in I}$ subject to the relations
\[
\{ \beta_i, \gamma_j \} = \delta_{ij}, \qquad
\{ \beta_i, \beta_j \} = \{ \gamma_i, \gamma_j \} = 0, \qquad
\text{for all $i, j \in I$},
\]
where $\{x,y\}=xy+yx$ is the anticommutator.
This Clifford algebra has a representation in a Grassmann algebra
\[
\Lambda ({\cal V}^- \oplus {\cal V}'^-)
\]
via multiplication and contraction operators. We consider quantum fields
\[
\beta(Z) = \sum_{i \in I} \beta_i g_i(Z), \qquad
\gamma(Z) = \sum_{i \in I} \gamma_i f_i(Z)
\]
(formal series), then, for each $F \in {\cal W}'$, construct
a quadratic element in the Clifford algebra
\begin{equation}  \label{quadratic-elts}
:\langle \beta(\_), F(\_), \gamma(\_) \rangle:,
\end{equation}
where $\langle \_,\_, \_ \rangle$ denotes invariant trilinear form
\eqref{3form-obvious-intro} and the colons indicate what is called the
{\em normal ordering}.
Then the Lie algebra generated by the quadratic elements
\eqref{quadratic-elts} has a natural $SU(2,2)$ invariance.
However -- As was explained earlier -- when we restrict $F$ to the doubly
regular components of ${\cal W}'$, we get zero since the invariant trilinear
form is trivial and the "interaction" between left and right regular and doubly
regular functions disappears! 
In this paper, we construct an invariant trilinear form for a subgroup
$Sp(4,\BB R)$ of $SU(2,2)$, and we show how one can modify the spaces of
regular functions while keeping the spaces of doubly regular functions
to make the trilinear forms nontrivial.
This leads to spinor representations of certain Lie algebras generated by
the doubly regular components and will be studied in the forthcoming paper
\cite{future}.

This paper is organized as follows. In Section \ref{2} we review our notations
and recall basic facts about regular, anti regular and doubly regular functions.
In Section \ref{3} we construct the metaplectic representation of
$\mathfrak{sp}(8,\BB C)$ and decompose it into subspaces isomorphic to
the $n$-regular functions.
As a corollary of this construction, we obtain an isomorphism between
the spaces of the left and right regular functions as
$\mathfrak{sp}(4,\BB C)$-modules (Corollary \ref{left-right-regular-cor}).
We also show that the spaces of $n$-regular functions remain irreducible
after restricting to $\mathfrak{sp}(4,\BB C)$
(Proposition \ref{nreg-irreducibility-prop}).
In Section \ref{4} we start developing quaternionic analysis using our
new approach based on the conformal group symmetry reduction to
$Sp(4,\BB R)$.
Thus, we consider certain representations of $\mathfrak{sp}(4,\BB C)$
in the spaces of scalar valued functions, study their $K$-types and
construct a particular $\mathfrak{sp}(4,\BB C)$-equivariant projection
onto the trivial 1-dimensional component \eqref{M-integral}.
Then in Section \ref{5} we study various explicit realizations of the space
of regular functions restricted to $\mathfrak{sp}(4,\BB C)$.
We also analyze the restrictions of the quasi regular functions and
introduce spaces of solutions of Dirac-type equations on
$\HC^{\times}\text{-sym}$.
In Section \ref{6} we recall the doubly regular functions and their extension
by the biharmonic functions that comes naturally from our study of the
second order pole in quaternionic analysis \cite{ATMP}.
This space is again restricted to $\mathfrak{sp}(4,\BB C)$
-- the result is denoted by ${\cal W}'_{res}$ -- and we study
the $K$-types and decomposition of ${\cal W}'_{res}$ into
irreducible components.
Then we introduce associated spaces ${\cal W}^*_{res}$ and
${\cal W}'_{\frac12}$ that also contain the doubly regular functions as
irreducible components.
Finally, in Section \ref{7} we construct several
$\mathfrak{sp}(4,\BB C)$-invariant trilinear forms involving
the space of doubly regular functions.

The second author was partially supported by Indiana University's
IAS Collaborative Research Award.

\section{Quaternions and Regular Functions}  \label{2}

\subsection{Quaternions}

We continue to use notations established in \cite{FL1}.
In particular, $e_0$, $e_1$, $e_2$, $e_3$ denote the units of the classical
quaternions $\BB H$ corresponding to the more familiar $1$, $i$, $j$, $k$
(we reserve the symbol $i$ for $\sqrt{-1} \in \BB C$).
Thus, $\BB H$ is an algebra over $\BB R$ generated by
$e_0$, $e_1$, $e_2$, $e_3$ with multiplicative structure determined
by the rules
\[
e_0 e_i = e_i e_0 = e_i, \qquad
(e_i)^2 = e_1e_2e_3 = - e_0, \qquad
e_ie_j=-e_ie_j, \qquad 1 \le i< j \le 3,
\]
and the fact that $\BB H$ is a division ring.
Instead of $\BB H$, we usually consider the algebra of complexified
quaternions (also known as biquaternions) $\HC = \BB C \otimes_{\BB R} \BB H$
and write elements of $\HC$ as
\[
Z = z^0e_0 + z^1e_1 + z^2e_2 + z^3e_3, \qquad z^0,z^1,z^2,z^3 \in \BB C,
\]
so that $Z \in \BB H$ if and only if $z^0,z^1,z^2,z^3 \in \BB R$:
\[
\BB H = \{ X = x^0e_0 + x^1e_1 + x^2e_2 + x^3e_3; \: x^0,x^1,x^2,x^3 \in \BB R \}.
\]
For $Z = z^0e_0 + z^1e_1 + z^2e_2 + z^3e_3 \in \HC$, we denote by $Z^+$ its
{\em quaternionic conjugate}:
\[
Z^+ = z^0e_0 - z^1e_1 - z^2e_2 - z^3e_3
\]
(the bar notation is reserved for complex conjugation), and let
\[
N(Z) = ZZ^+ = Z^+Z  \in \BB C,
\]
\[
\HC^{\times} = \{ Z \in \HC ;\: N(Z) \ne 0 \}
= \{\text{invertible complexified quaternions}\}.
\]
We denote by $\BB S$ (respectively $\BB S'$)
the irreducible 2-dimensional left (respectively right) $\HC$-module,
as described in Subsection 2.3 of \cite{FL1}.
The spaces $\BB S$ and $\BB S'$ can be realized as respectively
columns and rows of complex numbers.
Then
\begin{equation*}  
\BB S \otimes \BB S' \simeq \HC.
\end{equation*}
We identify the quaternions $\HC$ with $2 \times 2$ matrices as in
Subsections 2.3, 3.1 of \cite{FL1}:
\begin{equation}  \label{H-matrix-realization}
\HC = \biggl\{ Z = z^0e_0+z^1e_1+z^2e_2+z^3e_3
= \begin{pmatrix} z^0-iz^3 & -iz^1-z^2 \\ -iz^1+z^2 & z^0+iz^3 \end{pmatrix};\:
z^0,z^1,z^2,z^3 \in \BB C \biggr\},
\end{equation}
where
\[
e_0 \longleftrightarrow
\bigl(\begin{smallmatrix} 1 & 0 \\ 0 & 1 \end{smallmatrix}\bigr), \qquad
e_1 \longleftrightarrow
\bigl(\begin{smallmatrix} 0 & -i \\ -i & 0 \end{smallmatrix}\bigr), \qquad
e_2 \longleftrightarrow
\bigl(\begin{smallmatrix} 0 & -1 \\ 1 & 0 \end{smallmatrix}\bigr), \qquad
e_3 \longleftrightarrow
\bigl(\begin{smallmatrix} -i & 0 \\ 0 & i \end{smallmatrix}\bigr).
\]
Then
\[
N(Z) = \det(Z),
\]
and the subspaces of symmetric matrices $\HC \text{-sym}$,
$\HC \text{-sym}^{\times}$ consist of matrices with zero $e_2$-coordinate:
\begin{align}
\HC \text{-sym} &= \{ Z = z^0e_0+z^1e_1+z^3e_3;\: z^0,z^1,z^3 \in \BB C \}
= \bigl\{
\bigl(\begin{smallmatrix} a & b \\ b & d \end{smallmatrix}\bigr);\:
a,b,d \in \BB C \bigr\},   \label{H-sym} \\
\HC^{\times} \text{-sym} &= \bigl\{
Z= \bigl(\begin{smallmatrix} a & b \\ b & d \end{smallmatrix}\bigr);\:
a,b,d \in \BB C,\: N(Z) \ne 0 \bigr\}.   \label{H-sym-x}
\end{align}

For $l=0,\tfrac12,1,\tfrac32,\dots$, we denote by $V_l$ the irreducible
representation of $SU(2)$ (or its Lie algebra) of dimension $2l+1$.
Abusing notation, we also use $V_l$ to denote any irreducible representation
of $U(2)$ of dimension $2l+1$ regardless of the action of the center.

\subsection{Regular Functions}  \label{1-reg-review}

The starting point of quaternionic analysis are the notions of
left and right regular functions.
For a suitable review we refer to Section 2 of \cite{FL1}.
(See also \cite{CSSS} and \cite{Su} for introductions to the subject.)
Recall the differential operators
\begin{align}
\nabla^+ &= e_0 \tfrac{\partial}{\partial z^0}
+ e_1 \tfrac{\partial}{\partial z^1}
+ e_2 \tfrac{\partial}{\partial z^2}
+ e_3 \tfrac{\partial}{\partial z^3}, \\
\nabla &= e_0 \tfrac{\partial}{\partial z^0}
- e_1 \tfrac{\partial}{\partial z^1}
- e_2 \tfrac{\partial}{\partial z^2}
- e_3 \tfrac{\partial}{\partial z^3}, \\
\square = \nabla\nabla^+ = \nabla^+\nabla
&= \tfrac{\partial^2}{(\partial z^0)^2} + \tfrac{\partial^2}{(\partial z^1)^2}
+ \tfrac{\partial^2}{(\partial z^2)^2} + \tfrac{\partial^2}{(\partial z^3)^2}.
\label{Laplacian}
\end{align}
In this paper we are interested in regular functions defined on open subsets of
$\HC$. In this case we require such functions to be holomorphic.

\begin{df}
Let $U$ be an open subset of $\HC$.
A holomorphic function $f: U \to \BB S$ is {\em left regular} if it satisfies
$\nabla^+ f =0$ at all points in $U$.

Similarly, a holomorphic function $g: U \to \BB S'$ is
{\em right regular}\footnote{The arrow over $\nabla^+$ indicates that the
  operator is applied to function $g$ on the right.}
if $g \overleftarrow{\nabla}^+ =0$ at all points in $U$.
\end{df}

Let $\tilde{\cal V}$ and $\tilde{\cal V}'$ denote respectively the
spaces of (holomorphic) left and right anti regular functions on $\HC$,
possibly with singularities.

\begin{thm}  \label{r-action-thm}
\begin{enumerate}
\item
The space $\tilde{\cal V}$ of left regular functions
$\HC \to \BB S$ (possibly with singularities)
is invariant under the following action of $GL(2,\HC)$:
\begin{multline}  \label{pi_l}
\pi_l(h): \: f(Z) \: \mapsto \: \bigl( \pi_l(h)f \bigr)(Z) =
\frac {(cZ+d)^{-1}}{N(cZ+d)} \cdot f \bigl( (aZ+b)(cZ+d)^{-1} \bigr),  \\
h^{-1} = \bigl( \begin{smallmatrix} a & b \\ c & d \end{smallmatrix} \bigr)
\in GL(2,\HC).
\end{multline}
\item
The space $\tilde{\cal V}'$ of right regular functions
$\HC \to \BB S'$ (possibly with singularities)
is invariant under the following action of $GL(2,\HC)$:
\begin{multline}  \label{pi_r}
\pi_r(h): \: g(Z) \: \mapsto \: \bigl( \pi_r(h)g \bigr)(Z) =
g \bigl( (a'-Zc')^{-1}(-b'+Zd') \bigr) \cdot \frac {(a'-Zc')^{-1}}{N(a'-Zc')}, \\
h = \bigl( \begin{smallmatrix} a' & b' \\ c' & d' \end{smallmatrix} \bigr)
\in GL(2,\HC).
\end{multline}
\end{enumerate}
\end{thm}

Differentiating $\pi_l$ and $\pi_r$, we obtain actions of the Lie algebra
$\mathfrak{gl}(2,\HC)$, which we still denote by $\pi_l$ and $\pi_r$
respectively. We spell out these Lie algebra actions
(special case of Lemma 2 in \cite{nreg}):

\begin{lem}  \label{pi-Lie_alg-action}
The Lie algebra action $\pi_l$ of $\mathfrak{gl}(2,\HC)$ on
$\tilde{\cal V}$ is given by
\begin{align*}
\pi_l \bigl( \begin{smallmatrix} A & 0 \\ 0 & 0 \end{smallmatrix} \bigr) &:
f(Z) \mapsto - \tr (AZ \partial) f,  \\
\pi_l \bigl( \begin{smallmatrix} 0 & B \\ 0 & 0 \end{smallmatrix} \bigr) &:
f(Z) \mapsto - \tr (B \partial) f,  \\
\pi_l \bigl( \begin{smallmatrix} 0 & 0 \\ C & 0 \end{smallmatrix} \bigr) &:
f(Z) \mapsto \tr (ZCZ \partial + CZ) f + CZf,  \\
\pi_l \bigl( \begin{smallmatrix} 0 & 0 \\ 0 & D \end{smallmatrix} \bigr) &:
f(Z) \mapsto \tr (ZD \partial +D) f + Df.
\end{align*}

Similarly, the Lie algebra action $\pi_r$ of $\mathfrak{gl}(2,\HC)$ on
$\tilde{\cal V}'$ is given by
\begin{align*}
\pi_r \bigl( \begin{smallmatrix} A & 0 \\ 0 & 0 \end{smallmatrix} \bigr) &:
g(Z) \mapsto - \tr (AZ \partial + A) g - gA,  \\
\pi_r \bigl( \begin{smallmatrix} 0 & B \\ 0 & 0 \end{smallmatrix} \bigr) &:
g(Z) \mapsto - \tr (B \partial) g,  \\
\pi_r \bigl( \begin{smallmatrix} 0 & 0 \\ C & 0 \end{smallmatrix} \bigr) &:
g(Z) \mapsto \tr (ZCZ \partial + ZC) g + gZC,  \\
\pi_r \bigl( \begin{smallmatrix} 0 & 0 \\ 0 & D \end{smallmatrix} \bigr) &:
g(Z) \mapsto \tr (ZD \partial) g.
\end{align*}
\end{lem}

Lemma \ref{pi-Lie_alg-action} implies that the Lie algebra actions
$\pi_l$ and $\pi_r$ preserve the spaces of polynomial regular functions
on $\HC^{\times}$
\begin{align*}
{\cal V} &= \bigl\{
  f \in \BB S \otimes \BB C[z_{11},z_{12},z_{21},z_{22}, N(Z)^{-1}] ;\:
  \nabla^+ f =0 \bigr\} \qquad \text{and}  \\
{\cal V}' &= \bigl\{
g \in \BB S' \otimes \BB C[z_{11},z_{12},z_{21},z_{22}, N(Z)^{-1}] ;\:
g \overleftarrow{\nabla}^+ =0 \bigr\}
\end{align*}
respectively. Then
\[
{\cal V} = {\cal V}^+ \oplus {\cal V}^-
\qquad \text{and} \qquad
{\cal V}' = {\cal V}'^+ \oplus {\cal V}'^-,
\]
where
\begin{align*}
{\cal V}^+ &= \bigl\{
  f \in \BB S \otimes \BB C[z_{11},z_{12},z_{21},z_{22}] ;\:
  \nabla^+ f =0 \bigr\},  \\
{\cal V}^- &= \bigl\{
  f \in \BB S \otimes \BB C[z_{11},z_{12},z_{21},z_{22}, N(Z)^{-1}] ;\:
N(Z)^{-1} \cdot Z^{-1} \cdot f(Z^{-1}) \in {\cal V}^+ \bigr\},  \\
{\cal V}'^+ &= \bigl\{
g \in \BB S' \otimes \BB C[z_{11},z_{12},z_{21},z_{22}] ;\:
g \overleftarrow{\nabla}^+ =0 \bigr\},  \\
{\cal V}'^- &= \bigl\{
g \in \BB S' \otimes \BB C[z_{11},z_{12},z_{21},z_{22}, N(Z)^{-1}] ;\:
N(Z)^{-1} \cdot g(Z^{-1}) \cdot Z^{-1} \in {\cal V}'^+ \bigr\}.
\end{align*}
We refer to ${\cal V}^+$ and ${\cal V}'^+$ as
``functions regular at the origin''.
Similarly, we refer to ${\cal V}^-$ and ${\cal V}'^-$ as
``functions regular at infinity''.

\subsection{Anti Regular Functions}  \label{anti-reg-subsect}

Similarly, one can consider left and right anti regular functions defined on
open subsets of $\BB H$ and $\HC$.

\begin{df}  \label{r-definition}
Let $U$ be an open subset of $\BB H$.
A ${\cal C}^1$-function $f: U \to \BB S$ is
{\em left anti regular} if it satisfies
\[
e_0 \tfrac{\partial f}{\partial x^0}
+ e_1 \tfrac{\partial f}{\partial x^1}
+ e_2 \tfrac{\partial f}{\partial x^2}
+ e_3 \tfrac{\partial f}{\partial x^3}
=0 \qquad \text{at all points in $U$}.
\]

Similarly, a ${\cal C}^1$-function $g: U \to \BB S'$ is
{\em right anti regular} if
\[
\tfrac{\partial g}{\partial x^0} e_0
+ \tfrac{\partial g}{\partial x^1} e_1
+ \tfrac{\partial g}{\partial x^2} e_2
+ \tfrac{\partial g}{\partial x^3} e_3
=0 \qquad \text{at all points in $U$}.
\]
\end{df}

We also can talk about regular functions defined on open subsets of
$\HC$. In this case we require such functions to be holomorphic.

\begin{df}
Let $U$ be an open subset of $\HC$.
A holomorphic function $f: U \to \BB S$ is
{\em left anti regular} if it satisfies
$\nabla f =0$ at all points in $U$.

Similarly, a holomorphic function $g: U \to \BB S'$ is
{\em right anti regular} if
$g \overleftarrow{\nabla} =0$ at all points in $U$.
\end{df}

It is clear from \eqref{Laplacian} that anti regular functions are harmonic,
i.e. annihilated by $\square$.
One way to construct left anti regular functions is to start with a harmonic
function $\phi: \BB H \to \BB S$, then $\nabla^+ \phi$ is left anti regular.
Similarly, if $\phi: \BB H \to \BB S'$ is harmonic, then
$\phi \overleftarrow{\nabla^+}$ is right anti regular.

Let $\tilde{\cal V}_a$ and $\tilde{\cal V}'_a$ denote respectively the
spaces of (holomorphic) left and right anti regular functions on $\HC$,
possibly with singularities.

\begin{thm}  \label{ra-action-thm}
\begin{enumerate}
\item
The space $\tilde{\cal V}_a$ of left anti regular functions
$\HC \to \BB S$ (possibly with singularities)
is invariant under the following action of $GL(2,\HC)$:
\begin{multline}  \label{pi_la}
\pi_{la}(h): \: f(Z) \: \mapsto \: \bigl( \pi_{la}(h)f \bigr)(Z) =
\frac{a'-Zc'}{N(a'-Zc')^2} \cdot f \bigl( (a'-Zc')^{-1}(-b'+Zd') \bigr),  \\
h = \bigl( \begin{smallmatrix} a' & b' \\ c' & d' \end{smallmatrix} \bigr)
\in GL(2,\HC).
\end{multline}
\item
The space $\tilde{\cal V}'_a$ of right anti regular functions
$\HC \to \BB S'$ (possibly with singularities)
is invariant under the following action of $GL(2,\HC)$:
\begin{multline}  \label{pi_ra}
\pi_{ra}(h): \: g(Z) \: \mapsto \: \bigl( \pi_{ra}(h)g \bigr)(Z) =
g \bigl( (aZ+b)(cZ+d)^{-1} \bigr) \cdot \frac{cZ+d}{N(cZ+d)^2},  \\
h^{-1} = \bigl( \begin{smallmatrix} a & b \\ c & d \end{smallmatrix} \bigr)
\in GL(2,\HC).
\end{multline}
\end{enumerate}
\end{thm}

Differentiating $\pi_{la}$ and $\pi_{ra}$, we obtain actions of the Lie algebra
$\mathfrak{gl}(2,\HC)$, which we still denote by $\pi_{la}$ and $\pi_{ra}$
respectively. We spell out these Lie algebra actions:

\begin{lem}  \label{pi-a-Lie_alg-action}
The Lie algebra action $\pi_{la}$ of $\mathfrak{gl}(2,\HC)$ on
$\tilde{\cal V}_a$ is given by
\begin{align*}
\pi_{la} \bigl( \begin{smallmatrix} A & 0 \\ 0 & 0 \end{smallmatrix} \bigr) &:
f(Z) \mapsto - \tr (AZ \partial + 2A) f + Af,  \\
\pi_{la} \bigl( \begin{smallmatrix} 0 & B \\ 0 & 0 \end{smallmatrix} \bigr) &:
f(Z) \mapsto - \tr (B \partial) f,  \\
\pi_{la} \bigl( \begin{smallmatrix} 0 & 0 \\ C & 0 \end{smallmatrix} \bigr) &:
f(Z) \mapsto \tr (ZCZ \partial + 2ZC) f - ZCf,  \\
\pi_{la} \bigl( \begin{smallmatrix} 0 & 0 \\ 0 & D \end{smallmatrix} \bigr) &:
f(Z) \mapsto \tr (ZD \partial) f.
\end{align*}

Similarly, the Lie algebra action $\pi_{ra}$ of $\mathfrak{gl}(2,\HC)$ on
$\tilde{\cal V}'_a$ is given by
\begin{align*}
\pi_{ra} \bigl( \begin{smallmatrix} A & 0 \\ 0 & 0 \end{smallmatrix} \bigr) &:
g(Z) \mapsto - \tr (AZ \partial) g,  \\
\pi_{ra} \bigl( \begin{smallmatrix} 0 & B \\ 0 & 0 \end{smallmatrix} \bigr) &:
g(Z) \mapsto - \tr (B \partial) g,  \\
\pi_{ra} \bigl( \begin{smallmatrix} 0 & 0 \\ C & 0 \end{smallmatrix} \bigr) &:
g(Z) \mapsto \tr (ZCZ \partial + 2CZ) g - gCZ,  \\
\pi_{ra} \bigl( \begin{smallmatrix} 0 & 0 \\ 0 & D \end{smallmatrix} \bigr) &:
g(Z) \mapsto \tr (ZD \partial + 2D) g - gD.
\end{align*}
\end{lem}

\begin{prop}  \label{reg-anti-reg-iso-prop}
  We have vector space isomorphisms
  $(\pi_l, \tilde {\cal V}) \simeq (\pi_{ra}, \tilde{\cal V}'_a)$ and
  $(\pi_r, \tilde {\cal V}') \simeq (\pi_{la}, \tilde{\cal V}_a)$ that
  intertwine the actions of $GL(2,\HC)$.
\end{prop}

\begin{proof}
Recall the $\BB C$-linear maps $\BB S \to \BB S'$ and $\BB S' \to \BB S$:
\begin{equation}  \label{dagger}
\begin{pmatrix} s_1 \\ s_2 \end{pmatrix}^{\dagger} = (s_2, -s_1),
\qquad
(s'_1, s'_2)^{\dagger} = \begin{pmatrix} -s'_2 \\ s'_1 \end{pmatrix},
\qquad
\begin{pmatrix} s_1 \\ s_2 \end{pmatrix} \in \BB S ,\:
(s'_1, s'_2) \in \BB S',
\end{equation}
introduced in Subsection 4.3 in \cite{ATMP}.
These maps are similar to the quaternionic conjugation:
\[
(s^{\dagger})^{\dagger}=s, \qquad (s'^{\dagger})^{\dagger}=s', \qquad
(Zs)^{\dagger} = s^{\dagger} Z^+, \qquad
(s'Z)^{\dagger} = Z^+ s'^{\dagger},
\]
for all $Z \in \HC,$ $s \in \BB S$, $s' \in \BB S'$.

Then $f: \HC \to \BB S$ is left regular if and only if
$f^{\dagger} : \HC \to \BB S'$ is right anti regular, and
$g: \HC \to \BB S'$ is right regular if and only if
$g^{\dagger} : \HC \to \BB S$ is left anti regular.
It is also clear that these maps intertwine the actions of $GL(2,\HC)$.  
\end{proof}

As in the case of regular functions, by Lemma \ref{pi-Lie_alg-action},
the Lie algebra actions $\pi_{la}$ and $\pi_{ra}$ preserve the spaces of
polynomial anti regular functions on $\HC^{\times}$
\begin{align*}
{\cal V}_a &= \bigl\{
  f \in \BB S \otimes \BB C[z_{11},z_{12},z_{21},z_{22}, N(Z)^{-1}] ;\:
  \nabla f =0 \bigr\} \qquad \text{and}  \\
{\cal V}'_a &= \bigl\{
g \in \BB S' \otimes \BB C[z_{11},z_{12},z_{21},z_{22}, N(Z)^{-1}] ;\:
g \overleftarrow{\nabla} =0 \bigr\}
\end{align*}
respectively. Then
\[
{\cal V}_a = {\cal V}^+_a \oplus {\cal V}^-_a
\qquad \text{and} \qquad
{\cal V}'_a = {\cal V}'^+_a \oplus {\cal V}'^-_a,
\]
where
\begin{align*}
{\cal V}^+_a &= \bigl\{
  f \in \BB S \otimes \BB C[z_{11},z_{12},z_{21},z_{22}] ;\:
  \nabla f =0 \bigr\},  \\
{\cal V}^-_a &= \bigl\{
  f \in \BB S \otimes \BB C[z_{11},z_{12},z_{21},z_{22}, N(Z)^{-1}] ;\:
N(Z)^{-2} \cdot Z \cdot f(Z^{-1}) \in {\cal V}^+_a \bigr\},  \\
{\cal V}'^+_a &= \bigl\{
g \in \BB S' \otimes \BB C[z_{11},z_{12},z_{21},z_{22}] ;\:
g \overleftarrow{\nabla} =0 \bigr\},  \\
{\cal V}'^-_a &= \bigl\{
g \in \BB S' \otimes \BB C[z_{11},z_{12},z_{21},z_{22}, N(Z)^{-1}] ;\:
N(Z)^{-2} \cdot g(Z^{-1}) \cdot Z \in {\cal V}'^+_a \bigr\}.
\end{align*}
As usual, we refer to ${\cal V}^+_a$ and ${\cal V}'^+_a$ as
``functions regular at the origin''.
Similarly, we refer to ${\cal V}^-_a$ and ${\cal V}'^-_a$ as
``functions regular at infinity''.
The isomorphisms from Proposition \ref{reg-anti-reg-iso-prop}
restrict to isomorphisms of $\mathfrak{gl}(2,\HC)$-modules
\[
(\pi_l, {\cal V}^+) \simeq (\pi_{ra}, {\cal V}'^+_a), \quad
(\pi_l, {\cal V}^-) \simeq (\pi_{ra}, {\cal V}'^-_a), \quad
(\pi_r, {\cal V}'^+) \simeq (\pi_{la}, {\cal V}^+_a), \quad
(\pi_r, {\cal V}'^-) \simeq (\pi_{la}, {\cal V}^-_a).
\]

\subsection{Doubly Regular Functions}

Left and right doubly regular functions were introduced in \cite{ATMP}
and generalized further to $n$-regular functions in \cite{nreg}.
The left doubly regular functions have values in $\BB S \odot \BB S$
-- the symmetric part of the tensor product $\BB S \otimes \BB S$.
Similarly, the right doubly regular functions have values in
$\BB S' \odot \BB S'$.
Following \cite{ATMP}, we introduce four first order differential operators
\begin{align*}
\nabla^+ \otimes 1 &= (e_0 \otimes 1) \frac{\partial}{\partial z^0}
+ (e_1 \otimes 1) \frac{\partial}{\partial z^1}
+ (e_2 \otimes 1) \frac{\partial}{\partial z^2}
+ (e_3 \otimes 1) \frac{\partial}{\partial z^3}, \\
1 \otimes \nabla^+ &= (1 \otimes e_0) \frac{\partial}{\partial z^0}
+ (1 \otimes e_1) \frac{\partial}{\partial z^1}
+ (1 \otimes e_2) \frac{\partial}{\partial z^2}
+ (1 \otimes e_3) \frac{\partial}{\partial z^3}, \\
\nabla \otimes 1 &= (e_0 \otimes 1) \frac{\partial}{\partial z^0}
- (e_1 \otimes 1) \frac{\partial}{\partial z^1}
- (e_2 \otimes 1) \frac{\partial}{\partial z^2}
- (e_3 \otimes 1) \frac{\partial}{\partial z^3}, \\
1 \otimes \nabla &= (1 \otimes e_0) \frac{\partial}{\partial z^0}
- (1 \otimes e_1) \frac{\partial}{\partial z^1}
- (1 \otimes e_2) \frac{\partial}{\partial z^2}
- (1 \otimes e_3) \frac{\partial}{\partial z^3},
\end{align*}
which can be applied to functions with values in
$\BB S \otimes \BB S$ or $\BB S' \otimes \BB S'$ as follows.
If $U$ is an open subset of $\HC$ and $f: U \to \BB S \otimes \BB S$
is a holomorphic function, then these operators can be applied to $f$
on the left. For example,
$$
(\nabla^+ \otimes 1) f = (e_0 \otimes 1) \frac{\partial f}{\partial z^0}
+ (e_1 \otimes 1) \frac{\partial f}{\partial z^1}
+ (e_2 \otimes 1) \frac{\partial f}{\partial z^2}
+ (e_3 \otimes 1) \frac{\partial f}{\partial z^3}.
$$
Similarly, these operators can be applied on the right to
holomorphic functions $g: U \to \BB S' \otimes \BB S'$;
we often indicate this with an arrow above the operator. For example,
$$
g (\overleftarrow{\nabla^+ \otimes 1})
= \frac{\partial g}{\partial z^0} (e_0 \otimes 1)
+ \frac{\partial g}{\partial z^1} (e_1 \otimes 1)
+ \frac{\partial g}{\partial z^2} (e_2 \otimes 1)
+ \frac{\partial g}{\partial z^3} (e_3 \otimes 1).
$$

\begin{df}
Let $U$ be an open subset of $\HC$.
A holomorphic function $f: U \to \BB S \odot \BB S$ is
{\em doubly left regular} if it satisfies $(\nabla^+ \otimes 1) f =0$
and $(1 \otimes \nabla^+) f =0$ for all points in $U$.

Similarly, a holomorphic function $g: U \to \BB S' \odot \BB S'$ is
{\em doubly right regular} if
$g (\overleftarrow{\nabla^+ \otimes 1}) =0$ and
$g (\overleftarrow{1 \otimes \nabla^+}) =0$ for all points in $U$.
\end{df}

It is easy to see that doubly left and right regular functions are harmonic.
Let ${\cal DR}$ and ${\cal DR}'$ denote respectively the spaces of
(holomorphic) doubly left and right regular functions on $\HC$,
possibly with singularities.

\begin{thm} [Theorem 3 in \cite{ATMP}]  \label{dr-action-thm}
\begin{enumerate}
\item
The space ${\cal DR}$ of doubly left regular functions
$\HC \to \BB S \odot \BB S$ (possibly with singularities)
is invariant under the following action of $GL(2,\HC)$:
\begin{multline*}
\pi_{dl}(h): \: f(Z) \: \mapsto \: \bigl( \pi_{dl}(h)f \bigr)(Z) =
\frac{(cZ+d)^{-1} \otimes (cZ+d)^{-1}}{N(cZ+d)}
\cdot f\bigl( (aZ+b)(cZ+d)^{-1} \bigr),  \\
h^{-1} = \bigl( \begin{smallmatrix} a & b \\ c & d \end{smallmatrix} \bigr)
\in GL(2,\HC).
\end{multline*}
\item
The space ${\cal DR}'$ of doubly right regular functions
$\HC \to \BB S' \odot \BB S'$ (possibly with singularities)
is invariant under the following action of $GL(2,\HC)$:
\begin{multline*}
\pi_{dr}(h): \: g(Z) \: \mapsto \: \bigl( \pi_{dr}(h)g \bigr)(Z) =
g \bigl( (a'-Zc')^{-1}(-b'+Zd') \bigr)
\cdot \frac{(a'-Zc')^{-1} \otimes (a'-Zc')^{-1}}{N(a'-Zc')}, \\
h = \bigl( \begin{smallmatrix} a' & b' \\ c' & d' \end{smallmatrix} \bigr)
\in GL(2,\HC).
\end{multline*}
\end{enumerate}
\end{thm}

Differentiating $\pi_{dl}$ and $\pi_{dr}$, we obtain actions of the Lie algebra
$\mathfrak{gl}(2,\HC)$, which we still denote by $\pi_{dl}$ and $\pi_{dr}$
respectively. We spell out these Lie algebra actions
(Lemma 4 in \cite{ATMP}):

\begin{lem}  \label{dr-action-lem}
The Lie algebra action $\pi_{dl}$ of $\mathfrak{gl}(2,\HC)$ on ${\cal DR}$
is given by
\begin{align*}
\pi_{dl} \bigl( \begin{smallmatrix} A & 0 \\ 0 & 0 \end{smallmatrix} \bigr)
&: f(Z) \mapsto - \tr (AZ \partial) f,  \\
\pi_{dl} \bigl( \begin{smallmatrix} 0 & B \\ 0 & 0 \end{smallmatrix} \bigr)
&: f(Z) \mapsto - \tr (B \partial) f,  \\
\pi_{dl} \bigl( \begin{smallmatrix} 0 & 0 \\ C & 0 \end{smallmatrix} \bigr)
&: f(Z) \mapsto \tr (ZCZ \partial +CZ) f + (CZ \otimes 1 + 1 \otimes CZ) f,  \\
\pi_{dl} \bigl( \begin{smallmatrix} 0 & 0 \\ 0 & D \end{smallmatrix} \bigr)
&: f(Z) \mapsto \tr (ZD \partial +D) f + (D \otimes 1 + 1 \otimes D) f.
\end{align*}

Similarly, the Lie algebra action $\pi_{dr}$ of $\mathfrak{gl}(2,\HC)$ on
${\cal DR}'$ is given by
\begin{align*}
\pi_{dr} \bigl( \begin{smallmatrix} A & 0 \\ 0 & 0 \end{smallmatrix} \bigr)
&: g(Z) \mapsto - \tr (AZ \partial +A) g - g (A \otimes 1 + 1 \otimes A),  \\
\pi_{dr} \bigl( \begin{smallmatrix} 0 & B \\ 0 & 0 \end{smallmatrix} \bigr)
&: g(Z) \mapsto - \tr (B \partial) g,  \\
\pi_{dr} \bigl( \begin{smallmatrix} 0 & 0 \\ C & 0 \end{smallmatrix} \bigr)
&: g(Z) \mapsto \tr (ZCZ \partial +ZC) g + g(ZC \otimes 1 + 1 \otimes ZC),  \\
\pi_{dr} \bigl( \begin{smallmatrix} 0 & 0 \\ 0 & D \end{smallmatrix} \bigr)
&: g(Z) \mapsto \tr (ZD \partial) g.
\end{align*}
\end{lem}

Lemma \ref{dr-action-lem} implies that the Lie algebra actions
$\pi_{dl}$ and $\pi_{dr}$ preserve the spaces of polynomial doubly regular
functions on $\HC^{\times}$
\begin{align*}
{\cal F} &= \bigl\{ f \in (\BB S \odot \BB S) \otimes
\BB C[z_{11},z_{12},z_{21},z_{22}, N(Z)^{-1}] ;\:
(\nabla^+ \otimes 1) f = (1 \otimes \nabla^+) f = 0 \bigr\} \qquad \text{and}  \\
{\cal G} &= \bigl\{ g \in (\BB S' \odot \BB S') \otimes
\BB C[z_{11},z_{12},z_{21},z_{22}, N(Z)^{-1}] ;\:
g (\overleftarrow{\nabla^+ \otimes 1})
= g (\overleftarrow{1 \otimes \nabla^+}) = 0 \bigr\}
\end{align*}
respectively. Then
\[
{\cal F} = {\cal F}^+ \oplus {\cal F}^-
\qquad \text{and} \qquad
{\cal G} = {\cal G}^+ \oplus {\cal G}^-,
\]
where
\begin{align*}
{\cal F}^+ &= \bigl\{ f \in (\BB S \odot \BB S) \otimes
\BB C[z_{11},z_{12},z_{21},z_{22}] ;\:
(\nabla^+ \otimes 1) f = (1 \otimes \nabla^+) f = 0 \bigr\},  \\
{\cal F}^- &= \bigl\{ f \in (\BB S \odot \BB S) \otimes
\BB C[z_{11},z_{12},z_{21},z_{22}, N(Z)^{-1}] ;\:
\tfrac{Z^{-1} \otimes Z^{-1}}{N(Z)} \cdot f(Z^{-1}) \in {\cal F}^+ \bigr\},  \\
{\cal G}^+ &= \bigl\{ g \in (\BB S' \odot \BB S') \otimes
\BB C[z_{11},z_{12},z_{21},z_{22}] ;\:
g (\overleftarrow{\nabla^+ \otimes 1})
= g (\overleftarrow{1 \otimes \nabla^+}) = 0 \bigr\},  \\
{\cal G}^- &= \bigl\{ g \in (\BB S' \odot \BB S') \otimes
\BB C[z_{11},z_{12},z_{21},z_{22}, N(Z)^{-1}] ;\:
g(Z^{-1}) \cdot \tfrac{Z^{-1} \otimes Z^{-1}}{N(Z)} \in {\cal G}^+ \bigr\}.
\end{align*}
As usual, we refer to ${\cal F}^+$ and ${\cal G}'^+$ as
``functions regular at the origin''.
Similarly, we refer to ${\cal F}^-$ and ${\cal F}'^-$ as
``functions regular at infinity''.

It is shown in \cite{ATMP} and summarized in
Subsection \ref{IrredDecom-subsection}
that the spaces of doubly regular functions  $(\pi_{dl},{\cal F}^{\pm})$ and
$(\pi_{dr},{\cal G}^{\pm})$ are irreducible and appear as irreducible components
of $(\rho_2, {\cal W})$ and its dual space $(\rho'_2, {\cal W}')$.

\section{Metaplectic Representation of $\mathfrak{sp}(4r,\BB R)$}  \label{3}

\subsection{Realization of $\mathfrak{sp}(4r,\BB R)$ Using Quadratic Operators}

We fix a positive integer $r$ and consider the following operators acting
on the space of Laurent polynomials in $2r$ variables
\[
W = \BB C \bigl[ w_1,\dots,w_r,w^*_{r+1},\dots,w^*_{2r},
  w_1^{-1},\dots,w_r^{-1},(w^*_{r+1})^{-1},\dots,(w^*_{2r})^{-1} \bigr].
\]
First, we have the {\em creation operators}
\[
a_1,\dots,a_r,a^*_{r+1},\dots,a^*_{2r}
\]
that correspond to multiplication by $w_1,\dots,w_r,w^*_{r+1},\dots,w^*_{2r}$
respectively. Then we have the {\em annihilation operators}
\[
a^*_1,\dots,a^*_r,a_{r+1},\dots,a_{2r}
\]
that correspond to differentiation with respect to
$w_1,\dots,w_r,w^*_{r+1},\dots,w^*_{2r}$ respectively.
These operators satisfy the commutator relations of generators of a
Heisenberg algebra:
\begin{equation}  \label{Heisenberg-rels}
[a_k,a_l]=0, \quad [a^*_k,a^*_l]=0, \quad
[a_k,a^*_l] = \begin{cases} - \delta_{kl} & \text{if $k=1,\dots,r$},\\
  \delta_{kl} & \text{if $k=r+1,\dots,2r$}, \end{cases}
\qquad 1 \le k,l \le 2r.
\end{equation}

The Lie algebra $\mathfrak{sp}(4n,\BB R)$ will be constructed using operators
on $W$ that are quadratic elements of the form
\begin{equation}  \label{quadratic-operators}
a_ka_l, \quad a^*_ka^*_l, \quad :a_ka^*_l: = \tfrac12 (a_ka^*_l+a^*_la_k),
\qquad 1 \le k,l \le 2r.
\end{equation}
Note that $:a_ka^*_l: = a_ka^*_l = a^*_la_k$ if $k \ne l$ and
$:a_ka^*_k: = \tfrac12 (a_ka^*_k+a^*_ka_k)$.

We introduce notations
\[
\pm_j = \begin{cases} + & \text{if $j=1,\dots,r$}, \\
  - & \text{if $j=r+1,\dots,2r$}, \end{cases} \qquad
\mp_j = \begin{cases} - & \text{if $j=1,\dots,r$}, \\
  + & \text{if $j=r+1,\dots,2r$}. \end{cases}
\]

\begin{lem}  \label{commutator-lem1}
We have the following commutation relations between these quadratic operators
and $a_1,\dots,a_{2r},a^*_1,\dots,a^*_{2r}$:
\[
[:a_ka^*_l:, a_m] = \begin{cases} 0  & \text{if $m \ne l$}, \\
  \pm_l a_k & \text{if $m=l$}, \end{cases}
\qquad
[:a_ka^*_l:, a^*_m] = \begin{cases} 0  & \text{if $m \ne k$}, \\
  \mp_k a^*_l & \text{if $m=k$}, \end{cases}
\]
\[
[a_ka_l, a^*_m] = \begin{cases} 0  & \text{if $m \ne k,l$}, \\
  \mp_k a_l & \text{if $k \ne l$ and $m=k$}, \\
  \mp_l a_k & \text{if $k \ne l$ and $m=l$}, \\
  \mp_k 2a_k & \text{if $k=l=m$}, \end{cases}
\qquad [a_ka_l, a_m] = 0,
\]
\[
[a^*_ka^*_l, a_m] = \begin{cases} 0  & \text{if $m \ne k,l$}, \\
  \pm_k a^*_l & \text{if $k \ne l$ and $m=k$}, \\
  \pm_l a^*_k & \text{if $k \ne l$ and $m=l$}, \\
  \pm_k 2a^*_k & \text{if $k=l=m$}, \end{cases}
\qquad [a^*_ka^*_l, a^*_m] = 0,
\]
$1 \le k,l,m \le 2r$.
\end{lem}


From these commutation relations we can compute further commutators.

\begin{lem}  \label{commutator-lem2}
We have the following commutation relations between these quadratic operators:
\begin{align*}
[:a_ja^*_k:, :a_la^*_m:] &= \begin{cases} 0  & \text{if $j=k$ and $l=m$}, \\
  0  & \text{if $j \ne m$ and $k \ne l$}, \\
  \pm_j :a_ja^*_j: \mp_k :a_ka^*_k:  & \text{if $j \ne k$, $j=m$ and $k=l$}, \\
  \mp_j a_la^*_k & \text{if $j=m$ and $k \ne l$}, \\
  \pm_k a_ja^*_m & \text{if $j \ne m$ and $k=l$}, \end{cases}  \\
[:a_ja^*_k:, a_la_m] &= \begin{cases} 0  & \text{if $k \ne l,m$}, \\
  \pm_k a_ja_m  & \text{if $k=l$ and $l \ne m$}, \\
  \pm_k a_ja_l & \text{if $k=m$ and $l \ne m$}, \\
  \pm_k 2 a_ja_k & \text{if $k=l=m$}, \end{cases}  \\
[:a_ja^*_k:, a^*_la^*_m] &= \begin{cases} 0  & \text{if $j \ne l,m$}, \\
  \mp_j a^*_ka^*_m  & \text{if $j=l$ and $l \ne m$}, \\
  \mp_j a^*_ka^*_l & \text{if $j=m$ and $l \ne m$}, \\
  \mp_j 2 a^*_ja^*_k & \text{if $j=l=m$}, \end{cases}  \\
[a_ja_k, a^*_la^*_m] &= \begin{cases} 0 & \text{if $j \ne l,m$ and $k \ne l,m$},\\
  \mp_j a_ka^*_m  & \text{if $j=l$, $j \ne k,m$ and $k \ne m$}, \\
  \mp_j a_ka^*_l & \text{if $j=m$ and $j \ne k,l$ and $k \ne m$}, \\
  \mp_j 2 a_ja^*_k & \text{if $j=l=m$ and $j \ne k$}, \\
  \mp_k a_ja^*_m  & \text{if $k=l$, $k \ne j,m$ and $j \ne m$}, \\
  \mp_k a_ja^*_l & \text{if $k=m$ and $k \ne j,l$ and $j \ne m$}, \\
  \mp_k 2 a_ja^*_k & \text{if $k=l=m$ and $j \ne k$}, \\
  \mp_k 2 a_ja^*_m & \text{if $j=k=l$ and $m \ne j$}, \\
  \mp_k 2 a_ja^*_l & \text{if $j=k=m$ and $l \ne j$}, \\
  \mp_k :a_ja^*_j: \mp_j :a_ka^*_k: & \text{if $j=l$, $k=m$ and $j \ne k$}, \\
  \mp_k :a_ja^*_j: \mp_j :a_ka^*_k: & \text{if $j=m$, $k=l$ and $j \ne k$}, \\
  \mp_j 4 :a_ja^*_j: & \text{if $j=k=l=m$}, \end{cases}  \\
[a_ja_k, a_la_m] &= [a^*_ja^*_k, a^*_la^*_m] = 0,
\end{align*}
$1 \le j,k,l,m \le 2r$.
\end{lem}



Using quadratic operators \eqref{quadratic-operators} we can construct
various Lie algebras.

\begin{prop}  \label{iso-prop}
\begin{enumerate}
\item
Vector space
$\mathfrak{g}_1 = \BB R \Span \{:a_ka^*_l: ;\: 1 \le k,l \le 2r \}$
  is a real Lie algebra isomorphic to $\mathfrak{gl}(2r,\BB R)$;
\item
  Vector space
  $\mathfrak{g}_2 = \BB C \Span \{:a_ka^*_l: ;\: 1 \le k,l \le 2r \}$
  is a complex Lie algebra isomorphic to $\mathfrak{gl}(2r,\BB C)$;
\item
  Vector spaces
  $\mathfrak{g}_3 =
  \BB R \Span \{i:a_ja^*_j:,\: a_ka^*_l-a_la^*_k, \: ia_ka^*_l+ia_la^*_k ;\:
  1 \le j, k,l \le r,\: k<l \}$ and \\
  $\mathfrak{g}'_3 =
  \BB R \Span \{i:a_ja^*_j:,\: a_ka^*_l-a_la^*_k, \: ia_ka^*_l+ia_la^*_k ;\:
  r+1 \le j, k,l \le 2r,\: k<l \}$
  are real Lie algebras isomorphic to $\mathfrak{u}(r)$;
\item
  Vector space
  $\mathfrak{g}_4 =
  \BB R \Span \{i:a_ja^*_j:,\: a_ka^*_l-a_la^*_k, \: ia_ka^*_l+ia_la^*_k ;\:
  1 \le j, k,l \le 2r,\: k<l \}$
  is a real Lie algebra isomorphic to $\mathfrak{u}(r,r)$;  
\item
  Vector space
  $\mathfrak{g}_5 =
  \BB R \Span \{:a_ka^*_l:+:a_{l+r}a^*_{k+r}:,\: a_ka^*_{l+r}+a_la^*_{k+r}, \:
  a_{k+r}a^*_l+a_{l+r}a^*_k ;\: 1 \le k,l \le r \}$
  is a real Lie algebra isomorphic to $\mathfrak{sp}(2r,\BB R)$;
\item
  Vector space
  $\mathfrak{g}_6 =
  \BB C \Span \{:a_ka^*_l:+:a_{l+r}a^*_{k+r}:,\: a_ka^*_{l+r}+a_la^*_{k+r}, \:
  a_{k+r}a^*_l+a_{l+r}a^*_k ;\: 1 \le k,l \le r \}$
  is a complex Lie algebra isomorphic to $\mathfrak{sp}(2r,\BB C)$;
\item
  Vector space
  $\mathfrak{g}_7 =
  \BB R \Span \{:a_ka^*_l:,\: a_ka_l,\: a^*_ka^*_l ;\: 1 \le k,l \le 2r \}$
  is a real Lie algebra isomorphic to $\mathfrak{sp}(4r,\BB R)$;
\item
  Vector space
  $\mathfrak{g}_8 =
  \BB C \Span \{:a_ka^*_l:,\: a_ka_l,\: a^*_ka^*_l ;\: 1 \le k,l \le 2r \}$
  is a complex Lie algebra isomorphic to $\mathfrak{sp}(4r,\BB C)$.
\end{enumerate}
\end{prop}

\begin{proof}
\begin{enumerate}
\item
  From Lemma \ref{commutator-lem2} we see that $\mathfrak{g}_1$
  is a real Lie algebra. This Lie algebra acts on the space of
  all linear operators on $W$ by commutators.
  Lemma \ref{commutator-lem1} shows that $\mathfrak{g}_1$ preserves
  a real vector subspace $U_1=\BB R \Span\{a^*_1,\dots,a^*_{2r}\}$.
  Thus, we have a Lie algebra homomorphism
  $\mathfrak{g}_1 \to \mathfrak{gl}(U_1)$.
  Moreover, this homomorphism sends a generator $:a_ka^*_l:$ into a matrix
  with $\mp_k 1$ in the $(l,k)$-entry and zero elsewhere. This proves that\footnote{One can also identify $\mathfrak{g}_1$ with endomorphisms of
      $U'_1=\BB R \Span\{a_1,\dots,a_{2r}\}$. The resulting isomorphism
      $\mathfrak{g}_1 \simeq \mathfrak{gl}(U'_1) \simeq \mathfrak{gl}(2r,\BB R)$
      is related to the previous one by taking the negative transpose of
      matrices.}
  $\mathfrak{g}_1 \simeq \mathfrak{gl}(U_1) \simeq \mathfrak{gl}(2r,\BB R)$.
      
\item
  Follows immediately from the first part, since
  $\mathfrak{g}_2 = \BB C \otimes \mathfrak{g}_1$.

\item
  Let $U_2=\BB C \Span\{a^*_1,\dots,a^*_{2r}\} = \BB C \otimes U_1$.
  By Lemma \ref{commutator-lem1}, we have an isomorphism
  $\mathfrak{g}_2 \simeq \mathfrak{gl}(U_2) \simeq \mathfrak{gl}(2r,\BB C)$.
  Then $\mathfrak{g}_3$ and $\mathfrak{g}'_3$ are real subspaces of
  $\mathfrak{g}_2$ corresponding to the skew-Hermitian matrices of the form
  $\bigl( \begin{smallmatrix} A & 0 \\ 0 & 0 \end{smallmatrix}\bigr)$, $A^*=-A$,
  and
  $\bigl( \begin{smallmatrix} 0 & 0 \\ 0 & D \end{smallmatrix}\bigr)$, $D^*=-D$,
  respectively.
  This proves that $\mathfrak{g}_3$ and $\mathfrak{g}'_3$ are real Lie algebras
  isomorphic to $\mathfrak{u}(r)$.
  
\item
  Under the isomorphism
  $\mathfrak{g}_2 \simeq \mathfrak{gl}(U_2) \simeq \mathfrak{gl}(2r,\BB C)$,
  $\mathfrak{g}_4$ corresponds to the matrices of the form
  $\bigl( \begin{smallmatrix} A & B \\ B^* & D \end{smallmatrix}\bigr)$
  with $A^*=-A$, $D^*=-D$. Such matrices form a Lie algebra $\mathfrak{u}(r,r)$
  preserving the Hermitian form
  $\bigl( \begin{smallmatrix} 1 & 0 \\ 0 & -1 \end{smallmatrix}\bigr)$.
  This proves that $\mathfrak{g}_4$ is real Lie algebra isomorphic to
  $\mathfrak{u}(r,r)$.

\item
  Under the isomorphism
  $\mathfrak{g}_1 \simeq \mathfrak{gl}(U_1) \simeq \mathfrak{gl}(2r,\BB R)$,
  $\mathfrak{g}_5$ corresponds to the matrices of the form
  $\bigl( \begin{smallmatrix} A & B \\ C & -A^T \end{smallmatrix}\bigr)$
  with $B^T=B$, $C^T=C$. Such matrices form a Lie algebra
  $\mathfrak{sp}(2r,\BB R)$ preserving the symplectic form
  $\bigl( \begin{smallmatrix} 0 & 1 \\ -1 & 0 \end{smallmatrix}\bigr)$.
  This proves that $\mathfrak{g}_5$ is real Lie algebra isomorphic to
  $\mathfrak{sp}(2r,\BB R)$.

\item
  Follows immediately from the previous part, since
  $\mathfrak{g}_6 = \BB C \otimes \mathfrak{g}_5$.
  
\item
  From Lemma \ref{commutator-lem2} we see that $\mathfrak{g}_7$
  is a real Lie algebra. This Lie algebra acts on the space of
  all linear operators on $W$ by commutators.
  Lemma \ref{commutator-lem1} shows that $\mathfrak{g}_7$ preserves
  a real vector subspace $U_7$ with basis
  $\{a_1,\dots,a_{2r},a^*_1,\dots,a^*_r,-a^*_{r+1},\dots,-a^*_{2r}\}$.
  Thus, we have an injective Lie algebra homomorphism
  $\mathfrak{g}_7 \to \mathfrak{gl}(U_7)$. Under this homomorphism,
  $\mathfrak{g}_7$ corresponds to the real matrices of the form
  $\bigl( \begin{smallmatrix} A & B \\ C & -A^T \end{smallmatrix}\bigr)$
  with $B^T=B$, $C^T=C$.
  Such matrices form a Lie algebra $\mathfrak{sp}(4r,\BB R)$
  preserving the symplectic form
  $\bigl( \begin{smallmatrix} 0 & 1 \\ -1 & 0 \end{smallmatrix}\bigr)$.
  This proves that $\mathfrak{g}_7$ is real Lie algebra isomorphic to
  $\mathfrak{sp}(4r,\BB R)$.

\item
  Follows immediately from the previous part, since
  $\mathfrak{g}_8 = \BB C \otimes \mathfrak{g}_7$.
\end{enumerate}
\end{proof}

Thus, we obtain a representation of $\mathfrak{sp}(4r,\BB C) =
\BB C \Span \{:a_ka^*_l:,\: a_ka_l,\: a^*_ka^*_l ;\: 1 \le k,l \le 2r \}$ on $W$.
Inside $W$ we have $\mathfrak{sp}(4r,\BB C)$-invariant submodules
\[
W^+ = \BB C [w_1,\dots,w_r,w^*_{r+1},\dots,w^*_{2r}]
\]
(genuine polynomials) and
\[
W^{aux} = \BB C \Span \bigl\{ w_1^{\alpha_1} \cdots w_r^{\alpha_r}
(w^*_{r+1})^{\alpha_{r+1}} \cdots (w^*_{2r})^{\alpha_{2r}} \in W ;\:
  \text{at least one $\alpha_k \ge 0$} \bigr\}.
\]
Let $W^-$ be the quotient module $W/W^{aux}$. As a vector space,
\[
W^- = w_1^{-1} \cdots w_r^{-1}(w^*_{r+1})^{-1} \cdots (w^*_{2r})^{-1} \cdot
\BB C \bigl[ w_1^{-1},\dots,w_r^{-1},(w^*_{r+1})^{-1},\dots,(w^*_{2r})^{-1} \bigr].
\]
The operators
\[
a^*_1,\dots,a^*_r,a_{r+1},\dots,a_{2r}
\]
act on $W^-$ by differentiation with respect to
$w_1,\dots,w_r,w^*_{r+1},\dots,w^*_{2r}$ as before.
And the action of operators
\[
a_1,\dots,a_r,a^*_{r+1},\dots,a^*_{2r}
\]
on $W^-$ is modified so that if a multiplication by
$w_1,\dots,w_r,w^*_{r+1},\dots,w^*_{2r}$ results in a monomial with
a non-negative power, that monomial is replaced with zero.
The $\mathfrak{sp}(4r,\BB C)$-modules $(\pi^+, W^+)$ and $(\pi^-, W^-)$
are called the {\em metaplectic representations}
(they extend to the metaplectic Lie group).

We assign degrees to the generators of $W$:
\[
\deg w_k =-1,\quad k=1,\dots,r, \qquad \deg w^*_k=1, \qquad k=r+1,\dots,2r.
\]
Then we have decompositions
\[
W^+ = \bigoplus_{n \in \BB Z} W^+_n \qquad \text{and} \qquad
W^- = \bigoplus_{n \in \BB Z} W^-_n,
\]
where $W^{\pm}_n$ consist of homogeneous polynomials of degree $n$.
Each $W^{\pm}_n$ is invariant under the action of $\mathfrak{gl}(2r,\BB C) =
\BB C \Span \{:a_ka^*_l: ;\: 1 \le k,l \le 2r \}$, and we denote this action
by $\pi^+_n$ and $\pi^-_n$ respectively.
It can be shown that the representations $(\pi^+_n, W^+_n)$ and
$(\pi^-_n, W^-_n)$ of $\mathfrak{gl}(2r,\BB C)$ are irreducible.
The $\mathfrak{gl}(2r,\BB C)$-modules $W^{\pm}_n$ can be completed to unitary
representations of the group $U(r,r)$ with Lie algebra
\[
\mathfrak{u}(r,r) =
\BB R \Span \{i:a_ja^*_j:,\: a_ka^*_l-a_la^*_k, \: ia_ka^*_l+ia_la^*_k ;\:
1 \le j, k,l \le 2r,\: k<l \}.
\]
We have decompositions
\[
W^+ = V_+ \otimes V_+^*, \qquad
V_+=\BB C[w_1,\dots,w_r],\: V_+^*=\BB C[w^*_{r+1},\dots,w^*_{2r}]
\]
and $W^- = V_- \otimes V_-^*$, where
\begin{align*}
V_-&= w_1^{-1} \cdots w_r^{-1} \cdot \BB C\bigl[w_1^{-1},\dots,w_r^{-1}\bigr], \\
V_-^*&= (w^*_{r+1})^{-1} \cdots (w^*_{2r})^{-1} \cdot
\BB C\bigl[(w^*_{r+1})^{-1},\dots,(w^*_{2r})^{-1}\bigr].
\end{align*}
Let $K_r \simeq GL(r,\BB C) \times GL(r,\BB C)$ be the Lie group obtained by
exponentiating
\[
\mathfrak{k}_r = \BB C \Span \{ :a_ka^*_l: ;\: 1 \le k,l \le r
\text{ or } r+1 \le k,l \le 2r \}.
\]
Then the $K_r$-types of $(\pi^+_n, W^+_n)$ are realized by polynomials that are
homogeneous in $w_1,\dots,w_r$ and at the same time homogeneous in
$w^*_{r+1},\dots,w^*_{2r}$.
More precisely, the $K_r$-types of $W^+_n$ are elements of
$V_+(d) \otimes V_+^*(d+n)$, where $V_+(d)$ and $V_+^*(d+n)$ denote homogeneous
polynomials in $V_+$ and $V_+^*$ of degree $d$ and $d+n$ respectively.
Similarly, the $K_r$-types of $(\pi^-_n, W^-_n)$ are realized by polynomials
that are homogeneous in $w_1^{-1},\dots,w_r^{-1}$ and at the same time
homogeneous in $(w^*_{r+1})^{-1},\dots,(w^*_{2r})^{-1}$.
Thus, the $K_r$-types of $W^-_n$ are elements of
$V_-(-d-r-n) \otimes V_-^*(-d-r)$, where $V_-(-d-r-n)$ and $V_-^*(-d-r)$
denote homogeneous polynomials in $V_-$ and $V_-^*$ of degree $-(d+r+n)$ and
$-(d+r)$ respectively.

\subsection{The Metaplectic Representation $W^+$ of $\mathfrak{sp}(8,\BB C)$
and the $n$-Regular Functions}

We restrict ourselves to the case of $r=2$.
Then
\[
W^+= \BB C[w_1,w_2,w^*_3,w^*_4]
= V_+ \otimes V_+^*, \qquad V_+=\BB C[w_1,w_2],\: V_+^*=\BB C[w^*_3,w^*_4].
\]
The isomorphism from Proposition \ref{iso-prop}
\[
\BB R \Span \{:a^*_ka_l: ;\: 1 \le k,l \le 4 \} \simeq \mathfrak{gl}(4,\BB R)
\]
identifies $:a^*_ka_l:$ with a $4 \times 4$ matrix with $\mp_l 1$ in the
$(k,l)$-entry and zero elsewhere, where $\mp_l$ is ``$-$'' if $l=1,2$ and
``$+$'' if $l=3,4$.
We write $K$ in place of $K_2 \simeq GL(2,\BB C) \times GL(2,\BB C)$
-- the Lie group obtained by exponentiating
\[
\mathfrak{k} =
\BB C \Span \{ :a^*_ka_l: ;\: 1 \le k,l \le 2 \text{ or } 3 \le k,l \le 4 \}.
\]

Observe that $W^+_0$ (polynomials of degree zero in $W^+$) has a basis
consisting of monomials of the form
\begin{equation}  \label{monomial-basis}
  p^l_{\mu,\nu}(w_1,w_2,w^*_3,w^*_4) = \tfrac{(-1)^{2l}}
       {(l-\nu)!(l+\nu)!} (w_1)^{l-\nu}(w_2)^{l+\nu}(w^*_3)^{l-\mu}(w^*_4)^{l+\mu},
\qquad
\begin{smallmatrix} l = 0, \frac12, 1, \frac32, \dots, \\
  -l \le \mu,\:\nu \le l, \\ \mu,\:\nu \in \BB Z +l. \end{smallmatrix}
\end{equation}
For a fixed $l$, these monomials span the $K$-type $V_l \boxtimes V_l$.

To each monomial \eqref{monomial-basis} we associate the matrix coefficient
of $SU(2)$ described by equation (27) of \cite{FL1} (cf. \cite{V}):
\begin{equation}  \label{t}
t^l_{\nu\,\underline{\mu}}(Z) = \frac 1{2\pi i}
\oint (sz_{11}+z_{21})^{l-\mu} (sz_{12}+z_{22})^{l+\mu} s^{-l+\nu} \,\frac{ds}s,
\qquad
\begin{smallmatrix} l = 0, \frac12, 1, \frac32, \dots, \\
  -l \le \mu,\:\nu \le l, \\ \mu,\:\nu \in \BB Z +l, \end{smallmatrix}
\end{equation}
where $Z=\bigl(\begin{smallmatrix} z_{11} & z_{12} \\
z_{21} & z_{22} \end{smallmatrix}\bigr) \in \HC$,
the integral is taken over a loop in $\BB C$ going once around the origin
in the counterclockwise direction.
We regard these functions as polynomials on $\HC$, and we frequently use them
to construct $K$-type bases of various spaces as in, for example,
\cite{FL1}, \cite{desitter}, \cite{ATMP}, \cite{qreg}.
This results in a vector space isomorphism
\[
\Phi^+_0 : W^+_0 \xrightarrow{\sim} {\cal H}^+
= \bigl\{ \phi \in \BB C[z_{11},z_{12},z_{21},z_{22}] ;\: \square \phi =0 \bigr\}
= \BB C\text{-span of } \bigl\{  t^l_{\nu \, \underline{\mu}}(Z) \bigr\}.
\]

By direct computation and Lemma 22 in \cite {FL1} we obtain:

\begin{lem}  \label{intertwiner-lem}
We have the following intertwining relations:
\begin{align*}
\Phi^+_0 \begin{pmatrix} a^*_1a_3 & a^*_1a_4 \\ a^*_2a_3 & a^*_2a_4 \end{pmatrix}
p^l_{\mu,\nu} &= - \left(\begin{smallmatrix}
  (l-\mu) t^{l-\frac12}_{\nu+\frac12\,\underline{\mu+\frac12}}(Z) &
  (l+\mu) t^{l-\frac12}_{\nu+\frac12\,\underline{\mu-\frac12}}(Z) \\
  (l-\mu) t^{l-\frac12}_{\nu-\frac12\,\underline{\mu+\frac12}}(Z) &
  (l+\mu) t^{l-\frac12}_{\nu-\frac12\,\underline{\mu-\frac12}}(Z)
\end{smallmatrix}\right)
= - \bigl(\begin{smallmatrix} \partial_{11} & \partial_{12} \\
\partial_{21} & \partial_{22} \end{smallmatrix}\bigr) t^l_{\nu\,\underline{\mu}}(Z),\\
\Phi^+_0 \begin{pmatrix} a^*_3a_1 & a^*_3a_2 \\ a^*_4a_1 & a^*_4a_2 \end{pmatrix}
p^l_{\mu,\nu} &= - \left(\begin{smallmatrix}
  (l-\nu+1) t^{l+\frac12}_{\nu-\frac12\,\underline{\mu-\frac12}}(Z) &
  (l+\nu+1) t^{l+\frac12}_{\nu+\frac12\,\underline{\mu-\frac12}}(Z) \\
  (l-\nu+1) t^{l+\frac12}_{\nu-\frac12\,\underline{\mu+\frac12}}(Z) &
  (l+\nu+1) t^{l+\frac12}_{\nu+\frac12\,\underline{\mu+\frac12}}(Z)
\end{smallmatrix}\right).
\end{align*}
\end{lem}

\begin{prop}  \label{harmonic+iso-prop}
  The isomorphism $\Phi^+_0$ intertwines the actions $\pi^+_0$ on $W^+_0$ and
  $\pi^0_l$ on ${\cal H}^+$ of $\mathfrak{sl}(4,\BB C) = \mathfrak{sl}(2,\HC)$.
\end{prop}

\begin{proof}
  The action $\pi^0_l$ on ${\cal H}^+$ of $\mathfrak{sl}(2,\HC)$ is spelled
  out in Lemma 17 in \cite{FL1}.
  By equation (15) in \cite{qreg} and Lemma \ref{intertwiner-lem},
  $\Phi^+_0$ intertwines the actions on $W^+_0$ and ${\cal H}^+$ of matrices
  of the form
  $\bigl(\begin{smallmatrix} 0 & B \\ C & 0 \end{smallmatrix}\bigr)
  \in \mathfrak{gl}(2,\HC)$, $B, C \in \HC$.
  Since these matrices generate $\mathfrak{sl}(2,\HC)$, the result follows.
\end{proof}

\begin{rem}
The isomorphism $\Phi^+_0$ is {\em not} an intertwiner of the actions of
$\mathfrak{gl}(4,\BB C) = \mathfrak{gl}(2,\HC)$.
Scalar matrices $\lambda \in \mathfrak{gl}(2,\HC)$ act trivially
on $W^+_0$ via the $\pi^+_0$ action, by multiplication by $2\lambda$ 
on ${\cal H}^+$ via the $\pi^0_l$ action and by multiplication by $-2\lambda$ 
on ${\cal H}^+$ via the $\pi^0_r$ action.
\end{rem}

We turn our attention to $W^+_n$ -- polynomials of degree $n$ in $W^+$ --
with $n>0$. This space has a basis consisting of monomials of the form
\begin{equation}  \label{n-monomial-basis}
  P^{(n)}_{l,\mu,\nu}(w_1,w_2,w^*_3,w^*_4) = \tfrac{(-1)^{2l}}{(l-\nu)!(l+\nu)!}
  (w_1)^{l-\nu}(w_2)^{l+\nu}(w^*_3)^{l+\frac{n}2-\mu}(w^*_4)^{l+\frac{n}2+\mu},
\quad
\begin{smallmatrix} l = 0, \frac12, 1, \frac32, \dots, \\
  -l-\tfrac{n}2 \le \mu \le l+\tfrac{n}2, \\  -l \le \nu \le l, \\
  \mu+\tfrac{n}2,\:\nu \in \BB Z +l. \end{smallmatrix}
\end{equation}
For a fixed $l$, these monomials span the $K$-type
$V_l \boxtimes V_{l+\frac{n}2}$.
By direct computation we obtain:

\begin{lem}  \label{pi+-action-lem}
For $n>0$,
elements $\bigl(\begin{smallmatrix} 0 & B \\ C & 0 \end{smallmatrix}\bigr)
\in \mathfrak{sl}(4,\BB C)$, $B, C \in \mathfrak{gl}(2,\BB C)$, act on
$W^+_n$ via $\pi^+_n$ as follows:
\begin{align*}
\begin{pmatrix} a^*_1a_3 & a^*_1a_4 \\ a^*_2a_3 & a^*_2a_4 \end{pmatrix}
P^{(n)}_{l,\mu,\nu} &= - \left(\begin{smallmatrix}
  (l+\tfrac{n}2-\mu) P^{(n)}_{l-\frac12,\mu+\frac12,\nu+\frac12} &
  (l+\tfrac{n}2+\mu) P^{(n)}_{l-\frac12,\mu-\frac12,\nu+\frac12} \\
  (l+\tfrac{n}2-\mu) P^{(n)}_{l-\frac12,\mu+\frac12,\nu-\frac12} &
  (l+\tfrac{n}2+\mu) P^{(n)}_{l-\frac12,\mu-\frac12,\nu-\frac12}
\end{smallmatrix}\right),  \\
\begin{pmatrix} a^*_3a_1 & a^*_3a_2 \\ a^*_4a_1 & a^*_4a_2 \end{pmatrix}
P^{(n)}_{l,\mu,\nu} &= - \left(\begin{smallmatrix}
  (l-\nu+1) P^{(n)}_{l+\frac12,\mu-\frac12,\nu-\frac12} &
  (l+\nu+1) P^{(n)}_{l+\frac12,\mu-\frac12,\nu+\frac12} \\
  (l-\nu+1) P^{(n)}_{l+\frac12,\mu+\frac12,\nu-\frac12} &
  (l+\nu+1) P^{(n)}_{l+\frac12,\mu+\frac12,\nu+\frac12} \end{smallmatrix}\right).
\end{align*}
\end{lem}

For each monomial $P^{(n)}_{l,\mu,\nu}$, we define a function on $\HC$
with values in $\underbrace{\BB S \odot \dots \odot \BB S}_{\text{$n$ times}}$:
\[
\Phi^+_n \bigl( P^{(n)}_{l,\mu,\nu} \bigr) =
\Phi^+_0 \biggl( \underbrace{\begin{pmatrix} a_3 \\ a_4 \end{pmatrix}
  \otimes \dots \otimes 
  \begin{pmatrix} a_3 \\ a_4 \end{pmatrix}}_{\text{$n$ times}} \biggr)
P^{(n)}_{l,\mu,\nu}.
\]
(Note that the operators $a_3$ and $a_4$ act by differentiation with respect to
$w^*_3$ and $w^*_4$ respectively. Thus, applying them $n$ times to
$P^{(n)}_{l,\mu,\nu}$ results in a polynomial of degree zero.)

In Section 7 of \cite{nreg} we introduced the spaces of polynomial
left $n$-regular functions $(\pi_{nl}, {\cal F}^+_n)$, $(\pi_{nl}, {\cal F}^-_n)$
as well as polynomial right $n$-regular functions $(\pi_{nr}, {\cal G}^-_n)$,
$(\pi_{nr}, {\cal G}^+_n)$.
When $n=1$, these are just the classical regular functions:
\[
(\pi_{1l}, {\cal F}^+_1) = (\pi_l, {\cal V}^+), \quad
(\pi_{1l}, {\cal F}^-_1) = (\pi_l, {\cal V}^-), \quad
(\pi_{1r}, {\cal G}^+_1) = (\pi_r, {\cal V}'^+), \quad
(\pi_{1r}, {\cal G}^-_1) = (\pi_r, {\cal V}'^-).  
\]

\begin{lem}  \label{n-reg+-lem}
Each function
\[
\Phi^+_n \bigl( P^{(n)}_{l,\mu,\nu} \bigr) :
\HC \to \underbrace{\BB S \odot \dots \odot \BB S}_{\text{$n$ times}}
\]
is left $n$-regular.
\end{lem}

\begin{proof}
Using Lemma \ref{intertwiner-lem}, for each $k=1,\dots,n$, we have:
\begin{multline*}
\nabla^+_k \Phi^+_n \bigl( P^{(n)}_{l,\mu,\nu} \bigr)
=-2(1 \otimes \cdots \underset{\text{$k$-th place}}{\otimes \partial^+ \otimes}
\cdots \otimes 1)
\Phi^+_0 \biggl( \begin{pmatrix} a_3 \\ a_4 \end{pmatrix}
  \otimes \dots \otimes 
  \begin{pmatrix} a_3 \\ a_4 \end{pmatrix} \biggr)
P^{(n)}_{l,\nu\,\underline{\mu}}  \\
= -2 \Phi^+_0 \biggl( \begin{pmatrix} a_3 \\ a_4 \end{pmatrix}
\otimes \dots \otimes \underset{\text{$k$-th place}}
  {\begin{pmatrix} a^*_2a_4 & -a^*_2a_3 \\ -a^*_1a_4 & a^*_1a_3 \end{pmatrix}
  \begin{pmatrix} a_3 \\ a_4 \end{pmatrix}} \otimes \dots
  \begin{pmatrix} a_3 \\ a_4 \end{pmatrix} \biggr)
P^{(n)}_{l,\mu,\nu} =0.
\end{multline*}
\end{proof}

The map $\Phi^+_n$ extends by linearity to a map of vector spaces
\[
\Phi^+_n: W^+_n \to {\cal F}^+_n.
\]

\begin{lem}  \label{Z-mult-lem1}
Let $C = \bigl(\begin{smallmatrix} c_{11} & c_{12} \\
  c_{21} & c_{22} \end{smallmatrix}\bigr) \in \HC$, then, for any $p \in W^+_1$,
\[
CZ \cdot \Phi^+_0 \Bigl(
\bigl(\begin{smallmatrix} a_3 \\ a_4 \end{smallmatrix}\bigr) (p)\Bigr)
= - \Phi^+_0 \Bigl(
c_{11} \bigl(\begin{smallmatrix} a_1 \\ 0 \end{smallmatrix}\bigr)
+ c_{12} \bigl(\begin{smallmatrix} a_2 \\ 0 \end{smallmatrix}\bigr)
+ c_{21} \bigl(\begin{smallmatrix} 0 \\ a_1 \end{smallmatrix}\bigr)
+ c_{22} \bigl(\begin{smallmatrix} 0 \\ a_2 \end{smallmatrix}\bigr) \Bigr) (p).
\]
\end{lem}

\begin{proof}
The result follows from Lemma 23 in \cite{FL1} and observation
\[  
a_1 P^{(1)}_{l,\mu,\nu} = -(l-\nu+1) p^{l+\frac12}_{\mu,\nu-\frac12}
\qquad \text{and} \qquad
a_2 P^{(1)}_{l,\mu,\nu} = -(l+\nu+1) p^{l+\frac12}_{\mu,\nu+\frac12}.
\]
\end{proof}

\begin{thm}  \label{n-reg-iso-thm}
  For each $n>0$, the map $\Phi^+_n$ is an isomorphism between representations
  $(\pi^+_n, W^+_n)$ and $(\pi_{nl}, {\cal F}^+_n)$ of
  $\mathfrak{sl}(4,\BB C) = \mathfrak{sl}(2,\HC)$.
\end{thm}

\begin{proof}
  The action of $\mathfrak{sl}(2,\HC)$ on ${\cal F}^+_n$ via $\pi_{nl}$
  is spelled out in Lemma 2 in \cite{nreg}. By equation (15) in \cite{qreg}
  and Lemmas \ref{intertwiner-lem}, \ref{Z-mult-lem1}, $\Phi^+_n$ intertwines
  the actions on $W^+_n$ and ${\cal F}^+_n$ of matrices of the form
  $\bigl(\begin{smallmatrix} 0 & B \\ C & 0 \end{smallmatrix}\bigr)
  \in \mathfrak{sl}(2,\HC)$, $B, C \in \HC$.
  Since these matrices generate $\mathfrak{sl}(2,\HC)$, it follows that
  $\Phi^+_n$ is an intertwining map between $(\pi^+_n, W^+_n)$ and
  $(\pi_{nl}, {\cal F}^+_n)$.
  It is clear that $\Phi^+_n$ is injective.
  By Theorem 26 in \cite{nreg}, $(\pi_{nl}, {\cal F}^+_n)$ is irreducible.
  Hence $\Phi^+_n$ is onto, and the result follows.
\end{proof}

Next, we consider the case of polynomials of negative degrees $W^+_{-n}$, $n>0$.
This space has a basis consisting of monomials of the form
\begin{equation}  \label{-n-monomial-basis}
  Q^{(n)}_{l,\mu,\nu}(w_1,w_2,w^*_3,w^*_4) =
  \tfrac{(-1)^{2l}}{(l+\frac{n}2-\nu)!(l+\frac{n}2+\nu)!}
  (w_1)^{l+\frac{n}2-\nu}(w_2)^{l+\frac{n}2+\nu}(w^*_3)^{l-\mu}(w^*_4)^{l+\mu},
\end{equation}
\[
l = 0, \tfrac12, 1, \tfrac32, \dots, \quad
-l \le \mu \le l, \quad
-l-\tfrac{n}2 \le \nu \le l+\tfrac{n}2, \quad
\mu,\: \nu+\tfrac{n}2 \in \BB Z +l.
\]
For a fixed $l$, these monomials span the $K$-type
$V_{l+\frac{n}2} \boxtimes V_l$.
By direct computation we obtain:

\begin{lem}  \label{pi+-action-lem2}
For $n<0$, elements
$\bigl(\begin{smallmatrix} 0 & B \\ C & 0 \end{smallmatrix}\bigr)
\in \mathfrak{sl}(4,\BB C)$, $B, C \in \mathfrak{gl}(2,\BB C)$, act on
$W^+_n$ via $\pi^+_n$ as follows:
\begin{align*}
\begin{pmatrix} a^*_1a_3 & a^*_1a_4 \\ a^*_2a_3 & a^*_2a_4 \end{pmatrix}
Q^{(n)}_{l,\mu,\nu} &= - \left(\begin{smallmatrix}
  (l-\mu) Q^{(n)}_{l-\frac12,\mu+\frac12,\nu+\frac12} &
  (l+\mu) Q^{(n)}_{l-\frac12,\mu-\frac12,\nu+\frac12} \\
  (l-\mu) Q^{(n)}_{l-\frac12,\mu+\frac12,\nu-\frac12} &
  (l+\mu) Q^{(n)}_{l-\frac12,\mu-\frac12,\nu-\frac12} \end{smallmatrix}\right),  \\
\begin{pmatrix} a^*_3a_1 & a^*_3a_2 \\ a^*_4a_1 & a^*_4a_2 \end{pmatrix}
Q^{(n)}_{l,\mu,\nu} &= - \left(\begin{smallmatrix}
  (l+\tfrac{n}2-\nu+1) Q^{(n)}_{l+\frac12,\mu-\frac12,\nu-\frac12} &
  (l+\tfrac{n}2+\nu+1) Q^{(n)}_{l+\frac12,\mu-\frac12,\nu+\frac12} \\
  (l+\tfrac{n}2-\nu+1) Q^{(n)}_{l+\frac12,\mu+\frac12,\nu-\frac12} &
  (l+\tfrac{n}2+\nu+1) Q^{(n)}_{l+\frac12,\mu+\frac12,\nu+\frac12}
\end{smallmatrix}\right).
\end{align*}
\end{lem}

For each monomial $Q^{(n)}_{l,\mu,\nu}$, we define a function on $\HC$
with values in $\underbrace{\BB S' \odot \dots \odot \BB S'}_{\text{$n$ times}}$:
\[
\Psi^+_n \bigl( Q^{(n)}_{l,\mu,\nu} \bigr) =
\Phi^+_0 \bigl( \underbrace{(a^*_1, a^*_2) \otimes \dots \otimes 
  (a^*_1, a^*_2)}_{\text{$n$ times}} \bigr) Q^{(n)}_{l,\mu,\nu}.
\]
(Note that the operators $a^*_1$ and $a^*_2$ act by differentiation with
respect to $w_1$ and $w_2$ respectively. Thus, applying them $n$ times to
$Q^{(n)}_{l,\mu,\nu}$ results in a polynomial of degree zero.)

\begin{lem}  \label{n-reg--lem}
Each function
\[
\Psi^+_n \bigl( Q^{(n)}_{l,\mu,\nu} \bigr) :
\HC \to \underbrace{\BB S' \odot \dots \odot \BB S'}_{\text{$n$ times}}
\]
is right $n$-regular.
\end{lem}

\begin{proof}
Using Lemma \ref{intertwiner-lem}, for each $k=1,\dots,n$, we have:
\begin{multline*}
\Psi^+_n \bigl( Q^{(n)}_{l,\mu,\nu} \bigr) \overleftarrow{\nabla^+_k}
= -2 \Phi^+_0 \bigl( (a^*_1, a^*_2) \otimes \dots \otimes (a^*_1, a^*_2) \bigr)
Q^{(n)}_{l,\mu,\nu} (\overleftarrow{1 \otimes \cdots
  \underset{\text{$k$-th place}}{\otimes \partial^+ \otimes}
  \cdots \otimes 1})  \\
= -2 \Phi^+_0 \bigl( (a^*_1, a^*_2) \otimes \dots \otimes
\underset{\text{$k$-th place}}{(a^*_1, a^*_2)
    \begin{pmatrix} a^*_2a_4 & -a^*_2a_3 \\ -a^*_1a_4 & a^*_1a_3 \end{pmatrix}}
  \otimes \cdots \otimes (a^*_1, a^*_2) \bigr) Q^{(n)}_{l,\mu,\nu} =0.
\end{multline*}
\end{proof}

The map $\Psi^+_n$ extends by linearity to a map of vector spaces
\[
\Psi^+_n: W^+_{-n} \to {\cal G}^+_n,
\]
where ${\cal G}^+_n$ is the space of polynomial right $n$-regular functions
on $\HC$ introduced in Section 7 of \cite{nreg}.

\begin{lem}  \label{Z-mult-lem2}
Let $C = \bigl(\begin{smallmatrix} c_{11} & c_{12} \\
 c_{21} & c_{22} \end{smallmatrix}\bigr) \in \HC$, then, for any $q \in W^+_{-1}$,
\[
\Phi^+_0 \bigl( (a^*_1,a^*_2)q \bigr) \cdot ZC
= - \Phi^+_0 \bigl( c_{11} (a^*_3, 0) + c_{12} (0, a^*_3)
+ c_{21} (a^*_4, 0) +c_{22} (0, a^*_4) \bigr) (q).
\]
\end{lem}

\begin{proof}
The result follows from Lemma 23 in \cite{FL1} and observation
\[  
a^*_3 Q^{(1)}_{l,\mu,\nu} = - p^{l+\frac12}_{\mu-\frac12,\nu}
\qquad \text{and} \qquad
a^*_4 Q^{(1)}_{l,\mu,\nu} = - p^{l+\frac12}_{\mu+\frac12,\nu}.
\]
\end{proof}

\begin{thm}  \label{n-reg-iso-thm2}
  For each $n>0$, the map $\Psi^+_n$ is an isomorphism between representations
  $(\pi^+_{-n}, W^+_{-n})$ and $(\pi_{nr}, {\cal G}^+_n)$ of
  $\mathfrak{sl}(4,\BB C) = \mathfrak{sl}(2,\HC)$.
\end{thm}

\begin{proof}
  The action of $\mathfrak{sl}(2,\HC)$ on ${\cal G}^+_n$ via $\pi_{nr}$
  is spelled out in Lemma 2 in \cite{nreg}.   By equation (15) in \cite{qreg}
  and Lemmas \ref{intertwiner-lem}, \ref{Z-mult-lem2}, $\Psi^+_n$ intertwines
  the actions on $W^+_{-n}$ and ${\cal G}^+_n$ of matrices of the form
  $\bigl(\begin{smallmatrix} 0 & B \\ C & 0 \end{smallmatrix}\bigr)
  \in \mathfrak{sl}(2,\HC)$, $B, C \in \HC$.
  Since these matrices generate $\mathfrak{sl}(2,\HC)$, it follows that
  $\Psi^+_n$ is an intertwining map between $(\pi^+_{-n}, W^+_{-n})$ and
  $(\pi_{nr}, {\cal G}^+_n)$.
  It is clear that $\Psi^+_n$ is injective.
  By Theorem 26 in \cite{nreg}, $(\pi_{nr}, {\cal G}^+_n)$ is irreducible.
  Hence $\Psi^+_n$ is onto, and the result follows.
\end{proof}

Theorems \ref{n-reg-iso-thm} and \ref{n-reg-iso-thm2} as well as
Proposition \ref{harmonic+iso-prop} are well known, see e.g. \cite{JV1}.
We have stated these results in our notations for future use in the paper.

\subsection{The Metaplectic Representation $W^-$ of $\mathfrak{sp}(8,\BB C)$
and the $n$-Regular Functions Regular at Infinity}

We turn our attention to
\[
W^- = w_1^{-1}w_2^{-1}(w^*_3)^{-1}(w^*_4)^{-1} \cdot
\BB C\bigl[w_1^{-1},w_2^{-1},(w^*_3)^{-1},(w^*_4)^{-1}\bigr]
= V_- \otimes V_-^*,
\]
\[
V_-= w_1^{-1} w_2^{-1} \cdot \BB C\bigl[w_1^{-1},w_2^{-1}\bigr],\quad
V_-^*= (w^*_3)^{-1}(w^*_4)^{-1} \cdot \BB C\bigl[(w^*_3)^{-1},(w^*_4)^{-1}\bigr].
\]
Observe that $W^-_0$ (polynomials of degree zero in $W^-$) has a basis
consisting of monomials of the form
\begin{equation}  \label{monomial-tilde-basis}
  \tilde p^l_{\mu,\nu}(w_1,w_2,w^*_3,w^*_4) = (l-\nu)!(l+\nu)!
  (w_1)^{-(l-\nu+1)}(w_2)^{-(l+\nu+1)}(w^*_3)^{-(l-\mu+1)}(w^*_4)^{-(l+\mu+1)},
\end{equation}
\[
l = 0, \tfrac12, 1, \tfrac32, \dots, \quad
  -l \le \mu,\:\nu \le l, \quad \mu,\:\nu \in \BB Z +l.
\]
For a fixed $l$, these monomials span the $K$-type $V_l \boxtimes V_l$.
To each monomial \eqref{monomial-tilde-basis} we associate a harmonic function
$N(Z)^{-1} \cdot t^l_{\mu \, \underline{\nu}}(Z^{-1})$ on $\HC$.
This results in a vector space isomorphism
\begin{align*}
  \Phi^-_0 : W^-_0 \xrightarrow{\sim} {\cal H}^-
  &= \bigl\{ \phi \in \BB C[z_{11},z_{12},z_{21},z_{22}, N(Z)^{-1}] ;\:
  N(Z)^{-1} \cdot \phi(Z^{-1}) \in {\cal H^+} \bigr\}  \\
  &= \BB C\text{-span of } \bigl\{
  N(Z)^{-1} \cdot t^l_{\mu \, \underline{\nu}}(Z^{-1}) \bigr\}.
\end{align*}

By direct computation and identity (13) in \cite{qreg} we obtain:

\begin{lem}  \label{-intertwiner-lem}
We have the following intertwining relations:
\begin{align*}
\Phi^-_0 \begin{pmatrix} a^*_1a_3 & a^*_1a_4 \\ a^*_2a_3 & a^*_2a_4 \end{pmatrix}
\tilde p^l_{\mu,\nu} &=
\frac1{N(Z)} \left(\begin{smallmatrix}
  (l-\mu+1) t^{l+\frac12}_{\mu-\frac12\,\underline{\nu-\frac12}}(Z^{-1}) &
  (l+\mu+1) t^{l+\frac12}_{\mu+\frac12\,\underline{\nu-\frac12}}(Z^{-1}) \\
  (l-\mu+1) t^{l+\frac12}_{\mu-\frac12\,\underline{\nu+\frac12}}(Z^{-1}) &
  (l+\mu+1) t^{l+\frac12}_{\mu+\frac12\,\underline{\nu+\frac12}}(Z^{-1})
\end{smallmatrix}\right)  \\
&= - \bigl(\begin{smallmatrix} \partial_{11} & \partial_{12} \\
  \partial_{21} & \partial_{22} \end{smallmatrix}\bigr)
\bigl( N(Z)^{-1} \cdot t^l_{\mu\,\underline{\nu}}(Z^{-1}) \bigr),  \\
\Phi^-_0 \begin{pmatrix} a^*_3a_1 & a^*_3a_2 \\ a^*_4a_1 & a^*_4a_2 \end{pmatrix}
\tilde p^l_{\mu,\nu} &= \frac1{N(Z)}
\left(\begin{smallmatrix}
  (l-\nu) t^{l-\frac12}_{\mu+\frac12\,\underline{\nu+\frac12}}(Z^{-1}) &
  (l+\nu) t^{l-\frac12}_{\mu+\frac12\,\underline{\nu-\frac12}}(Z^{-1}) \\
  (l-\nu) t^{l-\frac12}_{\mu-\frac12\,\underline{\nu+\frac12}}(Z^{-1}) &
  (l+\nu) t^{l-\frac12}_{\mu-\frac12\,\underline{\nu-\frac12}}(Z^{-1})
\end{smallmatrix}\right).
\end{align*}
\end{lem}

\begin{prop}  \label{harmonic-iso-prop}
  The isomorphism $\Phi^-_0$ intertwines the actions $\pi^-_0$ on $W^-_0$ and
  $\pi^0_l$ on ${\cal H}^-$ of $\mathfrak{sl}(4,\BB C) = \mathfrak{sl}(2,\HC)$.
\end{prop}

\begin{proof}
  The action $\pi^0_l$ on ${\cal H}^-$ of $\mathfrak{sl}(2,\HC)$ is spelled
  out in Lemma 17 in \cite{FL1}.
  By equation (16) in \cite{qreg} and Lemma \ref{-intertwiner-lem},
  $\Phi^-_0$ intertwines the actions on $W^-_0$ and ${\cal H}^-$ of matrices
  of the form
  $\bigl(\begin{smallmatrix} 0 & B \\ C & 0 \end{smallmatrix}\bigr)
  \in \mathfrak{gl}(2,\HC)$, $B, C \in \HC$.
  Since these matrices generate $\mathfrak{sl}(2,\HC)$, the result follows.
\end{proof}

We turn our attention to $W^-_n$ -- polynomials of degree $n$ in $W^-$ --
with $n>0$. This space has a basis consisting of monomials of the form
\begin{multline}  \label{tilde-n-monomial-basis}
  \tilde P^{(n)}_{l,\mu,\nu}(w_1,w_2,w^*_3,w^*_4)  \\
  = (l+\tfrac{n}2-\nu)!(l+\tfrac{n}2+\nu)!   (w_1)^{-(l+\frac{n}2-\nu+1)}
  (w_2)^{-(l+\frac{n}2+\nu+1)}(w^*_3)^{-(l-\mu+1)}(w^*_4)^{-(l+\mu+1)},
\end{multline}
\[
l = 0, \tfrac12, 1, \tfrac32, \dots, \quad
-l \le \mu \le l, \quad
-l-\tfrac{n}2 \le \nu \le l+\tfrac{n}2, \quad
\mu,\:\nu+\tfrac{n}2 \in \BB Z +l.
\]
For a fixed $l$, these monomials span the $K$-type
$V_{l+\frac{n}2} \boxtimes V_l$.
By direct computation we obtain:

\begin{lem}  \label{pi--action-lem}
For $n>0$, elements
$\bigl(\begin{smallmatrix} 0 & B \\ C & 0 \end{smallmatrix}\bigr)
\in \mathfrak{sl}(4,\BB C)$, $B, C \in \mathfrak{gl}(2,\BB C)$, act on
$W^-_n$ via $\pi^-_n$ as follows:
\begin{align*}
\begin{pmatrix} a^*_1a_3 & a^*_1a_4 \\ a^*_2a_3 & a^*_2a_4 \end{pmatrix}
\tilde P^{(n)}_{l,\mu,\nu} &= \left(\begin{smallmatrix}
  (l-\mu+1) \tilde P^{(n)}_{l+\frac12,\mu-\frac12,\nu-\frac12} &
  (l+\mu+1) \tilde P^{(n)}_{l+\frac12,\mu+\frac12,\nu-\frac12} \\
  (l-\mu+1) \tilde P^{(n)}_{l+\frac12,\mu-\frac12,\nu+\frac12} &
  (l+\mu+1) \tilde P^{(n)}_{l+\frac12,\mu+\frac12,\nu+\frac12}
\end{smallmatrix}\right),  \\
\begin{pmatrix} a^*_3a_1 & a^*_3a_2 \\ a^*_4a_1 & a^*_4a_2 \end{pmatrix}
\tilde P^{(n)}_{l,\mu,\nu} &= \left(\begin{smallmatrix}
  (l+\tfrac{n}2-\nu) \tilde P^{(n)}_{l-\frac12,\mu+\frac12,\nu+\frac12} &
  (l+\tfrac{n}2+\nu) \tilde P^{(n)}_{l-\frac12,\mu+\frac12,\nu-\frac12} \\
  (l+\tfrac{n}2-\nu) \tilde P^{(n)}_{l-\frac12,\mu-\frac12,\nu+\frac12} &
  (l+\tfrac{n}2+\nu) \tilde P^{(n)}_{l-\frac12,\mu-\frac12,\nu-\frac12}
\end{smallmatrix}\right).
\end{align*}
\end{lem}

For each monomial $\tilde P^{(n)}_{l,\mu,\nu}$, we define a function on
$\HC^{\times}$
with values in $\underbrace{\BB S \odot \dots \odot \BB S}_{\text{$n$ times}}$:
\[
\Phi^-_n \bigl( \tilde P^{(n)}_{l,\mu,\nu} \bigr) =
\Phi^-_0 \biggl( \underbrace{\begin{pmatrix} a_3 \\ a_4 \end{pmatrix}
  \otimes \dots \otimes 
  \begin{pmatrix} a_3 \\ a_4 \end{pmatrix}}_{\text{$n$ times}} \biggr)
\tilde P^{(n)}_{l,\mu,\nu}.
\]
(Note that applying the operators $a_3$ and $a_4$ $n$ times to
$\tilde P^{(n)}_{l,\mu,\nu}$ results in a polynomial of degree zero.)
The same argument as in the proof of Lemma \ref{n-reg+-lem} shows:

\begin{lem}
Each function
\[
\Phi^-_n \bigl( \tilde P^{(n)}_{l,\mu,\nu} \bigr) :
\HC^{\times} \to \underbrace{\BB S \odot \dots \odot \BB S}_{\text{$n$ times}}
\]
is left $n$-regular.
\end{lem}

The map $\Phi^-_n$ extends by linearity to a map of vector spaces
\[
\Phi^-_n: W^-_n \to {\cal F}^-_n,
\]
where ${\cal F}^-_n$ is the space of polynomial left $n$-regular functions
on $\HC^{\times}$ that are regular at infinity introduced in
Section 7 of \cite{nreg}.

\begin{lem}  \label{-Z-mult-lem1}
Let $C = \bigl(\begin{smallmatrix} c_{11} & c_{12} \\
  c_{21} & c_{22} \end{smallmatrix}\bigr) \in \HC$, then, for any
$\tilde p \in W^-_1$,
\[
CZ \cdot \Phi^-_0 \Bigl(
\bigl(\begin{smallmatrix} a_3 \\ a_4 \end{smallmatrix}\bigr) (\tilde p)\Bigr)
= - \Phi^-_0 \Bigl(
c_{11} \bigl(\begin{smallmatrix} a_1 \\ 0 \end{smallmatrix}\bigr)
+ c_{12} \bigl(\begin{smallmatrix} a_2 \\ 0 \end{smallmatrix}\bigr)
+ c_{21} \bigl(\begin{smallmatrix} 0 \\ a_1 \end{smallmatrix}\bigr)
+ c_{22} \bigl(\begin{smallmatrix} 0 \\ a_2 \end{smallmatrix}\bigr) \Bigr)
(\tilde p).
\]
\end{lem}

\begin{proof}
The result follows from Lemma 23 in \cite{FL1} and observation
\[  
a_1 \tilde P^{(1)}_{l,\mu,\nu} = (l-\nu+\tfrac12) \tilde p^l_{\mu,\nu+\frac12}
\qquad \text{and} \qquad
a_2 \tilde P^{(1)}_{l,\mu,\nu} = (l+\nu+\tfrac12) \tilde p^l_{\mu,\nu-\frac12}.
\]
\end{proof}

\begin{thm}  \label{-n-reg-iso-thm}
  For each $n>0$, the map $\Phi^-_n$ is an isomorphism between representations
  $(\pi^-_n, W^-_n)$ and $(\pi_{nl}, {\cal F}^-_n)$ of
  $\mathfrak{sl}(4,\BB C) = \mathfrak{sl}(2,\HC)$.
\end{thm}

\begin{proof}
  The action of $\mathfrak{sl}(2,\HC)$ on ${\cal F}^-_n$ via $\pi_{nl}$
  is spelled out in Lemma 2 in \cite{nreg}. By equation (16) in \cite{qreg}
  and Lemmas \ref{-intertwiner-lem}, \ref{-Z-mult-lem1}, $\Phi^-_n$ intertwines
  the actions on $W^-_n$ and ${\cal F}^-_n$ of matrices of the form
  $\bigl(\begin{smallmatrix} 0 & B \\ C & 0 \end{smallmatrix}\bigr)
  \in \mathfrak{sl}(2,\HC)$, $B, C \in \HC$.
  Since these matrices generate $\mathfrak{sl}(2,\HC)$, it follows that
  $\Phi^-_n$ is an intertwining map between $(\pi^-_n, W^-_n)$ and
  $(\pi^-_{nl}, {\cal F}^-_n)$.
  It is clear that $\Phi^-_n$ is injective.
  By Theorem 26 in \cite{nreg}, $(\pi_{nl}, {\cal F}^-_n)$ is irreducible.
  Hence $\Phi^-_n$ is onto, and the result follows.
\end{proof}

We conclude with the case of polynomials of negative degrees $W^-_{-n}$, $n>0$.
This space has a basis consisting of monomials of the form
\begin{multline}  \label{-tilde-n-monomial-basis}
  \tilde Q^{(n)}_{l,\mu,\nu}(w_1,w_2,w^*_3,w^*_4) \\
  = (l-\nu)!(l+\nu)! (w_1)^{-(l-\nu+1)}(w_2)^{-(l+\nu+1)}(w^*_3)^{-(l+\tfrac{n}2-\mu+1)}
  (w^*_4)^{-(l+\tfrac{n}2+\mu+1)},
\end{multline}
\[
l = 0, \tfrac12, 1, \tfrac32, \dots, \quad
-l-\tfrac{n}2 \le \mu \le l+\tfrac{n}2, \quad
-l \le \nu \le l, \quad
\mu+\tfrac{n}2,\: \nu \in \BB Z +l.
\]
For a fixed $l$, these monomials span the $K$-type
$V_l \boxtimes V_{l+\frac{n}2}$.
By direct computation we obtain:

\begin{lem}  \label{pi--action-lem2}
For $n<0$, elements
$\bigl(\begin{smallmatrix} 0 & B \\ C & 0 \end{smallmatrix}\bigr)
\in \mathfrak{sl}(4,\BB C)$, $B, C \in \mathfrak{gl}(2,\BB C)$, act on
$W^-_n$ via $\pi^-_n$ as follows:
\begin{align*}
\begin{pmatrix} a^*_1a_3 & a^*_1a_4 \\ a^*_2a_3 & a^*_2a_4 \end{pmatrix}
\tilde Q^{(n)}_{l,\mu,\nu} &= \left(\begin{smallmatrix}
  (l+\tfrac{n}2-\mu+1) \tilde Q^{(n)}_{l+\frac12,\mu-\frac12,\nu-\frac12} &
  (l+\tfrac{n}2+\mu+1) \tilde Q^{(n)}_{l+\frac12,\mu+\frac12,\nu-\frac12} \\
  (l+\tfrac{n}2-\mu+1) \tilde Q^{(n)}_{l+\frac12,\mu-\frac12,\nu+\frac12} &
  (l+\tfrac{n}2+\mu+1) \tilde Q^{(n)}_{l+\frac12,\mu+\frac12,\nu+\frac12}
\end{smallmatrix}\right),  \\
\begin{pmatrix} a^*_3a_1 & a^*_3a_2 \\ a^*_4a_1 & a^*_4a_2 \end{pmatrix}
\tilde Q^{(n)}_{l,\mu,\nu} &= \left(\begin{smallmatrix}
  (l-\nu) \tilde Q^{(n)}_{l-\frac12,\mu+\frac12,\nu+\frac12} &
  (l+\nu) \tilde Q^{(n)}_{l-\frac12,\mu+\frac12,\nu-\frac12} \\
  (l-\nu) \tilde Q^{(n)}_{l-\frac12,\mu-\frac12,\nu+\frac12} &
  (l+\nu) \tilde Q^{(n)}_{l-\frac12,\mu-\frac12,\nu-\frac12} \end{smallmatrix}\right).
\end{align*}
\end{lem}

For each monomial $\tilde Q^{(n)}_{l,\mu,\nu}$, we define a function on
$\HC^{\times}$ with values in
$\underbrace{\BB S' \odot \dots \odot \BB S'}_{\text{$n$ times}}$:
\[
\Psi^-_n \bigl( \tilde Q^{(n)}_{l,\mu,\nu} \bigr) =
\Phi^-_0 \bigl( \underbrace{(a^*_1, a^*_2) \otimes \dots \otimes 
  (a^*_1, a^*_2)}_{\text{$n$ times}} \bigr) \tilde Q^{(n)}_{l,\mu,\nu}.
\]
(Note that applying the operators $a^*_1$ and $a^*_2$ $n$ times to
$\tilde Q^{(n)}_{l,\mu,\nu}$ results in a polynomial of degree zero.)
The same argument as in the proof of Lemma \ref{n-reg--lem} shows:

\begin{lem}
Each function
\[
\Psi^-_n \bigl( \tilde Q^{(n)}_{l,\mu,\nu} \bigr) :
\HC^{\times} \to \underbrace{\BB S' \odot \dots \odot \BB S'}_{\text{$n$ times}}
\]
is right $n$-regular.
\end{lem}

The map $\Psi^-_n$ extends by linearity to a map of vector spaces
\[
\Psi^-_n: W^-_{-n} \to {\cal G}^-_n,
\]
where ${\cal G}^-_n$ is the space of polynomial right $n$-regular functions
on $\HC^{\times}$ introduced in Section 7 of \cite{nreg}.

\begin{lem}  \label{-Z-mult-lem2}
Let $C = \bigl(\begin{smallmatrix} c_{11} & c_{12} \\
  c_{21} & c_{22} \end{smallmatrix}\bigr) \in \HC$,
then, for any $\tilde q \in W^-_{-1}$,
\[
\Phi^-_0 \bigl( (a^*_1,a^*_2) \tilde q \bigr) \cdot ZC
= - \Phi^-_0 \bigl( c_{11} (a^*_3, 0) + c_{12} (0, a^*_3)
+ c_{21} (a^*_4, 0) +c_{22} (0, a^*_4) \bigr) (\tilde q).
\]
\end{lem}

\begin{proof}
The result follows from Lemma 23 in \cite{FL1} and observation
\[  
a^*_3 \tilde Q^{(1)}_{l,\mu,\nu} = \tilde p^l_{\mu+\frac12,\nu}
\qquad \text{and} \qquad
a^*_4 \tilde Q^{(1)}_{l,\mu,\nu} = \tilde p^l_{\mu-\frac12,\nu}.
\]
\end{proof}

\begin{thm}  \label{-n-reg-iso-thm2}
  For each $n>0$, the map $\Psi^-_n$ is an isomorphism between representations
  $(\pi^-_{-n}, W^-_{-n})$ and $(\pi_{nr}, {\cal G}^-_n)$ of
  $\mathfrak{sl}(4,\BB C) = \mathfrak{sl}(2,\HC)$.
\end{thm}

\begin{proof}
  The action of $\mathfrak{sl}(2,\HC)$ on ${\cal G}^-_n$ via $\pi_{nr}$
  is spelled out in Lemma 2 in \cite{nreg}. By equation (16) in \cite{qreg}
  and Lemmas \ref{-intertwiner-lem}, \ref{-Z-mult-lem2}, $\Psi^-_n$ intertwines
  the actions on $W^-_{-n}$ and ${\cal G}^-_n$ of matrices of the form
  $\bigl(\begin{smallmatrix} 0 & B \\ C & 0 \end{smallmatrix}\bigr)
  \in \mathfrak{sl}(2,\HC)$, $B, C \in \HC$.
  Since these matrices generate $\mathfrak{sl}(2,\HC)$, it follows that
  $\Psi^-_n$ is an intertwining map between $(\pi^-_{-n}, W^-_{-n})$ and
  $(\pi_{nr}, {\cal G}^-_n)$.
  It is clear that $\Psi^-_n$ is injective.
  By Theorem 26 in \cite{nreg}, $(\pi_{nr}, {\cal G}^-_n)$ is irreducible.
  Hence $\Psi^-_n$ is onto, and the result follows.
\end{proof}

\subsection{Invariant Bilinear Pairing}

In this subsection we describe an $\mathfrak{sp}(8,\BB C)$-invariant bilinear
pairing between $(\pi^+,W^+)$ and $(\pi^-,W^-)$.
This pairing can be used to construct projectors onto $W^{\pm}_n$.

\begin{thm}
We have an $\mathfrak{sp}(8,\BB C)$-invariant bilinear pairing between
$(\pi^+,W^+)$ and $(\pi^-,W^-)$
\begin{equation}  \label{invar-pairing}
  \langle f, g \rangle = \frac1{(2\pi i)^4} \oint f(iw_1,iw_2,iw^*_3,iw^*_4)
  \cdot g (w_1,w_2,w^*_3,w^*_4) \,dw_1 dw_2 dw^*_3 dw^*_4,
\end{equation}
$f \in W^+$, $g \in W^-$.
The integral is taken over the torus $(S^1)^4$ inside $\BB C^4$, each
circle $S^1$ going once around the origin in the counterclockwise direction.
\end{thm}

\begin{proof}
By Proposition \ref{iso-prop}, $\mathfrak{sp}(8,\BB C)$ is generated
as a vector space by operators $:a_ka^*_l:$, $a_ka_l$, $a^*_ka^*_l$,
$1 \le k,l \le 4$.
Computing the actions of these quadratic operators on monomials
\[
(w_1)^{\alpha_1}(w_2)^{\alpha_2}(w^*_3)^{\alpha_3}(w^*_4)^{\alpha_4}
\qquad \text{and} \qquad
(w_1)^{-(\beta_1+1)}(w_2)^{-(\beta_2+1)}(w^*_3)^{-(\beta_3+1)}(w^*_4)^{-(\beta_4+1)},
\]
$\alpha_1, \alpha_2, \alpha_3, \alpha_4,
\beta_1, \beta_2, \beta_3, \beta_4 \ge 0$,
spanning $W^+$ and $W^-$ respectively, it is immediate that
the bilinear form \eqref{invar-pairing} is $\mathfrak{sp}(8,\BB C)$-invariant.
\end{proof}

\begin{cor}
For each $n \in \BB Z$, this pairing \eqref{invar-pairing} restricts to an
$\mathfrak{sl}(4,\BB C)$-invariant bilinear pairing $\langle f, g \rangle_n$
between $(\pi^+_n, W^+_n)$ and $(\pi^-_{-n}, W^-_{-n})$.
\end{cor}

\begin{rem}
When $n \in \BB Z$ is fixed, in the integral formula \eqref{invar-pairing} for
$\langle f, g \rangle_n$, the integration with respect to one of the variables
$w_1$, $w_2$, $w^*_3$ or $w^*_4$ can be dropped without changing the result.
\end{rem}

Using the pairing \eqref{invar-pairing} we can construct the projectors onto
$W^{\pm}_n$. These projectors can be interpreted as analogues of the
Cauchy-Fueter formula for $n$-regular functions (see Theorem 9 in \cite{nreg}
for a different analogue).

\subsection{Isomorphisms between Left and Right $n$-Regular Functions}

Consider a map $\sigma: W \to W$ defined by switching the variables
$w_1 \leftrightarrow w^*_3$ and $w_2 \leftrightarrow w^*_4$:
\[
(\sigma f)(w_1,w_2,w^*_3,w^*_4) = f(w^*_3,w^*_4,w_1,w_2).
\]
Clearly, $\sigma$ restricts to vector space isomorphisms
$\sigma^+_n: W^+_n \simeq W^+_{-n}$ and $\sigma^-_n: W^-_n \simeq W^-_{-n}$
for all $n \in \BB Z$.

\begin{prop}  \label{left-right-reg-iso1}
The map $\sigma: W \to W$ is an automorphism of representations
$(\pi^+, W^+)$ and $(\pi^-, W^-)$ restricted to
\[
\bigl\{ \bigl(\begin{smallmatrix} X & Y \\ Y & X \end{smallmatrix}\bigr);\:
X, Y \in \mathfrak{sp}(4,\BB C) \bigr\}
\simeq  \mathfrak{sp}(4,\BB C) \otimes
\bigl\{ \bigl(\begin{smallmatrix} x & y \\ y & x \end{smallmatrix}\bigr);\:
x, y \in \BB C \bigr\}.
\]

For each $n \in \BB Z$, the map $\sigma$ restricts to isomorphisms of
representations $\sigma^+_n: (\pi^+_n, W^+_n) \simeq (\pi^+_{-n}, W^+_{-n})$ and
$\sigma^-_n: (\pi^-_n, W^-_n) \simeq (\pi^-_{-n}, W^-_{-n})$ restricted to
$\mathfrak{sp}(4,\BB C)$.
\end{prop}

\begin{proof}
We have the following commutation relations:
\[
\sigma (a_1 f) = a^*_3 (\sigma f), \quad
\sigma (a^*_1 f) = a_3 (\sigma f), \quad
\sigma (a_3 f) = a^*_1 (\sigma f), \quad
\sigma (a^*_3 f) = a_1 (\sigma f),
\]
\[
\sigma (a_2 f) = a^*_4 (\sigma f), \quad
\sigma (a^*_2 f) = a_4 (\sigma f), \quad
\sigma (a_4 f) = a^*_2 (\sigma f), \quad
\sigma (a^*_4 f) = a_2 (\sigma f).
\]
Therefore, operators of the form
\[
a_ka_l+a^*_{k+2}a^*_{l+2}, \quad a^*_ka^*_l+a_{k+2}a_{l+2}, \quad
a_ka_{l+2}+a^*_{k+2}a^*_l,
\]
\[
:a_ka^*_l:+:a_{l+2}a^*_{k+2}:,\quad
a_ka^*_{l+2}+a_la^*_{k+2}, \quad
a_{k+2}a^*_l+a_{l+2}a^*_k , \qquad 1 \le k,l \le 2,
\]
commute with $\sigma$.
These operators act on the space of all linear operators on $W$ by commutators,
and, by Lemma \ref{commutator-lem1}, they preserve a complex vector subspace
$U$ with basis $\{a_1,a_2,a_3,a_4,a^*_3,a^*_4,a^*_1,a^*_2\}$.
Thus, we have a faithful realization of these operators as $8 \times 8$ matrices
of the form $\bigl(\begin{smallmatrix} X & Y \\ Y & X \end{smallmatrix}\bigr)$,
where $X$ and $Y$ have the form
$\bigl( \begin{smallmatrix} A & B \\ C & -A^T \end{smallmatrix}\bigr)$
with $B^T=B$, $C^T=C$. In other words, $X, Y \in \mathfrak{sp}(4,\BB C)$.

The restrictions $\sigma^+_n: W^+_n \simeq W^+_{-n}$ commute with operators
\[
:a_ka^*_l:+:a_{l+2}a^*_{k+2}:,\quad
a_ka^*_{l+2}+a_la^*_{k+2}, \quad
a_{k+2}a^*_l+a_{l+2}a^*_k , \qquad 1 \le k,l \le 2.
\]
These operators correspond to matrices of the form
$\bigl(\begin{smallmatrix} A & 0 \\ 0 & A \end{smallmatrix}\bigr)$ with
$A \in \mathfrak{sp}(4,\BB C)$.

The proof that $\sigma^-: (\pi^-_n, W^-_n) \simeq (\pi^-_{-n}, W^-_{-n})$
when restricted to $\mathfrak{sp}(4,\BB C)$ is similar.
\end{proof}




From Theorems \ref{n-reg-iso-thm}, \ref{n-reg-iso-thm2}, \ref{-n-reg-iso-thm}
and \ref{-n-reg-iso-thm2} we obtain isomorphisms between left and right
$n$-regular functions:

\begin{cor}  \label{left-right-regular-cor}
We have isomorphisms of $\mathfrak{sp}(4,\BB C)$-modules
$(\pi_{nl}, {\cal F}^+_n) \simeq (\pi_{nr}, {\cal G}^+_n)$ and
$(\pi_{nl}, {\cal F}^-_n) \simeq (\pi_{nr}, {\cal G}^-_n)$,
for all positive integers $n$.
\end{cor}

\begin{prop}  \label{nreg-irreducibility-prop}
  After restricting to $\mathfrak{sp}(4,\BB C)$,
  the $n$-regular functions $(\pi_{nl}, {\cal F}^{\pm}_n)$ and
  $(\pi_{nr}, {\cal G}^{\pm}_n)$ remain irreducible, for all $n>0$.
\end{prop}

\begin{proof}
By Theorems \ref{n-reg-iso-thm}, \ref{n-reg-iso-thm2}, \ref{-n-reg-iso-thm}
and \ref{-n-reg-iso-thm2}, it is sufficient to show that 
$(\pi^{\pm}_{\pm n}, W^{\pm}_{\pm n})$ remain irreducible after restricting to
$\mathfrak{sp}(4,\BB C)$.
By Theorems 26,  27 in \cite{nreg}, $(\pi^{\pm}_{\pm n}, W^{\pm}_{\pm n})$ are
irreducible unitary representations of $\mathfrak{u}(2,2)$.
When restricted to $\mathfrak{sp}(4,\BB R) \subset \mathfrak{u}(2,2)$,
they remain unitary.
Consider the case of $(\pi^+_n, W^+_n)$, the other cases are similar.
If $(\pi^+_n, W^+_n)$ is no longer irreducible after the restriction to
$\mathfrak{sp}(4,\BB C)$, it must be a direct sum of two proper
subrepresentations: $W^+_n = U_1 \oplus U_2$.
If $f_1 \in U_1$ and $f_2 \in U_2$ are two non-zero representatives,
applying to $f_1$ and $f_2$ operators belonging to $\mathfrak{sp}(4,\BB C)$
of the form
\[
a_{k+2}a^*_l+a_{l+2}a^*_k , \qquad 1 \le k,l \le 2,
\]
we see that both $U_1$ and $U_2$
contain non-zero polynomials in $w^*_3$ and $w^*_4$ only.
These polynomials must have degree $n$.
However, polynomials of degree $n$ in $w^*_3$ and $w^*_4$ form an irreducible
$U(2)$-type $V_{\frac{n}2}$ of $W^+_n$ and must be a subset of exactly
one of $U_1$, $U_2$.
This contradiction proves that $(\pi^+_n, W^+_n)$ is irreducible
as an $\mathfrak{sp}(4,\BB C)$-module.
\end{proof}

\begin{rem}
When $n=0$, by Propositions \ref{harmonic+iso-prop}, \ref{harmonic-iso-prop},
$(\pi^+_0, W^+_0) \simeq (\pi^0_l, {\cal H}^+)$
and $(\pi^-_0, W^-_0) \simeq (\pi^0_l, {\cal H}^-)$
-- the harmonic functions regular at the origin and infinity respectively.
Restricting these to $\mathfrak{sp}(4,\BB C)$ results in {\em reducible}
modules.
\end{rem}

\subsection{Metaplectic Representation and Twistors}

We conclude this section with a brief outline the relation of the metaplectic
representation of $\mathfrak{sp}(8,\BB C)$ with twistor theory.
For more details see, for example, \cite{BE, HT, WW} and references therein.
Start with an all-encompassing space of Laurent polynomials
\[
W = \BB C \bigl[ w_1,w_2,w^*_3,w^*_4,
  w_1^{-1},w_2^{-1},(w^*_3)^{-1},(w^*_4)^{-1} \bigr],
\]
a subspace of polynomials
\[
W^+ = \BB C \bigl[ w_1,w_2,w^*_3,w^*_4 \bigr]
\]
and a quotient space
\[
W^- = w_1^{-1} w_2^{-1} (w^*_3)^{-1} (w^*_4)^{-1} \cdot
\BB C \bigl[ w_1^{-1},w_2^{-1},(w^*_3)^{-1},(w^*_4)^{-1} \bigr].
\]
Express these spaces as
\begin{align*}
W^+ = V_+ \otimes V^*_+, \qquad
&V_+ = \BB C \bigl[ w_1,w_2], \qquad
V^*_+ = \BB C \bigr[w^*_3,w^*_4 \bigr],  \\
W^- = V_- \otimes V^*_-, \qquad
&V_- = w_1^{-1} w_2^{-1} \cdot \BB C \bigl[ w_1^{-1}, w_2^{-1}],  \\
\hphantom{W^- = V_- \otimes V^*_-,} \qquad
&V^*_- = (w^*_3)^{-1} (w^*_4)^{-1} \cdot
\BB C \bigr[(w^*_3)^{-1},(w^*_4)^{-1} \bigr].
\end{align*}
Note that one could also introduce
\[
W^* = \BB C \bigl[ w^*_1,w^*_2,w_3,w_4,
  (w^*_1)^{-1},(w^*_2)^{-1},w_3^{-1},w_4^{-1} \bigr],
\]
and similar sub and quotient spaces
\begin{align*}
(W^*)^+ &= \BB C \bigl[ w^*_1,w^*_2,w_3,w_4 \bigr],  \\
(W^*)^- &= (w^*_1)^{-1} (w^*_2)^{-1} w_3^{-1} w_4^{-1} \cdot
\BB C \bigl[ (w^*_1)^{-1},(w^*_2)^{-1},w_3^{-1},w_4^{-1} \bigr].
\end{align*}
Then we have decompositions
\[
W^+ = \bigoplus_{n \in \BB Z} W^+_n \qquad \text{and} \qquad
W^- = \bigoplus_{n \in \BB Z} W^-_n,
\]
where $W^{\pm}_n$ consist of homogeneous polynomials of degree $n$
and realize irreducible representations belonging to the minimal discrete
series of $SU(2,2)$.
Decompositions of $(W^*)^+$ and $(W^*)^-$  are similar.

Now, reshuffle these Laurent polynomials and consider
\[
H = \BB C \bigl[ w_1,w_2,w_3,w_4,
  w_1^{-1},w_2^{-1},w_3^{-1},w_4^{-1} \bigr]
\]
together with subquotients
\begin{align*}
H^+ &=
w_3^{-1}w_4^{-1} \cdot \BB C \bigl[ w_1,w_2,w_3^{-1},w_4^{-1} \bigr],  \\
H^- &=
w_1^{-1}w_2^{-1} \cdot \BB C \bigl[ w_1^{-1},w_2^{-1},w_3,w_4 \bigr].
\end{align*}
Additionally, introduce the dual space
\[
H^* = \BB C \bigl[ w^*_1,w^*_2,w^*_3,w^*_4,
  (w^*_1)^{-1},(w^*_2)^{-1},(w^*_3)^{-1},(w^*_4)^{-1} \bigr]
\]
together with subquotients
\begin{align*}
(H^*)^+ &= (w^*_3)^{-1}(w^*_4)^{-1} \cdot
\BB C \bigl[ w^*_1,w^*_2,(w^*_3)^{-1},(w^*_4)^{-1} \bigr],  \\
(H^*)^- &= (w^*_1)^{-1}(w^*_2)^{-1} \cdot
\BB C \bigl[ (w^*_1)^{-1},(w^*_2)^{-1},w^*_3,w^*_4 \bigr].
\end{align*}
Then we have decompositions
\[
H^+ = \bigoplus_{m \in \BB Z} H^+_m \qquad \text{and} \qquad
H^- = \bigoplus_{m \in \BB Z} H^-_m,
\]
and similar decompositions of $(H^*)^+$ and $(H^*)^-$,
where $H^{\pm}_m$ consist of homogeneous polynomials of degree $m$.

On the other hand, in twistor theory we have a realization of
representations of $SU(2,2)$ via the sheaf cohomology.
Recall the standard Penrose notations
\[
\BB T = \BB C^4, \qquad \BB T^* = (\BB C^4)^*,  \qquad
\BB P = \BB CP^3=\BB P \BB T,
\]
\[
\BB P^+= \BB P^I = \BB P \setminus S^2_{\infty}, \qquad
\BB P^-= \BB P^J = \BB P \setminus S^2_0.
\]
The Klein correspondence produces two fiber bundles
\[
\begin{tikzcd}
\BB C P^1 \ar[r,hook] & \BB P^+ \ar[d,->>]  \\ & \BB H
\end{tikzcd}
\qquad \text{and} \qquad
\begin{tikzcd}
\BB C P^1 \ar[r,hook] & \BB P^-
\ar[d,->>]  \\ & \BB H^{\times} \cup \{\infty\}.
\end{tikzcd}
\]
Then the degenerate series of representations of $SU(2,2)$ can be realized as
\[
H^1(\BB P^+, {\cal O}(m)) \qquad \text{and} \qquad
H^1(\BB P^-, {\cal O}(m)),  \qquad m \in \BB Z.
\]
In this context the sheaf cohomology is computed via the \v{C}ech
cohomology approach, and in our case $\BB P^+$, $\BB P^-$ are covered
by just two open sets, each open set being diffeomorphic to $\BB C^3$.
The representatives of the basis of cocycles of the \v{C}ech cohomology are
precisely the spaces $H^+_m$ and $H^-_m$.
To get the precise relation between the metaplectic representation and
twistor theory, we note the degree shift in $H^{\pm}_m$ by $2$ due to the
factors $(w^*_3)^{-1}(w^*_4)^{-1}$ and $(w^*_1)^{-1}(w^*_2)^{-1}$.
Explicitly, for $n >0$,
by Theorems \ref{n-reg-iso-thm} and \ref{n-reg-iso-thm2}, we have:
\begin{align*}
W^+_n &= \{\text{left $n$-regular functions regular at $0$}\},  \\
W^-_n &= \{\text{left $n$-regular functions regular at $\infty$}\}
\end{align*}
in the metaplectic realization match respectively
\begin{align*}
H^1(\BB P^+, {\cal O}(-n-2)) &= H^+_{-n-2}
= \Bigl\{\begin{smallmatrix} \text{spin $\frac{n}2$ zero mass} \\
\text{states (particles) on $\BB P^+$} \end{smallmatrix}\Bigr\},  \\
H^1(\BB P^-, {\cal O}(-n-2)) &= H^-_{-n-2}
= \Bigl\{\begin{smallmatrix} \text{spin $\frac{n}2$ zero mass} \\
\text{states (particles) on $\BB P^-$} \end{smallmatrix}\Bigr\}
\end{align*}
in the \v{C}ech cohomology realization.
Similarly, by Theorems \ref{-n-reg-iso-thm} and \ref{-n-reg-iso-thm2},
\begin{align*}
W^+_{-n} &= \{\text{right $n$-regular functions regular at $0$}\},  \\
W^-_{-n} &= \{\text{right $n$-regular functions regular at $\infty$}\}
\end{align*}
in the metaplectic realization match respectively
\begin{align*}
H^1(\BB P^+, {\cal O}(n-2)) &= H^+_{n-2} =
\Bigl\{\begin{smallmatrix} \text{spin $-\frac{n}2$ zero mass}  \\
\text{states (particles) on $\BB P^+$} \end{smallmatrix}\Bigr\},  \\
H^1(\BB P^-, {\cal O}(n-2)) &= H^-_{n-2} =
\Bigl\{\begin{smallmatrix} \text{spin $-\frac{n}2$ zero mass}  \\
\text{states (particles) on $\BB P^-$} \end{smallmatrix}\Bigr\}
\end{align*}
in the \v{C}ech cohomology realization.
The spin $0$ case corresponds to the harmonic functions
(Propositions  \ref{harmonic+iso-prop} and \ref{harmonic-iso-prop}).
This is why in twistor theory there always is a shift by $2$ in the line bundles
${\cal O}(\mp n-2)$.


Also, one can work with $\BB P^* =\BB P \BB T^*$ instead of $\BB P$,
then one would get another realization of the $n$-regular functions.
(This corresponds to switching from $W$ to $W^*$.)
Thus, our metaplectic construction and twistor realization are complementary
to each other.

The relation between our treatment of the metaplectic representation and
twistor theory is beneficial for both: Our approach yields explicit formulas
for the action of $\mathfrak{sl}(4,\BB C)$ in the \v{C}ech cohomology basis.
On the other hand, twistor theory suggests a more invariant realization of the
metaplectic representation that does not require a specific choice of
coordinates and bases.

\section{Restriction of Geometric Constructions to $\mathfrak{sp}(4,\BB C)$}  \label{4}

In this section we start developing quaternionic analysis using our
new approach based on the conformal group symmetry reduction to
$Sp(4,\BB R)$.

\subsection{Basics}

We regard the group $U(2,2)$ as a subgroup of $GL(2,\HC)$,
as described in Subsection 3.5 of \cite{FL1}. That is,
$U(2,2)$ is realized as the subgroup of $GL(2,\HC)$ consisting of elements
preserving the Hermitian form on $\BB C^4$ given by the $4 \times 4$ matrix
$\bigl(\begin{smallmatrix} 1 & 0 \\ 0 & -1 \end{smallmatrix}\bigr)$.
Explicitly:
\begin{multline}  \label{U(2,2)}
U(2,2) = \Bigl\{ \bigl(\begin{smallmatrix} a & b \\
  c & d \end{smallmatrix}\bigr) \in GL(2,\HC) ;\:
\bigl(\begin{smallmatrix} a & b \\ c & d \end{smallmatrix}\bigr)
\bigl(\begin{smallmatrix} 1 & 0 \\ 0 & -1 \end{smallmatrix}\bigr)
\bigl(\begin{smallmatrix} a^* & c^* \\ b^* & d^* \end{smallmatrix}\bigr)
= \bigl(\begin{smallmatrix} 1 & 0 \\ 0 & -1 \end{smallmatrix}\bigr) \Bigr\}  \\
= \biggl\{ \bigl(\begin{smallmatrix} a & b \\ c & d \end{smallmatrix}\bigr)
\in GL(2,\HC) ;\: a,b,c,d \in \HC,\:
\begin{smallmatrix} a^*a = 1+c^*c \\ d^*d = 1+b^*b \\ a^*b=c^*d
\end{smallmatrix} \biggr\}.
\end{multline}
The group $U(2,2)$ acts on $\HC$ (with singularities) by fractional
linear transformations $Z \mapsto (aZ+b)(cZ+d)^{-1}$
preserving $U(2) \subset \HC$ and open domains
\[
\BB D^+ = \{ Z \in \HC;\: ZZ^*<1 \}, \qquad
\BB D^- = \{ Z \in \HC;\: ZZ^*>1 \},
\]
where the inequalities $ZZ^*<1$ and $ZZ^*>1$ mean that the self-adjoint
matrix $ZZ^*-1$ is negative and positive definite respectively.
The Lie algebra of $U(2,2)$ is
\begin{equation}  \label{u(2,2)-algebra}
\mathfrak{u}(2,2) = \Bigl\{
\bigl(\begin{smallmatrix} A & B \\ B^* & D \end{smallmatrix}\bigr) ;\:
A,B,D \in \HC ,\: A=-A^*, D=-D^* \Bigr\}.
\end{equation}

A ``standard'' realization of the Lie algebra $\mathfrak{sp}(4,\BB R)$ is as
the subalgebra of $\mathfrak{gl}(4,\BB R)$ preserving the symplectic form
$\bigl( \begin{smallmatrix} 0 & 1 \\ -1 & 0 \end{smallmatrix}\bigr)$:
\[
\mathfrak{sp}(4,\BB R)' = \Bigl\{
\bigl(\begin{smallmatrix} A & B \\ C & -A^T \end{smallmatrix}\bigr) ;\:
A,B,C \in \mathfrak{gl}(2,\BB R) ,\: B^T=B, C^T=C \Bigr\}.
\]
Consider an element
$\frac1{\sqrt{2}} \bigl(\begin{smallmatrix} -i & -1 \\
  1 & i \end{smallmatrix}\bigr) \in GL(2,\HC)$
with inverse $\frac1{\sqrt{2}} \bigl(\begin{smallmatrix} i & 1 \\
  -1 & -i \end{smallmatrix}\bigr)$
and a conjugate copy of $\mathfrak{sp}(4,\BB R)$:
\begin{align*}
\mathfrak{sp}(4,\BB R)
&= \tfrac12 \bigl(\begin{smallmatrix} -i & -1 \\ 1 & i \end{smallmatrix}\bigr)
\mathfrak{sp}(4,\BB R)'
\bigl(\begin{smallmatrix} i & 1 \\ -1 & -i \end{smallmatrix}\bigr)  \\
&= \Bigl\{\bigl(\begin{smallmatrix} A & B \\
  \bar B & -A^T \end{smallmatrix}\bigr) ;\:
A,B \in \HC ,\: A^*=-A, B^T=B \Bigr\}.
\end{align*}
This realizes $\mathfrak{sp}(4,\BB R)$ as a subalgebra of $\mathfrak{u}(2,2)$.
Then the action of the corresponding Lie group $Sp(4,\BB R) \subset U(2,2)$
by fractional linear transformations preserves the space of symmetric matrices
$\HC \text{-sym}$ defined by \eqref{H-sym}.
Complexify $\mathfrak{sp}(4,\BB R)$:
\[
\mathfrak{sp}(4,\BB C) = \BB C \otimes \mathfrak{sp}(4,\BB R)
= \Bigl\{\bigl(\begin{smallmatrix} A & B \\ C & -A^T \end{smallmatrix}\bigr);\:
A,B,C \in \HC ,\: B^T=B,\: C^T=C \Bigr\}.
\]
We realize $\mathfrak{u}(2)$ as diagonal elements of $\mathfrak{sp}(4,\BB R)$:
\begin{equation}  \label{u(2)}
\mathfrak{u}(2) = \Bigl\{ \bigl(\begin{smallmatrix} A & 0 \\
0 & -A^T \end{smallmatrix}\bigr) \in \mathfrak{sp}(4,\BB R);\:
A \in \HC ,\: A^*=-A \Bigr\}.
\end{equation}
The corresponding group is
\[
U(2) = \Bigl\{ \bigl(\begin{smallmatrix} a & 0 \\
0 & a^{-1T} \end{smallmatrix}\bigr) \in U(2,2);\: a \in \HC,\: aa^*=1 \Bigr\}.
\]
We frequently decompose representations of $\mathfrak{sp}(4,\BB C)$
into $K$-types -- representations of either group $U(2)$ or its Lie algebra
\eqref{u(2)}.

\subsection{Solid Harmonic Functions $R^m_l(Z)$}

In this paper we consider various $\mathfrak{sp}(4,\BB C)$-modules,
and their $K$-types will be described using the solid harmonic functions
$R^m_l(Z)$. These in turn are defined in terms of the associated Legendre
polynomials.

We use integer parameters $l, m$ with $l \ge 0$.
Recall that the {\em associated Legendre polynomials} can be defined as
(see, for example, \cite{GR,V})
\[
P^m_l(x) = \frac{(-1)^m}{2^l l!} (1-x^2)^{\frac{m}2} \frac{d^{l+m}}{dx^{l+m}}
\bigl[ (x^2-1)^l \bigr],
\qquad -l \le m \le l.
\]
These are $(1-x^2)^{\frac{m}2}$ times a polynomial of degree $l-m$ that is
either even or odd depending on the parity of $l+m$. We have:
\begin{equation}  \label{Pm-m}
P^{-m}_l(x) = (-1)^m \tfrac{(l-m)!}{(l+m)!} P^m_l(x)
\end{equation}
(identity 8.737(1) in \cite{GR}).
The orthogonality relation
\[
\int_{-1}^1 P^m_l(x) \cdot P^{-m}_{l'}(x)\,dx = (-1)^m \frac2{2l+1} \delta_{ll'}
\]
(identity 7.112(1) in \cite{GR} can be rewritten as
\begin{equation}  \label{P-orthogonality}
\int_0^{\pi} P^m_l(\cos\theta) \cdot P^{-m}_{l'}(\cos\theta) \cdot
\sin\theta\,d\theta = (-1)^m \frac2{2l+1} \delta_{ll'}.
\end{equation}

We consider the 3-dimensional space of symmetric quaternions
\[
\BB H \text{-sym} = \{ Z = x^0e_0+x^1e_1+x^3e_3;\: x^0,x^1,x^3 \in \BB R \}.
\]
In matrix realization \eqref{H-matrix-realization}, elements of
$\BB H \text{-sym}$ appear as
\[
Z = \begin{pmatrix} z & it \\ it & \bar z \end{pmatrix},
\qquad z=x+iy,\: \bar z=x-iy \in \BB C,\: t \in \BB R.
\]
We use the spherical coordinates on $\BB H \text{-sym}$:
\[
Z = \begin{pmatrix} z & it \\ it & \bar z \end{pmatrix}
= r \begin{pmatrix} \sin\theta e^{i\phi} & i\cos\theta \\
  i\cos\theta & \sin\theta e^{-i\phi} \end{pmatrix},
\qquad r \ge 0,\: 0 \le \theta \le \pi,\: 0 \le \phi \le 2\pi.
\]
The {\em solid harmonic polynomials} can be defined in terms of the
{\em spherical harmonics} $Y^m_l(\theta,\phi)$:
\[
R^m_l(Z) = \text{(constant factor)} \cdot r^l \cdot Y^m_l(\theta,\phi)
\]
(see, for example, \cite{H}).
However, there are different conventions for the constant factor.
For this reason, we define the solid harmonic polynomials as
\[
R^m_l(Z) = 
r^l \cdot P^m_l(\cos\theta) e^{im\phi}
= (|z|^2+t^2)^{\frac{l}2} \cdot \Bigl( \frac{z}{|z|} \Bigr)^m \cdot
P^m_l \bigl( t(|z|^2+t^2)^{-\frac12} \bigr).
\]
Observe that
\[
r^{l-|m|} \cdot P^m_l(\cos\theta)
  = (\sin\theta)^{|m|} \cdot \text{homogeneous polynomial in $x$, $y$, $t$
    of degree $l-|m|$}
\]
and  
\[
r^{|m|} (\sin\theta)^{|m|} e^{im\phi} =
\begin{cases} (x+iy)^m & \text{if $m \ge 0$}, \\
(x-iy)^{|m|} & \text{if $m \le 0$}. \end{cases}
\]
Hence $R^m_l(Z)$ is a homogeneous polynomial in $x,y,t$ or
$z_{11}, z_{12}, z_{22}$ of degree $l$. These polynomials are harmonic:
\begin{align*}
\Delta_3 R^m_l(Z) &= 0, \qquad \text{where } \\
\Delta_3 &= \bigl(\tfrac{\partial}{\partial x}+i\tfrac{\partial}{\partial y} \bigr)
\bigl(\tfrac{\partial}{\partial x}-i\tfrac{\partial}{\partial y} \bigr)
+ \tfrac{\partial^2}{\partial t^2}
= 4 \tfrac{\partial}{\partial z_{11}} \tfrac{\partial}{\partial z_{22}}
+ \tfrac{\partial^2}{\partial t^2}.
\end{align*}

Relation \eqref{Pm-m} implies that
\[
R^m_l(Z^+) = (-1)^l \tfrac{(l+m)!}{(l-m)!} R^{-m}_l(Z), \qquad
R^m_l(Z^{-1}) = (-1)^l \tfrac{(l+m)!}{(l-m)!} N(Z)^{-l} \cdot R^{-m}_l(Z).
\]

Let $S^2_r = \{ X \in \BB H;\: X^T=X,\: N(X)=r^2 \}$
be the sphere of radius $r$ centered at the origin, and let
$dS = r^2 \sin\theta \,d\theta d\phi$ be the standard Euclidean area element
on that sphere. Then \eqref{P-orthogonality} implies orthogonality relations
for the solid harmonic polynomials:
\begin{equation}  \label{R-orthogonality}
\frac1{4\pi}\iint_{Z \in S^2_r} R^m_l(Z) \cdot R^{-m'}_{l'}(Z) \,dS
= (-1)^m \frac{r^{2l+2}}{2l+1} \delta_{ll'} \delta_{mm'}.
\end{equation}

Recall the symmetric quaternions $\HC\text{-sym}$, $\HC^{\times}\text{-sym}$
defined by \eqref{H-sym}-\eqref{H-sym-x}.

\begin{prop}  \label{basis-prop}
The functions
\[
N(Z)^k \cdot R^m_l(Z), \qquad
k,l,m \in \BB Z, \quad k,l\ge 0, \quad -l \le m \le l,
\]
form a basis of $\BB C [z_{11}, z_{12}, z_{22}]$
-- polynomials on $\HC \text{-sym}$.

And the functions
\[
N(Z)^k \cdot R^m_l(Z), \qquad
k,l,m \in \BB Z, \quad l\ge 0, \quad -l \le m \le l,
\]
form a basis of $\BB C [z_{11}, z_{12}, z_{22}, N(Z)^{-1}]$
-- polynomials on $\HC^{\times} \text{-sym}$.
\end{prop}

We express the products $t \cdot R^m_l(Z)$, $z_{11} \cdot R^m_l(Z)$ and
$z_{22} \cdot R^m_l(Z)$ in terms of these basis functions.
Let $s=t/r$. Identity 8.731(2) in \cite{GR}
\[
(l-m+1) P^m_{l+1}(s) = (2l+1) sP^m_l(s) - (l+m) P^m_{l-1}(s)
\]
implies that
\begin{equation}  \label{tR}
t \cdot R^m_l(Z) = \frac{l-m+1}{2l+1} R^m_{l+1}(Z)
+ \frac{l+m}{2l+1} N(Z) \cdot R^m_{l-1}(Z).
\end{equation}
Identity 8.734(2) in \cite{GR}
\[
\sqrt{1-s^2} P^m_l(s) =
-\frac1{2l+1} \bigl( P^{m+1}_{l+1}(s) - P^{m+1}_{l-1}(s) \bigr)
\]
implies that
\begin{equation}  \label{z_{11}R}
z_{11} \cdot R^m_l(Z) =
-\frac1{2l+1} R^{m+1}_{l+1}(Z) + \frac1{2l+1} N(Z) \cdot R^{m+1}_{l-1}(Z).
\end{equation}
Identity 8.735(5) in \cite{GR}
\[
\sqrt{1-s^2} P^m_l(s) = \frac1{2l+1} \bigl( (l-m+1)(l-m+2)P^{m-1}_{l+1}(s)
- (l+m-1)(l+m)P^{m-1}_{l-1}(s) \bigr)
\]
implies that
\begin{equation}  \label{z_{22}R}
z_{22} \cdot R^m_l(Z) = \frac{(l-m+1)(l-m+2)}{2l+1} R^{m-1}_{l+1}(Z)
- \frac{(l+m)(l+m-1)}{2l+1} N(Z) \cdot R^{m-1}_{l-1}(Z).
\end{equation}


Next, we express the derivatives $\frac{\partial}{\partial t} R^m_l(Z)$,
$\frac{\partial}{\partial z_{11}} R^m_l(Z)$ and
$\frac{\partial}{\partial z_{22}} R^m_l(Z)$ in terms of these basis functions.
Using identity 8.731(1) in \cite{GR}
\[
(s^2-1) \frac{d}{ds} P^m_l(s) = ls P^m_l(s) - (l+m) P^m_{l-1}(s)
\]
and writing $\cos\theta = t/r$ imply that
\begin{equation}  \label{dtR}
\frac{\partial}{\partial t} R^m_l(Z) = (l+m) R^m_{l-1}(Z).
\end{equation}
Identities 8.735(2) and 8.733(1) in \cite{GR}
\begin{align*}
\sqrt{1-s^2} P^m_l(s) &= (l-m+1) s P^{m-1}_l(s) - (l+m-1) P^{m-1}_{l-1}(s),  \\
(1-s^2) \frac{d}{ds} P^m_l(s) &=
(l+m)(l-m+1) \sqrt{1-s^2} P^{m-1}_l(s) + msP^m_l(s)
\end{align*}
together imply
\begin{multline}  \label{dz_{11}R}
2 \frac{\partial}{\partial z_{11}} R^m_l(Z) = 
\Bigl(\frac{\partial}{\partial x}-i\frac{\partial}{\partial y} \Bigr) R^m_l(Z)\\
= r^{l-1} \cdot e^{i(m-1)\phi} \Bigl(
l \sqrt{1-s^2} P^m_l(s) + \frac{m}{\sqrt{1-s^2}} P^m_l(s)
- s \sqrt{1-s^2} (P^m_l)'(s) \Bigr)  \\
= r^{l-1} \cdot e^{i(m-1)\phi} \Bigl(
l \sqrt{1-s^2} P^m_l(s) + \frac{m}{\sqrt{1-s^2}} P^m_l(s)  \\
\hspace{1.5in} - \frac{s}{\sqrt{1-s^2}}
\bigl( (l+m)(l-m+1) \sqrt{1-s^2} P^{m-1}_l(s) + msP^m_l(s) \bigr) \Bigr)  \\
= (l+m) r^{l-1} \cdot e^{i(m-1)\phi} \bigl(
\sqrt{1-s^2} P^m_l(s) - (l-m+1)sP^{m-1}_l(s)  \bigr)  \\
= -(l+m)(l+m-1) r^{l-1} \cdot e^{i(m-1)\phi} \cdot P^{m-1}_{l-1}(t/r)
= -(l+m)(l+m-1) R^{m-1}_{l-1}(Z).
\end{multline}
Identities 8.735(3) and 8.733(1) in \cite{GR}
\begin{align*}
P^m_{l-1}(s) - s P^m_l(s) &= (l-m+1) \sqrt{1-s^2} P^{m-1}_l(s),  \\
(1-s^2) \frac{d}{ds} P^m_l(s) &=
-\sqrt{1-s^2} P^{m+1}_l(s) - msP^m_l(s)
\end{align*}
together imply
\begin{multline}  \label{dz_{22}R}
2 \frac{\partial}{\partial z_{22}} R^m_l(Z) =
\Bigl(\frac{\partial}{\partial x}+i\frac{\partial}{\partial y} \Bigr) R^m_l(Z)\\
= r^{l-1} \cdot e^{i(m+1)\phi} \Bigl(
l \sqrt{1-s^2} P^m_l(s) - \frac{m}{\sqrt{1-s^2}} P^m_l(s)
- s \sqrt{1-s^2} (P^m_l)'(s) \Bigr)  \\
= r^{l-1} \cdot e^{i(m+1)\phi} \Bigl(
l \sqrt{1-s^2} P^m_l(s) - \frac{m}{\sqrt{1-s^2}} P^m_l(s)  \\
\hspace{2in} + \frac{s}{\sqrt{1-s^2}}
\bigl( \sqrt{1-s^2} P^{m+1}_l(s) + msP^m_l(s) \bigr) \Bigr)  \\
= r^{l-1} \cdot e^{i(m+1)\phi} \bigl(
(l-m) \sqrt{1-s^2} P^m_l(s) + sP^{m+1}_l(s)  \bigr)  \\
= r^{l-1} \cdot e^{i(m+1)\phi} \cdot P^{m+1}_{l-1}(t/r)
= R^{m+1}_{l-1}(Z).
\end{multline}
Combining the formulas \eqref{tR}-\eqref{dz_{22}R}, we obtain expressions for
the partial derivatives of the basis functions in Proposition \ref{basis-prop}:
\begin{equation}  \label{dtNR}
\tfrac{\partial}{\partial t} \bigl( N(Z)^k \cdot R^m_l(Z) \bigr)
= \tfrac{(2k+2l+1)(l+m)}{2l+1} N(Z)^k \cdot R^m_{l-1}(Z)
+ \tfrac{2k(l-m+1)}{2l+1} N(Z)^{k-1} \cdot R^m_{l+1}(Z),
\end{equation}
\begin{multline}  \label{dz_{11}NR}
2 \tfrac{\partial}{\partial z_{11}} \bigl( N(Z)^k \cdot R^m_l(Z) \bigr)
= - \tfrac{(2k+2l+1)(l+m)(l+m-1)}{2l+1} N(Z)^k \cdot R^{m-1}_{l-1}(Z)  \\
+ \tfrac{2k(l-m+1)(l-m+2)}{2l+1} N(Z)^{k-1} \cdot R^{m-1}_{l+1}(Z),
\end{multline}
\begin{equation}  \label{dz_{22}NR}
2 \tfrac{\partial}{\partial z_{22}} \bigl( N(Z)^k \cdot R^m_l(Z) \bigr)
= \tfrac{2k+2l+1}{2l+1} N(Z)^k \cdot R^{m+1}_{l-1}(Z)
- \tfrac{2k}{2l+1} N(Z)^{k-1} \cdot R^{m+1}_{l+1}(Z).
\end{equation}


\subsection{Representation $(\rho_M,M)$}  \label{rho_M-subsection}

In this subsection we construct a particular $\mathfrak{sp}(4,\BB C)$-module
$(\rho_M,M)$ that will be crucial to our construction of
$\mathfrak{sp}(4,\BB C)$-invariant trilinear forms in Section \ref{7}.

Let $\alpha, \beta \in \BB Z$, and consider actions of $Sp(4,\BB C)$
on the space of $\BB C$-valued functions on $\HC^{\times} \text{-sym}$:
\[
\rho_{\alpha,\beta}(h): \: f(Z) \: \mapsto \:
\bigl( \rho_{\alpha,\beta}(h)f \bigr)(Z) =
\frac{f \bigl( (aZ+b)(cZ+d)^{-1} \bigr)}
{N(cZ+d)^{\alpha} \cdot N(a'-Zc')^{\beta}},
\]
$h = \bigl( \begin{smallmatrix} a' & b' \\ c' & d' \end{smallmatrix} \bigr)
  \in Sp(4,\BB R) \subset GL(2,\HC)$ and
$h^{-1} = \bigl( \begin{smallmatrix} a & b \\ c & d \end{smallmatrix} \bigr)$.
Differentiating, we obtain actions $\rho_{\alpha,\beta}$ of the Lie algebra
$\mathfrak{sp}(4,\BB C)$.
Since we are dealing with symmetric matrices, we redefine
\[
\partial = \begin{pmatrix}
  \frac{\partial}{\partial z_{11}} &  -\frac{i}2 \frac{\partial}{\partial t} \\
  -\frac{i}2 \frac{\partial}{\partial t} & \frac{\partial}{\partial z_{22}}
\end{pmatrix}.
\]

\begin{lem}  \label{rho-alpha-beta-action}
The Lie algebra actions $\rho_{\alpha,\beta}$
of $\mathfrak{sp}(4,\BB C)$ are given by
\begin{align*}
\rho_{\alpha,\beta}
\bigl( \begin{smallmatrix} A & 0 \\ 0 & -A^T \end{smallmatrix} \bigr) &:
f(Z) \mapsto -\tr\bigl(AZ \partial + ZA^T\partial + (\alpha+\beta) A\bigr) f, \\
\rho_{\alpha,\beta}
\bigl( \begin{smallmatrix} 0 & B \\ 0 & 0 \end{smallmatrix} \bigr) &:
f(Z) \mapsto - \tr (B \partial) f,  \\
\rho_{\alpha,\beta}
\bigl( \begin{smallmatrix} 0 & 0 \\ C & 0 \end{smallmatrix} \bigr) &:
f(Z) \mapsto \tr \bigl(ZCZ \partial + (\alpha+\beta) CZ\bigr) f.
\end{align*}
In particular, $\rho_{\alpha,\beta}$ and $\rho_{\alpha',\beta'}$ describe the same
actions of $\mathfrak{sp}(4,\BB C)$ and $Sp(4,\BB R)$ if and only if
$\alpha+\beta=\alpha'+\beta'$.
\end{lem}
  
\begin{cor}
The spaces
\[
\BB C [z_{11}, z_{12}, z_{22}, N(Z)^{-1}]
\qquad \text{and} \qquad
N(Z)^{\frac12} \cdot \BB C [z_{11}, z_{12}, z_{22}, N(Z)^{-1}]
\]
are invariant under the $\rho_{\alpha,\beta}$ actions of $\mathfrak{sp}(4,\BB C)$.
\end{cor}

\begin{rem}
Technically, $N(Z)^{\frac12}$ is not a well-defined function on
$\HC^{\times} \text{-sym}$. We can think of 
$N(Z)^{\frac12} \cdot \BB C [z_{11}, z_{12}, z_{22}, N(Z)^{-1}]$
as $\BB C$-valued functions on a subset of the classical quaternions $\BB H$
\[
\BB H^{\times} \text{-sym} =
\{ Z = x^0e_0+x^1e_1+x^3e_3;\: Z \ne 0,\: x^0,x^1,x^3 \in \BB R \}
\]
or a suitable open neighborhood of $\BB H^{\times} \text{-sym}$ in
$\HC^{\times} \text{-sym}$.
\end{rem}

We have an analogue of Theorems 31-32 in \cite{ATMP} (for convenience they
are restated as Theorems \ref{rho-decomposition}-\ref{rho'-decomposition} here):

\begin{thm}  \label{rho-alpha-beta-thm}
The only proper $\mathfrak{sp}(4,\BB C)$-invariant subspaces of
$\bigl(\rho_{\alpha,\beta},\BB C [z_{11}, z_{12}, z_{22}, N(Z)^{-1}] \bigr)$ are
\begin{align*}
&\BB C\text{-span of }
\bigl\{ N(Z)^k \cdot R^m_l(Z);\: k \ge 0 \bigr\}, \\
&\BB C\text{-span of }
\bigl\{ N(Z)^k \cdot R^m_l(Z);\: k \le -(l+\alpha+\beta) \bigr\}, \\
&\BB C\text{-span of }
\bigl\{ N(Z)^k \cdot R^m_l(Z);\: 0 \le k \le -(l+\alpha+\beta)\bigr\}
\qquad \text{if $\alpha+\beta \le 0$},
\end{align*}
and the sum of the first two subspaces (see Figure \ref{decomposition-fig1}).

The only proper $\mathfrak{sp}(4,\BB C)$-invariant subspaces of
$\bigl(\rho_{\alpha,\beta},
N(Z)^{\frac12} \cdot \BB C [z_{11}, z_{12}, z_{22}, N(Z)^{-1}] \bigr)$ are
\begin{align*}
&\BB C\text{-span of }
\bigl\{ N(Z)^k \cdot R^m_l(Z);\: k \ge -(l+\tfrac12) \bigr\}, \\
&\BB C\text{-span of }
\bigl\{ N(Z)^k \cdot R^m_l(Z);\: k \le \tfrac12-\alpha-\beta \bigr\}, \\
&\BB C\text{-span of } \bigl\{ N(Z)^k \cdot R^m_l(Z);\:
-(l+\tfrac12) \le k \le \tfrac12 -\alpha-\beta \bigr\}
\end{align*}
and -- when $\alpha+\beta \ge 3$ --  the sum of the first two subspaces
(see Figure \ref{decomposition-fig2}).
\end{thm}

\begin{figure}
\begin{center}
\setlength{\unitlength}{1mm}
\begin{picture}(120,60)
\multiput(10,7)(10,0){11}{\circle*{1}}
\multiput(10,17)(10,0){11}{\circle*{1}}
\multiput(10,27)(10,0){11}{\circle*{1}}
\multiput(10,37)(10,0){11}{\circle*{1}}
\multiput(10,47)(10,0){11}{\circle*{1}}

\thicklines
\put(60,0){\vector(0,1){60}}
\put(0,7){\vector(1,0){120}}

\thinlines
\put(58,7){\line(0,1){10}}
\put(60,5){\line(1,0){10}}
\put(71.4,8.4){\line(-1,1){8.6}}
\qbezier(70,5)(74.8,5)(71.4,8.4)
\qbezier(58,7)(58,5)(60,5)
\qbezier(58,17)(58,21.4)(62.8,17)

\put(56.5,7){\line(0,1){45}}
\put(60,3.5){\line(1,0){55}}
\qbezier(56.5,7)(56.5,3.5)(60,3.5)

\put(15,3){\line(1,0){55}}
\put(71.8,9.8){\line(-1,1){42.2}}
\qbezier(70,3)(78.6,3)(71.8,9.8)

\put(62,56){$l$}
\put(117,9){$k$}
\end{picture}
\end{center}
\caption{Decomposition of
$\bigl(\rho_{\alpha,\beta},\BB C [z_{11}, z_{12}, z_{22}, N(Z)^{-1}] \bigr)$
into irred. components for $\alpha+\beta=-1$.}
\label{decomposition-fig1}
\end{figure}
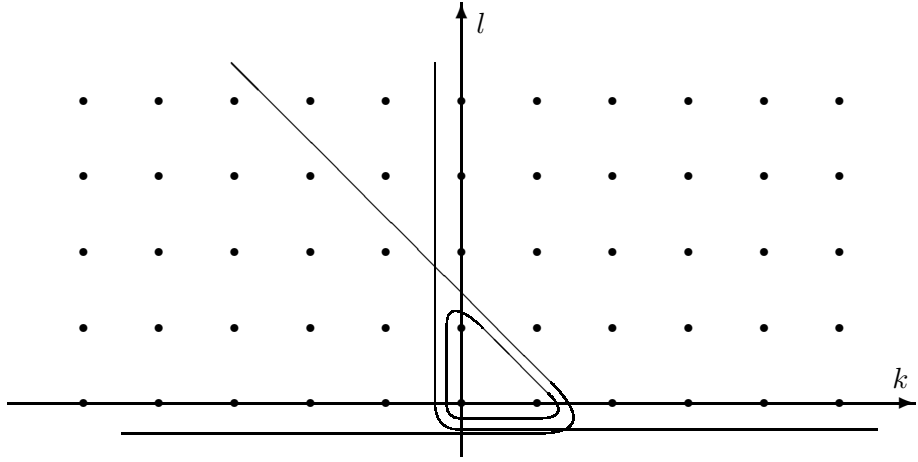

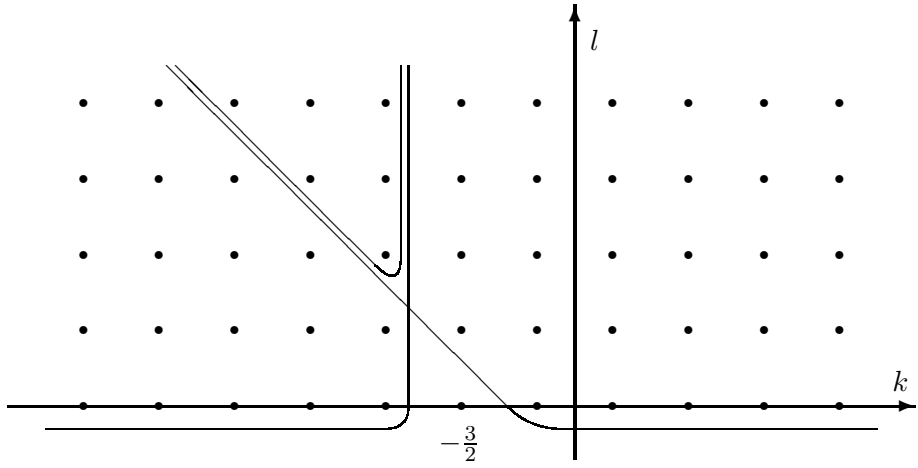
\begin{figure}
\begin{center}
\setlength{\unitlength}{1mm}
\begin{picture}(120,60)
\multiput(10,7)(10,0){11}{\circle*{1}}
\multiput(10,17)(10,0){11}{\circle*{1}}
\multiput(10,27)(10,0){11}{\circle*{1}}
\multiput(10,37)(10,0){11}{\circle*{1}}
\multiput(10,47)(10,0){11}{\circle*{1}}

\thicklines
\put(75,0){\vector(0,1){60}}
\put(0,7){\vector(1,0){120}}

\thinlines
\put(52,27){\line(0,1){25}}
\put(48.6,25.6){\line(-1,1){26.4}}
\qbezier(52,27)(52,22.2)(48.6,25.6)

\put(73,4){\line(1,0){42}}
\put(66,7){\line(-1,1){45}}
\qbezier(73,4)(69,4)(66,7)

\put(5,4){\line(1,0){45}}
\put(53,7){\line(0,1){45}}
\qbezier(50,4)(53,4)(53,7)

\put(77,54){$l$}
\put(117,9){$k$}
\put(57,1){$-\tfrac32$}
\end{picture}
\end{center}
\caption{Decomp. of $\bigl(\rho_{\alpha,\beta},
  N(Z)^{\frac12} \cdot \BB C [z_{11}, z_{12}, z_{22}, N(Z)^{-1}] \bigr)$
  into irred. components for $\alpha+\beta=3$.}
\label{decomposition-fig2}
\end{figure}

\begin{proof}
The proof is similar to that of Theorem 7 in \cite{desitter}.
Note that the basis elements $N(Z)^k \cdot R^m_l(Z)$ have $k \in \BB Z$
in the first case and $k \in \tfrac12 + \BB Z$ in the second case.
For $k$ and $l$ fixed, these functions span an irreducible representation
$V_l$ of $\mathfrak{u}(2)$.
(Recall that $\mathfrak{u}(2)$ is realized as a subalgebra of
$\mathfrak{sp}(4,\BB C)$ as in \eqref{u(2)}.)

To determine the effect of matrices of the kind
$\bigl( \begin{smallmatrix} 0 & B \\ 0 & 0 \end{smallmatrix} \bigr)
\in \mathfrak{sp}(4,\BB C)$ with $B \in \HC$, $B^T=B$,
we use Lemma \ref{rho-alpha-beta-action} describing their action together with
\eqref{dtNR}-\eqref{dz_{22}NR} and compute
\begin{multline}  \label{B-action}
\partial \bigl(N(Z)^k \cdot R^m_l(Z) \bigr) =
  \tfrac{2k+2l+1}{2(2l+1)} N(Z)^k
  \begin{pmatrix} -(l+m)(l+m-1)R^{m-1}_{l-1}(Z) & -i(l+m) R^m_{l-1}(Z) \\
    -i(l+m) R^m_{l-1}(Z) & R^{m+1}_{l-1}(Z) \end{pmatrix}  \\
  +\tfrac{k}{2l+1} N(Z)^{k-1}
  \begin{pmatrix} (l-m+1)(l-m+2)R^{m-1}_{l+1}(Z) & -i(l-m+1) R^m_{l+1}(Z) \\
    -i(l-m+1) R^m_{l+1}(Z) & -R^{m+1}_{l+1}(Z) \end{pmatrix}.
\end{multline}
Next we determine the effect of matrices of the kind
$\bigl( \begin{smallmatrix} 0 & 0 \\ C & 0 \end{smallmatrix} \bigr)
\in \mathfrak{sp}(4,\BB C)$ with $C \in \HC$, $C^T=C$.
Again, we use Lemma \ref{rho-alpha-beta-action} together with
\eqref{tR}-\eqref{z_{22}R}, \eqref{dtNR}-\eqref{dz_{22}NR} and compute
\begin{multline*}  
  Z \partial \bigl(N(Z)^k \cdot R^m_l(Z) \bigr) Z
  + (\alpha+\beta) N(Z)^k \cdot Z \cdot R^m_l(Z) \\
  = \tfrac{k+l+\alpha+\beta}{2l+1} N(Z)^k
  \begin{pmatrix} -R^{m+1}_{l+1}(Z) & i(l-m+1) R^m_{l+1}(Z) \\
    i(l-m+1) R^m_{l+1}(Z) & (l-m+1)(l-m+2)R^{m-1}_{l+1}(Z) \end{pmatrix}  \\
  +\tfrac{2k+2\alpha+2\beta-1}{2(2l+1)} N(Z)^{k+1}
  \begin{pmatrix} R^{m+1}_{l-1}(Z) & i(l+m) R^m_{l-1}(Z) \\
    i(l+m) R^m_{l-1}(Z) & -(l+m)(l+m-1)R^{m-1}_{l-1}(Z) \end{pmatrix}.
\end{multline*}

The actions of
$\bigl( \begin{smallmatrix} A & 0 \\ 0 & -A^T \end{smallmatrix} \bigr)$,
$\bigl( \begin{smallmatrix} 0 & B \\ 0 & 0 \end{smallmatrix} \bigr)$ and
$\bigl( \begin{smallmatrix} 0 & 0 \\ C & 0 \end{smallmatrix} \bigr)$
are illustrated in Figure \ref{actions}. In the diagram describing
$\rho_{\alpha,\beta}
\bigl( \begin{smallmatrix} 0 & B \\ 0 & 0 \end{smallmatrix} \bigr)$,
the vertical arrow disappears if $l=0$ or $2k+2l+1=0$,
and the diagonal arrow disappears if $k=0$.
Similarly, in the diagram describing
$\rho_{\alpha,\beta}
\bigl( \begin{smallmatrix} 0 & 0 \\ C & 0 \end{smallmatrix} \bigr)$,
the vertical arrow disappears if $k+l+\alpha+\beta=0$, and the
diagonal arrow disappears if $2k+2\alpha+2\beta-1=0$ or $l=0$.
This proves the result.
\end{proof}

\begin{figure}
\begin{center}
\begin{subfigure}[b]{0.31\textwidth}
\centering
\includegraphics[scale=0.2]{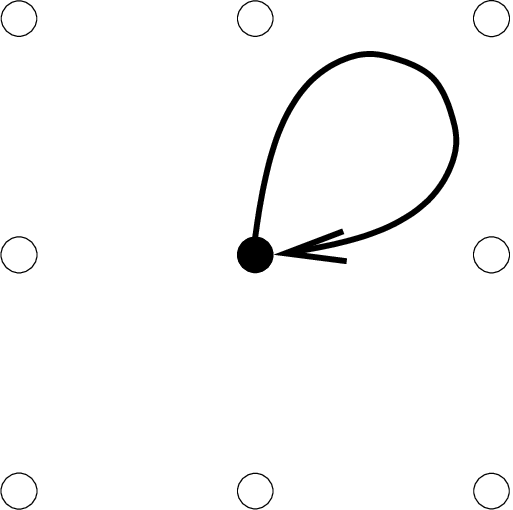}
\caption{Action of
  $\rho_{\alpha,\beta}
  \bigl(\begin{smallmatrix} A & 0 \\ 0 & -A^T \end{smallmatrix}\bigr)$}
\end{subfigure}
\quad
\begin{subfigure}[b]{0.31\textwidth}
\centering
\includegraphics[scale=0.2]{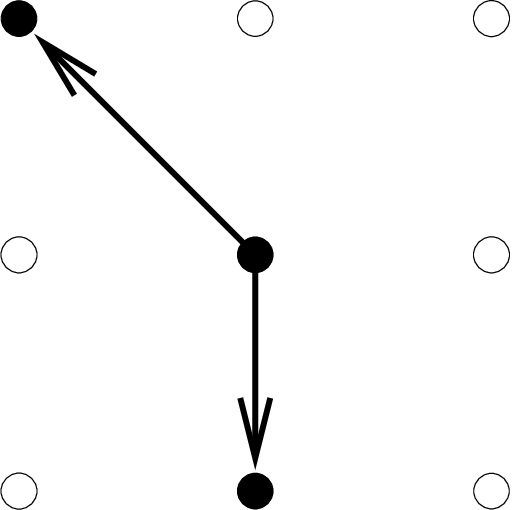}
\caption{Action of
  $\rho_{\alpha,\beta}
  \bigl(\begin{smallmatrix} 0 & B \\ 0 & 0 \end{smallmatrix}\bigr)$}
\end{subfigure}
\quad
\begin{subfigure}[b]{0.31\textwidth}
\centering
\includegraphics[scale=0.2]{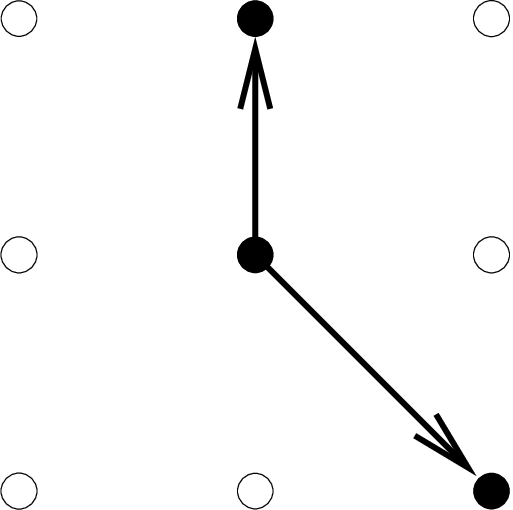}
\caption{Action of
  $\rho_{\alpha,\beta}
  \bigl(\begin{smallmatrix} 0 & 0 \\ C & 0 \end{smallmatrix}\bigr)$}
\end{subfigure}
\end{center}
\caption{}
\label{actions}
\end{figure}

Changing the value $\alpha+\beta$ has the effect of moving the sloped
``wall'' $k=-(l+\alpha+\beta)$ in Figure \ref{decomposition-fig1} and the
vertical ``wall'' $k=\tfrac12-\alpha-\beta$ in Figure \ref{decomposition-fig2}.
Let
\[
(\rho_M, M) = \bigl(\rho_{\alpha,\beta},
N(Z)^{\frac12} \cdot \BB C [z_{11}, z_{12}, z_{22}, N(Z)^{-1}] \bigr)
\qquad \text{with $\alpha+\beta=3$}.
\]
This representation is distinguished by the property that it is
the only $\mathfrak{sp}(4,\BB C)$-module among those listed in
Theorem \ref{rho-alpha-beta-thm} that has the trivial one-dimensional
representation appearing as a quotient (Figure \ref{decomposition-fig2}).
Thus, we have an $\mathfrak{sp}(4,\BB C)$-invariant map
\begin{equation}  \label{M-integral}
(\rho_M, M) \to \BB C.
\end{equation}
In this regard, the representation $(\rho_M, M)$ of $\mathfrak{sp}(4,\BB C)$
is similar to the representation $(\rho,\Sh)$ of $\mathfrak{sl}(2,\HC)$
(Theorem 31 and Figure 2 in \cite{ATMP}).

\subsection{Realization of the Map $(\rho_M,M) \to \BB C$ as an Integral}

We would like to express the map \eqref{M-integral} as an integral.
For this reason we introduce a set
\[
\Gamma = \{ U \in \HC \text{-sym};\: U^*U=I \}
\]
(the set of symmetric unitary matrices).

\begin{lem}
The set $\Gamma$ is $Sp(4,\BB R)$-invariant.
\end{lem}

\begin{proof}
Let $U \in \Gamma$ and
$X= \bigl(\begin{smallmatrix} A & B \\ \bar B & -A^T \end{smallmatrix}\bigr)
\in \mathfrak{sp}(4,\BB R)$ with $A,B \in \HC$, $A^*=-A$, $B^T=B$.
Working modulo terms of order $t^2$, $I+tX$ acts on $U$ by
\[
U \mapsto U + t(AU + B + UA^T - U \bar B U),
\]
which is symmetric and unitary:
\begin{multline*}
\bigl( U + t(AU + B + UA^T - U \bar B U) \bigr)
\bigl( U + t(AU + B + UA^T - U \bar B U) \bigr)^*  \\
= UU^* + t(A + BU^* + UA^TU^* - U \bar B)
+ t(A^* + UB^* + UA^{T*}U^* - U \bar B^* U^*) =I.
\end{multline*}
Since $Sp(4,\BB R)$ is connected, this proves the result.
\end{proof}  

The following lemma identifies $\Gamma$ is a quotient of $S^2 \times S^1$:

\begin{lem}
Let
\[
\widetilde{\Gamma} =
\{X \in \BB H;\: N(X)=1,\: X^T=X \} \times \{ z \in \BB C ;\: |z|=1 \}.
\]
Then we have a 2-to-1 covering
\[
\widetilde{\Gamma} \to \Gamma,\quad (X,z) \mapsto zX.
\]
This identifies $\Gamma$ with the quotient of
$\widetilde{\Gamma} = S^2 \times S^1$ modulo the
relation $(X,z) \sim (-X,-z)$. In particular, $\Gamma$ is {\em not} orientable.
\end{lem}

We introduce a space
$\widetilde{\BB H}^{\times} = \BB H^{\times} \times S^1$ together with a
map $\widetilde{\BB H}^{\times} \to \HC^{\times}$
\[
(X,z) \mapsto zX, \qquad X \in \BB H^{\times},\: |z|=1,
\]
which is a 2-to-1 covering of the image.
Then $\widetilde{\Gamma} = S^2 \times S^1$ is naturally a subset of
$\widetilde{\BB H}^{\times}$ and an orientable 2-to-1 covering of $\Gamma$.
The functions $N(Z)^k$, $k \in \tfrac12 \BB Z$, are well-defined on
$\widetilde{\BB H}^{\times}$.

The tangent space at $1 \in U(2)$ can be identified with
the Minkowski space $\BB M$ as in Subsection 3.5 in \cite{FL1}.
Likewise, the tangent space at $1 \in \Gamma = \HC \text{-sym} \cap U(2)$
can be identified with
\begin{equation}  \label{Gamma-orientation}
\BB M \cap \HC \text{-sym} =
\{ Z = -ix^0e_0+x^1e_1+x^3e_3;\: x^0,x^1,x^3 \in \BB R \}.
\end{equation}
We orient $\BB M \cap \HC \text{-sym}$ and $T_1\Gamma$ so that
$\{ -ie_0, e_1, e_3 \}$ is a positively oriented basis,
which in turn determines an orientation on $\widetilde{\Gamma}$.

Consider a holomorphic differential 3-form on $\HC \text{-sym}$:
\[
dZ^3 = dz^0 \wedge dz^1 \wedge dz^3.
\]
We have an analogue of Lemma 61 in \cite{FL1}.

\begin{lem}  \label{Jacobian_lemma}
On $\HC \text{-sym}$ we have:
\[
dZ^3 = N(cZ+d)^3  \,d \tilde Z^3 = N(a'-Zc')^3 \,d \tilde Z^3,
\]
where $\bigl(\begin{smallmatrix} a & b \\ c & d \end{smallmatrix}\bigr)
\in Sp(4,\BB C)$,
$\bigl(\begin{smallmatrix} a & b \\ c & d \end{smallmatrix}\bigr)^{-1}
= \bigl(\begin{smallmatrix} a' & b' \\ c' & d' \end{smallmatrix}\bigr)$,
and $\tilde Z = (aZ+b)(cZ+d)^{-1}$.
\end{lem}

Then the (locally defined) 3-form
\[
N(Z)^{-\frac32} \, dZ^3
\]
is (locally) invariant under the action of $U(2) \subset Sp(4,\BB C)$.
Thus, we have a $\mathfrak{u}(2)$-invariant measure on $\widetilde{\Gamma}$
-- the pull-back of $N(Z)^{-\frac32} \, dZ^3$.
We have an analogue of Lemma 60 in \cite{FL1}.

\begin{lem}
We have:
\[
\int_{\widetilde{\Gamma}} N(Z)^{-\frac32} \, dZ^3 = -8\pi^2i.
\]
\end{lem}

\begin{cor}  \label{M-integral-cor}
The $\mathfrak{sp}(4,\BB C)$-invariant map $(\rho_M,M) \to \BB C$
(map \eqref{M-integral}) can be expressed as an integral as
\[
f \mapsto \frac{i}{8\pi^2} \int_{\widetilde{\Gamma}} f(Z) \, dZ^3, \qquad f \in M.
\]
\end{cor}

We can give another description of the map \eqref{M-integral}.
For each real $s \in (0,1)$, let
\[
\Gamma_s = \{ U=z^0e_0+z^1e_1+z^2e_2+z^3e_3 \in \HC ;\: U^*U=I,\: |z^2|=s \}.
\]
We can parametrize $\Gamma_s$ using $0 \le \theta \le \pi$ and
$0 \le \phi, \psi \le 2\pi$ as follows
\begin{align}
\Gamma_s &= \Bigr\{ U=e^{i\psi} X \in \HC;\:
\begin{smallmatrix} X=x^0e_0+x^1e_1+x^2e_2+x^3e_3 \in \BB H, \\
  x^2=s,\: N(X)=1,\: 0 \le \psi \le 2\pi \end{smallmatrix} \Bigl\}  \\
&= \biggl\{ e^{i\psi} \begin{pmatrix} r\sin\theta e^{i\phi} & ir\cos\theta - s \\
  ir\cos\theta + s & r\sin\theta e^{-i\phi} \end{pmatrix} \in \HC;\:
\begin{smallmatrix} 0 \le \theta \le \pi, \\
  0 \le \phi, \psi \le 2\pi \end{smallmatrix} \biggr\},  \label{Gamma_s-param}
\end{align}
where $r=\sqrt{1-s^2}$.
In particular, $\Gamma_s$ is diffeomorphic to $S^2 \times S^1$.
Note that as $s \to 0^+$, the sets $\Gamma_s$ ``fold'' into $\Gamma$.

\begin{lem}  \label{Gamma_s-inversion-lemma}
For each $s \in (0,1)$, the set $\Gamma_s$ is invariant under the inversion
$Z \mapsto Z^{-1}$. In terms of the parametrization \eqref{Gamma_s-param},
the inversion switches
\[
\theta \mapsto \theta, \qquad
\phi \mapsto \pi-\phi \mod 2\pi, \qquad
\psi \mapsto \pi-\psi \mod 2\pi.
\]
\end{lem}

We can express the differential 3-form $dZ^3 = dz^0 \wedge dz^1 \wedge dz^3$
in terms of the parametrization \eqref{Gamma_s-param}:

\begin{lem}  \label{dZonGamma_s}
For each $s \in (0,1)$, on $\Gamma_s$ we have
\[
dZ^3 = dz^0 \wedge dz^1 \wedge dz^3 =
-i \sin\theta e^{3i\psi} \, d\theta \wedge d\phi \wedge d\psi.
\]
\end{lem}

The tangent space at $Z_0=\sqrt{1-s^2} + se^2 \in \Gamma_s$ can be identified
with \eqref{Gamma-orientation}.
As before, we orient this space and $T_{Z_0}\Gamma_s$ so that
$\{ -ie_0, e_1, e_3 \}$ is a positively oriented basis,
which in turn determines an orientation on $\Gamma_s$.

Let us revisit the
$\mathfrak{sp}(4,\BB C)$-invariant map $(\rho_M,M) \to \BB C$
(map \eqref{M-integral}). Note that the square root of $N(Z)$ can be
defined on $\Gamma_s$ by declaring
\[
N(e^{i\psi} X)^{\frac12} = e^{i\psi}, \qquad
\begin{smallmatrix} X=x^0e_0+x^1e_1+x^2e_2+x^3e_3 \in \BB H, \\
  x^2=s,\: N(X)=1,\: 0 \le \psi \le 2\pi. \end{smallmatrix}
\]
We can extend functions on $\HC\text{-sym}$ to functions on $\HC$ by making
them constant along the $e_2$-direction:
\[
f(Z_{sym}+z^2e_2)=f(Z_{sym}), \qquad Z_{sym} \in \HC\text{-sym},\:z^2 \in \BB C.
\]
Define the square root of $N(Z_{sym})$ along $\Gamma_s$ by declaring the value
of $N(e^{i\psi} X_{sym})^{\frac12}$ at $e^{i\psi} (X_{sym}+se_2)$ to be
\[
e^{i\psi} \sqrt{N(X_{sym})} = e^{i\psi} \sqrt{1-s^2}, \qquad
X_{sym} \in \BB H\text{-sym},\: N(X_{sym}+se_2)=1,\: 0 \le \psi \le 2\pi.
\]
From the orthogonality relations \eqref{R-orthogonality}
with $R^{-m'}_{l'}(Z) = R^0_0(Z)=1$ and Lemma \ref{dZonGamma_s}
we obtain:

\begin{cor}
The $\mathfrak{sp}(4,\BB C)$-invariant map $(\rho_M,M) \to \BB C$
(map \eqref{M-integral}) can be expressed as an integral as
\[
f \mapsto \lim_{s \to 0^+} \frac{i}{8\pi^2} \int_{\Gamma_s} f(Z) \, dZ^3,
\qquad f \in M.
\]
\end{cor}

Lemmas \ref{Gamma_s-inversion-lemma} and \ref{dZonGamma_s} imply
an analogue of Lemma \ref{Jacobian_lemma}:

\begin{lem}  \label{Gamma_s-Jacobian-lemma}
For each $s \in (0,1)$, the pull-back of $dZ^3 = dz^0 \wedge dz^1 \wedge dz^3$
restricted to $\Gamma_s$ under the inversion map $Z \mapsto Z^{-1}$ on
$\Gamma_s$ is
\[
- N(Z)^{-3} dZ^3.
\]
\end{lem}

\subsection{Reproducing Kernel Expansions}

In this subsection we find basis expansions of $N(Z_1 - Z_2)^{-\frac12}$ and
$N(Z_1 - Z_2)^{-\frac32}$ with $Z_1, Z_2 \in \HC\text{-sym}$.
These functions are reproducing kernels of certain spaces of functions on
$\HC^{\times}\text{-sym}$, and their expansions should be seen as analogues
of expansions of the reproducing kernels $N(Z_1 - Z_2)^{-1}$ and
$N(Z_1 - Z_2)^{-2}$ with $Z_1, Z_2 \in \HC$ given in \cite{FL1}.

\begin{lem}
Let $Z_1, Z_2 \in \BB H \text{-sym}$, then,
for each fixed non-negative integer $l$, we have:
\begin{equation}  \label{R-m-sum-identity}
\sum_{m=-l}^l (-1)^m R^m_l(Z_1) \cdot R^{-m}_l(Z_2)
= 2^{-l} \cdot P_l \bigl( \tr(Z_1Z_2^+) \bigr).
\end{equation}
\end{lem}

\begin{proof}
Assume first that $Z_1, Z_2 \in \BB H\text{-sym}$
and write using the spherical coordinates
\[
Z_1 = r_1 \begin{pmatrix} \sin\theta_1 e^{i\phi_1} & i\cos\theta_1 \\
  i\cos\theta_1 & \sin\theta_1 e^{-i\phi_1} \end{pmatrix}, \quad
Z_2 = r_2 \begin{pmatrix} \sin\theta_2 e^{i\phi_2} & i\cos\theta_2 \\
  i\cos\theta_2 & \sin\theta_2 e^{-i\phi_2} \end{pmatrix},
\]
where
$r_1= \sqrt{N(Z_1)} \ge 0$, $r_2 = \sqrt{N(Z_2)} \ge 0$,
$0 \le \theta_1, \theta_2 \le \pi$, $0 \le \phi_1, \phi_2 \le 2\pi$.
And assume further that $\theta_1+\theta_2 < \pi$.
Identity 8.794(1) in \cite{GR}
\begin{multline*}
P_l \bigl( \cos\theta_1 \cdot \cos\theta_2 +
\sin\theta_1 \cdot \sin\theta_2 \cdot \cos\phi \bigr)  \\
= P_l(\cos\theta_1) \cdot P_l(\cos\theta_2) + 
2\sum_{m=1}^{\infty} (-1)^m P^m_l(\cos\theta_1) \cdot
P^{-m}_l(\cos\theta_2) \cdot \cos(m\phi),
\end{multline*}
where $0 \le \theta_1, \theta_2 < \pi$, $\theta_1+\theta_2 < \pi$, $\phi$ real,
implies that
\begin{multline*}
\sum_{m=-l}^l (-1)^m R^m_l(Z_1) \cdot R^{-m}_l(Z_2)
= (r_1r_2)^l \sum_{m=-l}^l (-1)^m P^m_l(\cos\theta_1) \cdot
P^{-m}_l(\cos\theta_2) e^{im(\phi_1-\phi_2)}  \\
= (r_1r_2)^l \cdot P_l \bigl( \cos\theta_1 \cdot \cos\theta_2 +
\sin\theta_1 \cdot \sin\theta_2 \cdot \cos (\phi_1-\phi_2) \bigr)
= (r_1r_2)^l \cdot P_l \bigl( \tfrac1{2r_1r_2} \tr(Z_1Z_2^+) \bigr).
\end{multline*}
Since both sides are polynomials in $Z_1, Z_2$ of degree $l$, the identity
\eqref{R-m-sum-identity} holds for all $Z_1, Z_2 \in \HC\text{-sym}$.
\end{proof}

The following expansion should be regarded as an analogue of the matrix
coefficient expansion of the reproducing kernel $N(Z_1 - Z_2)^{-1}$
with $Z_1, Z_2 \in \HC$ given in Proposition 25 in \cite{FL1}.

\begin{prop}
We have the following series expansion
\[
N(Z_2)^{-\frac12} \cdot N(1 - Z_1Z_2^{-1})^{-\frac12}
= \sum_{l=0}^{\infty} \biggl( \sum_{m=-l}^l
(-1)^m R^m_l(Z_1) \cdot N(Z_2)^{-l-\frac12} \cdot R^{-m}_l(Z_2) \biggr),
\]
which converges uniformly on compact subsets in the region
$\{ Z_1, Z_2 \in \BB H\text{-sym} ;\: N(Z_1) < N(Z_2) \}$.
(The proof shows that the region of convergence is a much larger subset of
$\HC\text{-sym} \times \HC\text{-sym}$, but identifying this subset is
complicated due to the presence of the fractional powers.)
\end{prop}

\begin{proof}
Identity 8.921 in \cite{GR}
\begin{equation}  \label{GR8.921}
(1-2rz+r^2)^{-\frac12} = \sum_{l=0}^{\infty} r^l \cdot P_l(z), \qquad
|r| < \min \bigl\{ |z \pm \sqrt{z^2-1}| \bigr\}
\end{equation}
together with \eqref{R-m-sum-identity} imply that
\begin{multline*}
\sum_{l=0}^{\infty} \biggl( \sum_{m=-l}^l
(-1)^m R^m_l(Z_1) \cdot N(Z_2)^{-l-\frac12} \cdot R^{-m}_l(Z_2) \biggr)  \\
= N(Z_2)^{-\frac12} \sum_{l=0}^{\infty}
\Bigl( \tfrac{\sqrt{N(Z_1)}}{\sqrt{N(Z_2)}} \Bigr)^l
\cdot P_l \Bigl( \tfrac{\tr(Z_1Z_2^+)}{2\sqrt{N(Z_1) \cdot N(Z_2)}} \Bigr)  \\
= N(Z_2)^{-\frac12} \cdot \bigl( 1 - \tr(Z_1Z_2^+) \cdot N(Z_2)^{-1}
+ N(Z_1) \cdot N(Z_2)^{-1} \bigr)^{-\frac12}  \\
= N(Z_2)^{-\frac12} \cdot N(1 - Z_1Z_2^{-1})^{-\frac12}.
\end{multline*}
\end{proof}


The following expansion should be regarded as an analogue of the matrix
coefficient expansion of the reproducing kernel $N(Z_1 - Z_2)^{-2}$
with $Z_1, Z_2 \in \HC$ given in Proposition 27 in \cite{FL1}.

\begin{prop}
We have the following series expansion
\begin{multline}  \label{N-3/2-expansion}
N(Z_2)^{-\frac32} \cdot N(1 - Z_1Z_2^{-1})^{-\frac32}  \\
= \sum_{k,l=0}^{\infty} \biggl( \sum_{m=-l}^l
(-1)^m (2l+1) N(Z_1)^k \cdot R^m_l(Z_1)
\cdot N(Z_2)^{-k-l-\frac32} \cdot R^{-m}_l(Z_2) \biggr),
\end{multline}
which converges uniformly on compact subsets in the region
$\{ Z_1, Z_2 \in \BB H\text{-sym} ;\: N(Z_1) < N(Z_2) \}$.
(Again, the proof shows that the region of convergence is a much larger subset
of $\HC\text{-sym} \times \HC\text{-sym}$, but identifying this subset is
complicated due to the presence of the fractional powers.)
\end{prop}

\begin{proof}
Applying the degree operator $r \frac{\partial}{\partial r}$ to \eqref{GR8.921},
we find that
\[
\frac{1-r^2}{(1-2rz+r^2)^{\frac32}} =
\sum_{l=0}^{\infty} (2l+1) r^l \cdot P_l(z), \qquad
|r| < \min \bigl\{ |z \pm \sqrt{z^2-1}| \bigr\}.
\]
Using this identity together with \eqref{R-m-sum-identity}, we obtain that
\begin{multline*}
\sum_{k,l=0}^{\infty} \biggl( \sum_{m=-l}^l (-1)^m (2l+1) N(Z_1)^k \cdot R^m_l(Z_1)
\cdot N(Z_2)^{-k-l-\frac32} \cdot R^{-m}_l(Z_2) \biggr)  \\
= N(Z_2)^{-\frac32} \sum_{k,l=0}^{\infty} (2l+1) \frac{N(Z_1)^k}{N(Z_2)^k} \cdot
\Bigl( \tfrac{\sqrt{N(Z_1)}}{\sqrt{N(Z_2)}} \Bigr)^l
\cdot P_l \Bigl( \tfrac{\tr(Z_1Z_2^+)}{2\sqrt{N(Z_1) \cdot N(Z_2)}} \Bigr)  \\
= \frac{N(Z_2)^{-\frac12}}{N(Z_2)-N(Z_1)}
\sum_{l=0}^{\infty} (2l+1) \Bigl( \tfrac{\sqrt{N(Z_1)}}{\sqrt{N(Z_2)}} \Bigr)^l
\cdot P_l \Bigl( \tfrac{\tr(Z_1Z_2^+)}{2\sqrt{N(Z_1) \cdot N(Z_2)}} \Bigr)  \\
= N(Z_2)^{-\frac32} \cdot \bigl( 1 - \tr(Z_1Z_2^+) \cdot N(Z_2)^{-1}
+ N(Z_1) \cdot N(Z_2)^{-1} \bigr)^{-\frac32}  \\
= N(Z_2)^{-\frac32} \cdot N(1 - Z_1Z_2^{-1})^{-\frac32}.
\end{multline*}
\end{proof}

\begin{rem}
We prefer to write the reproducing kernels in the form
\[
N(Z_2)^{-\frac12} \cdot N(1 - Z_1Z_2^{-1})^{-\frac12}
\qquad \text{and} \qquad
N(Z_2)^{-\frac32} \cdot N(1 - Z_1Z_2^{-1})^{-\frac32}
\]
because expressions like $N(Z_2 - Z_1)^{-\frac12}$ and
$N(Z_2 - Z_1)^{-\frac32}$ are ambiguous.
\end{rem}

\subsection{Invariant Pairings and Reproducing Formulas}

Recall the $\mathfrak{sp}(4,\BB C)$-modules
\[
\bigl(\rho_{\alpha,\beta}, \BB C [z_{11}, z_{12}, z_{22}, N(Z)^{-1}] \bigr)
\quad \text{and} \quad
\bigl(\rho_{\alpha,\beta},
N(Z)^{\frac12} \cdot \BB C [z_{11}, z_{12}, z_{22}, N(Z)^{-1}] \bigr)
\]
introduced in Subsection \ref{rho_M-subsection}.
The multiplication map induces an $\mathfrak{sp}(4,\BB C)$-equivariant map
\begin{multline*}
\bigl(\rho_{\alpha_1,\beta_1},\BB C [z_{11}, z_{12}, z_{22}, N(Z)^{-1}] \bigr) \otimes
\bigl(\rho_{\alpha_2,\beta_2},
N(Z)^{\frac12} \cdot \BB C [z_{11}, z_{12}, z_{22}, N(Z)^{-1}] \bigr) \\
\to
\bigl(\rho_{\alpha_{1+2},\beta_{1+2}},
N(Z)^{\frac12} \cdot \BB C [z_{11}, z_{12}, z_{22}, N(Z)^{-1}] \bigr).
\end{multline*}
In particular, when $\alpha_1+\alpha_2+\beta_1+\beta_2=3$, we obtain an
$\mathfrak{sp}(4,\BB C)$-equivariant map
\[
\bigl(\rho_{\alpha_1,\beta_1},\BB C [z_{11}, z_{12}, z_{22}, N(Z)^{-1}] \bigr) \otimes
\bigl(\rho_{\alpha_2,\beta_2},
N(Z)^{\frac12} \cdot \BB C [z_{11}, z_{12}, z_{22}, N(Z)^{-1}] \bigr)
\to (\rho_M, M).
\]
Combining this map with the $\mathfrak{sp}(4,\BB C)$-invariant map
\eqref{M-integral}, we obtain an $\mathfrak{sp}(4,\BB C)$-invariant map
\[
\bigl(\rho_{\alpha_1,\beta_1},\BB C [z_{11}, z_{12}, z_{22}, N(Z)^{-1}] \bigr) \otimes
\bigl(\rho_{\alpha_2,\beta_2},
N(Z)^{\frac12} \cdot \BB C [z_{11}, z_{12}, z_{22}, N(Z)^{-1}] \bigr)
\to \BB C.
\]
In other words, we have an $\mathfrak{sp}(4,\BB C)$-invariant bilinear pairing
between
\[
\bigl(\rho_{\alpha_1,\beta_1},\BB C [z_{11}, z_{12}, z_{22}, N(Z)^{-1}] \bigr)
\quad \text{and} \quad
\bigl(\rho_{\alpha_2,\beta_2},
N(Z)^{\frac12} \cdot \BB C [z_{11}, z_{12}, z_{22}, N(Z)^{-1}] \bigr).
\]
Combining these observations with Corollary \ref{M-integral-cor}, we obtain
an analogue of Proposition 69 in \cite{FL1}.

\begin{prop}
When $\alpha_1+\alpha_2+\beta_1+\beta_2=3$, we have an
$\mathfrak{sp}(4,\BB C)$-invariant bilinear pairing between
\[
\bigl(\rho_{\alpha_1,\beta_1},\BB C [z_{11}, z_{12}, z_{22}, N(Z)^{-1}] \bigr)
\quad \text{and} \quad
\bigl(\rho_{\alpha_2,\beta_2},
N(Z)^{\frac12} \cdot \BB C [z_{11}, z_{12}, z_{22}, N(Z)^{-1}] \bigr).
\]
This pairing can be expressed as
\[
\langle f_1, f_2 \rangle_M =
\frac{i}{8\pi^2} \int_{\widetilde{\Gamma}} f_1(Z) \cdot f_2(Z) \, dZ^3.
\]
\end{prop}

From \eqref{R-orthogonality} we obtain the orthogonality relations.

\begin{lem}
We have the following orthogonality relations:
\[  
\Bigl\langle N(Z)^k \cdot R^m_l(Z),
N(Z)^{-k'-l'-\frac32} \cdot R^{-m'}_{l'}(Z) \Bigr\rangle_M
= \frac{(-1)^m}{2l+1} \delta_{kk'} \delta_{ll'} \delta_{mm'}.
\]
\end{lem}

Recall that, by Theorem \ref{rho-alpha-beta-thm}, $(\rho_M,M)$ has invariant
subspaces
\begin{align}
&\BB C\text{-span of }
\bigl\{ N(Z)^k \cdot R^m_l(Z);\: k \ge -(l+\tfrac12) \bigr\} \qquad \text{and}
\label{MM^+}  \\
&\BB C\text{-span of }
\bigl\{ N(Z)^k \cdot R^m_l(Z);\: k \le -\tfrac52 \bigr\}.  \label{MM^-}
\end{align}
From the series expansion \eqref{N-3/2-expansion} of
$N(W)^{-\frac32} \cdot N(1-ZW^{-1})^{-\frac32}$
we obtain an analogue of Corollary 14 in \cite{desitter}.

\begin{thm}
\begin{enumerate}
\item
If $Z \in \BB D^+$, then the map
\[
f \mapsto (\P^+ f)(Z) = 
\frac i{8\pi^2} \int_{W \in \widetilde{\Gamma}}
\frac{N(Z)^{-\frac32} \cdot f(W) \,dW^3} {N(1-ZW^{-1})^{\frac32}}
\]
annihilates \eqref{MM^-} and projects $M$ onto the quotient module
\begin{equation}    \label{quotient1}
\frac{N(Z)^{\frac12} \cdot \BB C [z_{11}, z_{12}, z_{22}, N(Z)^{-1}]}
{\BB C\text{-span of }
\bigl\{ N(Z)^k \cdot R^m_l(Z);\: k \le -\tfrac52 \bigr\}}.
\end{equation}
\item
If $Z \in \BB D^-$, then the map
\[
f \mapsto (\P^- f)(Z) = 
\frac i{8\pi^2} \int_{W \in \widetilde{\Gamma}}
\frac{N(Z)^{-\frac32} \cdot f(W) \,dW^3} {N(1-WZ^{-1})^{\frac32}}
\]
annihilates \eqref{MM^+} and projects $M$ onto the quotient module
\begin{equation}    \label{quotient2}
\frac{N(Z)^{\frac12} \cdot \BB C [z_{11}, z_{12}, z_{22}, N(Z)^{-1}]}
{\BB C\text{-span of }
\bigl\{ N(Z)^k \cdot R^m_l(Z);\: k \ge -(l+\tfrac12) \bigr\}}.
\end{equation}
\end{enumerate}
In particular, these maps $\P^+$ and $\P^-$ provide reproducing formulas
for functions in the quotient spaces \eqref{quotient1} and \eqref{quotient2}
respectively.
\end{thm}

\section{Realizations of Regular Functions and
  Related Spaces of Functions}  \label{5}

In this section we discuss different ways of realizing regular functions.
Additionally, we discuss related $\mathfrak{sp}(4,\BB C)$-modules of functions,
such as quasi regular functions and analogues of regular functions on
$\HC^{\times}\text{-sym}$.

\subsection{The Standard Realization of Regular Functions}  \label{standard-realiz-subsection}

In previous papers \cite{FL1}-\cite{qreg} we treated the left and right
regular functions as representations of either $\mathfrak{gl}(2,\HC)$ or
$\mathfrak{su}(2,2)$ (when restricted to
$\mathfrak{su}(2,2) \subset \mathfrak{gl}(2,\HC)$
these representations are unitary).
When we restrict these representations further to $\mathfrak{sp}(4,\BB C)$,
they remain irreducible unitary (Proposition \ref{nreg-irreducibility-prop}),
and the distinction between left and right regular functions disappears
-- they become isomorphic (Corollary \ref{left-right-regular-cor}).

Consider vector spaces
\begin{align*}
{\cal V}_{res} &=
\{\text{$\BB S$-valued polynomial functions on $\HC^{\times} \text{-sym}$}\}
= \BB S \otimes \BB C [z_{11}, z_{12}, z_{22}, N(Z)^{-1}],  \\
{\cal V}'_{res} &=
\{\text{$\BB S'$-valued polynomial functions on $\HC^{\times} \text{-sym}$}\}
= \BB S' \otimes \BB C [z_{11}, z_{12}, z_{22}, N(Z)^{-1}].
\end{align*}
The Lie algebra $\mathfrak{sp}(4,\BB C)$ acts on these spaces by
differentiating the following group actions:
\begin{align*}
\pi_l(h): \: f(Z) \: &\mapsto \: \bigl( \pi_l(h)f \bigr)(Z) =
\frac {(cZ+d)^{-1}}{N(cZ+d)} \cdot f \bigl( (aZ+b)(cZ+d)^{-1} \bigr),  \\
\pi_r(h): \: g(Z) \: &\mapsto \: \bigl( \pi_r(h)g \bigr)(Z) =
g \bigl( (a'-Zc')^{-1}(-b'+Zd') \bigr) \cdot \frac {(a'-Zc')^{-1}}{N(a'-Zc')},
\end{align*}
where $f \in {\cal V}_{res}$, $g \in {\cal V}'_{res}$,
$h = \bigl(\begin{smallmatrix} a' & b' \\ c' & d' \end{smallmatrix}\bigr)
\in Sp(4,\BB R) \subset GL(2,\HC)$ and 
$h^{-1} = \bigl(\begin{smallmatrix} a & b \\ c & d \end{smallmatrix}\bigr)$.
These are the same expressions as \eqref{pi_l}-\eqref{pi_r}, and
we think of $(\pi_l, {\cal V}_{res})$ and $(\pi_r, {\cal V}'_{res})$
as respectively representations $(\pi_l, {\cal V})$ and $(\pi_r, {\cal V}')$
of $\mathfrak{gl}(2,\HC)$ restricted to $\mathfrak{sp}(4,\BB C)$ and
$\HC^{\times} \text{-sym}$.
We essentially restate Lemma \ref{pi-Lie_alg-action}:

\begin{lem}  \label{pi-action-lem}
The Lie algebra action $\pi_l$ of $\mathfrak{sp}(4,\BB C)$ on ${\cal V}_{res}$
is given by
\begin{align*}
\pi_l \bigl( \begin{smallmatrix} A & 0 \\ 0 & -A^T \end{smallmatrix} \bigr) &:
f(Z) \mapsto - \tr (AZ \partial + ZA^T\partial + A) f - A^Tf,  \\
\pi_l \bigl( \begin{smallmatrix} 0 & B \\ 0 & 0 \end{smallmatrix} \bigr) &:
f(Z) \mapsto - \tr (B \partial) f,  \\
\pi_l \bigl( \begin{smallmatrix} 0 & 0 \\ C & 0 \end{smallmatrix} \bigr) &:
f(Z) \mapsto \tr (ZCZ \partial + CZ) f + CZf.
\end{align*}
The Lie algebra action $\pi_r$ of $\mathfrak{sp}(4,\BB C)$ on ${\cal V}'_{res}$
is given by
\begin{align*}
\pi_r \bigl( \begin{smallmatrix} A & 0 \\ 0 & -A^T \end{smallmatrix} \bigr) &:
g(Z) \mapsto - \tr (AZ \partial + ZA^T\partial + A) g - gA,  \\
\pi_r \bigl( \begin{smallmatrix} 0 & B \\ 0 & 0 \end{smallmatrix} \bigr) &:
g(Z) \mapsto - \tr (B \partial) g,  \\
\pi_r \bigl( \begin{smallmatrix} 0 & 0 \\ C & 0 \end{smallmatrix} \bigr) &:
g(Z) \mapsto \tr (ZCZ \partial + ZC) g + gZC.
\end{align*}
\end{lem}

The inversions on ${\cal V}_{res}$ and ${\cal V}'_{res}$ produced by
$\bigl(\begin{smallmatrix} 0 & 1 \\ 1 & 0 \end{smallmatrix}\bigr)
\in GL(2,\HC)$ are:
\[
\pi_l \bigl(\begin{smallmatrix} 0 & 1 \\ 1 & 0 \end{smallmatrix}\bigr):\:
f(Z) \mapsto \tfrac{Z^{-1}}{N(Z)} \cdot f(Z^{-1}), \qquad
\pi_r\bigl(\begin{smallmatrix} 0 & 1 \\ 1 & 0 \end{smallmatrix}\bigr):\:
g(Z) \mapsto - g(Z^{-1}) \cdot \tfrac{Z^{-1}}{N(Z)}.
\]

There are the spaces of functions regular at the origin and infinity:
\begin{align}
{\cal V}_{res}^+ &=
\{\text{$\BB S$-valued polynomial functions on $\HC \text{-sym}$}\}
= \BB S \otimes \BB C [z_{11}, z_{12}, z_{22}],  \label{V^+}  \\
{\cal V}_{res}^- &= \bigl\{
f \in \BB S \otimes \BB C[z_{11},z_{12},z_{22}, N(Z)^{-1}] ;\:
N(Z)^{-1} \cdot Z^{-1} \cdot f(Z^{-1}) \in {\cal V}_{res}^+ \bigr\}, \label{V^-} \\
{\cal V}'^+_{res} &=
\{\text{$\BB S'$-valued polynomial functions on $\HC \text{-sym}$}\}
= \BB S' \otimes \BB C [z_{11}, z_{12}, z_{22}], \label{V'^+}  \\
{\cal V}'^-_{res} &= \bigl\{
g \in \BB S' \otimes \BB C[z_{11},z_{12},z_{22}, N(Z)^{-1}] ;\:
N(Z)^{-1} \cdot g(Z^{-1}) \cdot Z^{-1} \in {\cal V}'^+_{res} \bigr\}.  \label{V'^-}
\end{align}

We can describe the $K$-types of $(\pi_l, {\cal V}_{res}^{\pm})$. Let
\[
{\bf v^t_{l,m}}(Z)
= \begin{pmatrix} (l+m+1) R^m_l(Z) \\ i R^{m+1}_l(Z) \end{pmatrix}, \qquad
\begin{smallmatrix} l=0,1,2,3,\dots, \\ -l-1 \le m \le l, \end{smallmatrix}
\]
and
\[
{\bf v^b_{l,m}}(Z)
= \begin{pmatrix} (l-m) R^m_l(Z) \\ -i R^{m+1}_l(Z) \end{pmatrix}, \qquad
\begin{smallmatrix} l=1,2,3,4,\dots, \\ -l \le m \le l-1. \end{smallmatrix}
\]
For a fixed $l$, these functions span $U(2)$-invariant subspaces of dimensions
$2l+2$ and $2l$ respectively.
Let $A_1=\bigl( \begin{smallmatrix} 0 & 1 \\ 0 & 0 \end{smallmatrix} \bigr)$
and $A_2=\bigl( \begin{smallmatrix} 0 & 0 \\ 1 & 0 \end{smallmatrix} \bigr)$,
then
\begin{align*}
\pi_l \bigl( \begin{smallmatrix} A_1 & 0 \\ 0 & -A_1^T \end{smallmatrix} \bigr)
   {\bf v^t_{l,m}}(Z) &= i(l-m+1)(l+m+1) {\bf v^t_{l,m-1}}(Z),  \\
\pi_l \bigl( \begin{smallmatrix} A_2 & 0 \\ 0 & -A_2^T \end{smallmatrix} \bigr)
   {\bf v^t_{l,m}}(Z) &= -i {\bf v^t_{l,m+1}}(Z);  \\
\pi_l \bigl( \begin{smallmatrix} A_1 & 0 \\ 0 & -A_1^T \end{smallmatrix} \bigr)
{\bf v^b_{l,m}}(Z) &= i(l-m)(l+m) {\bf v^b_{l,m-1}}(Z),  \\
\pi_l \bigl( \begin{smallmatrix} A_2 & 0 \\ 0 & -A_2^T \end{smallmatrix} \bigr)
{\bf v^b_{l,m}}(Z) &= -i {\bf v^b_{l,m+1}}(Z).
\end{align*}
The $K$-types of $(\pi_r, {\cal V}'^{\pm}_{res})$ are obtained by transposing
the ${\bf v^t_{l,m}}(Z)$'s and ${\bf v^b_{l,m}}(Z)$'s.

We have analogues of Lemma 23 in \cite{FL1}:
\begin{align}
Z \cdot \begin{pmatrix} (l+m+1) R^m_l(Z) \\ i R^{m+1}_l(Z) \end{pmatrix}
&= \begin{pmatrix} - R^{m+1}_{l+1}(Z) \\ i(l-m+1) R^m_{l+1}(Z) \end{pmatrix},
\label{Lemma23-analogue1}  \\
Z \cdot \begin{pmatrix} (l-m) R^m_l(Z) \\ -i R^{m+1}_l(Z) \end{pmatrix}
&= N(Z) \cdot \begin{pmatrix} R^{m+1}_{l-1}(Z) \\ i(l+m) R^m_{l-1}(Z)
\end{pmatrix}.  \label{Lemma23-analogue2}
\end{align}

We describe the effect of the inversions on the basis functions:

\begin{lem}  \label{V-inversion-lem}
Let $\bigl(\begin{smallmatrix} 0 & 1 \\ 1 & 0 \end{smallmatrix}\bigr)
\in GL(2,\HC)$, then
\begin{align*}
\pi_l \bigl(\begin{smallmatrix} 0 & 1 \\ 1 & 0 \end{smallmatrix}\bigr)
\bigl( N(Z)^k \cdot {\bf v^t_{l,m}}(Z) \bigr) &=
(-1)^l \tfrac{(l+m+1)!}{(l-m)!} N(Z)^{-(k+l+2)} \cdot {\bf v^b_{l+1,-m-1}}(Z),  \\
\pi_l \bigl(\begin{smallmatrix} 0 & 1 \\ 1 & 0 \end{smallmatrix}\bigr)
\bigl( N(Z)^k \cdot {\bf v^b_{l+1,m}}(Z) \bigr) &=
(-1)^l \tfrac{(l+m+1)!}{(l-m)!} N(Z)^{-(k+l+2)} \cdot {\bf v^t_{l,-m-1}}(Z).
\end{align*}
\end{lem}

\begin{thm}  \label{1reg-res-irred-thm}
The $\mathfrak{sp}(4,\BB C)$-modules
\[
(\pi_l, {\cal V}^+_{res}), \quad
(\pi_l, {\cal V}^-_{res}), \quad
(\pi_r, {\cal V}'^+_{res}), \quad
(\pi_r, {\cal V}'^-_{res})
\]
are irreducible, isomorphic respectively to
\[
(\pi_l, {\cal V}^+), \quad
(\pi_l, {\cal V}^-), \quad
(\pi_r, {\cal V}'^+), \quad
(\pi_r, {\cal V}'^-)
\]
and have $\mathfrak{sp}(4,\BB R)$-invariant unitary structures.
The $K$-types of $(\pi_l, {\cal V}_{res})$ are $V_{l+\frac12}$,
$l=0,1,2,3,\dots$, spanned by $N(Z)^k \cdot {\bf v^t_{l,m}}(Z)$
and by $N(Z)^k \cdot {\bf v^b_{l+1,m}}(Z)$,  $-l-1 \le m \le l$.

Furthermore,
\begin{align*}
{\cal V}_{res}^+ &= \BB C\text{-span of }
\bigl\{ N(Z)^k \cdot {\bf v^t_{l,m}}(Z),\: N(Z)^k \cdot {\bf v^b_{l+1,m}}(Z) ;\:
k,l \ge 0 \bigr\},  \\
{\cal V}_{res}^- &= \BB C\text{-span of }
\bigl\{ N(Z)^{-(k+l+2)} \cdot {\bf v^t_{l,m}}(Z),\:
N(Z)^{-(k+l+2)} \cdot {\bf v^b_{l+1,m}}(Z) ;\: k,l \ge 0 \bigr\}.
\end{align*}
\end{thm}

\begin{proof}
By Theorem 28 in \cite{FL1}, $(\pi_l, {\cal V}^{\pm})$ and
$(\pi_r, {\cal V}'^{\pm})$ are irreducible unitary representations of
$\mathfrak{u}(2,2)$.
By Proposition \ref{nreg-irreducibility-prop}, when restricted to
$\mathfrak{sp}(4,\BB R) \subset \mathfrak{u}(2,2)$,
they remain irreducible and unitary.



We have the restriction operators $\Res: {\cal V} \to {\cal V}_{res}$ and
$\Res: {\cal V}' \to {\cal V}'_{res}$ which take $\BB S$- or $\BB S'$-valued
polynomial functions on $\HC^{\times}$ and restrict them to
$\BB S$- or $\BB S'$-valued polynomial functions on $\HC^{\times} \text{-sym}$.
Clearly, the operators $\Res$ produce intertwining maps of
$\mathfrak{sp}(4,\BB C)$-modules
$(\pi_l, {\cal V}^{\pm}) \to (\pi_l, {\cal V}_{res})$ and
$(\pi_r, {\cal V}^{\pm}) \to (\pi_l, {\cal V}'_{res})$.
The images of $(\pi_l, {\cal V}^{\pm})$ lie in ${\cal V}^{\pm}_{res}$
and the images of $(\pi_r, {\cal V}'^{\pm})$ lie in ${\cal V}'^{\pm}_{res}$,
and -- by the irreducibility -- the maps $\Res$ are isomorphisms of
$(\pi_l, {\cal V}^{\pm})$ and $(\pi_r, {\cal V}'^{\pm})$ with their images.
For each integer $d$, let
\begin{align*}
{\cal V}^{\pm}(d) &= \{ f \in {\cal V}^{\pm};\:
  \text{$f$ is homogeneous of degree $d$} \},  \\
{\cal V}^{\pm}_{res}(d) &= \{ f \in {\cal V}^{\pm}_{res};\:
  \text{$f$ is homogeneous of degree $d$} \},  \\
{\cal V}'^{\pm}(d) &= \{ g \in {\cal V}'^{\pm};\:
  \text{$g$ is homogeneous of degree $d$} \},  \\
{\cal V}'^{\pm}_{res}(d) &= \{ g \in {\cal V}'^{\pm}_{res};\:
  \text{$g$ is homogeneous of degree $d$} \}.
\end{align*}
The dimensions of ${\cal V}^{\pm}(d)$ and ${\cal V}'^{\pm}(d)$ were computed
in Proposition 21 in \cite{FL1}:
\[
\dim {\cal V}^{\pm}(d) = \dim {\cal V}'^{\pm}(d) = (d+1)(d+2).
\]
A simple dimension count shows that
\[
\dim {\cal V}^{\pm}_{res}(d) = \dim {\cal V}'^{\pm}_{res}(d) = (d+1)(d+2)
\]
as well. This proves that the maps $\Res$ are isomorphisms
$(\pi_l, {\cal V}^{\pm}) \to (\pi_l, {\cal V}^{\pm}_{res})$ and
$(\pi_r, {\cal V}^{\pm}) \to (\pi_l, {\cal V}'^{\pm}_{res})$.
\end{proof}

\subsection{Restriction of the Anti Regular Functions}  \label{anti-reg-realiz-subsection}
  
Similarly, we can restrict anti regular functions to
$\mathfrak{sp}(4,\BB C)$-modules
\[
(\pi_{la}, {\cal V}^+_{a,res}), \quad
(\pi_{la}, {\cal V}^-_{a,res}), \quad
(\pi_{ra}, {\cal V}'^+_{a,res}), \quad
(\pi_{ra}, {\cal V}'^-_{a,res}),
\]
where
\begin{align*}
{\cal V}_{a,res}^+ &= \BB S \otimes \BB C [z_{11}, z_{12}, z_{22}],  \\
{\cal V}_{a,res}^- &= \bigl\{
  f \in \BB S \otimes \BB C[z_{11},z_{12},z_{22}, N(Z)^{-1}] ;\:
  N(Z)^{-2} \cdot Z \cdot f(Z^{-1}) \in {\cal V}_{a,res}^+ \bigr\},  \\
{\cal V}'^+_{a,res} &= \BB S' \otimes \BB C [z_{11}, z_{12}, z_{22}],  \\
{\cal V}'^-_{a,res} &= \bigl\{
  g \in \BB S' \otimes \BB C[z_{11},z_{12},z_{22}, N(Z)^{-1}] ;\:
  N(Z)^{-2} \cdot g(Z^{-1}) \cdot Z \in {\cal V}'^+_{a,res} \bigr\}.
\end{align*}
These modules are isomorphic respectively to
\[
(\pi_{la}, {\cal V}^+_a), \quad
(\pi_{la}, {\cal V}^-_a), \quad
(\pi_{ra}, {\cal V}'^+_a), \quad
(\pi_{ra}, {\cal V}'^-_a)
\]
from Subsection \ref{anti-reg-subsect} restricted to
$\mathfrak{sp}(4,\BB C)$.

We can describe the $K$-types of $(\pi_{la}, {\cal V}^{\pm}_{a,res})$. Let
\[
{\bf u^t_{l,m}}(Z)
= \begin{pmatrix} R^m_l(Z) \\ i(l+m) R^{m-1}_l(Z) \end{pmatrix}, \qquad
\begin{smallmatrix} l=0,1,2,3,\dots, \\ -l \le m \le l+1; \end{smallmatrix}
\]
and
\[
{\bf u^b_{l,m}}(Z)
= \begin{pmatrix} R^m_l(Z) \\ -i(l-m+1) R^{m-1}_l(Z) \end{pmatrix}, \qquad
\begin{smallmatrix} l=1,2,3,4,\dots, \\ -l+1 \le m \le l. \end{smallmatrix}
\]
For a fixed $l$, these functions span $U(2)$-invariant subspaces of dimensions
$2l+2$ and $2l$ respectively.
Let $A_1=\bigl( \begin{smallmatrix} 0 & 1 \\ 0 & 0 \end{smallmatrix} \bigr)$
and $A_2=\bigl( \begin{smallmatrix} 0 & 0 \\ 1 & 0 \end{smallmatrix} \bigr)$,
then
\begin{align*}
\pi_{la} \bigl( \begin{smallmatrix} A_1 & 0 \\ 0 & -A_1^T \end{smallmatrix}\bigr)
   {\bf u^t_{l,m}}(Z) &= i(l+m)(l-m+2) {\bf u^t_{l,m-1}}(Z),  \\
\pi_{la} \bigl( \begin{smallmatrix} A_2 & 0 \\ 0 & -A_2^T \end{smallmatrix}\bigr)
   {\bf u^t_{l,m}}(Z) &= -i {\bf u^t_{l,m+1}}(Z);  \\
\pi_{la} \bigl( \begin{smallmatrix} A_1 & 0 \\ 0 & -A_1^T \end{smallmatrix}\bigr)
{\bf u^b_{l,m}}(Z) &= i(l-m+1)(l+m-1) {\bf u^b_{l,m-1}}(Z),  \\
\pi_{la} \bigl( \begin{smallmatrix} A_2 & 0 \\ 0 & -A_2^T \end{smallmatrix}\bigr)
{\bf u^b_{l,m}}(Z) &= -i {\bf u^b_{l,m+1}}(Z).
\end{align*}
The $K$-types of $(\pi_{ra}, {\cal V}'^{\pm}_{a,res})$ are obtained by transposing
the ${\bf u^t_{l,m}}(Z)$'s and ${\bf u^b_{l,m}}(Z)$'s.

We describe the effect of the inversions on the basis functions:

\begin{lem}  \label{U-inversion-lem}
Let $\bigl(\begin{smallmatrix} 0 & 1 \\ 1 & 0 \end{smallmatrix}\bigr)
\in GL(2,\HC)$, then
\begin{align*}
\pi_{la} \bigl(\begin{smallmatrix} 0 & 1 \\ 1 & 0 \end{smallmatrix}\bigr)
\bigl( N(Z)^k \cdot {\bf u^t_{l,m}}(Z) \bigr) &= (-1)^{l+1}
\tfrac{(l+m)!}{(l-m+1)!} N(Z)^{-(k+l+2)} \cdot {\bf u^b_{l+1,-m+1}}(Z),  \\
\pi_{la} \bigl(\begin{smallmatrix} 0 & 1 \\ 1 & 0 \end{smallmatrix}\bigr)
\bigl( N(Z)^k \cdot {\bf u^b_{l+1,m}}(Z) \bigr) &= (-1)^{l+1}
\tfrac{(l+m)!}{(l-m+1)!} N(Z)^{-(k+l+2)} \cdot {\bf u^t_{l,-m+1}}(Z).
\end{align*}
\end{lem}

\begin{thm}
  The $K$-types of $(\pi_{la}, {\cal V}^{\pm}_{a,res})$ are $V_{l+\frac12}$,
  $l=0,1,2,3,\dots$,
  spanned by $N(Z)^k \cdot {\bf u^t_{l,m}}(Z)$ and by
  $N(Z)^k \cdot {\bf u^b_{l+1,m}}(Z)$, $-l \le m \le l+1$.

Furthermore,
\begin{align*}
{\cal V}^+_{a,res} &= \BB C\text{-span of }
\bigl\{ N(Z)^k \cdot {\bf u^t_{l,m}}(Z),\: N(Z)^k \cdot {\bf u^b_{l+1,m}}(Z) ;\:
k,l \ge 0 \bigr\},  \\
{\cal V}^-_{a,res} &= \BB C\text{-span of }
\bigl\{ N(Z)^{-(k+l+2)} \cdot {\bf u^t_{l,m}}(Z),\:
N(Z)^{-(k+l+2)} \cdot {\bf u^b_{l+1,m}}(Z) ;\: k,l \ge 0 \bigr\}.
\end{align*}
\end{thm}

\subsection{Realization of Regular Functions Using
  Half-Integer Powers of $N(Z)$}

We consider
\begin{align*}
{\cal V}_{a,res} &=
\{\text{$\BB S$-valued polynomial functions on $\HC^{\times} \text{-sym}$}\}  \\
&= \BB S \otimes \BB C [z_{11}, z_{12}, z_{22}, N(Z)^{-1}],  \\
{\cal V}'_{\frac12} &=
\BB S' \otimes N(Z)^{\frac12} \cdot \BB C [z_{11}, z_{12}, z_{22}, N(Z)^{-1}].
\end{align*}
The Lie algebra $\mathfrak{sp}(4,\BB C)$ acts on these spaces by
differentiating the following group actions:
\begin{align*}
\pi_{la}(h): \: f(Z) \: &\mapsto \: \bigl( \pi_{la}(h)f \bigr)(Z) =
\frac{a'-Zc'}{N(a'-Zc')^2} \cdot f \bigl( (a'-Zc')^{-1}(-b'+Zd') \bigr),  \\
\pi_r(h): \: g(Z) \: &\mapsto \: \bigl( \pi_r(h)g \bigr)(Z) =
g \bigl( (a'-Zc')^{-1}(-b'+Zd') \bigr) \cdot \frac{(a'-Zc')^{-1}}{N(a'-Zc')},
\end{align*}
where $f \in {\cal V}_{a,res}$, $g \in {\cal V}'_{\frac12}$,
$h = \bigl(\begin{smallmatrix} a' & b' \\ c' & d' \end{smallmatrix}\bigr)
\in Sp(4,\BB R) \subset GL(2,\HC)$.

On the one hand, $(\pi_{la}, {\cal V}_{a,res})$ contains the left anti regular
functions $(\pi_{la}, {\cal V}^{\pm}_{a,res})$ as
$\mathfrak{sp}(4,\BB C)$-submodules, and it has been established
that these are isomorphic to the regular functions $(\pi_l, {\cal V}^{\pm})$
(Proposition \ref{reg-anti-reg-iso-prop} and Subsections
\ref{standard-realiz-subsection}-\ref{anti-reg-realiz-subsection}).
On the other hand, $(\pi_r, {\cal V}'_{\frac12})$ is the dual of
$(\pi_{la}, {\cal V}_{a,res})$, since we have an
$\mathfrak{sp}(4,\BB C)$-invariant pairing
\begin{equation}  \label{half-integer-pairing}
(\pi_r, {\cal V}'_{\frac12}) \times (\pi_{la}, {\cal V}_{a,res})
\to (\rho_M,M) \to \BB C,
\end{equation}
where the first arrow is the multiplication map
${\cal V}'_{\frac12} \times {\cal V}_{a,res} \ni (g,f) \mapsto gf \in M$
and the second arrow is the projection \eqref{M-integral}.
At the same time, the dual of the space of regular functions is the space
of regular functions itself (functions regular at infinity are
dual to the functions regular at the origin).
We conclude that $(\pi_r, {\cal V}'_{\frac12})$ contains the regular functions
as a {\em quotient space}.

Let
\[
{\bf v'^t_{l,m}}(Z) = \bigl( (l+m+1) R^m_l(Z),\: i R^{m+1}_l(Z) \bigr), \qquad
\begin{smallmatrix} l=0,1,2,3,\dots, \\ -l-1 \le m \le l, \end{smallmatrix}
\]
and
\[
{\bf v'^b_{l,m}}(Z) = \bigl( (l-m) R^m_l(Z),\: -i R^{m+1}_l(Z) \bigr), \qquad
\begin{smallmatrix} l=1,2,3,4,\dots, \\ -l \le m \le l-1, \end{smallmatrix}
\]
be the transposes of ${\bf v^t_{l,m}}(Z)$ and ${\bf v^b_{l,m}}(Z)$ respectively.
By \eqref{B-action}, the subspace of ${\cal V}'_{\frac12}$
\begin{equation}  \label{I^+}
{\cal I}^+_{{\cal V}'_{\frac12}} = \BB C\text{-span of }
\bigl\{ N(Z)^k \cdot {\bf v'^t_{l,m}}(Z),\: N(Z)^{k-1} \cdot {\bf v'^b_{l+1,m}}(Z)
;\: k \ge -(l+\tfrac12),\: l \ge 0 \bigr\}
\end{equation}
is $\mathfrak{sp}(4,\BB C)$-invariant.
Applying the inversion
$\pi_r \bigl(\begin{smallmatrix} 0 & 1 \\ 1 & 0 \end{smallmatrix}\bigr)$,
we obtain another $\mathfrak{sp}(4,\BB C)$-invariant subspace of
${\cal V}'_{\frac12}$.
Using Lemma \ref{V-inversion-lem}, we can identify this subspace as
\begin{equation}  \label{I^-}
{\cal I}^-_{{\cal V}'_{\frac12}} = \BB C\text{-span of }
\bigl\{ N(Z)^k \cdot {\bf v'^t_{l,m}}(Z),\: N(Z)^{k-1} \cdot {\bf v'^b_{l+1,m}}(Z)
;\: k \le -\tfrac12,\: l \ge 0 \bigr\}.
\end{equation}
Therefore,
\begin{multline}  \label{V^0}
{\cal V}'^0_{\frac12} =
{\cal I}^+_{{\cal V}'_{\frac12}} \cap {\cal I}^-_{{\cal V}'_{\frac12}}  \\
= \BB C\text{-span of } \bigl\{ N(Z)^k \cdot {\bf v'^t_{l,m}}(Z),\:
N(Z)^{k-1} \cdot {\bf v'^b_{l+1,m}}(Z) ;\:
-(l+\tfrac12) \le k \le -\tfrac12,\: l \ge 0 \bigr\}
\end{multline}
is an $\mathfrak{sp}(4,\BB C)$-submodule of ${\cal V}'_{\frac12}$.
Working with $K$-types as in the proof of Theorem \ref{rho-alpha-beta-thm},
we obtain:

\begin{prop}  \label{V'_{1/2}-prop}
The only proper $\mathfrak{sp}(4,\BB C)$-invariant subspaces of 
$(\pi_r, {\cal V}'_{\frac12})$ are
\[
{\cal I}^+_{{\cal V}'_{\frac12}}, \qquad {\cal I}^-_{{\cal V}'_{\frac12}}
\qquad \text{and} \qquad
{\cal V}'^0_{\frac12} =
{\cal I}^+_{{\cal V}'_{\frac12}} \cap {\cal I}^-_{{\cal V}'_{\frac12}}.
\]
These invariant subspaces can be characterized as the annihilators of
\[
(\pi_{la}, {\cal V}^+_{a,res}), \qquad (\pi_{la}, {\cal V}^-_{a,res})
\qquad \text{and} \qquad (\pi_{la}, {\cal V}^+_{a,res} \oplus {\cal V}^-_{a,res})
\]
respectively via the pairing \eqref{half-integer-pairing}.

In particular, the quotient $\mathfrak{sp}(4,\BB C)$-module
$\bigl( \pi_r, {\cal V}'_{\frac12}/{\cal V}'^0_{\frac12} \bigr)$
decomposes into a direct sum of two irreducible submodules
\begin{align*}
{\cal V}'^+_{\frac12} &= \BB C\text{-span of }
\bigl\{ N(Z)^k \cdot {\bf v'^t_{l,m}}(Z),\:
N(Z)^{k-1} \cdot {\bf v'^b_{l+1,m}}(Z) ;\:
k \ge \tfrac12,\: l \ge 0 \bigr\} , \\
{\cal V}'^-_{\frac12} &= \BB C\text{-span of }
\bigl\{ N(Z)^k \cdot {\bf v'^t_{l,m}}(Z),\:
N(Z)^{k-1} \cdot {\bf v'^b_{l+1,m}}(Z) ;\:
k \le -(l+\tfrac32),\: l \ge 0 \bigr\}
\end{align*}
isomorphic to $(\pi_l, {\cal V}^+) \simeq (\pi_r, {\cal V}'^+)$ and
$(\pi_l, {\cal V}^-) \simeq (\pi_r, {\cal V}'^-)$ respectively.
\end{prop}

\subsection{Invariant Pairing, Reproducing Kernel Expansion and
  a Reproducing Formula for Regular Functions}

In this subsection we derive a reproducing kernel expansion for the
regular functions realized as ${\cal V}'^{\pm}_{\frac12}$ and
obtain reproducing formulas that can be regarded as analogues of
Cauchy-Fueter formulas.

By Corollary \ref{M-integral-cor}, the $\mathfrak{sp}(4,\BB C)$-invariant
pairing \eqref{half-integer-pairing} between $(\pi_r, {\cal V}'_{\frac12})$
and $(\pi_{la}, {\cal V}_{a,res})$ can be expressed as
\[
\langle g,f \rangle_{1-reg} = \frac{i}{8\pi^2} \int_{\widetilde{\Gamma}}
g(Z) \cdot f(Z) \, dZ^3,
\qquad g \in {\cal V}'_{\frac12}, \: f \in {\cal V}_{a,res}.
\]
From \eqref{R-orthogonality} we obtain the orthogonality relations.

\begin{lem}
We have the following orthogonality relations:
\begin{align*}
\Bigl\langle N(Z)^{k+\frac12} \cdot {\bf v'^t_{l,m}}(Z),
N(Z)^{-k'-l'-2} \cdot {\bf u^t_{l',-m'}}(Z) \Bigr\rangle_{1-reg}
&= (-1)^m \delta_{kk'} \delta_{ll'} \delta_{mm'},  \\
\Bigl\langle N(Z)^{k+\frac12} \cdot {\bf v'^b_{l,m}}(Z),
N(Z)^{-k'-l'-2} \cdot {\bf u^b_{l',-m'}}(Z) \Bigr\rangle_{1-reg}
&= (-1)^m \delta_{kk'} \delta_{ll'} \delta_{mm'},  \\
\Bigl\langle N(Z)^{k+\frac12} \cdot {\bf v'^t_{l,m}}(Z),
N(Z)^{-k'-l'-2} \cdot {\bf u^b_{l',-m'}}(Z) \Bigr\rangle_{1-reg}
&= 0,  \\
\Bigl\langle N(Z)^{k+\frac12} \cdot {\bf v'^b_{l,m}}(Z),
N(Z)^{-k'-l'-2} \cdot {\bf u^t_{l',-m'}}(Z) \Bigr\rangle_{1-reg}
&= 0.
\end{align*}
\end{lem}

Observe that
\begin{multline}  \label{uv+uv}
{\bf u^t_{l,-m}}(Z_1) \cdot {\bf v'^t_{l,m}}(Z_2)
+ {\bf u^b_{l,-m}}(Z_1) \cdot {\bf v'^b_{l,m}}(Z_2)  \\
= (2l+1) \begin{pmatrix} R^m_l(Z_2) \cdot R^{-m}_l(Z_1) & 0 \\
  0 & -R^{m+1}_l(Z_2) \cdot R^{-m-1}_l(Z_1) \end{pmatrix}.
\end{multline}
We can obtain the reproducing kernels for ${\cal V}'^{\pm}_{\frac12}$
and ${\cal V}^{\pm}_{a,res}$:

\begin{prop}
We have the following series expansions
\begin{multline*}
N(Z_2)^{-\frac32} \cdot N(1 - Z_1Z_2^{-1})^{-\frac32}  \\
= \sum_{k,l=0}^{\infty} \biggl( \sum_{m=-l-1}^l (-1)^m
\frac{N(Z_1)^k}{N(Z_2)^{k+l+\frac32}} \cdot
\Bigl( {\bf u^t_{l,-m}}(Z_1) \cdot {\bf v'^t_{l,m}}(Z_2)
+ {\bf u^b_{l,-m}}(Z_1) \cdot {\bf v'^b_{l,m}}(Z_2) \Bigr) \biggr),
\end{multline*}
which converges uniformly on compact subsets in the region
$\{ Z_1, Z_2 \in \BB H\text{-sym} ;\: N(Z_1) < N(Z_2) \}$, and
\begin{multline*}
\frac{N(Z_2)^{\frac12}}{N(Z_1)^2} \cdot N(1 - Z_2Z_1^{-1})^{-\frac32}
\cdot Z_1 Z_2^{-1} =
\sum_{k,l=0}^{\infty} \biggl( \sum_{m=-l-1}^l (-1)^m
\frac{N(Z_2)^{k-\frac12}}{N(Z_1)^{k+l+2}}  \\
\times \Bigl( N(Z_2) \cdot {\bf u^t_{l,-m}}(Z_1) \cdot {\bf v'^t_{l,m}}(Z_2)
+ {\bf u^b_{l+1,-m}}(Z_1) \cdot {\bf v'^b_{l+1,m}}(Z_2) \Bigr) \biggr),
\end{multline*}
which converges uniformly on compact subsets in the region
$\{ Z_1, Z_2 \in \BB H\text{-sym} ;\: N(Z_2) < N(Z_1) \}$.
(Again, the proof shows that the region of convergence is a much larger subset
of $\HC\text{-sym} \times \HC\text{-sym}$, but identifying this subset is
complicated due to the presence of the fractional powers.)
\end{prop}

\begin{proof}
The first expansion follows from the series expansion
\eqref{N-3/2-expansion} and identity \ref{uv+uv}.
The second expansion follows from the first by applying the inversions.
\end{proof}

Let
\[
{\bf u'^t_{l,m}}(Z)
= \bigl( R^m_l(Z) ,\: i(l+m) R^{m-1}_l(Z) \bigr), \qquad
\begin{smallmatrix} l=0,1,2,3,\dots, \\ -l \le m \le l+1; \end{smallmatrix}
\]
and
\[
{\bf u'^b_{l,m}}(Z)
= \bigl( R^m_l(Z) ,\: -i(l-m+1) R^{m-1}_l(Z) \bigr), \qquad
\begin{smallmatrix} l=1,2,3,4,\dots, \\ -l+1 \le m \le l. \end{smallmatrix}
\]
be the transposes of ${\bf u^t_{l,m}}(Z)$ and ${\bf u^b_{l,m}}(Z)$ respectively.
Transposing the above series expansions, we obtain reproducing kernel
expansions for ${\cal V}^{\pm}_{\frac12}$ and ${\cal V}'^{\pm}_{a,res}$:

\begin{prop}
We have the following series expansions
\begin{multline*}
N(Z_2)^{-\frac32} \cdot N(1 - Z_1Z_2^{-1})^{-\frac32}  \\
= \sum_{k,l=0}^{\infty} \biggl( \sum_{m=-l}^{l+1} (-1)^m
\frac{N(Z_1)^k}{N(Z_2)^{k+l+\frac32}} \cdot
\Bigl( {\bf v^t_{l,-m}}(Z_2) \cdot {\bf u'^t_{l,m}}(Z_1)
+ {\bf v^b_{l,-m}}(Z_2) \cdot {\bf u'^b_{l,m}}(Z_1) \Bigr) \biggr),
\end{multline*}
which converges uniformly on compact subsets in the region
$\{ Z_1, Z_2 \in \BB H\text{-sym} ;\: N(Z_1) < N(Z_2) \}$, and
\begin{multline*}
\frac{N(Z_2)^{\frac12}}{N(Z_1)^2} \cdot N(1 - Z_2Z_1^{-1})^{-\frac32}
\cdot Z_2^{-1} Z_1 =
\sum_{k,l=0}^{\infty} \biggl( \sum_{m=-l}^{l+1} (-1)^m
\frac{N(Z_2)^{k-\frac12}}{N(Z_1)^{k+l+2}}  \\
\times \Bigl( N(Z_2) \cdot {\bf v^t_{l,-m}}(Z_2) \cdot {\bf u'^t_{l,m}}(Z_1)
+ {\bf v^b_{l+1,-m}}(Z_2) \cdot {\bf u'^b_{l+1,m}}(Z_1) \Bigr) \biggr),
\end{multline*}
which converges uniformly on compact subsets in the region
$\{ Z_1, Z_2 \in \BB H\text{-sym} ;\: N(Z_2) < N(Z_1) \}$.
(Again, the proof shows that the region of convergence is a much larger subset
of $\HC\text{-sym} \times \HC\text{-sym}$, but identifying this subset is
complicated due to the presence of the fractional powers.)
\end{prop}

From these series expansions we can obtain various reproducing formulas.
Here is an example.

\begin{thm}
\begin{enumerate}
\item
If $Z \in \BB D^+$, then the map
\[
g \mapsto (\P^+ g)(Z) = 
\frac i{8\pi^2} \int_{W \in \widetilde{\Gamma}}
\frac{N(Z)^{\frac12} \cdot g(W) \cdot WZ^{-1} \,dW^3}
     {N(W)^2 \cdot N(1-ZW^{-1})^{\frac32}}
\]
annihilates \eqref{I^-} and projects ${\cal V}'_{\frac12}$ onto the quotient
module
\begin{equation}    \label{quotient1V}
{\cal V}'^+_{\frac12} ={\cal V}'_{\frac12} / {\cal I}^-_{{\cal V}'_{\frac12}}
= \BB C\text{-span of } \biggl\{ \begin{matrix}
N(Z)^k \cdot {\bf v'^t_{l,m}}(Z), \\ N(Z)^{k-1} \cdot {\bf v'^b_{l+1,m}}(Z)
\end{matrix} ;\: k \ge \tfrac12,\: l \ge 0 \biggr\}.
\end{equation}
\item
If $Z \in \BB D^-$, then the map
\[
g \mapsto (\P^- g)(Z) = 
\frac i{8\pi^2} \int_{W \in \widetilde{\Gamma}}
\frac{N(Z)^{-\frac32} \cdot g(W) \,dW^3}{N(1-WZ^{-1})^{\frac32}}
\]
annihilates \eqref{I^+} and projects ${\cal V}'_{\frac12}$ onto the quotient
module
\begin{equation}    \label{quotient2V}
{\cal V}'^-_{\frac12} = {\cal V}'_{\frac12} / {\cal I}^+_{{\cal V}'_{\frac12}}
= \BB C\text{-span of } \biggl\{ \begin{matrix}
N(Z)^k \cdot {\bf v'^t_{l,m}}(Z),\\ N(Z)^{k-1} \cdot {\bf v'^b_{l+1,m}}(Z)
\end{matrix} ;\: k \le -(l+\tfrac32),\: l \ge 0 \biggr\}.
\end{equation}
\end{enumerate}
In particular, these maps $\P^+$ and $\P^-$ provide reproducing formulas
for functions in the quotient spaces \eqref{quotient1V} and \eqref{quotient2V}
respectively.
\end{thm}

\subsection{Explicit Isomorphisms between Realizations
$(\pi_r, {\cal V}'^{\pm}_{\frac12})$ and $(\pi_{ra}, {\cal V}'^{\pm}_{a,res})$
of Regular Functions}

In this subsection we construct isomorphisms between realization
$(\pi_r, {\cal V}'^{\pm}_{\frac12})$ of regular functions and
right anti regular functions $(\pi_{ra}, {\cal V}'^{\pm}_{a,res})$.
Then by Propositions \ref{reg-anti-reg-iso-prop} and
\ref{nreg-irreducibility-prop}, $\mathfrak{sp}(4,\BB C)$-modules
$(\pi_{ra}, {\cal V}'^{\pm}_{a,res}) \simeq (\pi_l, {\cal V}^{\pm})$.
Similarly, one can construct explicit isomorphisms between
$(\pi_l, {\cal V}^{\pm}_{\frac12})$ and left anti regular functions
$(\pi_{la}, {\cal V}^{\pm}_{a,res}) \simeq (\pi_r, {\cal V}'^{\pm}) $.

We consider a map from ${\cal V}'_{\frac12}$ into $\BB S'$-valued functions
on $\BB D^+ \sqcup \BB D^-$
\begin{equation}  \label{integral-iso}
(\Iso g)(W) = \frac{i}{8\pi^2} \int_{Z \in \widetilde{\Gamma}}
g(Z) \cdot \frac{Z-W}{N(Z-W)^2} \, dZ^3,
\qquad g \in {\cal V}'_{\frac12}.
\end{equation}
If $W \in \BB D^+ \sqcup \BB D^-$, the integrand has no singularities,
and the result is a holomorphic anti regular function in $W$
(annihilated by the operator $\overleftarrow{\nabla}$ applied on the right).
We will see shortly that the result depends on whether $W$ lies in
$\BB D^+$ or $\BB D^-$. Thus, the expression \eqref{integral-iso} determines
two different maps, and we use notations $\Iso^+$ and $\Iso^-$ to signify
that $W \in \BB D^+$ and $W \in \BB D^-$ respectively.
Recall that the action $\pi_{ra}$ was described in \eqref{pi_ra}.

\begin{thm}
The integral operators \eqref{integral-iso} have the following properties:
\begin{enumerate}
\item
  If $W \in \BB D^+$, then $\Iso^+$ maps ${\cal V}'_{\frac12}$ into
  ${\cal V}'^+_a$, annihilates ${\cal I}^-_{{\cal V}'_{\frac12}}$ and
  provides an isomorphism
  \[
  (\pi_r, {\cal V}'^+_{\frac12}) =
  \bigl( \pi_r, {\cal V}'_{\frac12} / {\cal I}^-_{{\cal V}'_{\frac12}} \bigr)
  \to (\pi_{ra}, {\cal V}'^+_a).
  \]

\item  
  If $W \in \BB D^-$, then $\Iso^-$ maps ${\cal V}'_{\frac12}$ into
  ${\cal V}'^-_a$, annihilates ${\cal I}^+_{{\cal V}'_{\frac12}}$ and
  provides an isomorphism
  \[
  (\pi_r, {\cal V}'^-_{\frac12}) =
  \bigl( \pi_r, {\cal V}'_{\frac12} / {\cal I}^+_{{\cal V}'_{\frac12}} \bigr)
  \to (\pi_{ra}, {\cal V}'^-_a).
  \]
\end{enumerate}
\end{thm}

\begin{proof}
  For $h = \bigl(\begin{smallmatrix} a' & b' \\ c' & d' \end{smallmatrix}\bigr)
  \in Sp(4,\BB R)$ sufficiently close to the identity, let
  $h^{-1} = \bigl(\begin{smallmatrix} a & b \\ c & d \end{smallmatrix}\bigr)$,
  $\tilde Z = (aZ+b)(cZ+d)^{-1}$ and $\tilde W = (aW+b)(cW+d)^{-1}$.
  By Lemma 10 in \cite{FL1} and Lemma \ref{Jacobian_lemma} we have:
  \begin{multline*}
\int_{Z \in \widetilde{\Gamma}}
(\pi_r(h)g)(Z) \cdot \frac{Z-W}{N(Z-W)^2} \, dZ^3
= \int_{Z \in \widetilde{\Gamma}} g(\tilde Z) \cdot
\frac{(a'-Zc')^{-1}}{N(a'-Zc')} \cdot \frac{Z-W}{N(Z-W)^2} \, dZ^3  \\
= \int_{Z \in \widetilde{\Gamma}} \frac{g(\tilde Z)}{N(a'-Zc')^3} \cdot
\frac{\tilde Z - \tilde W}{N(\tilde Z - \tilde W)^2} \cdot
\frac{cW+d}{N(cW+d)^2} \, dZ^3  \\
= \int_{\tilde Z \in \widetilde{\Gamma}} g(\tilde Z) \cdot
\frac{\tilde Z - \tilde W}{N(\tilde Z - \tilde W)} \, d\tilde Z^3
\cdot \frac{cW+d}{N(cW+d)^2}.
\end{multline*}
  This proves that the operators \eqref{integral-iso} intertwine the
  $\mathfrak{sp}(4,\BB C)$-actions $\pi_r$ and $\pi_{ra}$.

To show that $\Iso^{\pm}$ map ${\cal V}'_{\frac12}$ into ${\cal V}'^{\pm}_a$
and annihilate ${\cal I}^{\mp}_{{\cal V}'_{\frac12}}$ respectively, we use the
matrix coefficient expansions of the Cauchy-Fueter kernel
(Proposition 26 in \cite{FL1} and Proposition 113 in \cite{ATMP}).
Using the maps \eqref{dagger},
we obtain matrix coefficient expansions for $\frac{Z-W}{N(Z-W)^2}$:
\begin{equation}  \label{kernel-expansion1}
\frac{Z-W}{N(Z-W)^2}
  = \frac1{N(Z)} \sum_{l,m,n} \left(\begin{smallmatrix}
t^{l+\frac 12}_{m \, \underline{n+ \frac 12}}(Z^{-1})  \\
- t^{l+\frac 12}_{m \, \underline{n- \frac 12}}(Z^{-1})
\end{smallmatrix}\right)
\cdot
 \left(\begin{smallmatrix}
(l+m+ \frac 12) t^l_{n \, \underline{m- \frac 12}}(W), &
-(l-m+ \frac 12) t^l_{n \, \underline{m+ \frac 12}}(W) \end{smallmatrix}\right),
\end{equation}
which converges uniformly on compact subsets in the region
$\{ (Z,W) \in \HC^{\times} \times \HC; \: WZ^{-1} \in \BB D^+ \}$.
The sum is taken first over all
$m =-l-\frac 12 ,-l+\frac 32,\dots,l+\frac 12$ and $n =-l,-l+1,\dots,l$,
then over $l=0,\frac 12, 1, \frac 32,\dots$.
Similarly,
\begin{equation}  \label{kernel-expansion2}
\frac{Z-W}{N(Z-W)^2} = - \frac1{N(W)} \sum_{l,m,n}
\left(\begin{smallmatrix} t^{l-\frac 12}_{n - \frac 12 \, \underline{m}}(Z) \\
- t^{l-\frac 12}_{n + \frac 12 \, \underline{m}}(Z) \end{smallmatrix}\right)
\cdot
\left(\begin{smallmatrix} (l+m+\frac12) t^l_{m+\frac12 \, \underline{n}}(W^{-1}), &
- (l-m+\frac12) t^l_{m-\frac12 \, \underline{n}}(W^{-1}) \end{smallmatrix}\right),
\end{equation}
which converges uniformly on compact subsets in the region
$\{ (Z,W) \in \HC \times \HC^{\times}; \: ZW^{-1} \in \BB D^+ \}$.
The sum is taken first over all
$m =-l+\frac 12 ,-l+\frac 32,\dots,l-\frac 12$ and $n =-l,-l+1,\dots,l$,
then over $l =\frac 12, 1, \frac 32, 2, \dots$.

Let $g(Z) \in {\cal V}'_{\frac12}$, and assume that $g \ne 0$ and homogeneous
of degree $d$. If $d \ge 0$, the kernel expansion \eqref{kernel-expansion1} has
only finitely many terms of degree $-(d+3)$ in $Z$ that can potentially survive
integration over $\widetilde{\Gamma}$. And if $d<0$, no terms in
\eqref{kernel-expansion1} can survive integration over $\widetilde{\Gamma}$.
Any $g$ of the form $N(Z)^k \cdot {\bf v'^t_{l,m}}(Z)$ with
$k<-(l+\frac12)$ is annihilated by $\Iso^+$ and  generates entire
${\cal I}^-_{{\cal V}'_{\frac12}}$.
This proves $\Iso^+g \in {\cal V}'^+_a$, and that $\Iso^+g=0$ for
$g \in {\cal I}^-_{{\cal V}'_{\frac12}}$.
Similarly, using  \eqref{kernel-expansion2}, one shows that the image of
$\Iso^-$ lies in ${\cal V}'^-_a$ and that $\Iso^-$ annihilates
${\cal I}^+_{{\cal V}'_{\frac12}}$.
Then the result follows from the irreducibility of
$\mathfrak{sp}(4,\BB C)$-modules
$(\pi_r, {\cal V}'^{\pm}_{\frac12})$ and $(\pi_{ra}, {\cal V}'^{\pm}_{a,res})$.
\end{proof}

\subsection{Quasi Regular Functions Restricted to $\mathfrak{sp}(4,\BB C)$}  \label{QReg-subsection}

In this subsection we study the quasi (anti) regular functions restricted to
$\mathfrak{sp}(4,\BB C)$.
As $\mathfrak{gl}(2,\HC)$-modules, the quasi (anti) regular functions were
thoroughly studied in \cite{qreg}.

Recall from \cite{qreg} that
\begin{align*}
{\cal U} &= \{ f \in \BB C[z^0,z^1,z^2,z^3,N(Z)^{-1}] \otimes \BB S ;\:
\nabla\square f =0 \},  \\
{\cal U}' &= \{ g \in \BB C[z^0,z^1,z^2,z^3,N(Z)^{-1}] \otimes \BB S' ;\:
(\square g) \overleftarrow{\nabla} =0 \}
\end{align*}
be respectively the polynomial left and right quasi anti regular functions on
$\HC^{\times}$.
The Lie algebra $\mathfrak{gl}(2,\HC)$ acts on these spaces by
differentiating the following group actions:
\begin{align*}
\pi'_l(h): \: f(Z) \: &\mapsto \: \bigl( \pi'_l(h)f \bigr)(Z) =
\frac{a'-Zc'}{N(a'-Zc')} \cdot f \bigl( (a'-Zc')^{-1}(-b'+Zd') \bigr),  \\
\pi'_r(h): \: g(Z) \: &\mapsto \: \bigl( \pi'_r(h)g \bigr)(Z) =
g \bigl( (aZ+b)(cZ+d)^{-1} \bigr) \cdot \frac{cZ+d}{N(cZ+d)},
\end{align*}
where $f \in {\cal U}$, $g \in {\cal U}'$,
$h = \bigl(\begin{smallmatrix} a' & b' \\ c' & d' \end{smallmatrix}\bigr)
\in GL(2,\HC)$ and 
$h^{-1} = \bigl(\begin{smallmatrix} a & b \\ c & d \end{smallmatrix}\bigr)$.

Restricting to $\mathfrak{sp}(4,\BB C)$ and $\HC^{\times} \text{-sym}$,
we obtain representations $(\pi'_l, {\cal U}_{res})$ and
$(\pi'_r, {\cal U}'_{res})$ of $\mathfrak{sp}(4,\BB C)$, where
\begin{align*}
{\cal U}_{res} &= \{\text{$\BB S$-valued polynomial
  functions on $\HC^{\times} \text{-sym}$}\}  \\
&= \BB S \otimes \BB C [z_{11}, z_{12}, z_{22}, N(Z)^{-1}],  \\
{\cal U}'_{res} &= \{\text{$\BB S'$-valued polynomial
  functions on $\HC^{\times} \text{-sym}$}\}  \\
&= \BB S' \otimes \BB C [z_{11}, z_{12}, z_{22}, N(Z)^{-1}].
\end{align*}

We have the following analogue of Lemma 68 in \cite{FL1} and
Lemma 17 in \cite{qreg}:

\begin{lem}  \label{pi-prime-action-lem}
The Lie algebra action $\pi'_l$ of $\mathfrak{sp}(4,\BB C)$ on ${\cal U}_{res}$
is given by
\begin{align*}
\pi'_l \bigl( \begin{smallmatrix} A & 0 \\ 0 & -A^T \end{smallmatrix} \bigr) &:
f(Z) \mapsto - \tr (AZ \partial + ZA^T\partial + A) f + Af,  \\
\pi'_l \bigl( \begin{smallmatrix} 0 & B \\ 0 & 0 \end{smallmatrix} \bigr) &:
f(Z) \mapsto - \tr (B \partial) f,  \\
\pi'_l \bigl( \begin{smallmatrix} 0 & 0 \\ C & 0 \end{smallmatrix} \bigr) &:
f(Z) \mapsto \tr (ZCZ \partial + ZC) f - ZCf.
\end{align*}
The Lie algebra action $\pi'_r$ of $\mathfrak{sp}(4,\BB C)$ on ${\cal U}'_{res}$
is given by
\begin{align*}
\pi'_r \bigl( \begin{smallmatrix} A & 0 \\ 0 & -A^T \end{smallmatrix} \bigr) &:
g(Z) \mapsto - \tr (AZ \partial + ZA^T\partial + A) g + gA^T,  \\
\pi'_r \bigl( \begin{smallmatrix} 0 & B \\ 0 & 0 \end{smallmatrix} \bigr) &:
g(Z) \mapsto - \tr (B \partial) g,  \\
\pi'_r \bigl( \begin{smallmatrix} 0 & 0 \\ C & 0 \end{smallmatrix} \bigr) &:
g(Z) \mapsto \tr (ZCZ \partial + CZ) g - gCZ.
\end{align*}
\end{lem}

Note that the $\mathfrak{sp}(4,\BB C)$-module structures of
$(\pi'_l, {\cal U}_{res})$ and $(\pi'_r, {\cal U}'_{res})$ differ from
the restrictions of the anti regular functions
$(\pi_{la}, {\cal V}_{a,res})$ and $(\pi_{ra}, {\cal V}'_{a,res})$
only by the powers of $N(a'-Zc')$ and $N(cZ+d)$.
For this reason, their $K$-types are very similar. Recall
\begin{align*}
{\bf u^t_{l,m}}(Z)
&= \begin{pmatrix} R^m_l(Z) \\ i(l+m) R^{m-1}_l(Z) \end{pmatrix}, \qquad
\begin{smallmatrix} l=0,1,2,3,\dots, \\ -l \le m \le l+1, \end{smallmatrix}  \\
{\bf u^b_{l,m}}(Z)
&= \begin{pmatrix} R^m_l(Z) \\ -i(l-m+1) R^{m-1}_l(Z) \end{pmatrix}, \qquad
\begin{smallmatrix} l=1,2,3,4,\dots, \\ -l+1 \le m \le l; \end{smallmatrix}  \\
{\bf u'^t_{l,m}}(Z)
&= \Bigl( R^m_l(Z), \: i(l+m) R^{m-1}_l(Z) \Bigr), \qquad
\begin{smallmatrix} l=0,1,2,3,\dots, \\ -l \le m \le l+1, \end{smallmatrix}  \\
{\bf u'^b_{l,m}}(Z)
&= \Bigl( R^m_l(Z), \: -i(l-m+1) R^{m-1}_l(Z) \Bigr), \qquad
\begin{smallmatrix} l=1,2,3,4,\dots, \\ -l+1 \le m \le l. \end{smallmatrix}
\end{align*}
In this context Lemma \ref{U-inversion-lem} becomes

\begin{lem}  \label{U-inversion-lem2}
Let $\bigl(\begin{smallmatrix} 0 & 1 \\ 1 & 0 \end{smallmatrix}\bigr)
\in GL(2,\HC)$, then
\begin{align*}
\pi'_l \bigl(\begin{smallmatrix} 0 & 1 \\ 1 & 0 \end{smallmatrix}\bigr)
\bigl( N(Z)^k \cdot {\bf u^t_{l,m}}(Z) \bigr) &= (-1)^{l+1}
\tfrac{(l+m)!}{(l-m+1)!} N(Z)^{-(k+l+1)} \cdot {\bf u^b_{l+1,-m+1}}(Z),  \\
\pi'_l \bigl(\begin{smallmatrix} 0 & 1 \\ 1 & 0 \end{smallmatrix}\bigr)
\bigl( N(Z)^k \cdot {\bf u^b_{l+1,m}}(Z) \bigr) &= (-1)^{l+1}
\tfrac{(l+m)!}{(l-m+1)!} N(Z)^{-(k+l+1)} \cdot {\bf u^t_{l,-m+1}}(Z);  \\
\pi'_r \bigl(\begin{smallmatrix} 0 & 1 \\ 1 & 0 \end{smallmatrix}\bigr)
\bigl( N(Z)^k \cdot {\bf u'^t_{l,m}}(Z) \bigr) &= (-1)^l
\tfrac{(l+m)!}{(l-m+1)!} N(Z)^{-(k+l+1)} \cdot {\bf u'^b_{l+1,-m+1}}(Z),  \\
\pi'_r \bigl(\begin{smallmatrix} 0 & 1 \\ 1 & 0 \end{smallmatrix}\bigr)
\bigl( N(Z)^k \cdot {\bf u'^b_{l+1,m}}(Z) \bigr) &= (-1)^l
\tfrac{(l+m)!}{(l-m+1)!} N(Z)^{-(k+l+1)} \cdot {\bf u'^t_{l,-m+1}}(Z).
\end{align*}
\end{lem}

We have $\mathfrak{sp}(4,\BB C)$-submodules of
$(\pi'_l, {\cal U}_{res})$ and $(\pi'_r, {\cal U}'_{res})$ consisting of the
functions regular at the origin and infinity:
\begin{align*}
{\cal U}_{res}^+ &=
\{\text{$\BB S$-valued polynomial functions on $\HC \text{-sym}$}\}  \\
&= \BB S \otimes \BB C [z_{11}, z_{12}, z_{22}],  \\
{\cal U}_{res}^- &= \bigl\{
  f \in \BB S \otimes \BB C[z_{11},z_{12},z_{22}, N(Z)^{-1}] ;\:
  N(Z)^{-1} \cdot Z \cdot f(Z^{-1}) \in {\cal U}_{res}^+ \bigr\},  \\
{\cal U}'^+_{res} &=
\{\text{$\BB S'$-valued polynomial functions on $\HC \text{-sym}$}\}  \\
&= \BB S' \otimes \BB C [z_{11}, z_{12}, z_{22}],  \\
{\cal U}'^-_{res} &= \bigl\{
  g \in \BB S' \otimes \BB C[z_{11},z_{12},z_{22}, N(Z)^{-1}] ;\:
  N(Z)^{-1} \cdot g(Z^{-1}) \cdot Z \in {\cal U}'^+_{res} \bigr\}.
\end{align*}

\begin{thm}  \label{U-K-type-decomp-thm}
The $K$-types of $(\pi'_l, {\cal U}_{res})$ are $V_{l+\frac12}$,
$l=0,1,2,3,\dots$,
spanned by $N(Z)^k \cdot {\bf u^t_{l,m}}(Z)$ and by
$N(Z)^k \cdot {\bf u^b_{l+1,m}}(Z)$, $-l \le m \le l+1$.
Furthermore,
\begin{align*}
{\cal U}^+_{res} &= \BB C\text{-span of }
\bigl\{ N(Z)^k \cdot {\bf u^t_{l,m}}(Z),\: N(Z)^k \cdot {\bf u^b_{l+1,m}}(Z) ;\:
k,l \ge 0 \bigr\},  \\
{\cal U}^-_{res} &= \BB C\text{-span of }
\bigl\{ N(Z)^{-(k+l+1)} \cdot {\bf u^t_{l,m}}(Z),\:
N(Z)^{-(k+l+1)} \cdot {\bf u^b_{l+1,m}}(Z) ;\: k,l \ge 0 \bigr\}.
\end{align*}

Similarly, the $K$-types of $(\pi'_r, {\cal U}'_{res})$ are $V_{l+\frac12}$,
$l=0,1,2,3,\dots$,
spanned by $N(Z)^k \cdot {\bf u'^t_{l,m}}(Z)$ and by
$N(Z)^k \cdot {\bf u'^b_{l+1,m}}(Z)$, $-l \le m \le l+1$.
Furthermore,
\begin{align*}
{\cal U}'^+_{res} &= \BB C\text{-span of }
\bigl\{ N(Z)^k \cdot {\bf u'^t_{l,m}}(Z),\: N(Z)^k \cdot {\bf u'^b_{l+1,m}}(Z) ;\:
k,l \ge 0 \bigr\},  \\
{\cal U}'^-_{res} &= \BB C\text{-span of }
\bigl\{ N(Z)^{-(k+l+1)} \cdot {\bf u'^t_{l,m}}(Z),\:
N(Z)^{-(k+l+1)} \cdot {\bf u'^b_{l+1,m}}(Z) ;\: k,l \ge 0 \bigr\}.
\end{align*}
\end{thm}
  
Let ${\cal U}^{\pm}(1)$ and ${\cal U}'^{\pm}(1)$ denote the respective kernels
of the natural restriction maps ${\cal U}^{\pm} \to {\cal U}^{\pm}_{res}$ and
${\cal U}'^{\pm} \to {\cal U}'^{\pm}_{res}$:
\begin{align*}
{\cal U}^{\pm}(1) &=
\Bigl\{ f \in {\cal U}^{\pm};\: f \bigr|_{\HC^{\times} \text{-sym}} =0 \Bigr\},  \\
{\cal U}'^{\pm}(1) &=
\Bigl\{ g \in {\cal U}'^{\pm};\: g \bigr|_{\HC^{\times} \text{-sym}} =0 \Bigr\}.
\end{align*}
And let ${\cal U}^{\pm}(2)$ and ${\cal U}'^{\pm}(2)$ denote the subspaces of
${\cal U}^{\pm}(1)$ and ${\cal U}'^{\pm}(1)$ respectively vanishing on the
symmetric quaternions to the order of two. More precisely, recall
that the quaternionic unit $e_2$ corresponds to the skew-symmetric
matrix $\bigl(\begin{smallmatrix} 0 & -1 \\ 1 & 0 \end{smallmatrix}\bigr)$.
Thus, the symmetric quaternions are characterized as those whose
$e_2$-coordinate $z^2$ is zero. Then
\begin{align*}
{\cal U}^{\pm}(2) &= \Bigl\{ f \in {\cal U}^{\pm} ;\:
f \bigr|_{\HC^{\times} \text{-sym}} =0 \text{ and }
\frac{\partial f}{\partial z^2} \Bigr|_{\HC^{\times} \text{-sym}} =0 \Bigr\},  \\
{\cal U}'^{\pm}(2) &= \Bigl\{ g \in {\cal U}'^{\pm} ;\:
g \bigr|_{\HC^{\times} \text{-sym}} =0 \text{ and }
\frac{\partial g}{\partial z^2} \Bigr|_{\HC^{\times} \text{-sym}} =0 \Bigr\}.
\end{align*}

\begin{thm}
The representations $(\pi'_l, {\cal U}^{\pm})$ and $(\pi'_r, {\cal U}'^{\pm})$
restricted to $\mathfrak{sp}(4,\BB C)$ decompose into three irreducible
components each:
\[
{\cal U}^{\pm}_{res} = {\cal U}^{\pm} / {\cal U}^{\pm}(1), \quad
{\cal U}^{\pm}(1) / {\cal U}^{\pm}(2), \quad {\cal U}^{\pm}(2)
\]
and
\[
{\cal U}'^{\pm}_{res} = {\cal U}'^{\pm} / {\cal U}'^{\pm}(1), \quad
{\cal U}'^{\pm}(1) / {\cal U}'^{\pm}(2), \quad {\cal U}'^{\pm}(2).
\]
The middle quotients are isomorphic to the (anti) regular functions:
\begin{equation}  \label{UV-iso}
\bigl( \pi'_l, {\cal U}^{\pm}(1)/{\cal U}^{\pm}(2) \bigr) \simeq
(\pi_{la}, {\cal V}^{\pm}_a) \quad \text{and} \quad
\bigl( \pi'_r, {\cal U}'^{\pm}(1)/{\cal U}'^{\pm}(2) \bigr) \simeq
(\pi_{ra}, {\cal V}'^{\pm}_a).
\end{equation}
\end{thm}

\begin{proof}
Consider the maps ${\cal U}^{\pm}(1) \to {\cal V}^{\pm}_a$ and
${\cal U}'^{\pm}(1) \to {\cal V}'^{\pm}_a$ given by the differentiation in the
$e_2$-direction followed by the restriction to $\HC^{\times} \text{-sym}$:
\begin{equation}  \label{UV-intertwiner}
f(Z) \mapsto \frac{\partial}{\partial z^2} f(Z) \Bigr|_{\HC^{\times} \text{-sym}}.
\end{equation}
The Lie algebra actions $\pi_{la}$, $\pi_{ra}$, $\pi'_l$ and $\pi'_r$ are
described in Lemmas \ref{pi-Lie_alg-action} and \ref{pi-prime-action-lem}.
Clearly, the map \eqref{UV-intertwiner} intertwines the actions of
$\bigl(\begin{smallmatrix} 0 & B \\ 0 & 0 \end{smallmatrix}\bigr)
\in \mathfrak{sp}(4,\BB C)$, $B \in \HC\text{-sym}$.
To show that the map \eqref{UV-intertwiner} intertwines the actions of
$\bigl(\begin{smallmatrix} 0 & 0 \\ C & 0 \end{smallmatrix}\bigr)
\in \mathfrak{sp}(4,\BB C)$, $C \in \HC\text{-sym}$,
we consider the effect of inversions:
\begin{multline*}
\frac{\partial}{\partial z^2} \Bigl(
\pi'_l \bigl(\begin{smallmatrix} 0 & 1 \\ 1 & 0 \end{smallmatrix}\bigr)
f(Z) \Bigr) \Bigr|_{\HC^{\times} \text{-sym}}
= - \frac{\partial}{\partial z^2}
\Bigl( \frac{Z}{N(Z)} \cdot f(Z^+/N(Z)) \Bigr) \Bigr|_{\HC^{\times} \text{-sym}}  \\
= - \frac{Z}{N(Z)} \cdot \frac{\partial}{\partial z^2}
\bigl( f(Z^+/N(Z)) \bigr) \Bigr|_{\HC^{\times} \text{-sym}}
= \Bigl( \frac{Z}{N(Z)^2} \cdot \frac{\partial}{\partial z^2} f(Z)
\Bigr|_{Z^{-1}} \Bigr)\Bigr|_{\HC^{\times} \text{-sym}}  \\
= - \pi_{la} \bigl(\begin{smallmatrix} 0 & 1 \\ 1 & 0 \end{smallmatrix}\bigr)
\frac{\partial}{\partial z^2} f(Z) \Bigr|_{\HC^{\times} \text{-sym}}.
\end{multline*}
It follows that the map \eqref{UV-intertwiner} intertwines the actions of
$\bigl(\begin{smallmatrix} 0 & 0 \\ C & 0 \end{smallmatrix}\bigr)
\in \mathfrak{sp}(4,\BB C)$ as well.
By the irreducibility of regular functions, the map \eqref{UV-intertwiner}
descends to the isomorphisms \eqref{UV-iso}.

It remains to prove the irreducibility of the quotients.
We compute the effect of
$\pi'_l \bigl( \begin{smallmatrix} 0 & B \\ 0 & 0 \end{smallmatrix} \bigr)$,
$B \in \HC\text{-sym}$, on the $K$-types:
\begin{multline*}
\tfrac{\partial}{\partial t} \bigl( N(Z)^k \cdot {\bf u^t_{l,m}}(Z) \bigr)
= \tfrac{(2k+2l+1)(l+m)}{2l+1} N(Z)^k \cdot {\bf u^t_{l-1,m}}(Z)  \\
+ \tfrac{2k(l-m+2)}{2l+3} N(Z)^{k-1} \cdot {\bf u^t_{l+1,m}}(Z)
- \tfrac{2k(2m-1)}{(2l+1)(2l+3)} N(Z)^{k-1} \cdot {\bf u^b_{l+1,m}}(Z),
\end{multline*}
\begin{multline*}
2 \tfrac{\partial}{\partial z_{11}} \bigl( N(Z)^k \cdot {\bf u^t_{l,m}}(Z) \bigr)
= -\tfrac{(2k+2l+1)(l+m)(l+m-1)}{2l+1} N(Z)^k \cdot {\bf u^t_{l-1,m-1}}(Z)  \\
+ \tfrac{2k(l-m+2)(l-m+3)}{2l+3} N(Z)^{k-1} \cdot {\bf u^t_{l+1,m-1}}(Z)
- \tfrac{2k(l+m)(l-m+2)}{(2l+1)(2l+3)} N(Z)^{k-1} \cdot {\bf u^b_{l+1,m-1}}(Z),
\end{multline*}
\begin{multline*}
2 \tfrac{\partial}{\partial z_{22}} \bigl( N(Z)^k \cdot {\bf u^t_{l,m}}(Z) \bigr)
= \tfrac{2k+2l+1}{2l+1} N(Z)^k \cdot {\bf u^t_{l-1,m+1}}(Z)  \\
- \tfrac{2k}{2l+3} N(Z)^{k-1} \cdot {\bf u^t_{l+1,m+1}}(Z)
- \tfrac{4k}{(2l+1)(2l+3)} N(Z)^{k-1} \cdot {\bf u^b_{l+1,m+1}}(Z);
\end{multline*}
\begin{multline*}
\tfrac{\partial}{\partial t} \bigl( N(Z)^k \cdot {\bf u^b_{l,m}}(Z) \bigr)
= - \tfrac{(2k+2l+1)(2m-1)}{(2l+1)(2l-1)} N(Z)^k \cdot {\bf u^t_{l-1,m}}(Z)  \\
+ \tfrac{(2k+2l+1)(l+m-1)}{2l-1} N(Z)^k \cdot {\bf u^b_{l-1,m}}(Z)
+ \tfrac{2k(l-m+1)}{2l+1} N(Z)^{k-1} \cdot {\bf u^b_{l+1,m}}(Z),
\end{multline*}
\begin{multline*}
2 \tfrac{\partial}{\partial z_{11}} \bigl( N(Z)^k \cdot {\bf u^b_{l,m}}(Z) \bigr)
= - \tfrac{2(2k+2l+1)(l-m+1)(l+m-1)}{(2l+1)(2l-1)} N(Z)^k \cdot {\bf u^t_{l-1,m-1}}(Z)  \\
- \tfrac{(2k+2l+1)(l+m-1)(l+m-2)}{2l-1} N(Z)^k \cdot {\bf u^b_{l-1,m-1}}(Z)
+ \tfrac{2k(l-m+1)(l-m+2)}{2l+1} N(Z)^{k-1} \cdot {\bf u^b_{l+1,m-1}}(Z),
\end{multline*}
\begin{multline*}
2 \tfrac{\partial}{\partial z_{22}} \bigl( N(Z)^k \cdot {\bf u^b_{l,m}}(Z) \bigr)
= - \tfrac{2(2k+2l+1)}{(2l+1)(2l-1)} N(Z)^k \cdot {\bf u^t_{l-1,m+1}}(Z)  \\
+ \tfrac{2k+2l+1}{2l-1} N(Z)^k \cdot {\bf u^b_{l-1,m+1}}(Z)
- \tfrac{2k}{2l+1} N(Z)^{k-1} \cdot {\bf u^b_{l+1,m+1}}(Z).
\end{multline*}

Now we prove that the $\mathfrak{sp}(4,\BB C)$-module
${\cal U}^+_{res} = {\cal U}^+ / {\cal U}^+(1)$ is irreducible by showing that
any non-zero element $f_0 \in {\cal U}^+_{res}$ generates the whole space.
Since all $K$-types of $(\pi'_l, {\cal U}^+_{res})$ have multiplicity one,
a non-zero element $f_0 \in {\cal U}^+_{res}$ generates at least one
$K$-type of the kind $N(Z)^{k_0} \cdot {\bf u^t_{l_0,m}}(Z)$ or
$N(Z)^{k_0} \cdot {\bf u^b_{l_0,m}}(Z)$ for some $k_0 \ge 0$.
The above formulas show that such a $K$-type generates a $K$-type of the kind
${\bf u^t_{k_0+l_0,m}}(Z)$ or ${\bf u^b_{k_0+l_0,m}}(Z)$.
These $K$-types, in turn, generate the lowest $K$-type ${\bf u^t_{0,m}}(Z)$.
To show that this $K$-type generates entire ${\cal U}^+_{res}$, it is enough to
apply the inversion and show that $N(Z)^{-1} \cdot {\bf u^b_{1,m}}(Z)$
generates ${\cal U}^-_{res}$.
Indeed, the above formulas show that this $K$-type generates the $K$-types of the
kind $N(Z)^{-1} \cdot {\bf u^t_{0,m}}(Z)$, $N(Z)^{-l-1} \cdot {\bf u^t_{l,m}}(Z)$ and
$N(Z)^{-l-1} \cdot {\bf u^b_{l+1,m}}(Z)$, $l \ge 0$.
These $K$-types, in turn, generate $N(Z)^{-l-1} \cdot {\bf u^t_{l-k,m}}(Z)$ and
$N(Z)^{-l-1} \cdot {\bf u^b_{l-k+1,m}}(Z)$, $0 \le k \le l$, and,
by Theorem \ref{U-K-type-decomp-thm}, all of ${\cal U}^-_{res}$.

The remaining cases are similar, the only difference is a shift in the power of $N(Z)$
when applying the inversion.
\end{proof}

\subsection{Regular Functions on $\HC^{\times}\text{-sym}$}

In this subsection we introduce spaces of solutions of Dirac-type equations
on $\HC^{\times}\text{-sym}$ together with actions of
$\mathfrak{sp}(4,\BB C)$. On the one hand, these are symplectic analogues
of the well-familiar $\mathfrak{gl}(2,\HC)$-modules of left and right regular
functions $(\pi_l, {\cal V})$ and $(\pi_r, {\cal V}')$.
On the other hand, these spaces appear in the trilinear form that will be
constructed in Subsection \ref{3form-main-subsect}.

We consider solutions of Dirac-type equations on $\HC^{\times}\text{-sym}$:
\begin{align*}
\widetilde {\cal V}_{3-dim} &= \Bigl\{ \begin{smallmatrix}
\text{holomorphic $\BB S$-valued solutions of $\partial^+f=0$}  \\
\text{on $\HC \text{-sym}$ (possibly with singularities)}
\end{smallmatrix} \Bigr\},  \\
\widetilde {\cal V}'_{3-dim} &= \Bigl\{ \begin{smallmatrix}
\text{holomorphic $\BB S'$-valued solutions of
$g \overleftarrow{\partial^+}=0$} \\
\text{on $\HC \text{-sym}$ (possibly with singularities)}
\end{smallmatrix} \Bigr\},
\end{align*}
where $\partial^+$ is a Dirac-type operator
\[
\partial^+ = \begin{pmatrix}
  \frac{\partial}{\partial z_{22}} & \frac{i}2 \frac{\partial}{\partial t} \\
  \frac{i}2 \frac{\partial}{\partial t} & \frac{\partial}{\partial z_{11}}
\end{pmatrix}
\]
factoring the Laplacian on $\HC \text{-sym}$:
\[
\Delta_3 = 4 \partial^+ \partial = 4 \partial \partial^+.
\]

\begin{thm}
\begin{enumerate}
\item
The space $\widetilde {\cal V}_{3-dim}$ of left regular functions
$f: \HC \text{-sym} \to \BB S$ (possibly with singularities)
is invariant under the following action of $Sp(4,\BB R)$:
\begin{multline*}
\pi^{3-dim}_l(h): \: f(Z) \: \mapsto \: \bigl( \pi^{3-dim}_l(h)f \bigr)(Z) =
\frac {(cZ+d)^{-1}}{\sqrt{N(cZ+d)}} \cdot f \bigl( (aZ+b)(cZ+d)^{-1} \bigr),
\\
h^{-1} =
\bigl( \begin{smallmatrix} a & b \\ c & d \end{smallmatrix} \bigr)
\in Sp(4,\BB R).
\end{multline*}
\item
The space $\widetilde {\cal V}'_{3-dim}$ of right regular functions
$g: \HC \text{-sym} \to \BB S'$ (possibly with singularities)
is invariant under the following action of $Sp(4,\BB R)$:
\begin{multline*}
\pi^{3-dim}_r(h): \: g(Z) \: \mapsto \: \bigl( \pi^{3-dim}_r(h)g \bigr)(Z) =
g \bigl( (a'-Zc')^{-1}(-b'+Zd') \bigr)
\cdot \frac {(a'-Zc')^{-1}}{\sqrt{N(a'-Zc')}},  \\
h =
\bigl( \begin{smallmatrix} a' & b' \\ c' & d' \end{smallmatrix} \bigr)
\in Sp(4,\BB R).
\end{multline*}
\end{enumerate}
\end{thm}

Differentiating $\pi^{3-dim}_l$ and $\pi^{3-dim}_r$, we obtain actions
of the Lie algebra $\mathfrak{sp}(4,\BB C)$, which we still denote by
$\pi^{3-dim}_l$ and $\pi^{3-dim}_r$ respectively.
We spell out these Lie algebra actions (cf. Lemma \ref{pi-Lie_alg-action}):

\begin{lem}  \label{pi-3dim-Lie_alg-action}
The Lie algebra action $\pi^{3-dim}_l$ of $\mathfrak{sp}(4,\BB C)$ on
$\widetilde {\cal V}_{3-dim}$ is given by
\begin{align*}
\pi^{3-dim}_l \bigl( \begin{smallmatrix} A & 0 \\ 0 & -A^T
\end{smallmatrix} \bigr) &:
f(Z) \mapsto - \tr (AZ \partial + ZA^T \partial + \tfrac12 A) f - A^Tf,  \\
\pi^{3-dim}_l \bigl( \begin{smallmatrix} 0 & B \\ 0 & 0
\end{smallmatrix} \bigr) &:
f(Z) \mapsto - \tr (B \partial) f,  \\
\pi^{3-dim}_l \bigl( \begin{smallmatrix} 0 & 0 \\ C & 0
\end{smallmatrix} \bigr) &:
f(Z) \mapsto \tr (ZCZ \partial + \tfrac12 CZ) f + CZf.
\end{align*}

Similarly, the Lie algebra action $\pi^{3-dim}_r$ of
$\mathfrak{sp}(4,\BB C)$ on $\widetilde {\cal V}'_{3-dim}$ is given by
\begin{align*}
\pi^{3-dim}_r \bigl( \begin{smallmatrix} A & 0 \\ 0 & -A^T
\end{smallmatrix} \bigr) &:
g(Z) \mapsto - \tr (AZ \partial + ZA^T \partial + \tfrac12 A) g - gA,  \\
\pi^{3-dim}_r \bigl( \begin{smallmatrix} 0 & B \\ 0 & 0
\end{smallmatrix} \bigr) &:
g(Z) \mapsto - \tr (B \partial) g,  \\
\pi^{3-dim}_r \bigl( \begin{smallmatrix} 0 & 0 \\ C & 0
\end{smallmatrix} \bigr) &:
g(Z) \mapsto \tr (ZCZ \partial + \tfrac12 ZC) g + gZC.
\end{align*}
\end{lem}

Lemma \ref{pi-3dim-Lie_alg-action} implies that the Lie algebra actions
$\pi^{3-dim}_l$ and $\pi^{3-dim}_r$ preserve the spaces of polynomial
regular functions on $\HC^{\times}\text{-sym}$:
\begin{align*}
{\cal V}_{3-dim}^+ &=
\bigl\{ f \in \BB S \otimes \BB C [z_{11}, z_{12}, z_{22},N(Z)^{-1}];\:
\partial^+ f=0 \bigr\},  \\
{\cal V}_{3-dim}^- &= \Bigl\{ f \in \BB S \otimes
N(Z)^{\frac12} \cdot \BB C[z_{11},z_{12},z_{22}, N(Z)^{-1}] ;\:
\partial^+ f=0 \Bigr\},  \\
{\cal V}'^+_{3-dim} &=
\bigl\{ g \in \BB S' \otimes \BB C [z_{11}, z_{12}, z_{22},N(Z)^{-1}];\:
g \overleftarrow{\partial^+}=0 \bigr\},  \\
{\cal V}'^-_{3-dim} &= \Bigl\{ g \in \BB S' \otimes
N(Z)^{\frac12} \cdot \BB C[z_{11},z_{12},z_{22}, N(Z)^{-1}] ;\:
g \overleftarrow{\partial^+}=0 \bigr\}.
\end{align*}

We can identify the $K$-types of ${\cal V}_{3-dim}^{\pm}$ and
${\cal V}'^{\pm}_{3-dim}$.

\begin{thm}
The $\mathfrak{sp}(4,\BB C)$-modules
\[
(\pi^{3-dim}_l, {\cal V}_{3-dim}^+), \qquad
(\pi^{3-dim}_l, {\cal V}_{3-dim}^-), \qquad
(\pi^{3-dim}_r, {\cal V}'^+_{3-dim}), \qquad
(\pi^{3-dim}_r, {\cal V}'^-_{3-dim})
\]
are irreducible.
The $K$-types of these modules are $V_{l+\frac12}$, $l=0,1,2,3,\dots$,
spanned by respectively
\[
{\bf v^t_{l,m}}(Z), \qquad
N(Z)^{-l-\frac32} \cdot {\bf v^b_{l+1,m}}(Z), \qquad
{\bf v'^t_{l,m}}(Z), \qquad
N(Z)^{-l-\frac32} \cdot {\bf v'^b_{l+1,m}}(Z),
\]
where $-l-1 \le m \le l$. We have:
\begin{align*}
{\cal V}_{3-dim}^+ &=
\{\text{$\BB S$-valued polynomial solutions of $\partial^+f=0$ on
  $\HC \text{-sym}$}\}  \\
&= \bigl\{ f \in \BB S \otimes \BB C [z_{11}, z_{12}, z_{22}];\:
\partial^+ f=0 \bigr\}  \\
&= \BB C\text{-span of }
\bigl\{ {\bf v^t_{l,m}}(Z) ;\: l \ge 0,\: -l-1 \le m \le l \bigr\},  \\
{\cal V}_{3-dim}^- &= \Bigl\{
f \in \BB S \otimes N(Z)^{\frac12} \cdot \BB C[z_{11},z_{12},z_{22}, N(Z)^{-1}] ;\:
\tfrac{Z^{-1}}{N(Z)^{\frac12}} \cdot f(Z^{-1}) \in {\cal V}_{3-dim}^+ \Bigr\} \\
&= \BB C\text{-span of } \bigl\{ N(Z)^{-l-\frac12} \cdot {\bf v^b_{l,m}}(Z) ;\:
l \ge 1,\: -l \le m \le l-1 \bigr\},  \\
{\cal V}'^+_{3-dim} &=
\{\text{$\BB S'$-valued polynomial solutions of
  $g \overleftarrow{\partial^+}=0$ on $\HC \text{-sym}$}\}  \\
&= \bigl\{ g \in \BB S' \otimes \BB C [z_{11}, z_{12}, z_{22}];\:
g \overleftarrow{\partial^+}=0 \bigr\}  \\
&= \BB C\text{-span of }
\bigl\{ {\bf v'^t_{l,m}}(Z) ;\: l \ge 0,\: -l-1 \le m \le l \bigr\},  \\
{\cal V}'^-_{3-dim} &= \Bigl\{
g \in \BB S' \otimes N(Z)^{\frac12} \cdot \BB C[z_{11},z_{12},z_{22}, N(Z)^{-1}] ;\:
g(Z^{-1}) \cdot \tfrac{Z^{-1}}{N(Z)^{\frac12}} \in {\cal V}'^+_{3-dim} \Bigr\}  \\
&= \BB C\text{-span of } \bigl\{ N(Z)^{-l-\frac12} \cdot {\bf v'^b_{l,m}}(Z) ;\:
l \ge 1,\: -l \le m \le l-1 \bigr\}.
\end{align*}
\end{thm}

\begin{proof}
To find the $K$-types of ${\cal V}_{3-dim}^+$, we consider the space of {\em all}
polynomial $\BB S$-valued functions on $\HC^{\times} \text{-sym}$:
\[
\bigl( \pi^{3-dim}_l, \BB S \otimes \BB C [z_{11}, z_{12}, z_{22},N(Z)^{-1}] \bigr).
\]
This $\mathfrak{sp}(4,\BB C)$-module has $K$-types $V_{l+\frac12}$,
$l=0,1,2,3,\dots$, spanned by $N(Z)^k \cdot {\bf v^t_{l,m}}(Z)$
and by $N(Z)^k \cdot {\bf v^b_{l+1,m}}(Z)$,  $-l-1 \le m \le l$.
(Cf. Theorem \ref{1reg-res-irred-thm}.)
Using \eqref{dtNR}-\eqref{dz_{22}NR}, we compute:  
\begin{align*}
\partial^+ \bigl( N(Z)^k \cdot {\bf v^t_{l,m}}(Z) \bigr) &=
k N(Z)^{k-1} \cdot \begin{pmatrix} -R^{m+1}_{l+1}(Z) \\
  i(l-m+1) R^m_{l+1}(Z) \end{pmatrix},  \\
\partial^+ \bigl( N(Z)^k \cdot {\bf v^b_{l,m}}(Z) \bigr) &=
(2k+2l+1) N(Z)^k \cdot \begin{pmatrix} R^{m+1}_{l-1}(Z) \\
  i(l+m+1) R^m_{l-1}(Z) \end{pmatrix}.
\end{align*}
Therefore, the subspace of functions $f$ in
$\BB S \otimes \BB C [z_{11}, z_{12}, z_{22},N(Z)^{-1}]$ satisfying
$\partial^+ f=0$ is exactly
\[
\BB C\text{-span of }
\bigl\{ {\bf v^t_{l,m}}(Z) ;\: l \ge 0,\: -l-1 \le m \le l \bigr\}.
\]

The irreducibility of $(\pi^{3-dim}_l, {\cal V}_{3-dim}^+)$ follows from
explicit computation of $\pi^{3-dim}_l
\bigl(\begin{smallmatrix} 0 & B \\ 0 & 0 \end{smallmatrix}\bigr)$,
$B \in \HC \text{-sym}$, on the $K$-types:
\begin{align*}
\tfrac{\partial}{\partial t} {\bf v^t_{l,m}}(Z) &=
(l+m+1) {\bf v^t_{l-1,m}}(Z), \\
2\tfrac{\partial}{\partial z_{11}} {\bf v^t_{l,m}}(Z) &=
-(l+m)(l+m+1) {\bf v^t_{l-1,m-1}}(Z), \\
2\tfrac{\partial}{\partial z_{22}} {\bf v^t_{l,m}}(Z) &=
{\bf v^t_{l-1,m+1}}(Z);  \\
\tfrac{\partial}{\partial t}
\bigl( N(Z)^{-l-\frac32} \cdot {\bf v^b_{l+1,m}}(Z)\bigr) &=
-(l-m+1) N(Z)^{-l-\frac52} \cdot {\bf v^b_{l+2,m}}(Z), \\
\tfrac{\partial}{\partial z_{11}}
\bigl( N(Z)^{-l-\frac32} \cdot {\bf v^b_{l+1,m}}(Z)\bigr) &=
-(l-m+1)(l-m+2) N(Z)^{-l-\frac52} \cdot {\bf v^b_{l+2,m-1}}(Z), \\
\tfrac{\partial}{\partial z_{22}}
\bigl( N(Z)^{-l-\frac32} \cdot {\bf v^b_{l+1,m}}(Z)\bigr) &=
N(Z)^{-l-\frac52} \cdot {\bf v^b_{l+2,m+1}}(Z),
\end{align*}
and the effect of the inversions (cf. Lemma \ref{V-inversion-lem}):
\begin{align*}
\pi^{3-dim}_l
\bigl(\begin{smallmatrix} 0 & 1 \\ 1 & 0 \end{smallmatrix}\bigr)
 {\bf v^t_{l,m}}(Z) &= (-1)^l \tfrac{(l+m+1)!}{(l-m)!}
 N(Z)^{-l-\frac32} \cdot {\bf v^b_{l+1,-m-1}}(Z),  \\
\pi^{3-dim}_l
\bigl(\begin{smallmatrix} 0 & 1 \\ 1 & 0 \end{smallmatrix}\bigr)
\bigl( N(Z)^{-l-\frac32} \cdot {\bf v^b_{l+1,m}}(Z) \bigr) &=
(-1)^l \tfrac{(l+m+1)!}{(l-m)!} {\bf v^t_{l,-m-1}}(Z).
\end{align*}

The other cases are treated similarly.
\end{proof}

\section{Structure of $(\rho'_2, \mathring {\cal W}'_{res})$ and Related
$\mathfrak{sp}(4,\BB C)$-Modules}  \label{6}

In this section we introduce an $\mathfrak{sp}(4,\BB C)$-module
$(\rho'_2, \mathring {\cal W}'_{res})$ and identify some of its
irreducible components. Then we introduce closely related
$\mathfrak{sp}(4,\BB C)$-modules $(\rho^*_2, \mathring {\cal W}^*_{res})$
and $(\rho'_2, \mathring {\cal W}'_{\frac12})$
consisting of $\HC\text{-sym}$ valued functions.
All these modules contain an irreducible component isomorphic to
$(\pi_{dl},{\cal F}^+) \simeq (\pi_{dr},{\cal G}^+)$ and another irreducible
component isomorphic to $(\pi_{dl},{\cal F}^-) \simeq (\pi_{dr},{\cal G}^-)$.

\subsection{$K$-types of $(\rho'_2, \mathring {\cal W}'_{res})$}  \label{W'-subsection}

Recall from \cite{FL1,ATMP} that
\begin{align*}
{\cal W}' &= \{\text{$\HC$-valued polynomial functions on $\HC^{\times}$}\}  \\
&= \HC \otimes \BB C [z_{11}, z_{12}, z_{21}, z_{22}, N(Z)^{-1}].
\end{align*}
The Lie algebra $\mathfrak{gl}(2,\HC)$ acts on this space by
differentiating the following group action:
\[
\rho'_2(h): \: F(Z) \: \mapsto \: \bigl( \rho'_2(h)F \bigr)(Z) =
\frac {a'-Zc'}{N(a'-Zc')} \cdot F \bigl( (aZ+b)(cZ+d)^{-1} \bigr)
\cdot \frac {cZ+d}{N(cZ+d)},
\]
where $F \in {\cal W}'$,
$h = \bigl(\begin{smallmatrix} a' & b' \\ c' & d' \end{smallmatrix}\bigr)
\in GL(2,\HC)$ and 
$h^{-1} = \bigl(\begin{smallmatrix} a & b \\ c & d \end{smallmatrix}\bigr)$.
Restricting to $\mathfrak{sp}(4,\BB C)$ and $\HC^{\times} \text{-sym}$,
we obtain a representation $(\rho'_2, {\cal W}'_{res})$ of
$\mathfrak{sp}(4,\BB C)$, where
\begin{align*}
{\cal W}'_{res} &= \{\text{$\HC$-valued polynomial
  functions on $\HC^{\times} \text{-sym}$}\}  \\
&= \HC \otimes \BB C [z_{11}, z_{12}, z_{22}, N(Z)^{-1}].
\end{align*}
We have the following restatement of Lemma 26 in \cite{ATMP}:

\begin{lem}  \label{rho'_2-action-lem}
The Lie algebra action $\rho'_2$ of $\mathfrak{sp}(4,\BB C)$ on ${\cal W}'_{res}$
is given by
\begin{align*}
\rho'_2 \bigl( \begin{smallmatrix} A & 0 \\ 0 & -A^T \end{smallmatrix} \bigr) &:
F(Z) \mapsto - \tr (AZ \partial + ZA^T \partial + 2A) F +AF+FA^T,  \\
\rho'_2 \bigl( \begin{smallmatrix} 0 & B \\ 0 & 0 \end{smallmatrix} \bigr) &:
F(Z) \mapsto - \tr (B \partial) F,  \\
\rho'_2 \bigl( \begin{smallmatrix} 0 & 0 \\ C & 0 \end{smallmatrix} \bigr) &:
F(Z) \mapsto \tr (ZCZ \partial +2ZC) F - ZCF - FCZ.
\end{align*}
\end{lem}

Observe that $(\rho'_2, {\cal W}'_{res})$ decomposes into a direct sum of
two invariant subspaces -- functions with values in the symmetric matrices in
$\HC$, $\HC\text{-sym}$, and functions with antisymmetric values.
We are interested in the symmetric part only and call it
$\mathring {\cal W}'_{res}$:
\begin{align*}
\mathring {\cal W}'_{res} &= \{\text{$\HC\text{-sym}$-valued polynomial
  functions on $\HC^{\times} \text{-sym}$}\}  \\
&= \HC\text{-sym} \otimes \BB C [z_{11}, z_{12}, z_{22}, N(Z)^{-1}].
\end{align*}

In this subsection we describe the $K$-types of
$(\rho'_2, \mathring {\cal W}'_{res})$. Consider
\begin{align*}
{\bf T}^m_l(Z) &= \begin{pmatrix} R^{m+1}_l(Z) & i(l+m+1) R^m_l(Z) \\
i(l+m+1) R^m_l(Z) & -(l+m)(l+m+1) R^{m-1}_l(Z) \end{pmatrix},
\quad \text{\footnotesize $-l-1 \le m \le l+1$} ,\\
{\bf M}^m_l(Z) &= \begin{pmatrix} R^{m+1}_l(Z) & im R^m_l(Z) \\
im R^m_l(Z) & (l+m)(l-m+1) R^{m-1}_l \end{pmatrix},
\quad \text{\footnotesize $-l \le m \le l$},  \\
{\bf B}^m_l(Z) &= \begin{pmatrix} R^{m+1}_l(Z) & -i(l-m) R^m_l(Z) \\
-i(l-m) R^m_l(Z) & -(l-m)(l-m+1) R^{m-1}_l(Z) \end{pmatrix},
\quad \text{\footnotesize $-l+1 \le m \le l-1$}.
\end{align*}
For a fixed $l$, these functions span $U(2)$-invariant subspaces of dimensions
$2l+2$, $2l$ and $2l-1$ respectively.
(The letters {\bf T}, {\bf M} and {\bf B} stand respectively for ``top'',
``middle'' and ``bottom''.)
Let $A_1=\bigl( \begin{smallmatrix} 0 & 1 \\ 0 & 0 \end{smallmatrix} \bigr)$
and $A_2=\bigl( \begin{smallmatrix} 0 & 0 \\ 1 & 0 \end{smallmatrix} \bigr)$,
then
\begin{align*}
\rho'_2 \bigl(\begin{smallmatrix} A_1 & 0 \\ 0 & -A_1^T \end{smallmatrix}\bigr)
{\bf T}^m_l(Z) &= i(l+m+1)(l-m+2) {\bf T}^{m-1}_l(Z),  \\
\rho'_2 \bigl(\begin{smallmatrix} A_2 & 0 \\ 0 & -A_2^T \end{smallmatrix}\bigr)
{\bf T}^m_l(Z) &= -i {\bf T}^{m+1}_l(Z);  \\
\rho'_2 \bigl(\begin{smallmatrix} A_1 & 0 \\ 0 & -A_1^T \end{smallmatrix}\bigr)
{\bf M}^m_l(Z) &= i(l+m)(l-m+1) {\bf M}^{m-1}_l(Z),  \\
\rho'_2 \bigl(\begin{smallmatrix} A_2 & 0 \\ 0 & -A_2^T \end{smallmatrix}\bigr)
{\bf M}^m_l(Z) &= -i {\bf M}^{m+1}_l(Z);  \\
\rho'_2 \bigl(\begin{smallmatrix} A_1 & 0 \\ 0 & -A_1^T \end{smallmatrix}\bigr)
{\bf B}^m_l(Z) &= i(l+m-1)(l-m) {\bf B}^{m-1}_l(Z),  \\
\rho'_2 \bigl(\begin{smallmatrix} A_2 & 0 \\ 0 & -A_2^T \end{smallmatrix}\bigr)
{\bf B}^m_l(Z) &= -i {\bf B}^{m+1}_l(Z).
\end{align*}

We describe the effect of the inversion on the basis functions:

\begin{lem}  \label{W'-inversion-lem}
Let $\bigl(\begin{smallmatrix} 0 & 1 \\ 1 & 0 \end{smallmatrix}\bigr)
\in GL(2,\HC)$, then
\begin{align*}
\rho'_2 \bigl(\begin{smallmatrix} 0 & 1 \\ 1 & 0 \end{smallmatrix}\bigr)
\bigl( N(Z)^k \cdot {\bf T}^m_l(Z) \bigr) =
&\tfrac{(-1)^{l+1}}{2l+3} \tfrac{(l+m+1)!}{(l-m+1)!} N(Z)^{-(k+l+1)}
\cdot {\bf T}^{-m}_l(Z) \\
&+ (-1)^{l+1} \tfrac{2(l+1)}{2l+3} \tfrac{(l+m+1)!}{(l-m+1)!} N(Z)^{-(k+l+2)}
  \cdot {\bf B}^{-m}_{l+2}(Z),  \\
\rho'_2 \bigl(\begin{smallmatrix} 0 & 1 \\ 1 & 0 \end{smallmatrix}\bigr)
\bigl( N(Z)^k \cdot {\bf M}^m_l(Z) \bigr) =
&(-1)^l \tfrac{(l+m)!}{(l-m)!} N(Z)^{-(k+l+1)} \cdot {\bf M}^{-m}_l(Z),  \\
\rho'_2 \bigl(\begin{smallmatrix} 0 & 1 \\ 1 & 0 \end{smallmatrix}\bigr)
\bigl( N(Z)^k \cdot {\bf B}^m_{l+1}(Z) \bigr) =
&(-1)^l \tfrac{2(l+1)}{2l+1} \tfrac{(l+m)!}{(l-m)!} N(Z)^{-(k+l+1)}
\cdot {\bf T}^{-m}_{l-1}(Z) \\
&- \tfrac{(-1)^l}{2l+1} \tfrac{(l+m)!}{(l-m)!} N(Z)^{-(k+l+2)}
\cdot {\bf B}^{-m}_{l+1}(Z).
\end{align*}
\end{lem}

\begin{thm}
The functions
\begin{equation*}  
  N(Z)^k \cdot {\bf T}^m_l(Z), \qquad N(Z)^k \cdot {\bf M}^m_l(Z),
  \qquad N(Z)^k \cdot {\bf B}^m_l(Z)
\end{equation*}
form a $K$-type basis of $(\rho'_2,\mathring {\cal W}'_{res})$.
More precisely, for $k$, $l$ fixed, as representations of $SU(2)$,
\begin{align*}
  \BB C\text{-span of } \bigl\{ N(Z)^k \cdot {\bf T}^m_l(Z) ;\:
  -l-1 \le m \le l+1  \bigr\} &\simeq V_{l+1},  \\
  \BB C\text{-span of } \bigl\{ N(Z)^k \cdot {\bf M}^m_l(Z) ;\:
  -l \le m \le l \bigr\} &\simeq V_l,  \\
  \BB C\text{-span of } \bigl\{ N(Z)^k \cdot {\bf B}^m_l(Z) ;\:
  -l+1 \le m \le l-1 \bigr\} &\simeq V_{l-1}.
\end{align*}
\end{thm}

Note that the $K$-types $N(Z)^k \cdot {\bf B}^{m=0}_{l=1}(Z)$ and
$N(Z)^k \cdot {\bf M}^m_l(Z)$, $-l \le m \le l$, have multiplicity one.
On the other hand,
$N(Z)^k \cdot {\bf T}^m_l(Z)$ and $N(Z)^{k-1} \cdot {\bf B}^m_{l+2}(Z)$,
$-l-1 \le m \le l+1$, span isomorphic representations of $U(2)$.
We introduce an additional family of ``combination'' $K$-types of
$(\rho'_2, \mathring {\cal W}'_{res})$:
\[
{\bf C}^m_{k,l}(Z) = 2k N(Z)^{k-1} \cdot {\bf B}^m_{l+1}(Z)
- (2k+2l+1) N(Z)^k \cdot {\bf T}^m_{l-1}(Z),  \quad
\begin{smallmatrix} k \in \BB Z, \:  l \ge 0, \\
-l \le m \le l. \end{smallmatrix}
\]
From Lemma \ref{W'-inversion-lem} we find that the effect of the inversion
on these $K$-types is
\begin{equation}  \label{I-inversion}
\rho'_2 \bigl(\begin{smallmatrix} 0 & 1 \\ 1 & 0 \end{smallmatrix}\bigr)
\bigl( {\bf C}^m_{k,l}(Z) \bigr) =
(-1)^l \tfrac{(l+m)!}{(l-m)!} {\bf C}^{-m}_{-k-l,l}(Z).
\end{equation}

     

\subsection{Decomposition of
  $(\rho'_2, \mathring {\cal W}'^+_{res} \oplus \mathring {\cal W}'^-_{res})$
  into Irreducible Components}  \label{IrredDecom-subsection}

By Corollary 36 in \cite{ATMP}, $(\rho'_2, {\cal W}')$ contains an invariant
subspace
\begin{equation*}
{\cal M} = \bigl\{ F \in {\cal W}';\:
\M F=0, \: \tr ( \partial \circ \square F) =0 \bigr\}.
\end{equation*}
Inside ${\cal M}$ we have
\begin{align}
{\cal M}^+ &= \HC \otimes [z_{11},z_{12},z_{21},z_{22}] \cap {\cal M},
\label{M^+}\\
{\cal M}^- &= \bigl\{ F \in {\cal M};\:
\tfrac{Z}{N(Z)} \cdot F(Z^{-1}) \cdot \tfrac{Z}{N(Z)} \in {\cal M}^+ \bigr\}.
\label{M^-}
\end{align}
Note that ${\cal M}^+ \oplus {\cal M}^-$ is a proper subspace of ${\cal M}$;
indeed, it can be seen from Proposition 42 in \cite{ATMP} that the quotient
${\cal M} / ({\cal M}^+ \oplus {\cal M}^-)$ is one-dimensional.
In \cite{ATMP} we decomposed the $\mathfrak{gl}(2,\HC)$-module
$(\rho'_2,{\cal W}')$ into irreducible components.
It follows from this decomposition (Theorems 32, 49, Corollary 41
and Lemma 48 in \cite{ATMP}) that ${\cal M}^{\pm}$ contain three
$\mathfrak{gl}(2,\HC)$-irreducible components each. More precisely,
${\cal M}^+$ contains a unique irreducible submodule 
$\partial^+({\cal BH}^+)$ isomorphic to $(\rho', {\cal BH}^+/{\cal I}'_0)$
realized using the biharmonic functions regular at the origin, and
the quotient
\[
\biggl( \rho'_2, \frac{{\cal M}^+}{\partial^+({\cal BH}^+)} \biggr) \simeq
(\pi_{dl},{\cal F}^+) \oplus (\pi_{dr},{\cal G}^+)
\]
-- direct sum of the doubly left regular and doubly right regular functions
regular at the origin.
Similarly, ${\cal M}^-$ contains a unique irreducible submodule 
$\partial^+({\cal BH}^-)$ isomorphic to $(\rho', {\cal BH}^-/{\cal I}'_0)$
realized using the biharmonic functions regular at infinity, and
the quotient
\[
\biggl( \rho'_2, \frac{{\cal M}^-}{\partial^+({\cal BH}^-)} \biggr) \simeq
(\pi_{dl},{\cal F}^-) \oplus (\pi_{dr},{\cal G}^-)
\]
-- direct sum of the doubly left regular and doubly right regular functions
regular at infinity.
In particular, there are four indecomposable
submodules of ${\cal M} \subset {\cal W}'$ containing the doubly regular
functions:
\begin{equation*}  
\partial^+({\cal BH}^+) + {\cal F}^+, \qquad
\partial^+({\cal BH}^+) + {\cal G}^+, \qquad
\partial^+({\cal BH}^-) + {\cal F}^-, \qquad
\partial^+({\cal BH}^-) + {\cal G}^-,
\end{equation*}
each composed of two irreducible components.
The first two submodules contain an irreducible subspace
$\partial^+({\cal BH}^+)$, and the latter two contain an irreducible subspace
$\partial^+({\cal BH}^-)$.
The quotients
\begin{equation}  \label{2-reg-quotients}
\frac{\partial^+({\cal BH}^+) + {\cal F}^+}{\partial^+({\cal BH}^+)}, \qquad
\frac{\partial^+({\cal BH}^+) + {\cal G}^+}{\partial^+({\cal BH}^+)}, \qquad
\frac{\partial^+({\cal BH}^-) + {\cal F}^-}{\partial^+({\cal BH}^-)}, \qquad
\frac{\partial^+({\cal BH}^-) + {\cal G}^-}{\partial^+({\cal BH}^-)}
\end{equation}
are isomorphic respectively to
\[
(\pi_{dl},{\cal F}^+), \qquad
(\pi_{dr},{\cal G}^+), \qquad
(\pi_{dl},{\cal F}^-), \qquad
(\pi_{dr},{\cal G}^-).
\]
By Proposition \ref{nreg-irreducibility-prop}, after restricting to
$\mathfrak{sp}(4,\BB C)$, the quotients \eqref{2-reg-quotients} remain
irreducible.



We have the restriction operator
$\Res: {\cal W}' \to \mathring {\cal W}'_{res}$
which takes an $\HC$-valued polynomial function on $\HC^{\times}$ and restricts
it to an $\HC\text{-sym}$-valued polynomial function on
$\HC^{\times} \text{-sym}$. Clearly, $\Res$ produces intertwining maps of
$\mathfrak{sp}(4,\BB C)$-representations from
$(\rho'_2, \partial^+({\cal BH}^+))$,
$(\rho'_2, \partial^+({\cal BH}^+) + {\cal F}^+)$,
$(\rho'_2, \partial^+({\cal BH}^+) + {\cal G}^+)$,
$(\rho'_2, \partial^+({\cal BH}^-))$,
$(\rho'_2, \partial^+({\cal BH}^-) + {\cal F}^-)$,
$(\rho'_2, \partial^+({\cal BH}^-) + {\cal G}^-)$ into
$(\rho'_2,{\cal W}'_{res})$.
The goal of this subsection is to identify the images of these representations
and to decompose these images into irreducible components.

Introduce $\mathfrak{sp}(4,\BB C)$-invariant subspaces of
$\mathring {\cal W}'_{res}$:
\begin{align}
\mathring {\cal W}'^+_{res}
&= \HC\text{-sym} \otimes \BB C [z_{11}, z_{12}, z_{22}]  \label{W^+}  \\
&= \BB C\text{-span of } \bigl\{ N(Z)^k \cdot {\bf T}^m_l(Z),\:
N(Z)^k \cdot {\bf M}^m_l(Z),\: N(Z)^k \cdot {\bf B}^m_{l+1}(Z);\:
k,l \ge 0 \bigr\},  \\
&= \BB C\text{-span of } \bigl\{ N(Z)^k \cdot {\bf M}^m_l(Z),\:
N(Z)^k \cdot {\bf B}^m_{l+1}(Z),\: {\bf C}^m_{k,l+1}(Z);\: k,l \ge 0 \bigr\},  \\
\mathring {\cal W}'^-_{res} &= \bigl\{ F \in \mathring {\cal W}'_{res};\:
\tfrac{Z}{N(Z)} \cdot F(Z^{-1}) \cdot \tfrac{Z}{N(Z)}
\in \mathring {\cal W}'^+_{res} \bigr\}  \\
&= \BB C\text{-span of } \bigl\{ N(Z)^k \cdot {\bf M}^m_l(Z),\:
N(Z)^{k-1} \cdot {\bf B}^m_{l+1}(Z),\: {\bf C}^m_{k,l+1}(Z);\:
k \le -(l+1),\: l \ge 0 \bigr\}  \label{W^-}
\end{align}
(the last equality follows from Lemma \ref{W'-inversion-lem}
and \eqref{I-inversion}).
The inversion interchanges
\[
\partial^+({\cal BH}^+) + {\cal F}^+ \longleftrightarrow
\partial^+({\cal BH}^-) + {\cal F}^-, \qquad 
\partial^+({\cal BH}^+) + {\cal G}^+ \longleftrightarrow
\partial^+({\cal BH}^-) + {\cal G}^-,
\]
\[
\partial^+({\cal BH}^+) \longleftrightarrow \partial^+({\cal BH}^-), \qquad
\mathring {\cal W}'^+_{res} \longleftrightarrow \mathring {\cal W}'^-_{res}.
\]
It follows that
\begin{equation}  \label{containment+}
\Res \bigl( \partial^+({\cal BH}^+) \bigr), \quad
\Res \bigl( \partial^+({\cal BH}^+) + {\cal F}^+ \bigr), \quad
\Res \bigl( \partial^+({\cal BH}^+) + {\cal G}^+ \bigr) \quad
\subset \mathring {\cal W}'^+_{res}
\end{equation}
and
\[
\Res \bigl( \partial^+({\cal BH}^-) \bigr), \quad
\Res \bigl( \partial^+({\cal BH}^-) + {\cal F}^- \bigr), \quad
\Res \bigl( \partial^+({\cal BH}^-) + {\cal G}^- \bigr) \quad
\subset \mathring {\cal W}'^-_{res}.
\]

We compute the effect of
$\rho'_2 \bigl( \begin{smallmatrix} 0 & B \\ 0 & 0 \end{smallmatrix} \bigr)$,
$B \in \HC\text{-sym}$, on the $K$-types:
\begin{multline*}
\tfrac{\partial}{\partial t} \bigl( N(Z)^k \cdot {\bf T}^m_l(Z) \bigr)
= \tfrac{(2k+2l+1)(l+m+1)}{2l+1} N(Z)^k \cdot {\bf T}^m_{l-1}(Z)  \\
+ \tfrac{2k(l+1)(l-m+2)}{(2l+3)(l+2)} N(Z)^{k-1} \cdot {\bf T}^m_{l+1}(Z)
- \tfrac{2km}{(l+1)(l+2)} N(Z)^{k-1} \cdot {\bf M}^m_{l+1}(Z)  \\
- \tfrac{2k(l+m+1)}{(2l+1)(2l+3)(l+1)} N(Z)^{k-1} \cdot {\bf B}^m_{l+1}(Z),
\end{multline*}
\begin{multline*}
2 \tfrac{\partial}{\partial z_{11}} \bigl( N(Z)^k \cdot {\bf T}^m_l(Z) \bigr)
= -\tfrac{(2k+2l+1)(l+m)(l+m+1)}{2l+1} N(Z)^k \cdot {\bf T}^{m-1}_{l-1}(Z)  \\
+\tfrac{2k(l+1)(l-m+2)(l-m+3)}{(2l+3)(l+2)} N(Z)^{k-1} \cdot {\bf T}^{m-1}_{l+1}(Z)
- \tfrac{2k(l+m+1)(l-m+2)}{(l+1)(l+2)} N(Z)^{k-1} \cdot {\bf M}^{m-1}_{l+1}(Z)  \\
+\tfrac{2k(l+m)(l+m+1)}{(2l+1)(2l+3)(l+1)} N(Z)^{k-1} \cdot {\bf B}^{m-1}_{l+1}(Z),
\end{multline*}
\begin{multline*}
2 \tfrac{\partial}{\partial z_{22}} \bigl( N(Z)^k \cdot {\bf T}^m_l(Z) \bigr)
= \tfrac{2k+2l+1}{2l+1} N(Z)^k \cdot {\bf T}^{m+1}_{l-1}(Z)  \\
- \tfrac{2k(l+1)}{(2l+3)(l+2)} N(Z)^{k-1} \cdot {\bf T}^{m+1}_{l+1}(Z)
- \tfrac{2k}{(l+1)(l+2)} N(Z)^{k-1} \cdot {\bf M}^{m+1}_{l+1}(Z)  \\
- \tfrac{2k}{(2l+1)(2l+3)(l+1)} N(Z)^{k-1} \cdot {\bf B}^{m+1}_{l+1}(Z);
\end{multline*}
\begin{multline*}
\tfrac{\partial}{\partial t} \bigl( N(Z)^k \cdot {\bf M}^m_l(Z) \bigr)
= - \tfrac{m(2k+2l+1)}{l(2l+1)} N(Z)^k \cdot {\bf T}^m_{l-1}(Z)  \\
+ \tfrac{(l+1)(2k+2l+1)(l+m)}{l(2l+1)} N(Z)^k \cdot {\bf M}^m_{l-1}(Z)
+ \tfrac{2kl(l-m+1)}{(2l+1)(l+1)} N(Z)^{k-1} \cdot {\bf M}^m_{l+1}(Z)  \\
- \tfrac{2km}{(2l+1)(l+1)} N(Z)^{k-1} \cdot {\bf B}^m_{l+1}(Z),
\end{multline*}
\begin{multline*}
2 \tfrac{\partial}{\partial z_{11}} \bigl( N(Z)^k \cdot {\bf M}^m_l(Z) \bigr)
= - \tfrac{(2k+2l+1)(l+m)(l-m+1)}{l(2l+1)} N(Z)^k \cdot {\bf T}^{m-1}_{l-1}(Z)  \\
- \tfrac{(l+1)(2k+2l+1)(l+m)(l+m-1)}{l(2l+1)} N(Z)^k \cdot {\bf M}^{m-1}_{l-1}(Z)
+ \tfrac{2kl(l-m+1)(l-m+2)}{(2l+1)(l+1)} N(Z)^{k-1} \cdot {\bf M}^{m-1}_{l+1}(Z) \\
- \tfrac{2k(l+m)(l-m+1)}{(2l+1)(l+1)} N(Z)^{k-1} \cdot {\bf B}^{m-1}_{l+1}(Z),
\end{multline*}
\begin{multline*}
2 \tfrac{\partial}{\partial z_{22}} \bigl( N(Z)^k \cdot {\bf M}^m_l(Z) \bigr)
= - \tfrac{2k+2l+1}{l(2l+1)} N(Z)^k \cdot {\bf T}^{m+1}_{l-1}(Z)  \\
+ \tfrac{(l+1)(2k+2l+1)}{l(2l+1)} N(Z)^k \cdot {\bf M}^{m+1}_{l-1}(Z)
- \tfrac{2kl}{(2l+1)(l+1)} N(Z)^{k-1} \cdot {\bf M}^{m+1}_{l+1}(Z)  \\
- \tfrac{2k}{(2l+1)(l+1)} N(Z)^{k-1} \cdot {\bf B}^{m+1}_{l+1}(Z);
\end{multline*}
\begin{multline*}
\tfrac{\partial}{\partial t} \bigl( N(Z)^k \cdot {\bf B}^m_l(Z) \bigr)
= - \tfrac{(2k+2l+1)(l-m)}{l(2l-1)(2l+1)} N(Z)^k \cdot {\bf T}^m_{l-1}(Z)
- \tfrac{m(2k+2l+1)}{l(l-1)} N(Z)^k \cdot {\bf M}^m_{l-1}(Z)  \\
+ \tfrac{l(2k+2l+1)(l+m-1)}{(2l-1)(l-1)} N(Z)^k \cdot {\bf B}^m_{l-1}(Z)
+ \tfrac{2k(l-m)}{2l+1} N(Z)^{k-1} \cdot {\bf B}^m_{l+1}(Z),
\end{multline*}
\begin{multline*}
2 \tfrac{\partial}{\partial z_{11}} \bigl( N(Z)^k \cdot {\bf B}^m_l(Z) \bigr)
= - \tfrac{(2k+2l+1)(l-m)(l-m+1)}{l(2l-1)(2l+1)}
N(Z)^k \cdot {\bf T}^{m-1}_{l-1}(Z)  \\
- \tfrac{(2k+2l+1)(l-m)(l+m-1)}{l(l-1)} N(Z)^k \cdot {\bf M}^{m-1}_{l-1}(Z)
- \tfrac{l(2k+2l+1)(l+m-1)(l+m-2)}{(2l-1)(l-1)}
N(Z)^k \cdot {\bf B}^{m-1}_{l-1}(Z)  \\
+ \tfrac{2k(l-m)(l-m+1)}{2l+1} N(Z)^{k-1} \cdot {\bf B}^{m-1}_{l+1}(Z),
\end{multline*}
\begin{multline*}
2 \tfrac{\partial}{\partial z_{22}} \bigl( N(Z)^k \cdot {\bf B}^m_l(Z) \bigr)
= \tfrac{2k+2l+1}{l(2l-1)(2l+1)} N(Z)^k \cdot {\bf T}^{m+1}_{l-1}(Z)
- \tfrac{2k+2l+1}{l(l-1)} N(Z)^k \cdot {\bf M}^{m+1}_{l-1}(Z)  \\
+ \tfrac{l(2k+2l+1)}{(2l-1)(l-1)} N(Z)^k \cdot {\bf B}^{m+1}_{l-1}(Z)
- \tfrac{2k}{2l+1} N(Z)^{k-1} \cdot {\bf B}^{m+1}_{l+1}(Z);
\end{multline*}
\begin{align*}
\tfrac{\partial}{\partial t} {\bf C}^m_{k,l}(Z)
&= \tfrac{(2k+2l+1)(l+m)}{2l-1} {\bf C}^m_{k,l-1}(Z)
+ \tfrac{2k(l-m+1)}{2l+3} {\bf C}^m_{k-1,l+1}(Z),  \\
2 \tfrac{\partial}{\partial z_{11}} {\bf C}^m_{k,l}(Z)
&= -\tfrac{(2k+2l+1)(l+m)(l+m-1)}{2l-1} {\bf C}^{m-1}_{k,l-1}(Z)
+ \tfrac{2k(l-m+1)(l-m+2)}{2l+3} {\bf C}^{m-1}_{k-1,l+1}(Z),  \\
2 \tfrac{\partial}{\partial z_{22}} {\bf C}^m_{k,l}(Z)
&= \tfrac{2k+2l+1}{2l-1} {\bf C}^{m+1}_{k,l-1}(Z)
- \tfrac{2k}{2l+3} {\bf C}^{m+1}_{k-1,l+1}(Z).
\end{align*}

These formulas allow us to decompose $\mathring {\cal W}'^+_{res}$
and $\mathring {\cal W}'^-_{res}$ into irreducible components.

\begin{prop}  \label{W-indecomp-prop}
The $\mathfrak{sp}(4,\BB C)$-module $(\rho'_2, \mathring {\cal W}'^+_{res})$
is indecomposable and can be generated by a single generator of the form
${\bf M}^m_l(Z)$, $l \ge 1$, $-l \le m \le l$.
It also contains an irreducible submodule
\begin{equation}  \label{invar+-subspace}
\BB C\text{-span of } \bigl\{ {\bf C}^m_{k,l+1}(Z);\: k,l \ge 0 \bigr\}.
\end{equation}

Similarly, the $\mathfrak{sp}(4,\BB C)$-module
$(\rho'_2, \mathring {\cal W}'^-_{res})$ is indecomposable too and can be
generated by a single generator of the form
$N(Z)^{-(l+1)} \cdot {\bf M}^m_l(Z)$, $l \ge 1$, $-l \le m \le l$.
It also contains an irreducible submodule
\begin{equation}  \label{invar--subspace}
\BB C\text{-span of } \bigl\{ {\bf C}^m_{k,l+1}(Z);\:
k \le -(l+1),\: l \ge 0 \bigr\}.
\end{equation}
\end{prop}

\begin{proof}
From this description of the action of
$\rho'_2 \bigl( \begin{smallmatrix} 0 & B \\ 0 & 0 \end{smallmatrix} \bigr)$,
Lemma \ref{W'-inversion-lem} and equation \eqref{I-inversion}
we see that the subspaces \eqref{invar+-subspace} and \eqref{invar--subspace}
are indeed invariant and irreducible.

As was observed earlier, the $K$-types $N(Z)^k \cdot {\bf M}^m_l(Z)$,
$-l \le m \le l$, have multiplicity one.
It follows that a particular generator ${\bf M}^m_l(Z)$ with $l \ge 1$
generates all other $N(Z)^k \cdot {\bf M}^m_l(Z)$, $k, l \ge 0$,
as well as ${\bf C}^m_{k=0,l+1}(Z)$.
The latter $K$-types generate entire \eqref{invar+-subspace}.
Then the entire $\mathring {\cal W}'^+_{res}$ is generated.

The case of $(\rho'_2, \mathring {\cal W}'^-_{res})$ is similar.
\end{proof}

\begin{cor}  \label{Res-image-cor}
We have:
\begin{align*}
\Res \bigl( \partial^+({\cal BH}^+) + {\cal F}^+ \bigr) &=
\Res \bigl( \partial^+({\cal BH}^+) + {\cal G}^+ \bigr) =
\mathring {\cal W}'^+_{res},  \\
\Res \bigl( \partial^+({\cal BH}^-) + {\cal F}^- \bigr) &=
\Res \bigl( \partial^+({\cal BH}^-) + {\cal G}^- \bigr) =
\mathring {\cal W}'^-_{res}.
\end{align*}
\end{cor}

\begin{proof}
By Proposition 42 and Theorem 49 in \cite{ATMP},
${\cal F}^+$ can be generated by a single generator of the form
$\tilde F_{l=\frac12,m,n=-\frac12}(Z)$, $m=-\tfrac32,-\tfrac12,\tfrac12$,
and ${\cal G}^+$ can be generated by a single generator of the form
$\tilde G_{l=\frac12,m=-\frac12,n}(Z)$, $n=-\tfrac32,-\tfrac12,\tfrac12$.
When restricted to the symmetric quaternions, these functions produce
respectively scalar multiples of
\begin{equation}  \label{2-reg-gens}
F_{-1}= \begin{pmatrix} 2it & z_{22} \\ z_{22} & 0 \end{pmatrix}, \qquad
F_0= \begin{pmatrix} z_{11} & 0 \\  0 & -z_{22} \end{pmatrix}, \qquad
F_1= \begin{pmatrix} 0 & z_{11} \\ z_{11} & 2it \end{pmatrix} \qquad
\in \mathring {\cal W}'_{res}.
\end{equation}
Since
\[
R^{-1}_1(Z) = \tfrac12 z_{22}, \qquad R^0_1(Z) = t, \qquad R^1_1(Z) = -z_{11},
\]
up to scalar multiples, $F_{-1}$, $F_0$, $F_1$ are respectively the generators
${\bf M}^m_{l=1}(Z)$, $-1 \le m \le 1$.
Then the result follows from \eqref{containment+} and
Proposition \ref{W-indecomp-prop}.

The case of $\partial^+({\cal BH}^-) + {\cal F}^-$ and 
$\partial^+({\cal BH}^-) + {\cal G}^-$ follows by applying the inversion.
\end{proof}

\begin{thm}  \label{W'-decomp-thm}
The quotient
\begin{equation}  \label{W-quot-iso-2-reg+}
\bigl(\rho'_2, \mathring {\cal W}'^+_{res} /
\Res \bigl( \partial^+({\cal BH}^+) \bigr)\bigr)
\end{equation}
is irreducible and isomorphic to
$(\pi_{dl},{\cal F}^+) \simeq (\pi_{dr},{\cal G}^+)$
as $\mathfrak{sp}(4,\BB C)$-modules.

The $\mathfrak{sp}(4,\BB C)$-module $(\rho'_2, \mathring {\cal W}'^+_{res})$
contains an irreducible submodule \eqref{invar+-subspace}
as well as an infinite family of invariant one-dimensional subspaces spanned by
\[
N(Z)^k \cdot {\bf M}^{m=0}_{l=0}(Z),\qquad k \ge 0.
\]
The quotient
\[
\mathring {\cal W}'^+_{res} /
\BB C\text{-span of } \bigl\{ {\bf C}^m_{k,l+1}(Z);\: k,l \ge 0 \bigr\}
\]
contains another infinite family of invariant one-dimensional subspaces
spanned by
\[
N(Z)^{k'} \cdot {\bf B}^{m=0}_{l=1}(Z), \qquad k' \ge 0.
\]
The image of $\partial^+({\cal BH}^+)$ under the restriction to
the symmetric quaternions is
\begin{equation}  \label{Res(BH+)}
\Res \bigl( \partial^+({\cal BH}^+) \bigr)
= \BB C\text{-span of } \bigl\{ {\bf C}^m_{k,l+1}(Z),\:
N(Z)^k \cdot {\bf M}^0_0(Z),\: N(Z)^k \cdot {\bf B}^0_1(Z);\:
k,l \ge 0 \bigr\}.
\end{equation}

Similarly, the quotient
\begin{equation}  \label{W-quot-iso-2-reg-}
  \bigl(\rho'_2, \mathring {\cal W}'^-_{res} /
  \Res \bigl( \partial^+({\cal BH}^-) \bigr)\bigr)
\end{equation}
is irreducible and isomorphic to
$(\pi_{dl},{\cal F}^-) \simeq (\pi_{dr},{\cal G}^-)$
as $\mathfrak{sp}(4,\BB C)$-modules.

The $\mathfrak{sp}(4,\BB C)$-module
$(\rho'_2, \mathring {\cal W}'^-_{res})$
contains an irreducible submodule \eqref{invar--subspace}
as well as an infinite family of invariant one-dimensional subspaces spanned by
\[
N(Z)^k \cdot {\bf M}^{m=0}_{l=0}(Z),\qquad k \le -1.
\]
The quotient
\[
\mathring {\cal W}'^-_{res} / \BB C\text{-span of } \bigl\{ {\bf C}^m_{k,l+1}(Z)
;\:\ k \le -(l+1),\: l \ge 0 \bigr\}
\]
contains another infinite family of invariant one-dimensional subspaces
spanned by
\[
N(Z)^{k'-1} \cdot {\bf B}^{m=0}_{l=1}(Z), \qquad k' \le -1.
\]
The image of $\partial^+({\cal BH}^-)$ under the restriction to
the symmetric quaternions is
\begin{equation}  \label{Res(BH-)}
\Res \bigl( \partial^+({\cal BH}^-) \bigr)
= \BB C\text{-span of } \Bigl\{ {\bf C}^m_{k,l+1}(Z),\:
N(Z)^{k'} \cdot {\bf M}^0_0(Z),\: N(Z)^{k'-1} \cdot {\bf B}^0_1(Z);\:
\begin{smallmatrix} k \le -(l+1), \\ k' \le -1,\: l \ge 0 \end{smallmatrix}
\Bigr\}.
\end{equation}
\end{thm}

\begin{proof}
By Corollary 43 in \cite{ATMP}, the $K$-types ${\bf M}^m_l(Z)$ and
$N(Z)^{-l-1} \cdot {\bf M}^m_l(Z)$, $l \ge 1$, are not in
$\Res \bigl( \partial^+({\cal BH}^+) \bigr) \oplus
\Res \bigl( \partial^+({\cal BH}^-) \bigr)$.
Then Corollary \ref{Res-image-cor} together with
Proposition \ref{nreg-irreducibility-prop} imply
that we have isomorphisms of irreducible $\mathfrak{sp}(4,\BB C)$-modules
\begin{align}
\bigl(\rho'_2, \mathring {\cal W}'^+_{res} /
\Res \bigl( \partial^+({\cal BH}^+) \bigr)\bigr)
&\simeq (\pi_{dl},{\cal F}^+) \simeq (\pi_{dr},{\cal G}^+),  \label{iso+}  \\
\bigl(\rho'_2, \mathring {\cal W}'^+_{res} /
\Res \bigl( \partial^+({\cal BH}^-) \bigr)\bigr)
&\simeq (\pi_{dl},{\cal F}^-) \simeq (\pi_{dr},{\cal G}^-).
\end{align}

The statements about the irreducible invariant subspaces
\eqref{invar+-subspace} and \eqref{invar--subspace} were established
in Proposition \ref{W-indecomp-prop}.

The statements about the infinite families of invariant one-dimensional
subspaces follow from our description of the action of
$\rho'_2 \bigl( \begin{smallmatrix} 0 & B \\ 0 & 0 \end{smallmatrix} \bigr)$,
Lemma \ref{W'-inversion-lem} and equation \eqref{I-inversion}.

It remains to establish \eqref{Res(BH+)} and \eqref{Res(BH-)}.
By Proposition 42 and Corollary 43 in \cite{ATMP}, $\partial^+({\cal BH}^+)$
contains elements $\tilde H_{\frac{l}2,\frac{l}2,\frac{l}2}$ proportional to
$\bigl(\begin{smallmatrix} (z_{22})^l & 0 \\ 0 & 0 \end{smallmatrix}\bigr)$.
Then
$\Res \bigl(\begin{smallmatrix} (z_{22})^l & 0 \\ 0 & 0 \end{smallmatrix}\bigr)$
is proportional to ${\bf C}^{-l-1}_{0,l+1}(Z)$.
By Proposition \ref{W-indecomp-prop},
$\Res \bigl( \partial^+({\cal BH}^+) \bigr)$ contains the irreducible submodule
\eqref{invar+-subspace}.
On the other hand, by Lemma 48 and Theorem 49 in \cite{ATMP}, the quotients
\eqref{2-reg-quotients} cannot contain trivial one-dimensional
$SU(2)$-invariant subspaces.
Then Corollary \ref{Res-image-cor} implies that
$\Res \bigl( \partial^+({\cal BH}^+) \bigr)$ contains the submodule
\[
M^+ = \BB C\text{-span of } \bigl\{ {\bf C}^m_{k,l+1}(Z),\:
N(Z)^k \cdot {\bf M}^0_0(Z),\: N(Z)^k \cdot {\bf B}^0_1(Z);\:
k,l \ge 0 \bigr\}.
\]

Next, consider the quotient
$\bigl(\rho'_2, \mathring {\cal W}'^+_{res} / M^+ \bigr)$
which has $K$-types spanned by
\[
N(Z)^k \cdot {\bf M}^m_l(Z),\quad
N(Z)^k \cdot {\bf B}^m_{l+1}(Z),\qquad k \ge 0,\: l \ge 1.
\]
For a non-negative integer $d$, let
\begin{align*}
  \bigl(\mathring {\cal W}'^+_{res} / M^+ \bigr)(d)
  &= \bigl\{ F \in \mathring {\cal W}'^+_{res} / M^+;\:
  \text{$F$ is homogeneous of degree $d$} \bigr\},  \\
{\cal F}^+(d) &=
\{ f \in {\cal F}^+;\: \text{$f$ is homogeneous of degree $d$} \},  \\
{\cal G}^+(d) &=
\{ g \in {\cal G}^+;\: \text{$g$ is homogeneous of degree $d$} \}.
\end{align*}
A simple dimension count shows that
\[
\dim \bigl(\mathring {\cal W}'^+_{res} / M^+ \bigr)(d)
= \dim {\cal F}^+(d) = \dim {\cal G}^+(d) = d(d+2).
\]
Then \eqref{iso+} implies
$M^+ = \Res \bigl( \partial^+({\cal BH}^+) \bigr)$
and establishes \eqref{Res(BH+)}.

The case of \eqref{Res(BH-)} is similar.
\end{proof}

\begin{cor}  \label{2reg-generators-cor}
Each of the images of ${\bf M}^m_l(Z)$ and $N(Z)^{-l-1} \cdot {\bf M}^m_l(Z)$,
$l \ge 1$, $-l \le m \le l$, in the respective quotients
\eqref{W-quot-iso-2-reg+} and \eqref{W-quot-iso-2-reg-}
generates the irreducible component corresponding to the doubly regular
functions.
\end{cor}

\subsection{The Dual of $(\rho'_2, \mathring {\cal W}'_{res})$}

The dual space of $(\rho'_2, {\cal W}'_{res})$ consists of
polynomial-like functions on $\HC^{\times} \text{-sym}$ involving
half-integer powers of $N(Z)$:
\[
{\cal W}^*_{res} =
\HC \otimes N(Z)^{\frac12} \cdot \BB C [z_{11}, z_{12}, z_{22}, N(Z)^{-1}].
\]
The Lie algebra $\mathfrak{sp}(2,\BB C)$ acts on this space by
differentiating the following group action:
\[
\rho^*_2(h): \: G(Z) \: \mapsto \: \bigl( \rho^*_2(h)G \bigr)(Z) =
\frac{(cZ+d)^{-1}}{N(cZ+d)} \cdot G \bigl( (aZ+b)(cZ+d)^{-1} \bigr)
\cdot (a'-Zc')^{-1},
\]
where $G \in {\cal W}^*_{res}$,
$h = \bigl(\begin{smallmatrix} a' & b' \\ c' & d' \end{smallmatrix}\bigr)
\in Sp(4,\BB C)$ and 
$h^{-1} = \bigl(\begin{smallmatrix} a & b \\ c & d \end{smallmatrix}\bigr)$.

\begin{prop}  \label{W_res-pairing-prop}
There is a non-degenerate bilinear pairing between the
$\mathfrak{sp}(2,\BB C)$-modules $(\rho'_2, {\cal W}'_{res})$
and $(\rho^*_2, {\cal W}^*_{res})$
\[
\langle F, G \rangle_{{\cal W}_{res}} = 
\frac i{8\pi^2} \int_{\widetilde{\Gamma}} \tr \bigl(F(Z) \cdot G(Z)\bigr) \,dZ^3,
\qquad F \in {\cal W}'_{res}, \quad G \in {\cal W}^*_{res},
\]
where $R>0$. This bilinear pairing is $\mathfrak{sp}(4,\BB C)$-invariant and
independent of the choice of $R>0$.
\end{prop}

\begin{rem}
The dual of the $\mathfrak{gl}(2, \HC)$-module $(\rho'_2, {\cal W}')$ is
$(\rho_2, {\cal W})$ (Proposition 80 in \cite{FL1}).
However, the $\mathfrak{sp}(4, \BB C)$-module $(\rho^*_2, {\cal W}^*_{res})$
is quite different from $(\rho_2, {\cal W})$ restricted to
$\HC^{\times} \text{-sym}$ and $\mathfrak{sp}(4, \BB C)$.
\end{rem}

We have the following analogue of Lemma 25 in \cite{ATMP}:

\begin{lem}
The Lie algebra action $\rho^*_2$ of $\mathfrak{sp}(4,\BB C)$ on
${\cal W}^*_{res}$ is given by
\begin{align*}
\rho^*_2 \bigl(\begin{smallmatrix} A & 0 \\ 0 & -A^T \end{smallmatrix}\bigr) &:
G(Z) \mapsto - \tr (AZ \partial + ZA^T \partial + A) G -GA-A^TG,  \\
\rho^*_2 \bigl( \begin{smallmatrix} 0 & B \\ 0 & 0 \end{smallmatrix} \bigr) &:
G(Z) \mapsto - \tr (B \partial) G,  \\
\rho^*_2 \bigl( \begin{smallmatrix} 0 & 0 \\ C & 0 \end{smallmatrix} \bigr) &:
G(Z) \mapsto \tr (ZCZ \partial + ZC) G + CZG + GZC.
\end{align*}
\end{lem}

Like $(\rho'_2, {\cal W}'_{res})$, $(\rho^*_2, {\cal W}^*_{res})$ decomposes into
a direct sum of two invariant subspaces -- functions with values in the
symmetric matrices in $\HC$, $\HC\text{-sym}$, and functions with
antisymmetric values. We are interested in the symmetric part only and call it
$\mathring {\cal W}^*_{res}$:
\[
\mathring {\cal W}^*_{res} = \HC\text{-sym} \otimes
N(Z)^{\frac12} \cdot \BB C [z_{11}, z_{12}, z_{22}, N(Z)^{-1}].
\]
It is easy to see that the pairing from Proposition \ref{W_res-pairing-prop}
restricts to a non-degenerate bilinear pairing between the
$\mathfrak{sp}(2,\BB C)$-modules $(\rho'_2, \mathring {\cal W}'_{res})$
and $(\rho^*_2, \mathring {\cal W}^*_{res})$.

Next, we describe the $K$-types of $(\rho^*_2, \mathring {\cal W}^*_{res})$.
Consider
\begin{align*}
\leftidx{^*}{{\bf T}}^m_l(Z) &=
\begin{pmatrix} (l+m)(l+m+1) R^{m-1}_l(Z) & i(l+m+1) R^m_l(Z) \\
i(l+m+1) R^m_l(Z) & - R^{m+1}_l(Z) \end{pmatrix},
\quad \text{\footnotesize $-l-1 \le m \le l+1$} ,\\
\leftidx{^*}{{\bf M}}^m_l(Z) &=
\begin{pmatrix} (l+m)(l-m+1) R^{m-1}_l(Z) & -im R^m_l(Z) \\
-im R^m_l(Z) & R^{m+1}_l \end{pmatrix},
\quad \text{\footnotesize $-l \le m \le l$},  \\
\leftidx{^*}{{\bf B}}^m_l(Z) &=
\begin{pmatrix} (l-m)(l-m+1) R^{m-1}_l(Z) & -i(l-m) R^m_l(Z) \\
-i(l-m) R^m_l(Z) & - R^{m+1}_l(Z) \end{pmatrix},
\quad \text{\footnotesize $-l+1 \le m \le l-1$}.
\end{align*}
For a fixed $l$, these functions span $U(2)$-invariant spaces of dimensions
$2l+2$, $2l$ and $2l-1$ respectively.
Let $A_1=\bigl( \begin{smallmatrix} 0 & 1 \\ 0 & 0 \end{smallmatrix} \bigr)$
and $A_2=\bigl( \begin{smallmatrix} 0 & 0 \\ 1 & 0 \end{smallmatrix} \bigr)$,
then
\begin{align*}
\rho^*_2 \bigl(\begin{smallmatrix} A_1 & 0 \\ 0 & -A_1^T \end{smallmatrix}\bigr)
\leftidx{^*}{{\bf T}}^m_l(Z) &=i(l+m+1)(l-m+2) \leftidx{^*}{{\bf T}}^{m-1}_l(Z),\\
\rho^*_2 \bigl(\begin{smallmatrix} A_2 & 0 \\ 0 & -A_2^T \end{smallmatrix}\bigr)
\leftidx{^*}{{\bf T}}^m_l(Z) &= -i \leftidx{^*}{{\bf T}}^{m+1}_l(Z);  \\
\rho^*_2 \bigl(\begin{smallmatrix} A_1 & 0 \\ 0 & -A_1^T \end{smallmatrix}\bigr)
\leftidx{^*}{{\bf M}}^m_l(Z) &= i(l+m)(l-m+1) \leftidx{^*}{{\bf M}}^{m-1}_l(Z), \\
\rho^*_2 \bigl(\begin{smallmatrix} A_2 & 0 \\ 0 & -A_2^T \end{smallmatrix}\bigr)
\leftidx{^*}{{\bf M}}^m_l(Z) &= -i \leftidx{^*}{{\bf M}}^{m+1}_l(Z);  \\
\rho^*_2 \bigl(\begin{smallmatrix} A_1 & 0 \\ 0 & -A_1^T \end{smallmatrix}\bigr)
\leftidx{^*}{{\bf B}}^m_l(Z) &= i(l+m-1)(l-m) \leftidx{^*}{{\bf B}}^{m-1}_l(Z), \\
\rho^*_2 \bigl(\begin{smallmatrix} A_2 & 0 \\ 0 & -A_2^T \end{smallmatrix}\bigr)
\leftidx{^*}{{\bf B}}^m_l(Z) &= -i \leftidx{^*}{{\bf B}}^{m+1}_l(Z).
\end{align*}

We describe the effect of the inversion on the basis functions:

\begin{lem}  \label{W*-inversion-lem}
Let $\bigl(\begin{smallmatrix} 0 & 1 \\ 1 & 0 \end{smallmatrix}\bigr)
\in GL(2,\HC)$, then
\begin{align*}
\rho^*_2 \bigl(\begin{smallmatrix} 0 & 1 \\ 1 & 0 \end{smallmatrix}\bigr)
\bigl( N(Z)^{k+\frac12} \cdot \leftidx{^*}{{\bf T}}^m_l(Z) \bigr) =
&\tfrac{(-1)^{l+1}}{2l+3} \tfrac{(l+m+1)!}{(l-m+1)!} N(Z)^{-(k+l+3)+\frac12}
\cdot \leftidx{^*}{{\bf T}}^{-m}_l(Z) \\
&+ (-1)^{l+1} \tfrac{2(l+1)}{2l+3} \tfrac{(l+m+1)!}{(l-m+1)!}
N(Z)^{-(k+l+4)+\frac12} \cdot \leftidx{^*}{{\bf B}}^{-m}_{l+2}(Z),  \\
\rho^*_2 \bigl(\begin{smallmatrix} 0 & 1 \\ 1 & 0 \end{smallmatrix}\bigr)
\bigl( N(Z)^{k+\frac12} \cdot \leftidx{^*}{{\bf M}}^m_l(Z) \bigr) =
&(-1)^l \tfrac{(l+m)!}{(l-m)!} N(Z)^{-(k+l+3)+\frac12} \cdot
\leftidx{^*}{{\bf M}}^{-m}_l(Z),  \\
\rho^*_2 \bigl(\begin{smallmatrix} 0 & 1 \\ 1 & 0 \end{smallmatrix}\bigr)
\bigl( N(Z)^{k+\frac12} \cdot \leftidx{^*}{{\bf B}}^m_{l+1}(Z) \bigr) =
&(-1)^l \tfrac{2(l+1)}{2l+1} \tfrac{(l+m)!}{(l-m)!} N(Z)^{-(k+l+3)+\frac12}
\cdot \leftidx{^*}{{\bf T}}^{-m}_{l-1}(Z) \\
&- \tfrac{(-1)^l}{2l+1} \tfrac{(l+m)!}{(l-m)!} N(Z)^{-(k+l+4)+\frac12}
\cdot \leftidx{^*}{{\bf B}}^{-m}_{l+1}(Z).
\end{align*}
\end{lem}

\begin{thm}
The functions
\begin{equation*}  
N(Z)^{k+\frac12} \cdot \leftidx{^*}{{\bf T}}^m_l(Z), \quad
N(Z)^{k+\frac12} \cdot \leftidx{^*}{{\bf M}}^m_l(Z), \quad
N(Z)^{k+\frac12} \cdot \leftidx{^*}{{\bf B}}^m_l(Z)
\end{equation*}
form a $K$-type basis of $(\rho^*_2,\mathring {\cal W}^*_{res})$.
More precisely, for $k$, $l$ fixed, as representations of $SU(2)$,
\begin{align*}
\BB C\text{-span of }
\bigl\{ N(Z)^{k+\frac12} \cdot \leftidx{^*}{{\bf T}}^m_l(Z) ;\:
-l-1 \le m \le l+1  \bigr\} &\simeq V_{l+1},  \\
\BB C\text{-span of }
\bigl\{ N(Z)^{k+\frac12} \cdot \leftidx{^*}{{\bf M}}^m_l(Z) ;\:
  -l \le m \le l \bigr\} &\simeq V_l,  \\
\BB C\text{-span of }
\bigl\{ N(Z)^{k+\frac12} \cdot \leftidx{^*}{{\bf B}}^m_l(Z) ;\:
  -l+1 \le m \le l-1 \bigr\} &\simeq V_{l-1}.
\end{align*}
\end{thm}

Note that the $K$-types
$N(Z)^{k+\frac12} \cdot \leftidx{^*}{{\bf B}}^{m=0}_{l=1}(Z)$ and
$N(Z)^{k+\frac12} \cdot \leftidx{^*}{{\bf M}}^m_l(Z)$, $-l \le m \le l$,
have multiplicity one. On the other hand,
$N(Z)^{k+\frac12} \cdot \leftidx{^*}{{\bf T}}^m_l(Z)$ and
$N(Z)^{k-\frac12} \cdot \leftidx{^*}{{\bf B}}^m_{l+2}(Z)$, $-l-1 \le m \le l+1$,
span isomorphic representations of $U(2)$.
We introduce an additional family of ``combination'' $K$-types of
$(\rho^*_2, \mathring {\cal W}^*_{res})$:
\[
\leftidx{^*}{{\bf C}}^m_{k,l}(Z) =
(2k+1)l N(Z)^{k-\frac12} \cdot \leftidx{^*}{{\bf B}}^m_{l+1}(Z)
+ 2(k+l+1)(l+1) N(Z)^{k+\frac12} \cdot \leftidx{^*}{{\bf T}}^m_{l-1}(Z),
\]
$k \in \BB Z$, $l \ge 1$, $-l \le m \le l$.
From Lemma \ref{W*-inversion-lem} we find that the effect of the inversion
on these $K$-types is
\begin{equation}  \label{I*-inversion}
\rho^*_2 \bigl(\begin{smallmatrix} 0 & 1 \\ 1 & 0 \end{smallmatrix}\bigr)
\bigl( \leftidx{^*}{{\bf C}}^m_{k,l}(Z) \bigr) =
(-1)^{l+1} \tfrac{(l+m)!}{(l-m)!} \leftidx{^*}{{\bf C}}^{-m}_{-k-l-2,l}(Z).
\end{equation}

\begin{thm}  \label{W*-decomp-thm}
The $\mathfrak{sp}(4,\BB C)$-module $(\rho^*_2, {\cal W}^*_{res})$
contains invariant subspaces
\begin{align}
\BB C\text{-span of } &\bigl\{
N(Z)^{k+\frac12} \cdot \leftidx{^*}{{\bf M}}^m_l(Z),\:
\leftidx{^*}{{\bf C}}^m_{k,l+1}(Z);\: k \in \BB Z,\: l \ge 0 \bigr\},
\label{MI*-subspace}  \\
\BB C\text{-span of } &\left\{ \begin{matrix}
N(Z)^{k+\frac12} \cdot \leftidx{^*}{{\bf M}}^0_0(Z) \text{ with } k \in \BB Z, \\
N(Z)^{k+\frac12} \cdot \leftidx{^*}{{\bf M}}^m_l(Z) \text{ with }
k \in \BB Z,\: l \ge 1,\: -(l+1) \le k \le -2, \\
\leftidx{^*}{{\bf C}}^m_{k,l}(Z) \text{ with }
k \in \BB Z,\:l \ge 1,\: -(l+1) \le k \le -1 \end{matrix} \right\}.
\label{MI*-sub-subspace}
\end{align}
The quotient
\[
\tfrac{\BB C\text{-span of } \bigl\{
N(Z)^{k+\frac12} \cdot \leftidx{^*}{{\bf M}}^m_l(Z),\:
\leftidx{^*}{{\bf C}}^m_{k,l+1}(Z);\: k \in \BB Z,\: l \ge 0 \bigr\}}
{\BB C\text{-span of } \left\{ \begin{smallmatrix}
N(Z)^{k+\frac12} \cdot \leftidx{^*}{{\bf M}}^0_0(Z) \text{ with } k \in \BB Z, \\
N(Z)^{k+\frac12} \cdot \leftidx{^*}{{\bf M}}^m_l(Z) \text{ with }
k \in \BB Z,\: l \ge 1,\: -(l+1) \le k \le -2, \\
\leftidx{^*}{{\bf C}}^m_{k,l}(Z) \text{ with }
k \in \BB Z,\:l \ge 1,\: -(l+1) \le k \le -1 \end{smallmatrix} \right\}}
\]
decomposes into a direct sum of two irreducible modules with $K$-types
\begin{equation}  \label{W*-subquot1}
\BB C\text{-span of } \bigl\{
N(Z)^{k-\frac12} \cdot \leftidx{^*}{{\bf M}}^m_l(Z),\:
\leftidx{^*}{{\bf C}}^m_{k,l}(Z);\: k \ge 0,\: l \ge 1 \bigr\}
\end{equation}
isomorphic to $(\pi_{dl},{\cal F}^+) \simeq (\pi_{dr},{\cal G}^+)$
as $\mathfrak{sp}(4,\BB C)$-modules and
\begin{equation}    \label{W*-subquot2}
\BB C\text{-span of } \bigl\{
N(Z)^{k+\frac12} \cdot \leftidx{^*}{{\bf M}}^m_l(Z),\:
\leftidx{^*}{{\bf C}}^m_{k,l}(Z);\: l \ge 1,\: k \le -(l+2) \bigr\}
\end{equation}
isomorphic to $(\pi_{dl},{\cal F}^-) \simeq (\pi_{dr},{\cal G}^-)$
as $\mathfrak{sp}(4,\BB C)$-modules.
\end{thm}

\begin{proof}
We compute the effect of
$\rho^*_2 \bigl( \begin{smallmatrix} 0 & B \\ 0 & 0 \end{smallmatrix} \bigr)$,
$B \in \HC\text{-sym}$, on these $K$-types:
\begin{multline*}
\tfrac{\partial}{\partial t} \bigl( N(Z)^{k+\frac12} \cdot
\leftidx{^*}{{\bf M}}^m_l(Z) \bigr)
= \tfrac{m}{l(l+1)(2l+1)} \leftidx{^*}{{\bf C}}^m_{k,l}(Z)  \\
+ \tfrac{2(l+1)(k+l+1)(l+m)}{l(2l+1)} N(Z)^{k+\frac12} \cdot
\leftidx{^*}{{\bf M}}^m_{l-1}(Z)
+ \tfrac{l(2k+1)(l-m+1)}{(l+1)(2l+1)} N(Z)^{k-\frac12} \cdot
\leftidx{^*}{{\bf M}}^m_{l+1}(Z),
\end{multline*}
\begin{multline*}
2\tfrac{\partial}{\partial z_{11}} \bigl( N(Z)^{k+\frac12} \cdot
\leftidx{^*}{{\bf M}}^m_l(Z) \bigr)
= \tfrac{(l+m)(l-m+1)}{l(l+1)(2l+1)} \leftidx{^*}{{\bf C}}^{m-1}_{k,l}(Z)  \\
- \tfrac{2(l+1)(k+l+1)(l+m)(l+m-1)}{l(2l+1)} N(Z)^{k+\frac12} \cdot
\leftidx{^*}{{\bf M}}^{m-1}_{l-1}(Z)  \\
+ \tfrac{l(2k+1)(l-m+1)(l-m+2)}{(l+1)(2l+1)} N(Z)^{k-\frac12} \cdot
\leftidx{^*}{{\bf M}}^{m-1}_{l+1}(Z),
\end{multline*}
\begin{multline*}
2\tfrac{\partial}{\partial z_{22}} \bigl( N(Z)^{k+\frac12} \cdot
\leftidx{^*}{{\bf M}}^m_l(Z) \bigr)
= \tfrac{1}{l(l+1)(2l+1)} \leftidx{^*}{{\bf C}}^{m+1}_{k,l}(Z)  \\
+ \tfrac{2(l+1)(k+l+1)}{l(2l+1)} N(Z)^{k+\frac12} \cdot
\leftidx{^*}{{\bf M}}^{m+1}_{l-1}(Z)
- \tfrac{l(2k+1)}{(l+1)(2l+1)} N(Z)^{k-\frac12} \cdot
\leftidx{^*}{{\bf M}}^{m+1}_{l+1}(Z);
\end{multline*}
\begin{multline*}
\tfrac{\partial}{\partial t} \bigl( \leftidx{^*}{{\bf C}}^m_{k,l}(Z) \bigr)
= \tfrac{2(k+l+1)(2k+1)(2l+1)m}{l(l+1)} N(Z)^{k-\frac12} \cdot
\leftidx{^*}{{\bf M}}^m_l(Z)  \\
+ \tfrac{2(l+1)(k+l+1)(l+m)}{l(2l-1)} \leftidx{^*}{{\bf C}}^m_{k,l-1}(Z)
+ \tfrac{l(2k+1)(l-m+1)}{(l+1)(2l+3)} \leftidx{^*}{{\bf C}}^m_{k-1,l+1}(Z),
\end{multline*}
\begin{multline*}
2\tfrac{\partial}{\partial z_{11}} \bigl( \leftidx{^*}{{\bf C}}^m_{k,l}(Z) \bigr)
= \tfrac{2(k+l+1)(2k+1)(2l+1)(l+m)(l-m+1)}{l(l+1)} N(Z)^{k-\frac12}
\cdot \leftidx{^*}{{\bf M}}^{m-1}_l(Z)  \\
- \tfrac{2(l+1)(k+l+1)(l+m)(l+m-1)}{l(2l-1)}
\leftidx{^*}{{\bf C}}^{m-1}_{k,l-1}(Z)
+ \tfrac{l(2k+1)(l-m+1)(l-m+2)}{(l+1)(2l+3)}
\leftidx{^*}{{\bf C}}^{m-1}_{k-1,l+1}(Z),
\end{multline*}
\begin{multline*}
2\tfrac{\partial}{\partial z_{22}} \bigl( \leftidx{^*}{{\bf C}}^m_{k,l}(Z) \bigr)
= \tfrac{2(k+l+1)(2k+1)(2l+1)}{l(l+1)} N(Z)^{k-\frac12}
\cdot \leftidx{^*}{{\bf M}}^{m+1}_l(Z)  \\
+ \tfrac{2(l+1)(k+l+1)}{l(2l-1)} \leftidx{^*}{{\bf C}}^{m+1}_{k,l-1}(Z)
- \tfrac{l(2k+1)}{(l+1)(2l+3)} \leftidx{^*}{{\bf C}}^{m+1}_{k-1,l+1}(Z).
\end{multline*}

From this description of the action of
$\rho^*_2 \bigl( \begin{smallmatrix} 0 & B \\ 0 & 0 \end{smallmatrix} \bigr)$,
Lemma \ref{W*-inversion-lem} and equation \eqref{I*-inversion}
we see that the subspaces \eqref{MI*-subspace} and \eqref{MI*-sub-subspace}
are indeed $\mathfrak{sp}(4,\BB C)$-invariant.
From the description of the $K$-types, it is clear that the quotients
\eqref{W*-subquot1} and \eqref{W*-subquot2} are dual to the quotients
\eqref{W-quot-iso-2-reg-} and \eqref{W-quot-iso-2-reg+} respectively.
Then the result follows from Theorem \ref{W'-decomp-thm}.
\end{proof}

\subsection{Structure of $(\rho'_2, \mathring {\cal W}'_{\frac12})$}

We consider a space
\[
{\cal W}'_{\frac12} =
\HC \otimes N(Z)^{\frac12} \cdot \BB C [z_{11}, z_{12}, z_{22}, N(Z)^{-1}]
\]
with $\rho'_2$ action of $\mathfrak{sp}(4,\BB C)$ introduced in
Subsection \ref{W'-subsection}.
This $\mathfrak{sp}(4,\BB C)$-module appears in the trilinear form that
will be constructed in Subsection \ref{3form-main-subsect}.
Lemma \ref{rho'_2-action-lem} readily extends to
$(\rho'_2,{\cal W}'_{\frac12})$ and implies that $(\rho'_2,{\cal W}'_{\frac12})$
decomposes into a direct sum of two invariant subspaces -- functions with
values in the symmetric matrices in $\HC$, $\HC\text{-sym}$, and
functions with antisymmetric values. As usual, we are interested in the
symmetric part only and call it $\mathring {\cal W}'_{\frac12}$:
\[
\mathring {\cal W}'_{\frac12} = \HC\text{-sym} \otimes
N(Z)^{\frac12} \cdot  \BB C [z_{11}, z_{12}, z_{22}, N(Z)^{-1}].
\]

The $K$-types of $(\rho'_2, \mathring {\cal W}'_{\frac12})$ are very similar
to those of $(\rho'_2, \mathring {\cal W}'_{res})$. Let
\[
{\bf C}^m_{k+\frac12,l}(Z) = (2k+1) N(Z)^{k-\frac12} \cdot {\bf B}^m_{l+1}(Z)
- 2(k+l+1) N(Z)^{k+\frac12} \cdot {\bf T}^m_{l-1}(Z),
\]
$k\in \BB Z$, $l \ge 0$, $-l \le m \le l$.

\begin{thm}
The functions
\begin{equation*}
N(Z)^{k+\frac12} \cdot {\bf T}^m_l(Z), \quad N(Z)^{k+\frac12} \cdot {\bf M}^m_l(Z),
\quad {\bf C}^m_{k+\frac12,l}(Z), \qquad k \in \BB Z,
\end{equation*}
form a $K$-type basis of $(\rho'_2,\mathring {\cal W}'_{\frac12})$.
More precisely, for $k$, $l$ fixed, as representations of $SU(2)$,
\begin{align*}
  \BB C\text{-span of } \bigl\{ N(Z)^{k+\frac12} \cdot {\bf T}^m_l(Z) ;\:
  -l-1 \le m \le l+1  \bigr\} &\simeq V_{l+1},  \\
  \BB C\text{-span of } \bigl\{ N(Z)^{k+\frac12} \cdot {\bf M}^m_l(Z) ;\:
  -l \le m \le l \bigr\} &\simeq V_l,  \\
  \BB C\text{-span of } \bigl\{ {\bf C}^m_{k+\frac12,l}(Z) ;\:
  -l \le m \le l \bigr\} &\simeq V_l.
\end{align*}
\end{thm}

Note that the $K$-types
$N(Z)^{k+\frac12} \cdot {\bf M}^m_l(Z)$, $-l \le m \le l$,
and ${\bf C}^{m=0}_{k+\frac12,l=0}(Z)$ have multiplicity one.
On the other hand, $N(Z)^{k+\frac12} \cdot {\bf T}^m_l(Z)$ and
${\bf C}^m_{k+\frac12,l+1}(Z)$, $-l-1 \le m \le l+1$,
span isomorphic representations of $U(2)$.

We modify Lemma \ref{W'-inversion-lem} and equation \eqref{I-inversion}
to describe the effect of the inversion on this basis functions:

\begin{lem}  \label{W-half-inversion-lem}
Let $\bigl(\begin{smallmatrix} 0 & 1 \\ 1 & 0 \end{smallmatrix}\bigr)
\in GL(2,\HC)$, then
\begin{align*}
\rho'_2 \bigl(\begin{smallmatrix} 0 & 1 \\ 1 & 0 \end{smallmatrix}\bigr)
\bigl( N(Z)^{k+\frac12} \cdot {\bf T}^m_l(Z) \bigr) =
&\tfrac{(-1)^{l+1} (2k+1)}{2k+2l+3} \tfrac{(l+m+1)!}{(l-m+1)!}
N(Z)^{-(k+l+2)+\frac12} \cdot {\bf T}^{-m}_l(Z)  \\
&+ \tfrac{(-1)^l}{2k+2l+3} \tfrac{2(l+1)}{2l+3} \tfrac{(l+m+1)!}{(l-m+1)!}
\cdot {\bf C}^{-m}_{-(k+l+2)+\frac12,l+1}(Z),  \\
\rho'_2 \bigl(\begin{smallmatrix} 0 & 1 \\ 1 & 0 \end{smallmatrix}\bigr)
\bigl( N(Z)^{k+\frac12} \cdot {\bf M}^m_l(Z) \bigr) =
&(-1)^l \tfrac{(l+m)!}{(l-m)!} N(Z)^{-(k+l+2)+\frac12} \cdot {\bf M}^{-m}_l(Z),  \\
\rho'_2 \bigl(\begin{smallmatrix} 0 & 1 \\ 1 & 0 \end{smallmatrix}\bigr)
\bigl( {\bf C}^m_{k+\frac12,l}(Z) \bigr) =
&(-1)^l \tfrac{(l+m)!}{(l-m)!} {\bf C}^{-m}_{-(k+l+1)+\frac12,l}(Z).
\end{align*}
\end{lem}

\begin{thm}  \label{W-half-decomp-thm}
The $\mathfrak{sp}(4,\BB C)$-module $(\rho'_2, \mathring {\cal W}'_{\frac12})$
contains an invariant subspace
\begin{equation}  \label{W-half-subspace}
\BB C\text{-span of } \left\{ \begin{matrix}
N(Z)^{k+\frac12} \cdot {\bf M}^0_0(Z),\: {\bf C}^m_{k+\frac12,l}(Z)
\text{ with } k \in \BB Z,\:l \ge 0,  \\
\begin{matrix} N(Z)^{k+\frac12} \cdot {\bf T}^m_l(Z), \\
  N(Z)^{k+\frac12} \cdot {\bf M}^m_l(Z) \end{matrix}
\text{ with } k \in \BB Z,\: l \ge 0,\: -(l+1) \le k \le -1  \end{matrix}
\right\}.
\end{equation}
The quotient
\[
\tfrac{\mathring {\cal W}'_{\frac12}}
{\BB C\text{-span of } \left\{ \begin{smallmatrix}
N(Z)^{k+\frac12} \cdot {\bf M}^0_0(Z),\: {\bf C}^m_{k+\frac12,l}(Z)
\text{ with } k \in \BB Z,\:l \ge 0,  \\
\begin{smallmatrix} N(Z)^{k+\frac12} \cdot {\bf T}^m_l(Z), \\
  N(Z)^{k+\frac12} \cdot {\bf M}^m_l(Z) \end{smallmatrix}
\text{ with } k \in \BB Z,\: l \ge 0,\: -(l+1) \le k \le -1  \end{smallmatrix}
\right\}}
\]
decomposes into a direct sum of two irreducible modules with $K$-types
\begin{equation}  \label{W-half-subquot1}
\BB C\text{-span of } \bigl\{
N(Z)^{k+\frac12} \cdot {\bf T}^m_l(Z),\: N(Z)^{k+\frac12} \cdot {\bf M}^m_{l+1}(Z);\:
k \ge 0,\: l \ge 0 \bigr\}
\end{equation}
isomorphic to $(\pi_{dl},{\cal F}^+) \simeq (\pi_{dr},{\cal G}^+)$
as $\mathfrak{sp}(4,\BB C)$-modules and
\begin{equation}    \label{W-half-subquot2}
\BB C\text{-span of } \bigl\{
N(Z)^{k+\frac12} \cdot {\bf T}^m_l(Z),\: N(Z)^{k-\frac12} \cdot {\bf M}^m_{l+1}(Z);\:
l \ge 0,\: k \le -(l+2) \bigr\}
\end{equation}
isomorphic to $(\pi_{dl},{\cal F}^-) \simeq (\pi_{dr},{\cal G}^-)$
as $\mathfrak{sp}(4,\BB C)$-modules.
\end{thm}

\begin{proof}
We compute the effect of
$\rho'_2 \bigl( \begin{smallmatrix} 0 & B \\ 0 & 0 \end{smallmatrix} \bigr)$,
$B \in \HC\text{-sym}$, on the $K$-types:
\begin{multline*}
\tfrac{\partial}{\partial t} \bigl( N(Z)^{k+\frac12} \cdot {\bf T}^m_l(Z) \bigr)
= \tfrac{2(l+2)(k+l+1)(l+m+1)}{(l+1)(2l+3)}
N(Z)^{k+\frac12} \cdot {\bf T}^m_{l-1}(Z)  \\
+ \tfrac{(2k+1)(l+1)(l-m+2)}{(2l+3)(l+2)} N(Z)^{k-\frac12} \cdot {\bf T}^m_{l+1}(Z)
- \tfrac{m(2k+1)}{(l+1)(l+2)} N(Z)^{k-\frac12} \cdot {\bf M}^m_{l+1}(Z)  \\
- \tfrac{l+m+1}{(2l+1)(2l+3)(l+1)} {\bf C}^m_{k+\frac12,l}(Z),
\end{multline*}
\begin{multline*}
2 \tfrac{\partial}{\partial z_{11}}
\bigl( N(Z)^{k+\frac12} \cdot {\bf T}^m_l(Z) \bigr)
= -\tfrac{2(l+2)(k+l+1)(l+m)(l+m+1)}{(l+1)(2l+3)}
N(Z)^{k+\frac12} \cdot {\bf T}^{m-1}_{l-1}(Z)  \\
+\tfrac{(2k+1)(l+1)(l-m+2)(l-m+3)}{(2l+3)(l+2)}
N(Z)^{k-\frac12} \cdot {\bf T}^{m-1}_{l+1}(Z)
- \tfrac{(2k+1)(l+m+1)(l-m+2)}{(l+1)(l+2)}
N(Z)^{k-\frac12} \cdot {\bf M}^{m-1}_{l+1}(Z)  \\
+\tfrac{(l+m)(l+m+1)}{(2l+1)(2l+3)(l+1)} {\bf C}^{m-1}_{k+\frac12,l}(Z),
\end{multline*}
\begin{multline*}
2 \tfrac{\partial}{\partial z_{22}}
\bigl( N(Z)^{k+\frac12} \cdot {\bf T}^m_l(Z) \bigr)
= \tfrac{2(l+2)(k+l+1)}{(l+1)(2l+3)}
N(Z)^{k+\frac12} \cdot {\bf T}^{m+1}_{l-1}(Z)  \\
- \tfrac{(2k+1)(l+1)}{(2l+3)(l+2)} N(Z)^{k-\frac12} \cdot {\bf T}^{m+1}_{l+1}(Z)
- \tfrac{2k+1}{(l+1)(l+2)} N(Z)^{k-\frac12} \cdot {\bf M}^{m+1}_{l+1}(Z)  \\
- \tfrac1{(2l+1)(2l+3)(l+1)} {\bf C}^{m+1}_{k+\frac12,l}(Z);
\end{multline*}
\begin{multline*}
\tfrac{\partial}{\partial t} \bigl( N(Z)^{k+\frac12} \cdot {\bf M}^m_l(Z) \bigr)
= - \tfrac{2m(k+l+1)}{l(l+1)} N(Z)^{k+\frac12} \cdot {\bf T}^m_{l-1}(Z)
- \tfrac{m}{(2l+1)(l+1)} {\bf C}^m_{k+\frac12,l}(Z)  \\
+ \tfrac{2(l+1)(k+l+1)(l+m)}{l(2l+1)} N(Z)^{k+\frac12} \cdot {\bf M}^m_{l-1}(Z)
+ \tfrac{l(2k+1)(l-m+1)}{(2l+1)(l+1)} N(Z)^{k-\frac12} \cdot {\bf M}^m_{l+1}(Z),
\end{multline*}
\begin{multline*}
2\tfrac{\partial}{\partial z_{11}}
\bigl( N(Z)^{k+\frac12} \cdot {\bf M}^m_l(Z) \bigr)
= - \tfrac{2(k+l+1)(l+m)(l-m+1)}{l(l+1)}
N(Z)^{k+\frac12} \cdot {\bf T}^{m-1}_{l-1}(Z)  \\
- \tfrac{2(l+1)(k+l+1)(l+m)(l+m-1)}{l(2l+1)}
N(Z)^{k+\frac12} \cdot {\bf M}^{m-1}_{l-1}(Z)
+ \tfrac{l(2k+1)(l-m+1)(l-m+2)}{(2l+1)(l+1)}
N(Z)^{k-\frac12} \cdot {\bf M}^{m-1}_{l+1}(Z) \\
- \tfrac{(l+m)(l-m+1)}{(2l+1)(l+1)} {\bf C}^{m-1}_{k+\frac12,l}(Z),
\end{multline*}
\begin{multline*}
2 \tfrac{\partial}{\partial z_{22}}
\bigl( N(Z)^{k+\frac12} \cdot {\bf M}^m_l(Z) \bigr)
= - \tfrac{2(k+l+1)}{l(l+1)} N(Z)^{k+\frac12} \cdot {\bf T}^{m+1}_{l-1}(Z)
- \tfrac1{(2l+1)(l+1)} {\bf C}^{m+1}_{k+\frac12,l}(Z)  \\
+ \tfrac{2(l+1)(k+l+1)}{l(2l+1)} N(Z)^{k+\frac12} \cdot {\bf M}^{m+1}_{l-1}(Z)
- \tfrac{l(2k+1)}{(2l+1)(l+1)} N(Z)^{k-\frac12} \cdot {\bf M}^{m+1}_{l+1}(Z);
\end{multline*}
\begin{align*}
\tfrac{\partial}{\partial t} {\bf C}^m_{k+\frac12,l}(Z)
&= \tfrac{2(k+l+1)(l+m)}{2l-1} {\bf C}^m_{k+\frac12,l-1}(Z)
+ \tfrac{(2k+1)(l-m+1)}{2l+3} {\bf C}^m_{k-\frac12,l+1}(Z),  \\
2 \tfrac{\partial}{\partial z_{11}} {\bf C}^m_{k+\frac12,l}(Z)
&= -\tfrac{2(k+l+1)(l+m)(l+m-1)}{2l-1} {\bf C}^{m-1}_{k+\frac12,l-1}(Z)
+ \tfrac{(2k+1)(l-m+1)(l-m+2)}{2l+3} {\bf C}^{m-1}_{k-\frac12,l+1}(Z),  \\
2 \tfrac{\partial}{\partial z_{22}} {\bf C}^m_{k+\frac12,l}(Z)
&= \tfrac{2(k+l+1)}{2l-1} {\bf C}^{m+1}_{k+\frac12,l-1}(Z)
- \tfrac{2k+1}{2l+3} {\bf C}^{m+1}_{k-\frac12,l+1}(Z).
\end{align*}

From this description of the action of
$\rho'_2 \bigl( \begin{smallmatrix} 0 & B \\ 0 & 0 \end{smallmatrix} \bigr)$
and Lemma \ref{W-half-inversion-lem} we see that the subspace
\eqref{W-half-subspace} is indeed $\mathfrak{sp}(4,\BB C)$-invariant.
It remains to show that the quotient modules \eqref{W-half-subquot1} and
\eqref{W-half-subquot2} are isomorphic to the quotients
\eqref{W*-subquot1} and \eqref{W*-subquot2} respectively.

Identify
\begin{align*}
N(Z)^{k+\frac12} \cdot {\bf T}^m_l(Z) &\leftrightsquigarrow
(2k+1) N(Z)^{k-\frac12} \cdot \leftidx{^*}{{\bf M}}^m_{l+1}(Z),
\qquad l \ge 0,  \\
-(2l+1) N(Z)^{k+\frac12} \cdot {\bf M}^m_l(Z) &\leftrightsquigarrow
\leftidx{^*}{{\bf C}}^m_{k,l}(Z),
\qquad l \ge 1.
\end{align*}
This produces vector space isomorphisms between \eqref{W-half-subquot1},
\eqref{W-half-subquot2} and \eqref{W*-subquot1}, \eqref{W*-subquot2} respectively.
Comparing the actions of
$\rho'_2 \bigl( \begin{smallmatrix} 0 & B \\ 0 & 0 \end{smallmatrix} \bigr)$
listed above with the actions of
$\rho^*_2 \bigl( \begin{smallmatrix} 0 & B \\ 0 & 0 \end{smallmatrix} \bigr)$
found in the proof of Theorem \ref{W*-decomp-thm}, we see that these isomorphisms
intertwine the actions of
$\bigl( \begin{smallmatrix} 0 & B \\ 0 & 0 \end{smallmatrix} \bigr)
\in \mathfrak{sp}(4,\BB C)$.
From Lemmas \ref{W*-inversion-lem}, \ref{W-half-inversion-lem} and equation
\eqref{I*-inversion} we see that these isomorphisms also intertwine the actions of
$\bigl( \begin{smallmatrix} 0 & 0 \\ C & 0 \end{smallmatrix} \bigr)
\in \mathfrak{sp}(4,\BB C)$.
Hence the $\mathfrak{sp}(4,\BB C)$-modules \eqref{W-half-subquot1} and
\eqref{W-half-subquot2} are isomorphic to \eqref{W*-subquot1} and
\eqref{W*-subquot2} respectively.
\end{proof}

We conclude this subsection with a comment that the dual of
$(\rho'_2, \mathring {\cal W}'_{\frac12})$ is
\[
\bigl( \rho^*_2,
\HC\text{-sym} \otimes \BB C [z_{11}, z_{12}, z_{22}, N(Z)^{-1}] \bigr).
\]
This dual $\mathfrak{sp}(4,\BB C)$-module
contains an irreducible submodule isomorphic to
$(\pi_{dl},{\cal F}^+) \simeq (\pi_{dr},{\cal G}^+)$
and another irreducible submodule isomorphic to
$(\pi_{dl},{\cal F}^-) \simeq (\pi_{dr},{\cal G}^-)$.

\section{Invariant Trilinear Forms}  \label{7}

In this section we study $\mathfrak{gl}(2,\HC)$-invariant and
$\mathfrak{sp}(4,\BB C)$-invariant trilinear forms
involving the space of doubly regular functions.

\subsection{The Obvious Invariant Trilinear Form}

Recall the representations $(\Sh,\rho)$ and $(\Sh',\rho')$ of
$\mathfrak{gl}(2,\HC)$:
\[
\Sh = \Sh' = \bigl\{\text{$\BB C$-valued polynomial functions on
$\HC^{\times}$}\bigr\}
= \BB C[z_{11},z_{12},z_{21},z_{22}, N(Z)^{-1}];
\]
the Lie algebra $\g{gl}(2,\HC)$ acts on $\Sh$ and $\Sh'$ by differentiating
the following group actions:
\begin{align*}
\rho(h): \: f(Z) \: &\mapsto \: \bigl( \rho(h)f \bigr)(Z) =
\frac {f \bigl( (aZ+b)(cZ+d)^{-1} \bigr)}{N(cZ+d)^2 \cdot N(a'-Zc')^2}, \\
\rho'(h): \: f(Z) \: &\mapsto \: \bigl( \rho(h)f \bigr)(Z) =
f \bigl( (aZ+b)(cZ+d)^{-1} \bigr),
\end{align*}
where $f \in \Sh$ or $\Sh'$,
$h = \bigl(\begin{smallmatrix} a' & b' \\ c' & d' \end{smallmatrix}\bigr)
\in GL(2,\HC)$ and 
$h^{-1} = \bigl(\begin{smallmatrix} a & b \\ c & d \end{smallmatrix}\bigr)$.
Recall Theorems 31 and 32 from \cite{ATMP}.

\begin{thm}  \label{rho-decomposition}
The only proper $\mathfrak{gl}(2,\HC)$-invariant subspaces of $(\rho,\Sh)$ are
\begin{align*}
\Sh^+ &= \BB C\text{-span of }
\bigl\{ N(Z)^k \cdot t^l_{n\,\underline{m}}(Z);\: k \ge 0 \bigr\}, \\
\Sh^- &= \BB C\text{-span of }
\bigl\{ N(Z)^k \cdot t^l_{n\,\underline{m}}(Z);\: k \le -(2l+4) \bigr\}, \\
{\cal I}^+ &= \BB C\text{-span of }
\bigl\{ N(Z)^k \cdot t^l_{n\,\underline{m}}(Z);\: k \ge -(2l+1) \bigr\}, \\
{\cal I}^- &= \BB C\text{-span of }
\bigl\{ N(Z)^k \cdot t^l_{n\,\underline{m}}(Z);\: k \le -3 \bigr\}, \\
{\cal J} &= \BB C\text{-span of }
\bigl\{ N(Z)^k \cdot t^l_{n\,\underline{m}}(Z);\: -(2l+1) \le k \le -3 \bigr\}
\end{align*}
and their sums (see Figure \ref{Sh-decomposition-fig1}).

The irreducible components of $(\rho,\Sh)$ are the subrepresentations
\[
(\rho, \Sh^+), \qquad (\rho, \Sh^-), \qquad (\rho, {\cal J})
\]
and the quotients
\begin{equation*}  
\bigl( \rho, {\cal I}^+/(\Sh^+ \oplus {\cal J}) \bigr), \quad
\bigl( \rho, {\cal I}^-/(\Sh^- \oplus {\cal J}) \bigr), \quad
\bigl( \rho, \Sh/({\cal I}^++{\cal I}^-) \bigr)
\end{equation*}
(see Figure \ref{Sh-decomposition-fig1-comp}).
\end{thm}

\begin{figure}
\begin{center}
\setlength{\unitlength}{1mm}
\begin{picture}(120,70)
\multiput(10,10)(10,0){11}{\circle*{1}}
\multiput(10,20)(10,0){11}{\circle*{1}}
\multiput(10,30)(10,0){11}{\circle*{1}}
\multiput(10,40)(10,0){11}{\circle*{1}}
\multiput(10,50)(10,0){11}{\circle*{1}}
\multiput(10,60)(10,0){11}{\circle*{1}}

\thicklines
\put(80,0){\vector(0,1){70}}
\put(0,10){\vector(1,0){120}}

\thinlines
\put(78,10){\line(0,1){55}}
\put(80,8){\line(1,0){35}}
\qbezier(78,10)(78,8)(80,8)

\put(52,30){\line(0,1){35}}
\put(48.6,28.6){\line(-1,1){36.4}}
\qbezier(52,30)(52,25.2)(48.6,28.6)

\put(5,8){\line(1,0){35}}
\put(41.4,11.4){\line(-1,1){36.4}}
\qbezier(40,8)(44.8,8)(41.4,11.4)

\put(73,7){\line(1,0){42}}
\put(66,10){\line(-1,1){55}}
\qbezier(73,7)(69,7)(66,10)

\put(5,7){\line(1,0){45}}
\put(53,10){\line(0,1){55}}
\qbezier(50,7)(53,7)(53,10)


\put(82,67){$2l$}
\put(117,12){$k$}
\put(3,24){$\Sh^-$}
\put(33,62){${\cal J}$}
\put(112,44){$\Sh^+$}
\put(25,3){${\cal I}^-$}
\put(95,3){${\cal I}^+$}
\end{picture}
\end{center}
\caption{Decomposition of $(\rho,\Sh)$ into irreducible components.}
\label{Sh-decomposition-fig1}
\end{figure}
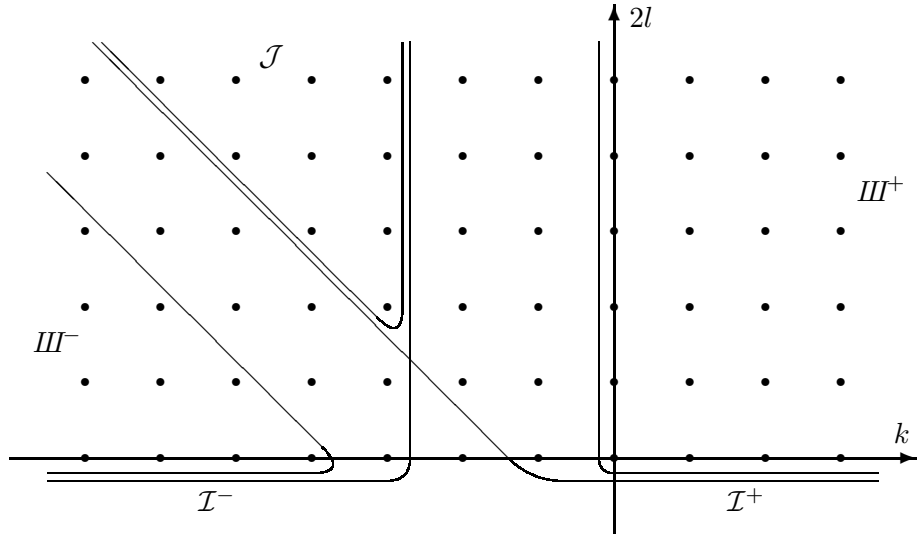

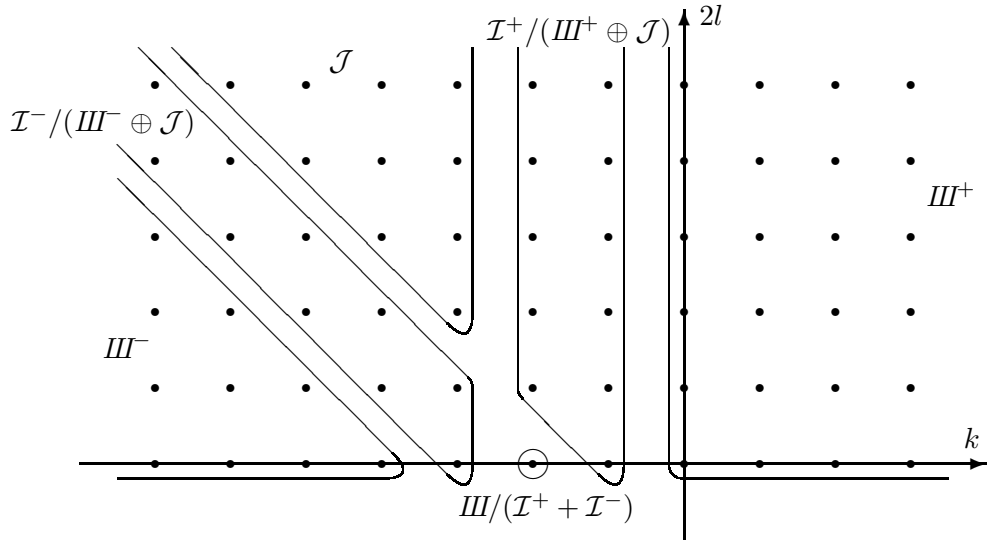
\begin{figure}
\begin{center}
\setlength{\unitlength}{1mm}
\begin{picture}(120,70)
\multiput(10,10)(10,0){11}{\circle*{1}}
\multiput(10,20)(10,0){11}{\circle*{1}}
\multiput(10,30)(10,0){11}{\circle*{1}}
\multiput(10,40)(10,0){11}{\circle*{1}}
\multiput(10,50)(10,0){11}{\circle*{1}}
\multiput(10,60)(10,0){11}{\circle*{1}}

\thicklines
\put(80,0){\vector(0,1){70}}
\put(0,10){\vector(1,0){120}}

\thinlines

\put(60,10){\circle{4}}

\put(72,10){\line(0,1){55}}
\put(58,20){\line(0,1){45}}
\put(68.6,8.6){\line(-1,1){10}}
\qbezier(72,10)(72,5.2)(68.6,8.6)
\qbezier(58,20)(58,19.2)(58.6,18.6)

\put(52,10){\line(0,1){10}}
\put(48.6,8.6){\line(-1,1){43.6}}
\put(51.4,21.4){\line(-1,1){43.6}}
\qbezier(52,10)(52,5.2)(48.6,8.6)
\qbezier(52,20)(52,20.8)(51.4,21.4)

\put(78,10){\line(0,1){55}}
\put(80,8){\line(1,0){35}}
\qbezier(78,10)(78,8)(80,8)

\put(52,30){\line(0,1){35}}
\put(48.6,28.6){\line(-1,1){36.4}}
\qbezier(52,30)(52,25.2)(48.6,28.6)

\put(5,8){\line(1,0){35}}
\put(41.4,11.4){\line(-1,1){36.3}}
\qbezier(40,8)(44.8,8)(41.4,11.4)

\put(82,68){$2l$}
\put(117,12){$k$}
\put(3,24){$\Sh^-$}
\put(33,62){${\cal J}$}
\put(112,44){$\Sh^+$}

\put(54,66){${\cal I}^+/(\Sh^+ \oplus {\cal J})$}
\put(-9,54){${\cal I}^-/(\Sh^- \oplus {\cal J})$}

\put(50,3){$\Sh/({\cal I}^++{\cal I}^-)$}
\end{picture}
\end{center}
\caption{Irreducible components of $(\rho,\Sh)$.}
\label{Sh-decomposition-fig1-comp}
\end{figure}

\begin{thm}  \label{rho'-decomposition}
The only proper $\mathfrak{gl}(2,\HC)$-invariant subspaces of $(\rho',\Sh')$ are
\begin{align*}
{\cal I}'_0 &= \BB C = \BB C\text{-span of }
\bigl\{ N(Z)^0 \cdot t^0_{0\,\underline{0}}(Z) \bigr\}, \\
{\cal BH}^+ &= \BB C\text{-span of }
\bigl\{ N(Z)^k \cdot t^l_{n\,\underline{m}}(Z);\: 0 \le k \le 1 \bigr\}, \\
{\cal BH}^- &= \BB C\text{-span of }
\bigl\{ N(Z)^k \cdot t^l_{n\,\underline{m}}(Z);\: -1 \le 2l+k \le 0 \bigr\}, \\
\Sh^+ &= \BB C\text{-span of }
\bigl\{ N(Z)^k \cdot t^l_{n\,\underline{m}}(Z);\: k \ge 0 \bigr\}, \\
\Sh'^- &= \BB C\text{-span of }
\bigl\{ N(Z)^k \cdot t^l_{n\,\underline{m}}(Z);\: k \le -2l \bigr\}, \\
{\cal I}'^+ &= \BB C\text{-span of }
\bigl\{ N(Z)^k \cdot t^l_{n\,\underline{m}}(Z);\: k \ge -(2l+1) \bigr\}, \\
{\cal I}'^- &= \BB C\text{-span of }
\bigl\{ N(Z)^k \cdot t^l_{n\,\underline{m}}(Z);\: k \le 1 \bigr\}, \\
{\cal J}' &= \BB C\text{-span of }
\bigl\{ N(Z)^k \cdot t^l_{n\,\underline{m}}(Z);\: -(2l+1) \le k \le 1 \bigr\}
\end{align*}
and their sums (see Figure \ref{Sh-decomposition-fig2}).

The irreducible components of $(\rho',\Sh')$ are the trivial
subrepresentation $(\rho', {\cal I}'_0)$ and the quotients
\[
(\rho', {\cal BH}^+/{\cal I}'_0), \quad (\rho', {\cal BH}^-/{\cal I}'_0),
\]
\[
(\rho', \Sh^+/{\cal BH}^+) = (\rho',\Sh/{\cal I}'^-), \quad
(\rho', \Sh'^-/{\cal BH}^-) = (\rho',\Sh/{\cal I}'^+),
\]
\[
\bigl( \rho', \Sh/({\cal I}'^++{\cal I}'^-) \bigr) =
\bigl( \rho', {\cal I}'^+/(\Sh^++{\cal BH}^-) \bigr) =
\bigl( \rho', {\cal I}'^-/(\Sh'^-+{\cal BH}^+) \bigr) =
\bigl( \rho', {\cal J}'/({\cal BH}^++{\cal BH}^-) \bigr)
\]
(see Figure \ref{Sh-decomposition-fig3},
which is essentially a shifted Figure \ref{Sh-decomposition-fig1-comp}).
\end{thm}

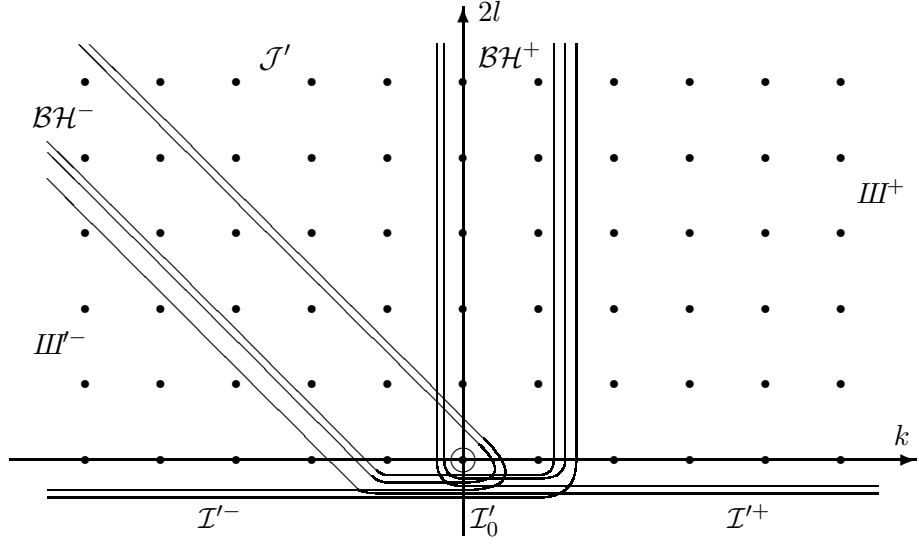
\begin{figure}
\begin{center}
\setlength{\unitlength}{1mm}
\begin{picture}(120,70)
\multiput(10,10)(10,0){11}{\circle*{1}}
\multiput(10,20)(10,0){11}{\circle*{1}}
\multiput(10,30)(10,0){11}{\circle*{1}}
\multiput(10,40)(10,0){11}{\circle*{1}}
\multiput(10,50)(10,0){11}{\circle*{1}}
\multiput(10,60)(10,0){11}{\circle*{1}}

\thicklines
\put(60,0){\vector(0,1){70}}
\put(0,10){\vector(1,0){120}}

\thinlines

\put(60,10){\circle{3}}

\put(50,8){\line(1,0){20}}
\put(72,10){\line(0,1){55}}
\put(48.6,8.6){\line(-1,1){43.6}}
\qbezier(48.6,8.6)(49.2,8)(50,8)
\qbezier(72,10)(72,8)(70,8)

\put(57.5,10){\line(0,1){55}}
\put(73.5,10){\line(0,1){55}}
\put(60,7.5){\line(1,0){10}}
\qbezier(57.5,10)(57.5,7.5)(60,7.5)
\qbezier(73.5,10)(73.5,7.5)(70,7.5)

\put(56.5,10){\line(0,1){55}}
\put(60,6.5){\line(1,0){55}}
\qbezier(56.5,10)(56.5,6.5)(60,6.5)

\put(50,7){\line(1,0){10}}
\put(62.1,12.1){\line(-1,1){52.9}}
\put(47.9,7.9){\line(-1,1){42.8}}
\qbezier(60,7)(67.2,7)(62.1,12.1)
\qbezier(47.9,7.9)(48.8,7)(50,7)

\put(5,6){\line(1,0){55}}
\put(62.8,12.8){\line(-1,1){52.2}}
\qbezier(60,6)(69.6,6)(62.8,12.8)

\put(50,5.5){\line(1,0){65}}
\put(46.1,6.1){\line(-1,1){41.1}}
\qbezier(46.1,6.1)(46.7,5.5)(50,5.5)

\put(5,5){\line(1,0){65}}
\put(75,10){\line(0,1){55}}
\qbezier(70,5)(75,5)(75,10)

\put(62,68){$2l$}
\put(117,12){$k$}
\put(62,62){${\cal BH}^+$}
\put(3,54){${\cal BH}^-$}
\put(3,24){$\Sh'^-$}
\put(33,62){${\cal J}'$}
\put(61,1){${\cal I}'_0$}
\put(112,44){$\Sh^+$}
\put(25,1){${\cal I}'^-$}
\put(95,1){${\cal I}'^+$}
\end{picture}
\end{center}
\caption{Decomposition of $(\rho',\Sh')$ into irreducible components.}
\label{Sh-decomposition-fig2}
\end{figure}

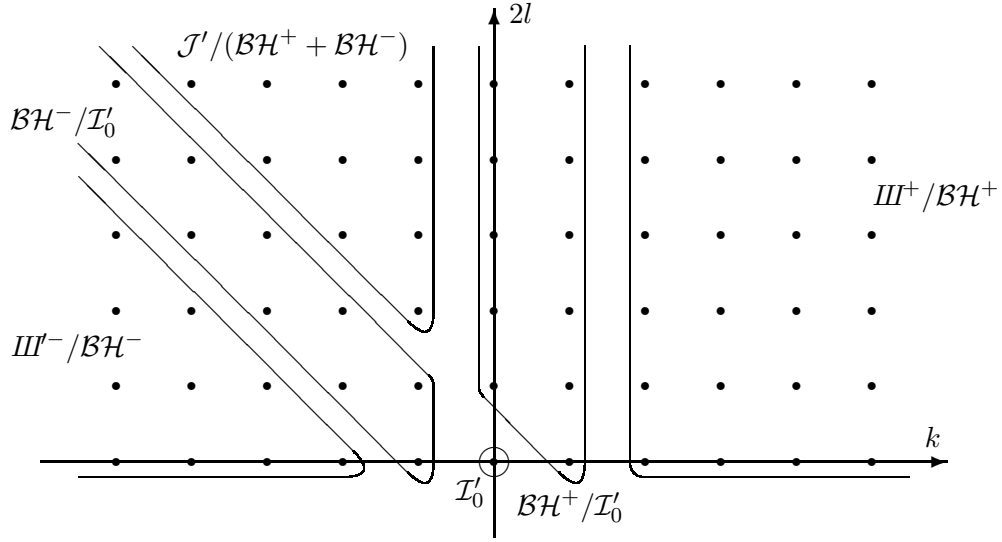
\begin{figure}
\begin{center}
\setlength{\unitlength}{1mm}
\begin{picture}(120,70)
\multiput(10,10)(10,0){11}{\circle*{1}}
\multiput(10,20)(10,0){11}{\circle*{1}}
\multiput(10,30)(10,0){11}{\circle*{1}}
\multiput(10,40)(10,0){11}{\circle*{1}}
\multiput(10,50)(10,0){11}{\circle*{1}}
\multiput(10,60)(10,0){11}{\circle*{1}}

\thicklines
\put(60,0){\vector(0,1){70}}
\put(0,10){\vector(1,0){120}}

\thinlines

\put(60,10){\circle{4}}

\put(72,10){\line(0,1){55}}
\put(58,20){\line(0,1){45}}
\put(68.6,8.6){\line(-1,1){10}}
\qbezier(72,10)(72,5.2)(68.6,8.6)
\qbezier(58,20)(58,19.2)(58.6,18.6)

\put(52,10){\line(0,1){10}}
\put(48.6,8.6){\line(-1,1){43.6}}
\put(51.4,21.4){\line(-1,1){43.6}}
\qbezier(52,10)(52,5.2)(48.6,8.6)
\qbezier(52,20)(52,20.8)(51.4,21.4)

\put(78,10){\line(0,1){55}}
\put(80,8){\line(1,0){35}}
\qbezier(78,10)(78,8)(80,8)

\put(52,30){\line(0,1){35}}
\put(48.6,28.6){\line(-1,1){36.4}}
\qbezier(52,30)(52,25.2)(48.6,28.6)

\put(5,8){\line(1,0){35}}
\put(41.4,11.4){\line(-1,1){36.3}}
\qbezier(40,8)(44.8,8)(41.4,11.4)

\put(62,68){$2l$}
\put(117,12){$k$}
\put(63,3){${\cal BH}^+/{\cal I}'_0$}
\put(-4,54){${\cal BH}^-/{\cal I}'_0$}
\put(-4,24){$\Sh'^-/{\cal BH}^-$}
\put(18,64){${\cal J}'/({\cal BH}^++{\cal BH}^-)$}
\put(55,5){${\cal I}'_0$}
\put(110,44){$\Sh^+/{\cal BH}^+$}
\end{picture}
\end{center}
\caption{Irreducible components of $(\rho',\Sh')$.}
\label{Sh-decomposition-fig3}
\end{figure}

The projection of $(\rho,\Sh)$ onto its trivial 1-dimensional component
$\Sh/({\cal I}^++{\cal I}^-)$ can be expressed as an integral
\begin{equation}  \label{Sh-projection}
f \: \mapsto \: \frac i{2\pi^3} \int_{U(2)_R} f(Z) \,dV,
\qquad f \in \Sh,
\end{equation}
where  $dV$ is a holomorphic 4-form defined by
\[
dV = dz^0 \wedge dz^1 \wedge dz^2 \wedge dz^3
= \frac14 dz_{11} \wedge dz_{12} \wedge dz_{21} \wedge dz_{22}.
\]

Regular functions were reviewed in Subsection \ref{1-reg-review}.
Consider a multiplication map
\begin{equation}  \label{gFf-multiplication}
(\pi_r, {\cal V}') \times (\rho'_2, {\cal W}') \times (\pi_l, {\cal V})
\to (\rho, \Sh), \qquad (g, F, f) \mapsto gFf.
\end{equation}

\begin{prop}  \label{gFf-prop}
The multiplication map \eqref{gFf-multiplication} is
$\mathfrak{gl}(2,\HC)$-equivariant.
\end{prop}

\begin{proof}
We will show that, for all $h \in GL(2,\HC)$, the map
\eqref{gFf-multiplication} commutes with the action of $h$. Writing
\[
h= \bigl(\begin{smallmatrix} a' & b' \\ c' & d' \end{smallmatrix}\bigr), \qquad
h^{-1}= \bigl(\begin{smallmatrix} a & b \\ c & d \end{smallmatrix}\bigr), \qquad
\tilde Z = (aZ+b)(cZ+d)^{-1},
\]
we obtain:
\begin{multline*}
\bigl( \pi_r(h)g \bigr)(Z) \cdot \bigl(\rho'_2(h)F \bigr)(Z) \cdot
\bigl(\pi_l(h) f \bigr)(Z)  \\
= g(\tilde Z) \cdot \frac{(a'-Zc')^{-1}}{N(a'-Zc')} \cdot
\frac{a'-Zc'}{N(a'-Zc')} \cdot F(\tilde Z) \cdot \frac{cZ+d}{N(cZ+d)}
\cdot \frac{(cZ+d)^{-1}}{N(cZ+d)} \cdot f(\tilde Z)  \\
= \frac{(gFf)(\tilde Z)}{N(cZ+d)^2 \cdot N(a'-Zc')^2}
= \bigl(\rho(h) (gFf)\bigr)(Z).
\end{multline*}
Then the  $\mathfrak{gl}(2,\HC)$-equivariance follows.
\end{proof}

Combining the multiplication map \eqref{gFf-multiplication}
with the projection \eqref{Sh-projection}, we obtain a
$\mathfrak{gl}(2,\HC)$-invariant trilinear form
\begin{equation}  \label{3form-obvious}
(\pi_r, {\cal V}') \times (\rho'_2, {\cal W}') \times (\pi_l, {\cal V})
\to \BB C, \qquad
(g, F, f) \mapsto \frac i{2\pi^3} \int_{U(2)_R} g(Z) \cdot F(Z) \cdot f(Z) \,dV.
\end{equation}

Recall from \cite{FL1,ATMP} that the dual $\mathfrak{gl}(2,\HC)$-module
of $(\rho'_2, {\cal W}')$ is $(\rho_2, {\cal W})$, where
\[
{\cal W} = \HC \otimes \BB C [z_{11}, z_{12}, z_{21}, z_{22}, N(Z)^{-1}]
\]
and the action of $\mathfrak{gl}(2,\HC)$ on ${\cal W}$ is obtained
by differentiating the following group action:
\[
\rho_2(h): \: G(Z) \: \mapsto \: \bigl( \rho_2(h)G \bigr)(Z) =
\frac {(cZ+d)^{-1}}{N(cZ+d)} \cdot G \bigl( (aZ+b)(cZ+d)^{-1} \bigr) \cdot
\frac {(a'-Zc')^{-1}}{N(a'-Zc')},
\]
$G \in {\cal W}$,
$h = \bigl(\begin{smallmatrix} a' & b' \\ c' & d' \end{smallmatrix}\bigr)
\in GL(2,\HC)$ and 
$h^{-1} = \bigl(\begin{smallmatrix} a & b \\ c & d \end{smallmatrix}\bigr)$.

Also recall that $(\rho'_2, {\cal W}')$ contains
$\mathfrak{gl}(2,\HC)$-invariant subspaces ${\cal M}^+$ and ${\cal M}^-$
described by \eqref{M^+}-\eqref{M^-}.

\begin{prop}  \label{3form-trivial}
The $\mathfrak{gl}(2,\HC)$-invariant trilinear form \eqref{3form-obvious}
is identically zero on
${\cal V}' \times ({\cal M}^+ \oplus {\cal M}^-) \times {\cal V}$.
\end{prop}

\begin{proof}
We can rewrite the product $gFf$ as $\tr(Ffg)$.
Then the trilinear form \eqref{3form-obvious} can be rewritten as a
composition of the $\mathfrak{gl}(2,\HC)$-equivariant multiplication map
\begin{equation}  \label{VV-mult}
\operatorname{Mult}:
(\pi_l, {\cal V}) \times (\pi_r, {\cal V}') \to (\rho_2, {\cal W}),
\qquad (f, g) \mapsto fg,
\end{equation}
followed by the $\mathfrak{gl}(2,\HC)$-invariant bilinear pairing between 
$(\rho_2, {\cal W})$ and $(\rho'_2, {\cal W}')$
from Proposition 80 in \cite{FL1}
\[
\langle F, G \rangle = \frac i{2\pi^3} \int_{U(2)_R} \tr(F(Z) \cdot G(Z)) \,dV,
\qquad F \in {\cal W}',\: G \in {\cal W}.
\]
By Proposition 89 in \cite{ATMP} (see Remark \ref{Prop89} below),
the image of ${\cal V} \otimes {\cal V}'$ in ${\cal W}$
under the multiplication map \eqref{VV-mult} is
\begin{equation}  \label{VV-image}
  {\cal Q}^+ \oplus ({\cal Q}^0 + N(Z)^{-2} \cdot Z^+) \oplus {\cal Q}^-.
\end{equation}
On the other hand, the irreducible components of ${\cal M}^+ \oplus {\cal M}^-$
are described in Subsection \ref{IrredDecom-subsection}, and none of them are
dual to the irreducible components of \eqref{VV-image}.
\end{proof}

\begin{rem}  \label{Prop89}
Proposition 89 in \cite{ATMP} should be corrected to: 
The image under the multiplication map
$\operatorname{Mult}: {\cal V} \otimes {\cal V}' \to {\cal W}$ is precisely
${\cal Q}^+ \oplus ({\cal Q}^0 + N(Z)^{-2} \cdot Z^+) \oplus {\cal Q}^-$.
\end{rem}

Because of Proposition \ref{3form-trivial} we are not interested in
the trilinear form \eqref{3form-obvious}.

\subsection{Another $\mathfrak{gl}(2,\HC)$-Invariant Trilinear Form}

Quasi anti regular functions were reviewed in
Subsection \ref{QReg-subsection}. Consider a multiplication map
\begin{equation}  \label{quasi-gFf-multiplication}
(\pi'_r, {\cal U}') \times (\rho_2, {\cal W}) \times
(\pi'_l, {\cal U}) \to (\rho, \Sh), \qquad (g, F, f) \mapsto gFf.
\end{equation}
The argument used to show Proposition \ref{gFf-prop} also shows:

\begin{prop}
The multiplication map \eqref{quasi-gFf-multiplication} is
$\mathfrak{gl}(2,\HC)$-equivariant.
\end{prop}

Combining the multiplication map \eqref{quasi-gFf-multiplication}
with the projection \eqref{Sh-projection}, we obtain a
$\mathfrak{gl}(2,\HC)$-invariant trilinear form
\begin{equation}  \label{3form-easy}
(\pi'_r, {\cal U}') \times (\rho_2, {\cal W}) \times
(\pi'_l, {\cal U}) \to \BB C, \qquad
(g, F, f) \mapsto \frac i{2\pi^3} \int_{U(2)_R} g(Z) \cdot F(Z) \cdot f(Z) \,dV.
\end{equation}

By Proposition 56 in \cite{qreg}, restrictions of \eqref{3form-easy} result
in non-trivial $\mathfrak{gl}(2,\HC)$-invariant trilinear forms
\begin{align*}
(\pi'_r, {\cal U}'^+) \times
\left(\rho_2, \begin{smallmatrix} \text{submodule of ${\cal W}$} \\
  \text{containing $(\pi_{dl},{\cal F}^-)$ or $(\pi_{dr},{\cal G}^-)$}
\end{smallmatrix}\right) \times (\pi'_l, {\cal U}^+) &\to \BB C,  \\
(\pi'_r, {\cal U}'^-) \times
\left(\rho_2, \begin{smallmatrix} \text{submodule of ${\cal W}$} \\
  \text{containing $(\pi_{dl},{\cal F}^+)$ or $(\pi_{dr},{\cal G}^+)$}
\end{smallmatrix}\right) \times (\pi'_l, {\cal U}^-) &\to \BB C.
\end{align*}

The decomposition of $(\rho_2, {\cal W})$ into irreducible components
was investigated in \cite{ATMP}. In particular, any
$\mathfrak{gl}(2,\HC)$-submodule of ${\cal W}$ containing
$(\pi_{dl},{\cal F}^{\pm})$ or $(\pi_{dr},{\cal G}^{\pm})$ also
contains a submodule ${\cal Q}^0$.
However, $(\rho_2, {\cal Q}^0)$ has Gelfand-Kirillov dimension $4$,
while $(\pi_{dl},{\cal F}^{\pm})$ and $(\pi_{dr},{\cal G}^{\pm})$
have Gelfand-Kirillov dimensions $3$.
Since realizations of doubly regular functions inside ${\cal W}$
require ``much larger'' modules than the doubly regular functions
themselves, we are not interested in the trilinear form
\eqref{3form-easy} either.

\subsection{A Scalar Version of a
  $\mathfrak{gl}(2,\HC)$-Invariant Trilinear Form}

For future reference, in this subsection we briefly describe a
$\mathfrak{gl}(2,\HC)$-invariant trilinear
form involving spaces of scalar valued functions.
While this form does not involve the doubly regular functions,
it is useful as a practice model.

Recall harmonic polynomial functions on $\HC^{\times}$:
\begin{align*}
  {\cal H}^+ &=   \bigl\{ \phi \in \BB C[z_{11},z_{12},z_{21},z_{22}] ;\:
  \square \phi =0 \bigr\},  \\
  {\cal H}^- &= \bigl\{ \phi \in \BB C[z_{11},z_{12},z_{21},z_{22}, N(Z)^{-1}] ;\:
  N(Z)^{-1} \cdot \phi(Z^{-1}) \in {\cal H^+} \bigr\},  \\
  {\cal H} &= \bigl\{ \phi \in \BB C[z_{11},z_{12},z_{21},z_{22}, N(Z)^{-1}] ;\:
  \square \phi =0 \bigr\}  \\
  &= {\cal H}^+ \oplus {\cal H}^-.
\end{align*}
The Lie algebra $\mathfrak{gl}(2,\HC)$ acts on ${\cal H}^+$, ${\cal H}^-$ and
${\cal H}$ via two slightly different actions $\pi^0_l$ and $\pi^0_r$ obtained
by differentiating the following group actions respectively:
\begin{align*}
\pi^0_l(h): \: \phi(Z) \: &\mapsto \: \bigl( \pi^0_l(h) \phi \bigr)(Z) =
\frac1{N(cZ+d)} \cdot \phi \bigl( (aZ+b)(cZ+d)^{-1} \bigr),  \\
\pi^0_r(h): \: \phi(Z) \: &\mapsto \: \bigl( \pi^0_r(h) \phi \bigr)(Z) =
\frac1{N(a'-Zc')} \cdot \phi \bigl( (a'-Zc')^{-1}(-b'+Zd') \bigr),
\end{align*}
where $\phi \in {\cal H}$,
$h = \bigl(\begin{smallmatrix} a' & b' \\ c' & d' \end{smallmatrix}\bigr)
\in GL(2,\HC)$ and 
$h^{-1} = \bigl(\begin{smallmatrix} a & b \\ c & d \end{smallmatrix}\bigr)$
(Subsections 2.4-2.5, 4.1 in \cite{FL1}).

We also recall a $\mathfrak{gl}(2,\HC)$-module $(\rho_1,\Zh)$:
\[
\Zh = \BB C[z_{11},z_{12},z_{21},z_{22}, N(Z)^{-1}],
\]
and the action of $\mathfrak{gl}(2,\HC)$ is obtained
by differentiating the following group action:
\[
\rho_1(h): \: f(Z) \: \mapsto \: \bigl( \rho_1(h) f \bigr)(Z) =
\frac {f\bigl( (aZ+b)(cZ+d)^{-1} \bigr)}{N(cZ+d) \cdot N(a'-Zc')},
\]
where $f \in \Zh$,
$h = \bigl(\begin{smallmatrix} a' & b' \\ c' & d' \end{smallmatrix}\bigr)
\in GL(2,\HC)$ and 
$h^{-1} = \bigl(\begin{smallmatrix} a & b \\ c & d \end{smallmatrix}\bigr)$.
It was shown in \cite{desitter} that we have a direct sum decomposition
into irreducible components
\[
(\rho_1,\Zh) = (\rho_1,\Zh^+) \oplus (\rho_1,\Zh^0) \oplus (\rho_1,\Zh^-),
\]
where
\begin{align*}
\Zh^+ &=   \BB C[z_{11},z_{12},z_{21},z_{22}] ,  \\
\Zh^- &= \bigl\{ f \in \BB C[z_{11},z_{12},z_{21},z_{22}, N(Z)^{-1}] ;\:
  N(Z)^{-2} \cdot f(Z^{-1}) \in \Zh^+ \bigr\},
\end{align*}
and $\Zh^0$ is the unique $\mathfrak{gl}(2,\HC)$-invariant irreducible
subspace of $\Zh$ that is a direct sum complement of $\Zh^+ \oplus \Zh^-$.

Consider a multiplication map
\begin{equation}  \label{scalar-gFf-multiplication}
(\pi^0_r, {\cal H}) \times (\rho_1, \Zh) \times
(\pi^0_l, {\cal H}) \to (\rho, \Sh), \qquad (\phi_1, f, \phi_2)
\mapsto \phi_1 f \phi_2.
\end{equation}
The argument used to show Proposition \ref{gFf-prop} also shows:

\begin{prop}
The multiplication map \eqref{scalar-gFf-multiplication} is
$\mathfrak{gl}(2,\HC)$-equivariant.
\end{prop}

Combining the multiplication map \eqref{scalar-gFf-multiplication}
with the projection \eqref{Sh-projection}, we obtain a
$\mathfrak{gl}(2,\HC)$-invariant trilinear form
\begin{equation}  \label{3form-sc}
(\pi^0_r, {\cal H}) \times (\rho_1, \Zh) \times
(\pi^0_l, {\cal H}) \to \BB C, \qquad
(\phi_1, f, \phi_2) \mapsto \frac i{2\pi^3} \int_{U(2)_R}
\phi_1(Z) \cdot f(Z) \cdot \phi_2(Z) \,dV.
\end{equation}
Restricting this trilinear form 
to the irreducible components, we obtain twelve potential
$\mathfrak{gl}(2,\HC)$-invariant trilinear forms:
\begin{equation}  \label{3form-sc-12}
(\pi^0_r, {\cal H}^{\pm}) \times (\rho_1, \Zh^{\,\epsilon}) \times
(\pi^0_l, {\cal H}^{\pm}) \to \BB C,
\qquad \text{where }\epsilon \in \{+,0,-\}.
\end{equation}

From Proposition 4, Lemma 8 and orthogonality relations (19) in \cite{desitter}
we obtain the following result:

\begin{prop}
Of the twelve potential $\mathfrak{gl}(2,\HC)$-invariant
trilinear forms \eqref{3form-sc-12}, exactly four are non-trivial:
\begin{align*}
(\pi^0_r, {\cal H}^+) \times (\rho_1, \Zh^-) \times
(\pi^0_l, {\cal H}^+) &\to \BB C,  \\
(\pi^0_r, {\cal H}^-) \times (\rho_1, \Zh^+) \times
(\pi^0_l, {\cal H}^-) &\to \BB C,  \\
(\pi^0_r, {\cal H}^+) \times (\rho_1, \Zh^0) \times
(\pi^0_l, {\cal H}^-) &\to \BB C,  \\
(\pi^0_r, {\cal H}^-) \times (\rho_1, \Zh^0) \times
(\pi^0_l, {\cal H}^+) &\to \BB C.
\end{align*}
\end{prop}

\subsection{A Scalar Version of an
  $\mathfrak{sp}(4,\BB C)$-Invariant Trilinear Form}

In this subsection we describe an $\mathfrak{sp}(4,\BB C)$-invariant
trilinear form involving spaces of scalar valued functions.
This form does not involve the doubly regular functions
but serves as a simple analogue of the trilinear form that will be
introduced in Subsection \ref{3form-main-subsect}.

Let ${\cal H}^+_{3-dim}$ and ${\cal H}^-_{3-dim}$ be the spaces of $\BB C$-valued
harmonic functions on $\HC^{\times} \text{-sym}$ regular at the origin and infinity
respectively:
\begin{align*}
{\cal H}^+_{3-dim} &= \{ \phi(Z) \in \BB C [z_{11},z_{12},z_{22}] ;\: \Delta_3 \phi(Z)=0 \},  \\
{\cal H}^-_{3-dim} &=
\{ \phi(Z) \in N(Z)^{\frac12} \cdot \BB C [z_{11},z_{12},z_{22},N(Z)^{-1}] ;\:
\Delta_3 \phi(Z)=0 \}.
\end{align*}
The Lie algebra $\mathfrak{sp}(4,\BB C)$ acts on these spaces by differentiating the
$Sp(4,\BB R)$ action given by
\[
\pi^0_{3-dim}(h): \: \phi(Z) \: \mapsto \: \bigl( \pi^0_{3-dim}(h) \phi\bigr)(Z) =
\frac1{\sqrt{N(cZ+d)}} \cdot \phi \bigl( (aZ+b)(cZ+d)^{-1} \bigr), 
\]
where $\phi  \in {\cal H}^{\pm}_{3-dim}$,
$h = \bigl(\begin{smallmatrix} a' & b' \\ c' & d' \end{smallmatrix}\bigr)
\in Sp(4,\BB R) \subset GL(2,\HC)$ and 
$h^{-1} = \bigl(\begin{smallmatrix} a & b \\ c & d \end{smallmatrix}\bigr)$.
The two spaces ${\cal H}^+_{3-dim}$ and ${\cal H}^-_{3-dim}$ are interchanged by
the inversion:
\[
\pi^0_{3-dim} \bigl(\begin{smallmatrix} 0 & 1 \\ 1 & 0 \end{smallmatrix}\bigr):\:
\phi(Z) \mapsto N(Z)^{-\frac12} \cdot \phi(Z^{-1}).
\]
We can describe the $K$-types:

\begin{thm}
The two $\mathfrak{sp}(4,\BB C)$-modules $(\pi^0_{3-dim}, {\cal H}^+_{3-dim})$ and
$(\pi^0_{3-dim}, {\cal H}^-_{3-dim})$ are irreducible. They have $K$-types
\begin{align*}
{\cal H}^+_{3-dim} &= \BB C\text{-span of }
\bigl\{ R^m_l(Z);\: l \ge 0,\: -l \le m \le l \bigr\},  \\
{\cal H}^-_{3-dim} &= \BB C\text{-span of }
\bigl\{ N(Z)^{-l-\frac12} \cdot R^m_l(Z);\: l \ge 0,\: -l \le m \le l \bigr\}.
\end{align*}
More precisely, for $l$ fixed, as representations of $SU(2)$,
\begin{align*}
\BB C\text{-span of } \bigl\{ R^m_l(Z);\: -l \le m \le l \bigr\} &\simeq V_l,  \\
\BB C\text{-span of } \bigl\{ N(Z)^{-l-\frac12} \cdot R^m_l(Z);\: -l \le m \le l \bigr\}
&\simeq V_l.
\end{align*}
\end{thm}

Next, we consider one of the $\mathfrak{sp}(4,\BB C)$-modules from
Subsection \ref{rho_M-subsection}, specifically the one obtained by differentiating
the action the group $Sp(4,\BB C)$ on the space of $\BB C$-valued functions on
$\HC^{\times} \text{-sym}$:
\[
\rho_1(h): \: f(Z) \: \mapsto \: \bigl( \rho_1(h)f \bigr)(Z) =
\frac{f \bigl( (aZ+b)(cZ+d)^{-1} \bigr)} {N(cZ+d) \cdot N(a'-Zc')},
\]
$h = \bigl( \begin{smallmatrix} a' & b' \\ c' & d' \end{smallmatrix} \bigr)
  \in Sp(4,\BB R) \subset GL(2,\HC)$ and
$h^{-1} = \bigl( \begin{smallmatrix} a & b \\ c & d \end{smallmatrix} \bigr)$.
(Thus, $(\alpha,\beta)=(1,1)$.)

We modify the trilinear form \eqref{3form-sc} by replacing the
$\mathfrak{gl}(2,\HC)$-modules $(\pi^0_r, {\cal H})$,
$(\rho_1, \Zh)$ and $(\pi^0_l, {\cal H})$ with the
$\mathfrak{sp}(4,\BB C)$-modules $(\pi^0_{3-dim}, {\cal H}_{3-dim}^{\pm})$,
$\bigl( \rho_1, N(Z)^{\frac12} \cdot \BB C [z_{11},z_{12},z_{22},N(Z)^{-1}] \bigr)$
and $(\pi^0_{3-dim}, {\cal H}_{3-dim}^{\pm})$
respectively. Thus, we consider multiplication maps
\begin{align}
(\pi^0_{3-dim}, {\cal H}^+_{3-dim}) \times
\bigl( \rho_1, N(Z)^{\frac12} \cdot \BB C [z_{11},z_{12},z_{22},N(Z)^{-1}] \bigr)
\times (\pi^0_{3-dim}, {\cal H}^+_{3-dim}) &\to (\rho_M, M),  \label{scalar-gFf-multiplication+}  \\
(\pi^0_{3-dim}, {\cal H}^-_{3-dim}) \times
\bigl( \rho_1, N(Z)^{\frac12} \cdot \BB C [z_{11},z_{12},z_{22},N(Z)^{-1}] \bigr)
\times (\pi^0_{3-dim}, {\cal H}^-_{3-dim}) &\to (\rho_M, M),  \label{scalar-gFf-multiplication-}
\end{align}
given by $(\phi_1, f, \phi_2) \mapsto \phi_1f\phi_2$.

\begin{prop}
The multiplication maps \eqref{scalar-gFf-multiplication+} and
\eqref{scalar-gFf-multiplication-} are $\mathfrak{sp}(4,\BB C)$-equivariant.
\end{prop}

Combining the multiplication maps \eqref{scalar-gFf-multiplication+} and
\eqref{scalar-gFf-multiplication-} with the projection \eqref{M-integral},
we obtain two $\mathfrak{sp}(4,\BB C)$-invariant trilinear forms
\begin{align*}
(\pi^0_{3-dim}, {\cal H}^+_{3-dim}) \times
\bigl( \rho_1, N(Z)^{\frac12} \cdot \BB C [z_{11},z_{12},z_{22},N(Z)^{-1}] \bigr)
\times (\pi^0_{3-dim}, {\cal H}^+_{3-dim}) &\to \BB C,  \\
(\pi^0_{3-dim}, {\cal H}^-_{3-dim}) \times
\bigl( \rho_1, N(Z)^{\frac12} \cdot \BB C [z_{11},z_{12},z_{22},N(Z)^{-1}] \bigr)
\times (\pi^0_{3-dim}, {\cal H}^-_{3-dim}) &\to \BB C,
\end{align*}
both given by the formula
\begin{equation}  \label{3form-scalar}
  (\phi_1, f, \phi_2) \mapsto
  \frac i{8\pi^2} \int_{\widetilde{\Gamma}} \phi_1(Z) \cdot f(Z) \cdot \phi_2(Z) \,dZ^3.
\end{equation}

Recall from Theorem \ref{rho-alpha-beta-thm} that
$\bigl( \rho_1, N(Z)^{\frac12} \cdot \BB C [z_{11},z_{12},z_{22},N(Z)^{-1}] \bigr)$
has invariant subspaces
\begin{align*}
&\BB C\text{-span of } \bigl\{ N(Z)^{k+\frac12} \cdot R^m_l(Z);\: k \ge -(l+1) \bigr\}, \\
&\BB C\text{-span of } \bigl\{ N(Z)^{k+\frac12} \cdot R^m_l(Z);\: k \le -2 \bigr\}.
\end{align*}

\begin{lem} 
The restriction of the trilinear form \eqref{3form-scalar} to
\begin{align*}
(\pi^0_{3-dim}, {\cal H}^+_{3-dim}) \times
\bigl( \rho_1,
\BB C\text{-span of } \bigl\{ N(Z)^{k+\frac12} \cdot R^m_l(Z);\: k \ge -(l+1) \bigr\} \bigr)
\times (\pi^0_{3-dim}, {\cal H}^+_{3-dim}) &\to \BB C,  \\
(\pi^0_{3-dim}, {\cal H}^-_{3-dim}) \times
\bigl( \rho_1,
\BB C\text{-span of } \bigl\{ N(Z)^{k+\frac12} \cdot R^m_l(Z);\: k \le -2 \bigr\} \bigr)
\times (\pi^0_{3-dim}, {\cal H}^-_{3-dim}) &\to \BB C
\end{align*}
is identically zero.
\end{lem}

The proof is the same as in the spinor case
(Lemma \ref{3form-main-restriction-lem}).

Therefore, the trilinear form \eqref{3form-scalar} descends to non-trivial
$\mathfrak{sp}(4,\BB C)$-invariant trilinear forms
\begin{align*}
(\pi^0_{3-dim}, {\cal H}^+_{3-dim}) \times
\Bigl( \rho_1, \tfrac{N(Z)^{\frac12} \cdot \BB C [z_{11},z_{12},z_{22},N(Z)^{-1}]}
{\BB C\text{-span of } \bigl\{ N(Z)^{k+\frac12} \cdot R^m_l(Z);\: k \ge -(l+1) \bigr\}} \Bigr)
\times (\pi^0_{3-dim}, {\cal H}^+_{3-dim}) &\to \BB C,  \\
(\pi^0_{3-dim}, {\cal H}^-_{3-dim}) \times
\Bigl( \rho_1, \tfrac{N(Z)^{\frac12} \cdot \BB C [z_{11},z_{12},z_{22},N(Z)^{-1}]}
{\BB C\text{-span of } \bigl\{ N(Z)^{k+\frac12} \cdot R^m_l(Z);\: k \le -2 \bigr\}} \Bigr)
\times (\pi^0_{3-dim}, {\cal H}^-_{3-dim}) &\to \BB C
\end{align*}
(cf. proof of Proposition \ref{3form-nonzero-prop}).

\begin{rem}
Similarly, one can construct $\mathfrak{sp}(4,\BB C)$-invariant trilinear forms
\[
(\pi^0_{3-dim}, {\cal H}^+_{3-dim}) \times
(\rho_1, \BB C [z_{11},z_{12},z_{22},N(Z)^{-1}]^{\pm}) \times
(\pi^0_{3-dim}, {\cal H}^-_{3-dim}) \to \BB C.
\]
\end{rem}

\subsection{An $\mathfrak{sp}(4,\BB C)$-Invariant Trilinear Form}

We modify the trilinear form \eqref{3form-easy} by replacing the
$\mathfrak{gl}(2,\HC)$-modules $(\pi'_r, {\cal U}')$,
$(\rho_2, {\cal W})$ and $(\pi'_l, {\cal U})$ with the
$\mathfrak{sp}(4,\BB C)$-modules $(\pi'_r, {\cal U}'_{res})$,
$(\rho^*_2, \mathring {\cal W}^*_{res})$ and $(\pi'_l, {\cal U}_{res})$
respectively. Thus, we consider a multiplication map
\begin{equation}  \label{quasi-gFf-res-multiplication}
(\pi'_r, {\cal U}'_{res}) \times (\rho^*_2, \mathring {\cal W}^*_{res}) \times
(\pi'_l, {\cal U}_{res}) \to (\rho_M, M), \qquad (g, F, f) \mapsto gFf.
\end{equation}
The argument used to show Proposition \ref{gFf-prop} also shows:

\begin{prop}
The multiplication map \eqref{quasi-gFf-res-multiplication} is
$\mathfrak{sp}(4,\BB C)$-equivariant.
\end{prop}

Combining the multiplication map \eqref{quasi-gFf-res-multiplication}
with the projection \eqref{M-integral}, we obtain an
$\mathfrak{sp}(4,\BB C)$-invariant trilinear form
\begin{equation}    \label{3form-quasi}
(\pi'_r, {\cal U}'_{res}) \times (\rho^*_2, \mathring {\cal W}^*_{res}) \times
(\pi'_l, {\cal U}_{res}) \to \BB C, \qquad
\end{equation}
\begin{equation}  \label{3form}
(g, F, f) \mapsto \frac i{8\pi^2} \int_{\widetilde{\Gamma}} g(Z) \cdot F(Z) \cdot f(Z) \,dZ^3.
\end{equation}

The $\mathfrak{sp}(4,\BB C)$-module
$(\rho^*_2, \mathring {\cal W}^*_{\frac12})$ has invariant subspaces:
\begin{align*}
\BB C\text{-span of } &\bigl\{
N(Z)^{k+\frac12} \cdot \leftidx{^*}{{\bf M}}^m_l(Z),\:
\leftidx{^*}{{\bf C}}^m_{k,l}(Z);\: k \ge -(l+1) \bigr\},  \\
\BB C\text{-span of } &\bigl\{
N(Z)^{k-\frac12} \cdot \leftidx{^*}{{\bf M}}^m_l(Z),\:
\leftidx{^*}{{\bf C}}^m_{k,l}(Z);\: k \le -1 \bigr\}.
\end{align*}

\begin{lem}  \label{3form-quasi-restriction-lem}
The restrictions of the trilinear form \eqref{3form-quasi} to
\begin{align*}
(\pi'_r, {\cal U}'^+_{res}) \times
\bigl( \rho^*_2, \BB C\text{-span of } \bigl\{
N(Z)^{k+\frac12} \cdot \leftidx{^*}{{\bf M}}^m_l(Z),\:
\leftidx{^*}{{\bf C}}^m_{k,l}(Z);\: k \ge -(l+1) \bigr\} \bigr)
\times (\pi'_l, {\cal U}^+_{res}) &\to \BB C,  \\
(\pi'_r, {\cal U}'^-_{res}) \times
\bigl( \rho^*_2, \BB C\text{-span of } \bigl\{
N(Z)^{k-\frac12} \cdot \leftidx{^*}{{\bf M}}^m_l(Z),\:
\leftidx{^*}{{\bf C}}^m_{k,l}(Z);\: k \le -1 \bigr\} \bigr)
\times (\pi'_l, {\cal U}^-_{res}) &\to \BB C
\end{align*}
are identically zero.
\end{lem}

\begin{proof}
Since the two restrictions are related by the inversion map, it is sufficient
to prove that the first restriction is zero.
Furthermore, since $(\pi'_r, {\cal U}'^+_{res})$ is irreducible,
it is sufficient to show that the trilinear form is zero for a particular generator
$g_0 \in {\cal U}'^+_{res}$, such as $g_0=(1,0)$.
Recall that the $K$-types of ${\cal U}^+_{res}$ are listed in
Theorem \ref{U-K-type-decomp-thm}.
Then the trilinear form \eqref{3form-quasi} produces integrals
of the kind
\begin{equation}  \label{zero-integral}
\int_{\widetilde{\Gamma}} (1,0) \cdot F(Z) \cdot f(Z) \,dZ^3,
\end{equation}
where $F(Z)$ is $N(Z)^{k+\frac12} \cdot \leftidx{^*}{{\bf M}}^m_l(Z)$
or $\leftidx{^*}{{\bf C}}^m_{k,l}(Z)$ with $k \ge -(l+1)$
and $f(Z)$ is $N(Z)^{k'} \cdot {\bf u^t_{l',m'}}(Z)$ or
$N(Z)^{k'} \cdot {\bf u^b_{l'+1,m'}}(Z)$ with $k, l \ge 0$.

Consider first the case
$F(Z) = N(Z)^{k+\frac12} \cdot \leftidx{^*}{{\bf M}}^m_l(Z)$.
By the orthogonality relations \eqref{R-orthogonality}, integral
\eqref{zero-integral} is zero unless $f(Z)$ is
$N(Z)^{k'} \cdot {\bf u^t_{l',m'}}(Z)$ or
$N(Z)^{k'} \cdot {\bf u^b_{l',m'}}(Z)$ with $l'=l$,
in which case the degree of homogeneity of the
integrand is $2l+2k+2k'+1 \ge -1$.
Since the degree cannot be $-3$, integral \eqref{zero-integral} is zero.

Now, consider the case
\[
F(Z) = \leftidx{^*}{{\bf C}}^m_{k,l}(Z)
= (2k+1)l N(Z)^{k-\frac12} \cdot \leftidx{^*}{{\bf B}}^m_{l+1}(Z)
+ 2(k+l+1)(l+1) N(Z)^{k+\frac12} \cdot \leftidx{^*}{{\bf T}}^m_{l-1}(Z).
\]
By the orthogonality relations \eqref{R-orthogonality},
\[
\iint_{Z \in S^2_r} \leftidx{^*}{{\bf T}}^m_l(Z)
\cdot {\bf u^b_{l',m'}}(Z) \,dS
= \iint_{Z \in S^2_r} \leftidx{^*}{{\bf B}}^m_l(Z)
\cdot {\bf u^t_{l',m'}}(Z) \,dS =0
\]
for all values of $l,l',m,m'$.
Hence, by the orthogonality relations \eqref{R-orthogonality}
again, integral \eqref{zero-integral} is zero unless $f(Z)$ is
$N(Z)^{k'} \cdot {\bf u^t_{l'-1,m'}}(Z)$ and $k \ge -l$ or
$N(Z)^{k'} \cdot {\bf u^b_{l'+1,m'}}(Z)$ with $l'=l$,
in both cases the degree of homogeneity of the integrand is at least $-1$
and cannot be $-3$. Hence integral \eqref{zero-integral} is zero.
\end{proof}

Therefore,  the trilinear form \eqref{3form-quasi} descends to
the quotients
\begin{align*}
(\pi'_r, {\cal U}'^+_{res}) \times
\Bigl( \rho^*_2, \tfrac{\mathring {\cal W}^*_{res}}
{\BB C\text{-span of } \bigl\{
N(Z)^{k+\frac12} \cdot \leftidx{^*}{{\bf M}}^m_l(Z),\:
\leftidx{^*}{{\bf C}}^m_{k,l}(Z);\: k \ge -(l+1) \bigr\}} \Bigr)
\times (\pi'_l, {\cal U}^+_{res}) &\to \BB C,  \\
(\pi'_r, {\cal U}'^-_{res}) \times
\Bigl( \rho^*_2, \tfrac{\mathring {\cal W}^*_{res}}
{\BB C\text{-span of } \bigl\{
N(Z)^{k-\frac12} \cdot \leftidx{^*}{{\bf M}}^m_l(Z),\:
\leftidx{^*}{{\bf C}}^m_{k,l}(Z);\: k \le -1 \bigr\}} \Bigr)
\times (\pi'_l, {\cal U}^-_{res}) &\to \BB C.
\end{align*}
By Theorem \ref{W*-decomp-thm}, the middle
$\mathfrak{sp}(4,\BB C)$-modules contain irreducible components
isomorphic to $(\pi_{dl},{\cal F}^-) \simeq (\pi_{dr},{\cal G}^-)$
and $(\pi_{dl},{\cal F}^+) \simeq (\pi_{dr},{\cal G}^+)$ respectively.
It is easy to see that these two invariant trilinear forms are  non-trivial
(cf. proof of Proposition \ref{3form-nonzero-prop}).

\subsection{Another $\mathfrak{sp}(4,\BB C)$-Invariant Trilinear Form}  \label{3form-main-subsect}

We modify the trilinear form \eqref{3form-obvious} by replacing the
$\mathfrak{gl}(2,\HC)$-modules $(\pi_r, {\cal V}')$,
$(\rho'_2, {\cal W}')$ and $(\pi_l, {\cal V})$ with the
$\mathfrak{sp}(4,\BB C)$-modules $(\pi^{3-dim}_r, {\cal V}'^{\pm}_{3-dim})$,
$(\rho'_2, \mathring {\cal W}'_{\frac12})$ and
$(\pi^{3-dim}_l, {\cal V}^{\pm}_{3-dim})$ respectively.
Thus, we consider multiplication maps
\begin{align}
(\pi^{3-dim}_r, {\cal V}'^+_{3-dim}) \times (\rho'_2, \mathring {\cal W}'_{\frac12})
\times (\pi^{3-dim}_l, {\cal V}^+_{3-dim}) &\to (\rho_M, M),
\label{gFf-res-multiplication+}  \\
(\pi^{3-dim}_r, {\cal V}'^-_{3-dim}) \times (\rho'_2, \mathring {\cal W}'_{\frac12})
\times (\pi^{3-dim}_l, {\cal V}^-_{3-dim}) &\to (\rho_M, M),
\label{gFf-res-multiplication-}
\end{align}
where both maps are given by the formula
\begin{equation*}
(g, F, f) \mapsto gFf.
\end{equation*}
The argument used to show Proposition \ref{gFf-prop} also shows:

\begin{prop}
The multiplication maps \eqref{gFf-res-multiplication+} and
\eqref{gFf-res-multiplication-} are $\mathfrak{sp}(4,\BB C)$-equivariant.
\end{prop}

Combining the multiplication maps
\eqref{gFf-res-multiplication+}-\eqref{gFf-res-multiplication-}
with the projection \eqref{M-integral}, we obtain two
$\mathfrak{sp}(4,\BB C)$-invariant trilinear forms

\begin{align}
(\pi^{3-dim}_r, {\cal V}'^+_{3-dim}) \times (\rho'_2, \mathring {\cal W}'_{\frac12})
\times (\pi^{3-dim}_l, {\cal V}^+_{3-dim}) &\to \BB C, \label{3form+}  \\
(\pi^{3-dim}_r, {\cal V}'^-_{3-dim}) \times (\rho'_2, \mathring {\cal W}'_{\frac12})
\times (\pi^{3-dim}_l, {\cal V}^-_{3-dim}) &\to \BB C, \label{3form-}
\end{align}
both given by the formula
\begin{equation}  \label{3form-main}
(g, F, f) \mapsto \frac i{8\pi^2}
\int_{\widetilde{\Gamma}} g(Z) \cdot F(Z) \cdot f(Z) \,dZ^3.
\end{equation}

The $\mathfrak{sp}(4,\BB C)$-module $(\rho'_2, \mathring {\cal W}'_{\frac12})$
has invariant subspaces:
\begin{align*}
&\BB C\text{-span of } \bigl\{ N(Z)^{k+\frac12} \cdot {\bf T}^m_l(Z), \:
  N(Z)^{k+\frac12} \cdot {\bf M}^m_l(Z), \: {\bf C}^m_{k+\frac12,l}(Z) ;\:
  k \ge -(l+1) \bigr\}, \\
&\BB C\text{-span of } \bigl\{ N(Z)^{k+\frac12} \cdot {\bf T}^m_l(Z), \:
  N(Z)^{k+\frac12} \cdot {\bf M}^m_l(Z), \: {\bf C}^m_{k+\frac32,l}(Z) ;\:
  k \le -1 \bigr\}.
\end{align*}

\begin{lem}  \label{3form-main-restriction-lem}
The restrictions of the trilinear forms \eqref{3form+}-\eqref{3form-} to
\begin{align*}
(\pi^{3-dim}_r, {\cal V}'^+_{3-dim}) \times
\biggl( \rho'_2,
\BB C\text{-span of } \biggl\{ \begin{smallmatrix}
  N(Z)^{k+\frac12} \cdot {\bf T}^m_l(Z), \:
  N(Z)^{k+\frac12} \cdot {\bf M}^m_l(Z), \\
  {\bf C}^m_{k+\frac12,l}(Z) \text{ with $k \ge -(l+1)$}
\end{smallmatrix} \biggr\} \biggr)
\times (\pi^{3-dim}_l, {\cal V}^+_{3-dim}) &\to \BB C,  \\
(\pi^{3-dim}_r, {\cal V}'^-_{3-dim}) \times
\biggl( \rho'_2,
\BB C\text{-span of } \biggl\{ \begin{smallmatrix}
  N(Z)^{k+\frac12} \cdot {\bf T}^m_l(Z), \:
  N(Z)^{k+\frac12} \cdot {\bf M}^m_l(Z), \\
  {\bf C}^m_{k+\frac32,l}(Z) \text{ with $k \le -1$}
\end{smallmatrix} \biggr\} \biggr)
\times (\pi^{3-dim}_l, {\cal V}^-_{3-dim}) &\to \BB C
\end{align*}
are identically zero.
\end{lem}

\begin{proof}
The proof is similar to that of Lemma \ref{3form-quasi-restriction-lem}.
Since the two restrictions are related by the inversion map, it is sufficient
to prove that the first restriction is zero.
Furthermore, since $(\pi_{r.3-dim}, {\cal V}'^+_{3-dim})$ is irreducible,
it is sufficient to show that the trilinear form is zero for a particular generator
$g_0 \in {\cal V}'^+_{3-dim}$, such as $g_0=(1,0)$.
Note that the middle $\mathfrak{sp}(4,\BB C)$-module is
\[
\BB C\text{-span of } \Bigl\{ N(Z)^{k+\frac12} \cdot R^m_l(Z) \cdot
\bigl(\begin{smallmatrix} a & b \\ b & d \end{smallmatrix}\bigr),\:
a,b,d \in \BB C,\: k \ge -(l+1) \Bigr\}.
\]
Then the trilinear form \eqref{3form+} produces integrals of the kind
\begin{equation}  \label{zero-integral2}
\int_{\widetilde{\Gamma}} (1,0) \cdot  N(Z)^{k+\frac12} \cdot R^m_l(Z) \cdot
\bigl(\begin{smallmatrix} a & b \\ b & d \end{smallmatrix}\bigr)
\cdot \Bigl( \begin{smallmatrix} (l'+m'+1) R^{m'}_{l'}(Z) \\
i R^{m'+1}_{l'}(Z) \end{smallmatrix} \Bigr)  \,dZ^3.
\end{equation}
By the orthogonality relations \eqref{R-orthogonality}, integral
\eqref{zero-integral2} is zero unless $l'=l$, in which case the degree
of homogeneity of the integrand is $2l+2k+1 \ge -1$.
Since the degree cannot be $-3$, integral \eqref{zero-integral2} is zero.
\end{proof}

Therefore,  the trilinear forms \eqref{3form+}-\eqref{3form-}
descend to the quotients
\begin{align*}
(\pi^{3-dim}_r, {\cal V}'^+_{3-dim}) \times
\left( \rho'_2, \tfrac{\mathring {\cal W}'_{\frac12}}
{\BB C\text{-span of } \biggl\{ \begin{smallmatrix}
  N(Z)^{k+\frac12} \cdot {\bf T}^m_l(Z), \:
  N(Z)^{k+\frac12} \cdot {\bf M}^m_l(Z), \\
  {\bf C}^m_{k+\frac12,l}(Z) \text{ with $k \ge -(l+1)$}
\end{smallmatrix} \biggr\}} \right)
\times (\pi^{3-dim}_l, {\cal V}^+_{3-dim}) &\to \BB C,  \\
(\pi^{3-dim}_r, {\cal V}'^-_{3-dim}) \times
\left( \rho'_2, \tfrac{\mathring {\cal W}'_{\frac12}}
{\BB C\text{-span of } \biggl\{ \begin{smallmatrix}
  N(Z)^{k+\frac12} \cdot {\bf T}^m_l(Z), \:
  N(Z)^{k+\frac12} \cdot {\bf M}^m_l(Z), \\
  {\bf C}^m_{k+\frac32,l}(Z) \text{ with $k \le -1$}
\end{smallmatrix} \biggr\}} \right)
\times (\pi^{3-dim}_l, {\cal V}^-_{3-dim}) &\to \BB C.
\end{align*}
By Theorem \ref{W-half-decomp-thm}, the middle
$\mathfrak{sp}(4,\BB C)$-modules contain irreducible components
isomorphic to $(\pi_{dl},{\cal F}^-) \simeq (\pi_{dr},{\cal G}^-)$
and $(\pi_{dl},{\cal F}^+) \simeq (\pi_{dr},{\cal G}^+)$ respectively.

\begin{prop}  \label{3form-nonzero-prop}
The above two trilinear forms are non-trivial.
\end{prop}

\begin{proof}
Since the two trilinear forms are related by the inversion map,
it is sufficient to prove that the first form is non-trivial.
We can take $g=(0,1)  \in {\cal V}'^+_{3-dim}$.
Consider elements of $\mathring {\cal W}'_{\frac12}$
\[
N(Z)^{-l-\frac32} \cdot {\bf M^m_l}(Z) =
N(Z)^{-l-\frac32} \cdot \begin{pmatrix} R^{m+1}_l(Z) & im R^m_l(Z) \\
  im R^m_l(Z) & (l+m)(l-m+1) R^{m-1}_l(Z)  \end{pmatrix},
\]
$-l \le m \le l$.
(We saw in Theorem \ref{W-half-decomp-thm} that these elements
generate the doubly regular functions inside $\mathring {\cal W}'_{\frac12}$.)
By \eqref{2-reg-gens}, the generators
$N(Z)^{-\frac52} \cdot {\bf M}^m_1(Z)$ with $l=1$ and $m=-1,0,1$
are respectively scalar multiples of
\begin{equation}  \label{main-thm-pf-F}
F_{-1} = \tfrac1{N(Z)^{\frac52}} \cdot \bigl(
\begin{smallmatrix} 2it & z_{22} \\ z_{22} & 0 \end{smallmatrix}
\bigr), \quad
F_0 = \tfrac1{N(Z)^{\frac52}} \cdot \bigl(
\begin{smallmatrix} z_{11} & 0 \\  0 & -z_{22} \end{smallmatrix}
\bigr), \quad
F_1 = \tfrac1{N(Z)^{\frac52}} \cdot \bigl(
\begin{smallmatrix} 0 & z_{11} \\ z_{11} & 2it \end{smallmatrix}
\bigr) \quad \in \mathring {\cal W}'_{\frac12}.
\end{equation}
Finally, choose
\[
{\bf v^t_{l,m}}(Z) =
\begin{pmatrix} (l+m+1) R^m_l(Z) \\ i R^{m+1}_l(Z) \end{pmatrix}
\qquad \in {\cal V}^+_{3-dim}
\]
with $l=1$, $m=-1,0,1$:
\begin{equation}  \label{main-thm-pf-g}
f_{-1} = \bigl( \begin{smallmatrix} z_{22} \\
  2it \end{smallmatrix} \bigr), \quad
f_0 = \bigl( \begin{smallmatrix} 2t \\
  -iz_{11} \end{smallmatrix} \bigr), \quad
f_1 = \bigl( \begin{smallmatrix} z_{11} \\
  0 \end{smallmatrix} \bigr) \quad
\in {\cal V}^+_{3-dim}.
\end{equation}
Then
\begin{align*}
g \cdot F_{-1} \cdot f_1 &=
N(Z)^{-\frac52} \cdot z_{11} z_{22}, \\
g \cdot F_0 \cdot f_0  &=
i N(Z)^{-\frac52} \cdot z_{11} z_{22}, \\
g \cdot F_1 \cdot f_{-1} &=
N(Z)^{-\frac52} \cdot (z_{11} z_{22} - 4t^2).
\end{align*}
Neither of the functions $z_{11} z_{22}$, $z_{11} z_{22} - 4t^2$
is harmonic on $\HC \text{-sym}$ (neither is annihilated by
$\Delta_3 = 4\frac{\partial}{\partial z_{11}} \frac{\partial}{\partial z_{22}}
+ \frac{\partial^2}{\partial t^2}$).
Hence they can be expressed in the form
\[
const \cdot N(Z) + \text{harmonic function}:
\]
\[
z_{11} z_{22} = \tfrac23 N(Z) + \tfrac13 (z_{11} z_{22} - 2t^2), \qquad
z_{11} z_{22} - 4t^2 = -\tfrac23 N(Z) + \tfrac53 (z_{11} z_{22} - 2t^2).
\]
Hence
\begin{align*}
g \cdot F_{-1} \cdot f_1  &= \tfrac23 N(Z)^{-\frac32}
+ N(Z)^{-\frac52} \cdot \text{harmonic function}, \\
g \cdot F_0 \cdot f_0 &= \tfrac{2i}3 N(Z)^{-\frac32}
+ N(Z)^{-\frac52} \cdot \text{harmonic function}, \\
g \cdot F_1 \cdot f_{-1} &= -\tfrac23  N(Z)^{-\frac32}
+ N(Z)^{-\frac52} \cdot \text{harmonic function},
\end{align*}
and the trilinear form \eqref{3form+} evaluated on the triples
\[
\langle g, F_{-1}, f_1 \rangle, \qquad \langle g, F_0, f_0 \rangle, \qquad
\langle g, F_1, f_{-1} \rangle
\]
is non-zero.
\end{proof}

\begin{rem}
Similarly, one can construct $\mathfrak{sp}(4,\BB C)$-invariant trilinear forms
\begin{align*}
(\pi^{3-dim}_r, {\cal V}'^+_{3-dim}) \times
(\rho'_2, \mathring {\cal W}'^{\pm}_{res} )
\times (\pi^{3-dim}_l, {\cal V}^-_{3-dim}) &\to \BB C,  \\
(\pi^{3-dim}_r, {\cal V}'^-_{3-dim}) \times
(\rho'_2, \mathring {\cal W}'^{\pm}_{res} )
\times (\pi^{3-dim}_l, {\cal V}^+_{3-dim}) &\to \BB C.
\end{align*}
\end{rem}

\separate

\noindent
{\em Department of Mathematics, Yale University,
P.O. Box 208283, New Haven, CT 06520-8283}\\
{\em Department of Mathematics, Indiana University,
Rawles Hall, 831 East 3rd St, Bloomington, IN 47405}   

\end{document}